\documentclass[12pt,reqno]{amsart}
\usepackage{amsfonts,amsrefs,latexsym,amsmath, amssymb, mathrsfs, verbatim}
\usepackage{url}
\usepackage[dvipsnames,usenames]{color}
\usepackage{pifont}
\usepackage{upgreek}
\usepackage{hyperref}
\usepackage{times}
\usepackage{calligra}
\usepackage{marvosym,rotating}
\usepackage{graphicx}
\usepackage[percent]{overpic}
\usepackage{caption}
\usepackage{fancyhdr}

\newtheorem{theorem}{Theorem}
\newtheorem{proposition}{Proposition}[section]
\newtheorem{lemma}[proposition]{Lemma}
\newtheorem{corollary}[proposition]{Corollary}

\theoremstyle{definition}
\newtheorem{definition}{Definition}[section]
\newtheorem{remark}{Remark}[section]

\DeclareMathAlphabet{\mathcalligra}{T1}{calligra}{m}{n}
\DeclareFontShape{T1}{calligra}{m}{n}{<->s*[2.2]callig15}{}

\newcommand{\rgeo}{\varrho}

\newcommand{\myexp}{e}

\newcommand{\mytr}{{\mbox{\upshape{tr}}_{\mkern-2mu \gsphere}}}

\newcommand{\gsphere}{g \mkern-8.5mu / }
\newcommand{\msphere}{m \mkern-8.5mu / }

\newcommand{\ginversesphere}{\gsphere^{-1}}
\newcommand{\minversesphere}{\msphere^{-1}}

\newcommand{\Euct}{e}

\newcommand{\vol}{\varpi}
\newcommand{\tvol}{\underline{\varpi}}
\newcommand{\conevol}{\overline{\varpi}}

\newcommand{\Eucspherevol}{\upsilon_{e \mkern-8.5mu /}}

\newcommand{\Fried}{\digamma}

\newcommand{\FutFailFac}{ {^{(+)} \mkern-1mu \aleph} }
\newcommand{\PastFailFac}{ {^{(-)} \mkern-1mu \aleph} }
\newcommand{\InitialFutFailFac}{ {^{(+)} \mkern-1mu \mathring{\aleph}} }

\newcommand{\D}{\mathscr{D}}

\newcommand{\angD}{ {\nabla \mkern-14mu / \,} }

\newcommand{\angLap}{ {\Delta \mkern-12mu / \, } }

\newcommand{\angLie}{ { \mathcal{L} \mkern-10mu / } }

\newcommand{\Lie}{\mathcal{L}}

\newcommand{\angG}{ {{G \mkern-12mu /} \, }}

\newcommand{\deform}[1]{{^{(#1)} \mkern-1mu \pi}}

\newcommand{\Lgeo}{L_{(Geo)}}
\newcommand{\Lunit}{L}
\newcommand{\uLgood}{\breve{\underline{L}}}
\newcommand{\uLunit}{\underline{L}}
\newcommand{\Rad}{\breve{R}}
\newcommand{\Radunit}{R}

\newcommand{\Timenormal}{N}

\newcommand{\Mult}{T}
\newcommand{\Mor}{\widetilde{K}}

\newcommand{\Rot}{O}

\newcommand{\Jenergycurrent}[1]{^{(#1)} \mkern-3mu J}

\newcommand{\enmomtensor}{Q}

\newcommand{\enzero}{\mathbb{E}}
\newcommand{\enone}{\widetilde{\mathbb{E}}}
\newcommand{\flzero}{\mathbb{F}}
\newcommand{\flone}{\widetilde{\mathbb{F}}}

\newcommand{\Morint}{\widetilde{\mathbb{K}}}

\newcommand{\chifullmodarg}[1]{{^{(#1)} \mkern-4mu \mathscr{X}}}

\newcommand{\waveinhom}{\mathfrak{F}}

\newcommand{\Conone}{A}
\newcommand{\Cononestar}{A_*}

\newcommand{\Contwo}{B}

\numberwithin{equation}{subsection}

\newcommand{\JX}{\Jenergycurrent{X}}

\newcommand{\bea}{\begin{eqnarray}}
\newcommand{\eea}{\end{eqnarray}}
\def\beaa{\begin{eqnarray*}}
\def\eeaa{\end{eqnarray*}}

\def\ba{\begin{array}}
\def\ea{\end{array}}
\def\be#1{\begin{equation} \label{#1}}
\def \eeq{\end{equation}}

\newcommand{\nn}{\nonumber}

\def\a{\alpha}
\def\b{\beta}
\def\ga{\gamma}
\def\Ga{\Gamma}
\def\de{\delta}

\def\ep{\epsilon}
\def\la{\lambda}

\def\si{\sigma}
\def\Si{\Sigma}

\def\nab{\nabla}

\def\Lb{{\underline{L}}}
\def\ub{\underline{u}} 
\def\Timenormal{{N}}

\def\AA{{\mathcal A}}

\def\MM{{\mathcal M}}
\def\NN{{\mathcal N}}
\def\II{{\mathcal I}}

\def\Lie{{\mathcal L}}

\def\lap{\Delta}
\def\pr{\partial}

\def\Gapsi{\,^{(\Psi)}  \mkern-2mu \Ga}
\def\c{\cdot}
\def\err{\mbox{\upshape Err}}

\def\RRR{{\mathbb R}}

\def\f14{{\frac{1}{4}}}
\def\f12{{\frac{1}{2}}}

\newcommand{\piX}{\deform{X}}

\newcommand{\piZ}{\deform{Z}}
\newcommand{\piO}{\deform{O}}

\newcommand{\piMor}{\deform{\Mor}}

\def\Lunit{L}
\def\trch{\mytr\upchi}

\begin{document}

\title{Introduction}
\author[GH]{Gustav Holzegel$^{\ddagger}$}
\author[SK]{Sergiu Klainerman$^{**}$}
\author[JS]{Jared Speck$^*$}
\author[WW]{Willie Wai-Yeung Wong$^{\dagger}$}

\thanks{$^{*}$Massachusetts Institute of Technology, Department of Mathematics, 77 Massachusetts Ave, Room E18-328, Cambridge, MA 02139-4307, USA. \texttt{jspeck@math.mit.edu}}

\thanks{$^{\dagger}$\'Ecole Polytechnique F\'ed\'erale de Lausanne,
Section de Math\'ematiques, Station 8, CH-1015, Lausanne, Switzerland.
\texttt{willie.wong@epfl.ch}}

\thanks{$^{\ddagger}$Imperial College London, Department of Mathematics, South Kensington Campus, London, SW7 2AZ, United Kingdom. \texttt{g.holzegel@imperial.ac.uk}}

\thanks{$^{**} $Princeton University, Department of Mathematics, 
Fine Hall, Washington Road,
Princeton NJ, 08544-1000, USA.
\texttt{seri@math.princeton.edu}}

\title[Small-Data Shock Formation in $3D$ Wave Equations]{Shock Formation in Small-Data Solutions to $3D$ Quasilinear Wave Equations: An Overview}
\maketitle

\centerline{\today}

\begin{abstract}
	In his 2007 monograph, 
	D. Christodoulou proved a remarkable result giving a detailed 
	description of shock formation, for small $H^s$-initial conditions ($s$ sufficiently large),  
	in solutions to the relativistic Euler equations in three space dimensions. His work provided a
	significant advancement over a large body of prior work concerning the long-time behavior of solutions to higher-dimensional
	quasilinear wave equations, initiated by F. John in the mid 1970's and continued 
	by            
	   S. Klainerman,  
	   T. Sideris,  
	   L. H\"ormander,   
	   H. Lindblad,  
	   S. Alinhac, and others.    
	   Our goal in this 
	   paper is to give an overview 	   
	   of his result, outline its main new ideas,    
	   and place it    
	   in the context of the above mentioned  
	      earlier work. We also introduce      
	      the recent work of J. Speck, 
	      which extends Christodoulou's result to show that  
	      for two important classes of quasilinear wave equations in three space dimensions,
	  		small-data shock formation 
				occurs precisely when the quadratic nonlinear terms fail the classic null condition.
\bigskip

\bigskip

\end{abstract}

\tableofcontents

\section{Introduction}
	\label{S:INTRO}
	\subsection{Motivation and background}
	This project was motivated by our desire to understand the work of Christodoulou 
  \cite{dC2007} concerning the formation of shocks in compressible, irrotational, relativistic,\footnote{The results were
	later extended to apply to the non-relativistic Euler equations in \cite{dCsM2012}.}   
	$3D$ fluids, starting from small, smooth initial conditions.   
	His work is a landmark    
	result in the venerable area of PDE known as   
	\textit{Systems of Nonlinear Hyperbolic Conservation Laws}. 
	This field originated from considerations concerning the propagation of one-dimensional sound waves through air,   
	by Monge, Poisson,
  Stokes, Riemann, Rankine, and Hougoniot  
	and was transformed into a systematic theory by Courant \cite{rCdH1948}, Friedrichs \cite{rCdH1948, kF1954}, Lax \cite{pL1957},
  Glimm \cite{jG1965}, Bressan \cite{aB1992}, and many others.

  \subsubsection{Singularities are an unavoidable aspect of the $1D$ theory}
  \label{SSS:SINGULARITIESIN1D}
  The crucial fact, already well understood by Riemann and Stokes,        
	which the theory had to deal with from its beginnings, was the
    observation that solutions to the equations develop singularities, 
    even when the data are small and smooth. 
    This fact is easy to exhibit in one space dimension and is well-captured by Burgers' equation:
     \begin{align}
     \label{Burger}
     \pr_t \Psi + \Psi \pr_x \Psi &= 0,
     	\\
     \Psi(0,x) & = \mathring{\Psi}(x).
     \end{align}
     In view of the equation, $\Psi$ must be constant along the characteristic curves $x(t,\a),$ 
     which are, in this case, solutions to the ODE initial value problems
      \begin{align}\label{eq:BurgerCharacterCoord}
      \frac{\partial}{\partial t} x(t,\a) & =\Psi(t,x(t,\a)),
				& & x(0,\a)= \a \in \RRR.
      \end{align}
      Thus, $\Psi\big(t,x(t,\a)\big) = \mathring{\Psi}(\a),$
      $\frac{\partial}{\partial t} x(t,\a) = \mathring{\Psi}(\a),$ 
      and $\frac{\partial^2}{\partial t \partial \a} x(t,\a) = \mathring{\Psi}'(\a).$ 
      Hence, we have $\frac{\partial}{\partial \a} x(t,\a) = 1 + t \mathring{\Psi}'(\a).$ In particular,  
			$\frac{\partial}{\partial \a} x(t,\a) =0$ when $1 + t \mathring{\Psi}'(\a)=0.$  
			It follows that a singularity must form in any solution launched by nontrivial, 
			smooth, compactly supported initial data 
			$\mathring{\Psi}.$ An alternative way to see the blow-up is 
       to differentiate Burgers' equation   
				in $x$ and derive the equation
      \beaa
     \pr_t(\pr_x \Psi) + \Psi \pr_x(\pr_x \Psi)=-(\pr_x \Psi)^2,
     \eeaa 
     which is the well-known Riccati equation 
		 \begin{align} \label{E:BURGERSRICCATI}
     \frac{d}{dt} y & =-y^2,  &&  \mbox{for} \,\, y(t):= \pr_x \Psi(t,x(t,\a)).
     \end{align}
       Note that the $L^\infty$ norm of $\Psi$ 
       is conserved. That is, the blow-up occurs in $\partial_x \Psi,$  
			while $\Psi$ itself remains bounded.
       It is also easy to check that the time of  
       blow up is no later than\footnote{Throughout the article, we
       sometimes write $A=\mathcal{O}(B)$ and equivalently $A \lesssim
       B$ to indicate $A \leq CB$ by some universal constant $C$;  
       see Footnote \ref{F:LESS} on pg.\ \pageref{F:LESS} for details.} 
       $\mathcal{O}(1/\ep),$ 
       where $\ep = - \min_{\alpha \in \mathbb{R}}
       \mathring{\Psi}'(\alpha)$ measures the smallness of the initial
       data. 
       
       Though a bit more difficult to prove, the same blow-up results hold true
       for general classes of systems of quasilinear  conservation laws (genuinely 
       nonlinear, strictly hyperbolic) in one space dimension.        
	The first results for $2\times 2$ strictly hyperbolic    
	systems verifying the genuine nonlinearity condition are due to O. Oleinik \cite{oO1957}  
	and P. Lax \cite{pL1957}. 
	The results were later extended to general such systems by F. John \cite{fJ1974}.              
	In \cite{sKaM1980}, Klainerman and Majda showed that the genuine nonlinearity condition can be relaxed in    
	  the case of $1D$ nonlinear vibrating string equations.   
         
         The great achievement of the  $1D$  theory of systems  of conservation laws
         was the understanding of how shocks form and how solutions can be extended
	 through shock singularities. This entails a complete
	 description of the shock boundary, as well as a formulation
	 of the equations capable of accommodating such singular
	 solutions. Such machinery is available
	 for general classes of
   hyperbolic conservation laws (mainly strictly  hyperbolic)   
   in one space dimension with general initial data             
   of small bounded variation; see, 
	 for example, \cite{mS2011} and \cite{cD2010}.  
          
     \subsubsection{Present-day limitations of the theory} 
     A primary goal of the field of conservation laws is replicating
	   the $1D$ success in higher space dimensions, which
	   entails understanding the mechanism of singularity
	   formation as well as how to define generalized solutions
	   extending past sufficiently mild singularities. In higher
	   dimensions, one faces several difficulties including 
	   that of finding a suitable definition of generalized solutions
	   and corresponding function spaces.

	   In $1D,$ the continuation of solutions to
	   conservations laws past shock fronts is comfortably
  	 achieved, in most cases, by considering the equations in a weak formulation
     for functions with finite spatial bounded variation (BV)
	   norm.\footnote{The theory of smooth solutions for $1D$
	   hyperbolic equations can be easily developed, starting with
	   Monge \cite{gM1850}, using the method of characteristics,
	   in any $L^p$ norm (in particular $L^\infty$ and $L^1$; the
	   latter being consistent with the BV norm).}
           In higher dimensions, however, the BV norm is
	   incompatible with the simple phenomenon of focusing  
	   of perfectly smooth waves, as can be seen for spherically symmetric  solutions
	   to the standard wave equation    
	   in $\RRR^{1+n},$ for any $n \geq 2;$ see \cite{jR1986}.
           Instead, the general theory of local well-posedness
	   for systems of hyperbolic conservation laws in higher
           dimensions is intimately tied to $L^2$-based $H^s$
	   Sobolev spaces.\footnote{The number of derivatives
	   required, $s$, depends on the number of space dimensions
	   and the strength of the nonlinearity.} This theory has
	   largely been developed using the framework of
	   Friedrichs'  symmetric hyperbolic systems \cite{kF1954},
	   and with further contributions by many others such as
           Sobolev, Schauder, Frankl, and Leray. The theory, however, 
           is quite far from accommodating discontinuous shock fronts\footnote{The theory can, however, be adapted  (within the $H^s$  framework !)   to treat, for a short time, one  single   shock wave, starting with an initial discontinuity across an admissible  regular 
                  hypersurface  in  higher dimensions;  see \cite{aM1981, aM1983a, aM1983b}.  }  and their interactions. 

	   In addition to the problem of defining generalized solutions, one
	   also encounters difficulties with understanding the
	   mechanism of singularity formation. A subtle point is that
	   in higher dimensions there can, in principle, be
	   singularities which differ from shocks in that they do not form from compression.
	   For example, a current venue of investigation
	   is the possibility of vorticity blow-up (the possible mechanism driving it remains an enigma) 
	   for the $3D$ compressible Euler equations of fluid dynamics.
	   Furthermore, the phenomenon of dispersion, typical to
	   higher dimensions, may delay or in some cases altogether
	   prevent the formation of singularities for small initial
	   data.

   \subsubsection{Quasilinear systems  of  wave equations}    
          An obvious way to separate the phenomena of compression  
          and dispersion from the effects of vorticity, in the case of the
          compressible Euler equations (relativistic or non-relativistic), is to restrict oneself to irrotational flows.  
          For such flows, the Euler equations reduce to a quasilinear wave equation
          for $\Phi,$ the fluid potential.\footnote{Roughly, the gradient of $\Phi$ is equal to 
          a rescaled version of the fluid velocity.}
          Since the irrotational Euler equations are derivable from a
          Lagrangian $\mathcal{L}(\partial \Phi),$
          the wave equation can be written in the following Euler-Lagrange form
          relative to standard rectangular coordinates:\footnote{We use Einstein's summation convention throughout.
          Lowercase Greek ``spacetime'' indices vary over $0,1,2,3$
	  and lowercase Latin ``spatial'' indices vary over $1,2,3.$}
          \begin{align} \label{modeleq:nongeo}
				 		\partial_{\alpha} 
				 		\left\lbrace
				 			\frac{\partial \mathcal{L}(\partial \Phi)}{\partial (\partial_{\alpha} \Phi)}
				 		\right\rbrace = 0.
					\end{align}
					We explain the connection between equation \eqref{modeleq:nongeo} 
					and (special) relativistic
					fluid mechanics in more detail
					in Subsect.~\ref{SS:CHRISTODOULOURESULTS}
					and Appendix~\ref{A:CHRISTODOULOUSEQUATIONS}.
          When expanded relative to rectangular coordinates, equation
          \eqref{modeleq:nongeo} takes the form
          \begin{align}  \label{modeleq:nongeoEXPANDED}
          	(g^{-1})^{\alpha\beta}(\partial\Phi)\partial_{\alpha} \partial_{\beta} \Phi & = 0.
          \end{align}
          
            \begin{remark}
              	Throughout this article, $\partial f = (\partial_t f, \partial_1 f, \partial_2 f, \partial_3 f)$ 
              	denotes the gradient of $f$ relative to the rectangular spacetime coordinates.
             \end{remark}
          
	  For physical choices of the Lagrangian
$\mathcal{L}(\partial\Phi),$ \eqref{modeleq:nongeoEXPANDED} is a wave
          equation: $(g^{-1})^{\a\b}(\cdot)$ 
         	is a non-degenerate symmetric quadratic form       
    			of signature $(-,+,+,+)$ depending smoothly on 
    			$\partial \Phi.$
	  We can always find an affine change of coordinates on $\mathbb{R}^{1+n}$ to obtain the relationship
    	\begin{equation} \label{E:GINVERSEISMINKOWSKIFORPHIEQUALS0}
    		(g^{-1})^{\a\b}(\partial \Phi = 0)
    		= (m^{-1})^{\a\b},
     \end{equation}
    where $(m^{-1})^{\a\b} = \mbox{diag}(-1,1,1,1)$
    is the standard inverse Minkowski metric; we will assume
		henceforth such a coordinate change has been made. 
		In \cite{dC2007}, Christodoulou studied a particular class of scalar equations of type 
     \eqref{modeleq:nongeo} that arise in irrotational relativistic fluid mechanics.\footnote{More precisely, 
		as we explain in Subsect.~\ref{SS:CHRISTODOULOURESULTS} and Appendix \ref{A:CHRISTODOULOUSEQUATIONS},
			the solutions considered in \cite{dC2007} differ from 
			solutions to equations of the form
			\eqref{modeleq:nongeo} by choices of normalizations.}
			Most of the results that we discuss in this introduction,
especially those concerning almost global
existence, can be extended\footnote{The shock formation results seem to be less stable under modifications of the equation; 
see Remark \ref{R:STRONGNULL}.} 
 to apply to the more general class of equations
      \begin{align}\label{modeleq:nongeo1}
				(g^{-1})^{\alpha\beta}(\partial\Phi) \partial_{\alpha}\partial_{\beta}\Phi 
				& = \NN(\Phi, \pr \Phi), 
			\end{align}
where $(g^{-1})^{\a \b}$ verifies \eqref{E:GINVERSEISMINKOWSKIFORPHIEQUALS0} and
$\NN$ is smooth in $(\Phi,\pr \Phi)$ and is quadratic or higher-order in $\pr\Phi$ for small  
     $(\Phi,\pr \Phi);$ that is, $\NN = \mathcal{O}(|\partial \Phi|^2)$ for small $(\Phi,\pr \Phi).$  

     At the beginning of the $20^{th}$ century, nonlinear wave equations made another dramatic appearance
     in General Relativity. Relative to the wave coordinates\footnote{The coordinate functions themselves verify 
    the covariant wave equation $\square_g x^\a=0.$ } $x^\a,$         
    the Einstein vacuum equations $\mathcal{R}_{\mu \nu} = 0$ (where
    $\mathcal{R}_{\mu \nu}$ is the Ricci curvature of the dynamic
    Lorentzian metric $g$) can be cast as a system of quasilinear wave
	equations in the components of $g$, in the form
               \begin{align}
                         \label{wave-Einst-vac}  
              (g^{-1})^{\a\b}\pr_\a\pr_\b  g_{\mu\nu} 
							& =\NN(g)(\pr g, \pr g),
							&& (\mu, \nu = 0,1,2,3),
              \end{align}
              where $\NN(g)(\pr g, \pr g)$ depends quadratically on $\pr g,$
              that is, on all spacetime derivatives of $g.$
              
          		 The above considerations have led to the study of general systems of nonlinear wave equations of the form
               \begin{align}
               \label{general-system}
              (g^{-1})^{\a\b}(\Psi)\pr_\a\pr_\b  \Psi^I & =\NN^I(\Psi,\pr \Psi), && (I=1,\ldots, K),
              \end{align}
              where $g$ is a smooth Lorentzian metric depending on the array
							$\Psi= \lbrace \Psi^I \rbrace_{I=1,\ldots,K}$  
              and $\NN$ is smooth in $(\Psi, \pr \Psi),$  
              at least quadratic in $\pr \Psi$ near  
              $(\Psi, \pr \Psi)=(0,0).$
              Note that \eqref{general-system} contains equations of type \eqref{modeleq:nongeoEXPANDED}
              by simply differentiating the latter and taking $\Psi=\pr\Phi.$   
              Note also that more general systems of the form              
              \begin{align}
               \label{general-systemII}
              (g^{-1})^{\a\b}(\Psi,\pr\Psi)\pr_\a\pr_\b  \Psi^I & = \NN^I(\Psi,\pr \Psi), 
              && (I=1,\ldots, K)
              \end{align}
              can, by differentiation, also be transformed into systems of type \eqref{general-system}. 
               Thus, the systems of the form \eqref{general-system} encompass the equations which arise in the irrotational compressible   
							Euler equations, both relativistic and non-relativistic under all physically reasonable equations of state,  
							and the Einstein vacuum equations \eqref{wave-Einst-vac} relative to wave coordinates.
		Furthermore, while the equations of nonlinear elasticity  
		do not fit\footnote{The general form of the equations of elasticity
		can, upon differentiation, be expressed as a
		generalization of equation \eqref{general-systemII} of the form
		$(g^{-1})_{IJ}^{\a\b}(\Psi)\pr_\a\pr_\b  \Psi^J = \NN^I(\Psi,\pr \Psi)$
		(with summation over $J$).
		Such equations give rise to more complicated geometries. In particular the principal part 
		is no longer the geometric wave operator of a
		Lorentzian manifold.}
		into the form of \eqref{general-system}, 
							the important special case of   
							homogeneous and isotropic materials can nevertheless be                
							reduced, by a simple separation between longitudinal and transversal waves, 
							to the same  framework; see John's work \cite{fJ1984}.

       \subsubsection{Results in $3D$ prior to Christodoulou's work}               
       \label{SSS:PRIORCHRISTODOULOU}
 In light of what we know in the one-dimensional case,  
 it makes sense to ask whether the mechanism of shock formation    
 remains the same in higher dimensions.  
 At first glance, we may expect a positive answer simply by observing that 
 plane wave solutions are effectively one dimensional. However, plane waves are non-generic and have infinite energy. 
The latter flaw can be ameliorated within the past domain of
dependence $\II^{-}(p)$ of an earliest singular point of the plane wave: 
we can simply cut-off the plane wave data    
outside of the intersection of $\II^{-}(p)$ with the initial Cauchy hypersurface $\lbrace t=0 \rbrace$  
to construct compactly supported initial data  
that lead to a shock singularity at $p.$ However, 
one can show that the cut-off data have large energy
and thus do not fit into the framework of small perturbations of the trivial state.                        
It turns out, in fact, that the large-time behavior of data of small size\footnote{The relevant energy norm depends on a finite number of derivatives.} 
$\mathring{\upepsilon},$      
in higher dimensions, is radically different from $1D.$    
This fact was first pointed out by F. John: in \cite{fJ1976a,fJ1976b}
he showed for quasilinear wave equations of type \eqref{modeleq:nongeoEXPANDED} 
that the dispersion of waves significantly delays the formation of singularities when the data are small.

		Starting with John's observation,
     Klainerman \cite{sK1980} was able to show, for a class of equations including those of form \eqref{modeleq:nongeo1}
     with $\NN$ independent of $\Phi,$   
			that the phenomenon of dispersion is sufficiently strong, in space dimensions 
			greater than or equal to $6,$  
     to completely avoid the formation of shocks for small initial data.      
     John and Klainerman \cite{sK1983, fJsK1984} 
     were later able to show the \emph{almost global existence}
result\footnote{The proof in \cite{fJsK1984} used some mild assumptions
	on the structure of the equation \eqref{modeleq:nongeo1}. Suitable assumptions are
	that the nonlinearity is independent of $\Phi$ itself 
	or that the equation can be written in divergence form up to cubic errors.}  
     that in $3$ space dimensions, if the data and a certain number
     their derivatives are of small size $\mathring{\upepsilon}$ 
     the singularities cannot form before time
			$\mathcal{O}(\exp(c \mathring{\upepsilon}^{-1})),$ which is
		significantly larger than the time $\mathcal{O}(\mathring{\upepsilon}^{-1})$ 
		in dimension $1.$ The result was significantly simplified and extended 
    in \cite{sK1985} using the geometric vectorfield method.    
    See also Theorem~\ref{T:JOHNHORMANDERLIFESPANLOWER} below for a sharp
		version of this result, due independently to John and H{\"o}rmander.
		Moreover, Klainerman \cite{sK1984, sK1986} 
		was later able to identify a structural condition on the form   
		of the quadratic terms in \eqref{modeleq:nongeo1}, 
		called the (classic) \textit{null condition}
		(see Definition \ref{D:CLASSICNULL}), which prevents the formation of  singularities when the data are sufficiently small
		in $3$ space dimensions. Two distinct proofs of the  
		result were given, one by Klainerman \cite{sK1986} based on the vectorfield method, 
		and the second by Christodoulou \cite{dC1986a} based on the 
		conformal method.\footnote{One should remark that the
geometric vectorfield method also yields a direct extension of
Klainerman's global existence result \cite{sK1980} to dimensions $4$
and $5,$ while for dimension $2$ (where the dispersion is even weaker than
in $3D$), versions of the null condition have been identified by Alinhac
\cite{sA2001a,sA2001b}.}

  		In the opposite direction,  
     	F. John gave \cite{fJ1981}
      a class of examples in $3D$
      of the form\footnote{Here and throughout, $\square_m  = - \partial_t^2 + \Delta$ is the standard flat d'Alembertian   
			in $\RRR^{1+n},$ where $n=3$ at present, and more generally $n$ will be clear from context. 
			John's class includes equations such as $\square_m \Phi= -(\pr_t\Phi)^2$ and   
			$\square_m \Phi = -\pr_t\Phi \pr_t^2\Phi$; his proof crucially uses the sign of the nonlinearity.} 
 		\beaa   
    \square_m \Phi= \NN(\pr \Phi,\pr^2\Phi),
    \eeaa
    where $\NN$ is quadratic in its arguments, 
    the classic null condition fails,
		and such that \emph{all} nontrivial compactly supported data
		lead to finite-time breakdown.
		Note that John's results are consistent with the almost global existence result of \cite{fJsK1984}.   
		Unlike, however, the one-dimensional argument that
    tracks solutions
    all the way to the first singularity,
    John's argument shows only that the existence of a global $C^3$
		solution would lead to a contradiction.\footnote{For John's
quasilinear equations, any rigorous proof of \emph{shock formation} would have to establish  
		the precise mechanism for the blow-up of the second derivatives of $\Phi$
    while also showing that the first derivatives remain bounded.} 
     T. Sideris \cite{tS1984} later proved a related but distinct result 
				showing that small initial data for the full compressible Euler equations in $3D,$
				under some adiabatic equations of state verifying a convexity assumption,
				also lead to  finite time break-down.
				Sideris' proof 
				was based on virial  inequalities and
thus provided an explicit upper bound on the solution's lifetime
				for small data verifying an open condition. Later, Guo and Tahvildar-Zadeh gave a similar proof of breakdown
				in solutions to the relativistic Euler equations in Minkowski spacetime
				\cite{yGsTZ1998},
				but their proof required the assumption of large data.
				We should mention here that the work \cite{dCsM2012}
				(following  \cite{dC2007}) shows  that
		the convexity assumptions used by Sideris
				are not necessary and that the first  singularity is in fact  caused by shock formation.

          Though for small initial data in $3D,$ 
					the arguments of John and Sideris  
          complement the global existence results that hold when the nonlinearities verify the null condition,  
					they fail to provide a satisfactory answer about   
					the nature of the
singularities, an understanding of which is clearly essential 
if one hopes to continue the
          solutions beyond them, as can be done in $1D$.
                 The first results in this direction
                are once more due to F. John \cite{fJ1981},  who analyzed 
	spherically symmetric solutions of  the model equation
                  \begin{align}
			\square_m \Phi & =- a^2(\pr_t \Phi )\cdot
\lap  \Phi, &a(0) &= 0, &a'(0) &\neq 0       \label{John-model}
		\end{align}
	in $3D$, where the final condition in \eqref{John-model} guarantees the \emph{failure} of the
classic null condition of Definition \ref{D:CLASSICNULL}. 
John's work showed that solutions corresponding to all sufficiently small nontrivial spherically symmetric data of compact support
necessarily have some second-order derivatives blowing-up near the
wave front, while all first derivatives remain
bounded.\footnote{John's analysis is restricted to a neighborhood of
the wave front (the ``wave zone''), where one expects (due to
dispersion; see next subsection) the first singularity to form. He does not provide information about the entire maximal development 
of the data. See Figure \ref{F:SHOCKFORMATION} and the discussion below.} 
His proof, based on the method of characteristics, makes   
essential use of the fact that spherically symmetric solutions
			of the equation verify a simplified  equation which is effectively
			one-dimensional. Although the passage from spherical symmetry
			to the general case is difficult, we nevertheless shall see
			that \emph{this simplified case provides the right 
			intuition about the behavior of general small-data shock-forming solutions.}

          In the last years of his life, F. John himself tried hard to extend his results to the general case. 
         Although he came close \cite{fJ1989, fJ1990}, he never was able to follow the solution all the way to the singularity. 
	The first results proving shock formation without symmetry assumptions
	are due to Alinhac; see \cite{sA1999a,sA1999b,sA2001b,sA2002}.  
			His results were highly motivated by John's earlier work \cite{fJ1987}
			(see also H{\"o}rmander's work \cite{lH1987}) which provided a lower bound on the 
				solution's lifespan that, as we take
				the size of the initial data to zero, 
				converges to Alinhac's blow-up time
	(see Theorem~\ref{T:JOHNHORMANDERLIFESPANLOWER} and the right-hand side of \eqref{E:ALINHACLIMITINGLIFESPAN}).
                Alinhac's results provided a major advance in our understanding of blow-up away from spherical symmetry. 
                However, they have some limitations. For example, 
                his proof works only for data that lead to a unique first blow-up point.
Hence, for equations invariant under the Euclidean rotations,
his results do not apply to some data containing a spherically symmetric sector.
A more significant limitation is that his results do not extend in an obvious
fashion to provide a complete description of the maximal development of the data;
see Subsect.~\ref{SS:COMPARISON} for more details. 
Christodoulou's work \cite{dC2007} eliminates these limitations, opens the way for obtaining a sharp understanding
	of shock formation in dimension 3, and properly sets up
	the difficult open problem of continuing the solution beyond the shock.
              
              \subsection{The dispersion of waves} 
                    \label{sect-Compr-dispers} 
                  	In this subsection and the next one, we
										discuss some of the main ideas, especially the role of dispersion, 
										in the development of the theory
                    of the long-time behavior of small-data solutions to nonlinear wave 
                    equations of type \eqref{modeleq:nongeo1} 
                    and \eqref{general-system}
                    in
										$\RRR^{1+n}$ prior to the work of Christodoulou \cite{dC2007}.
										We will especially focus on the case of $\RRR^{1+3}.$
                    In particular, we sketch in this subsection the proofs 
                    of the almost global existence
										result\footnote{We give here a version based on the vectorfield method
										introduced in \cite{sK1985} and not the original of \cite{fJsK1984}.}
										of \cite{fJsK1984} and the global existence result of \cite{sK1986}
                    for nonlinearities verifying the null condition
                    (see also \cite{dC1986a}).    
			In Subsect.~\ref{subs:radial-blow}, we study shock formation in detail
			for spherically symmetric solutions.

	\subsubsection{Local well-posedness}
		We start by recalling a classical local
 		well-posedness result for the scalar\footnote{The results can be extended, with
		minor modifications, to systems of the form
		\eqref{general-systemII} and that of nonlinear
		elasticity.}
		quasilinear wave equation \eqref{modeleq:nongeo1}; see, for example, \cite{cS2008}.
		We denote the initial data for $\Phi,$ given on the Cauchy hypersurface 
			$\Sigma_0 := \lbrace t=0 \rbrace \simeq \RRR^n \subset \RRR^{1+n}$ by
                         \begin{align}
                         \Phi(0,x) & = \mathring{\Phi}(x),  & & \pr_t \Phi(0,x)= \mathring{\Phi}_0(x). \label{in.cond}
                         \end{align}

\begin{proposition}[\textbf{Local well-posedness and continuation criteria}] \label{PROP:LOCALEXISTENCE}
Let $s \geq s_0 = \lfloor\frac{n}2\rfloor + 3$ be an integer.\footnote{By definition,
$\lfloor\frac{n}2\rfloor + 3$ is the smallest integer strictly larger than $n/2 + 2.$}
\medskip

\noindent \underline{\textbf{Local well-posedness.}}
Then there exists\footnote{In reality, 
for this theorem to hold (both local well-posedness and the breakdown
criterion to follow),  
we need additional
assumptions on the data and the coefficient matrix
$(g^{-1})^{\alpha\beta}$ ensuring that the
equation is hyperbolic in a suitable sense.
For convenience, we ignore this issue.} a unique classical solution
$\Phi$ to the equation \eqref{modeleq:nongeo1}
existing on a nontrivial spacetime slab of the form 
$[0,T) \times \mathbb{R}^n$ for some $T > 0.$
The solution has the following regularity properties:
\bea
\label{eq:Energy_s}
\| \partial\Phi(t,\cdot)\|_{H^{s-1}(\RRR^n)} 
\leq   
C_s
\left(
	\sum_{a=1}^3
	\| \partial_a \mathring{\Phi} \|_{H^{s-1}(\RRR^n)} 
	+ 
	\|\mathring{\Phi}_0\|_{H^{s-1}(\RRR^n)}
\right)
e^{ 
	C_s
	\int_0^t
		\| \Phi(\tau,\cdot) \|_{W^{2,\infty}} 
	d\tau,
}
\eea
where $C_s$ depends only on $s$ and  $W^{2,\infty}$  is the  $L^\infty$  based  Sobolev  norm, involving up to two derivatives 
 of $\phi$.

\medskip

\noindent \underline{\textbf{Continuation criterion.}}
The solution can be extended
 beyond $[0,T) \times \mathbb{R}^n$ as long as 
$\int_0^T
		\| \Phi(\tau,\cdot) \|_{W^{2,\infty}} 
	d\tau
	< \infty$. In particular, the
time of existence $T$ has a lower bound depending on 
$\|\mathring{\Phi}\|_{H^{s_0}(\RRR^n)} +
\|\mathring{\Phi}_0\|_{H^{s_0-1}(\RRR^n)}$.  

\end{proposition}

\begin{remark} 
 We note that a similar result holds    
 for the larger class of symmetric hyperbolic systems of Friedrichs \cite{kF1954} and, in particular,
 for systems of equations of type \eqref{general-system} relevant to
 the Einstein field equations. For this latter type, since the
 quasilinear term depends only on $\Psi$ and not its derivatives, we can
 close with one fewer derivative, that is, we can set 
 $s_0 = \lfloor \frac{n}2 \rfloor + 2$ instead.
\end{remark} 

  The a priori energy-type estimate \eqref{eq:Energy_s} is really at the heart of the proof.
                    It can be derived by differentiating the 
                   original nonlinear equation with respect to $\pr^{\vec{I}},$   
									for rectangular coordinate derivative multi-indices $\vec{I},$  
									multiplying the resulting equation by $\pr_t \pr^{\vec{I}} \Phi,$ 
									integrating by parts, and using simple commutator estimates; see \cite{sK1980} for example.   
									The local existence result  
									can then be proved by first replacing   
									$
										\int_0^t \|\Phi(\tau,\cdot)\|_{W^{2,\infty}(\RRR^n)} \, d \tau $  with the quantity  
									$t \sup_{0\le \tau\le t}\|\Phi(\tau,\cdot) \|_{H^{s_0}(\RRR^n)},$ $0\le t\le T,$  
									in view of the standard Sobolev inequality, and then devising a contraction argument with respect to the norm
                  $\sup_{0\le \tau \le T} \|\Phi(\tau,\cdot) \|_{H^s(\RRR^n)}$ for $s\ge s_0$ and sufficiently small $T.$ 
                    
                   We note in passing that
                   this method is very wasteful and that modern techniques
                   lead to an improved 
                   value of the minimal exponent $s_0.$  
                   The new methods avoid the crude use of Sobolev inequalities and rely instead on spacetime estimates  
                   such as Strichartz and bilinear estimates. 
                   For example, it was shown in \cite{hSdT2005}
                   that when $n \in \{3,4,5\}$, the  general\footnote{In the particular case  of        the   Einstein-vacuum equations expressed  with respect  to   wave coordinates,
                     the same result  was   proved  earlier  in \cite{sKiR2005d}. }  equation
                   $(g^{-1})^{\alpha\beta}(\Psi) \partial_{\alpha}\partial_{\beta} \Psi 
                   = \mathcal{N}^{\alpha \beta}(\Psi)\partial_{\alpha} \Psi \partial_{\beta} \Psi$
                   is locally well-posed for data $(\Psi, \partial_t \Psi) \in H^s \times H^{s-1}$
                   whenever $s > (n+1)/2.$
                   The best result in this direction is the recent resolution 
                   of the bounded $L^2$ curvature conjecture, 
                   see \cite{sKiRjS2012}, which for the Einstein-vacuum equations in $3$ space dimensions 
                   essentially leads to local well-posedness in $H^2.$ That is,
                   for the Einstein equations, this result
                   further improves those of \cite{hSdT2005}
                   from $s>2$ to $s=2.$
                                     
               \subsubsection{Beyond local existence via the vectorfield method}
	\label{SSS:BEYONDLOCALEXISTENCE} 
                      As we saw in Proposition
\ref{PROP:LOCALEXISTENCE}, to go beyond local existence,
the main step is to obtain control on the integral in the exponent of
\eqref{eq:Energy_s}. In the proof above, 
we crudely used the standard Sobolev inequality to bound the
integral 
$\int_0^t
	\| \Phi(\tau,\cdot) \|_{W^{2,\infty}} 
d\tau 
$
and we therefore did not account for the dispersive decay of solutions
to wave equations. 
If we could prove that the well-known uniform dispersive decay
rate $(1+t)^{-\frac{n-1}{2}}$
of solutions to the standard linear wave equation  
$\square_m \Phi=0$
also holds also for solutions
to the nonlinear wave equation \eqref{modeleq:nongeo1}
(and their up-to-second-order derivatives), 
then the exponential term on the right hand side of \eqref{eq:Energy_s} 
would be integrable for $n \geq 4$ and only logarithmically divergent for $n = 3.$
Note that the former estimate implies small-data global existence,\footnote{For $n \geq 5,$
this argument can be extended to show small-data global existence 
in the presence of arbitrary
nonlinear terms quadratic in $(\Phi, \partial \Phi, \partial^2 \Phi)$
in equation \eqref{modeleq:nongeo1}.
For $n=4,$ the argument can similarly be extended
as long as there are no quadratic terms of the form $\Phi^2.$} 
while the latter one implies the almost global existence
result of \cite{fJsK1984}. 
These estimates on decay rates are true as stated, but are nontrivial to prove.
The first results in this direction 
\cite{fJ1976a,fJ1976b,fJ1983,sK1980,fJsK1984}
were based on the explicit fundamental solution for $\square_m$
and as such were quite cumbersome and difficult to extend to more complicated situations.         
			The first modern proof, 
			based on the commuting vectorfield method and generalized energy estimates,
			appeared in \cite{sK1985}, though a related method had previously been used in linear theory     
			to derive local decay estimates in the exterior of a convex domain\footnote{The Minkowskian 
			\emph{Morawetz vectorfields} $(t^2+ r^2) \pr_t+ 2t r \pr_r$ and $f(r) \pr_r$, for appropriate functions $f(r)$,  
			also play fundamental roles in the modern vectorfield method and
			have their roots in \cite{cM1962}.}
			\cite{cM1962}.       
        We now provide a short summary of the commuting vectorfield method as it appears in \cite{sK1985}.
				The idea is to replace the multi-indexed rectangular spatial derivative operators 
                     $\pr^{\vec{I}}$ used in the derivation of   
										\eqref{eq:Energy_s} with a larger class 
                          of multi-indexed differential operators 
													$\mathscr{Z}_{(Flat)}^{\vec{I}} := Z_{(Flat;1)}^{I_1} \ldots Z_{(Flat;p)}^{I_p}$   
													that have good commutator properties with the Minkowski wave operator $\square_m,$ 
													where the vectorfields $Z_{(Flat;1)}, \cdots, Z_{(Flat;p)}$
													are the elements of the following subset
													of
conformal Killing fields\footnote{The vectorfield method has also been extended to apply
to some equations that are not invariant under the full Lie algebra of conformal symmetries
of Minkowski spacetime, but are instead invariant under only a
subalgebra; see, for example, \cite{sKtS1996,tS1996,tS1997,tS2000}. } of $m,$ expressed relative to rectangular coordinates:
                        \begin{multline} \label{E:MINKOWSKICONFORMALKILLING}
                      \mathscr{Z}_{(Flat)} := \{\pr_t, S_{(Flat)} =
t\partial_t + \sum_{a = 1}^n x^a \partial_a \}  \cup\{
\pr_i,  L_{(Flat;i)} = x^i \partial_t +
t\partial_i\}_{1\leq i \leq n} \\ \cup \{O_{(Flat;ij)} = x^i\partial_j -
x^j\partial_i\}_{1\leq i < j \leq n},
\end{multline}
which forms an $\RRR$-Lie algebra with the Lie bracket given by
the vectorfield commutator. 
        We can then derive energy-type estimates similar to those in \eqref{eq:Energy_s}, 
        but with $\pr^{\vec{I}}$ replaced by
$\mathscr{Z}_{(Flat)}^{\vec{I}},$ and with the $H^s$ norm on the left-hand side of \eqref{eq:Energy_s} replaced    
												by the norm\footnote{
												Note that the norm $\|\hspace{-0.7 pt}|\Phi | \hspace{-0.9 pt}  \|_{T,s}$ 
	does not directly control $\Phi$ itself or its $\mathscr{Z}_{(Flat)}$ derivatives in $L^2.$
	Various approaches for controlling these terms are described in \cite{sK2001, hLiR2010, cS2008}.}
                       \begin{align} \label{E:GENERALIZEDENERGYNORM}
                    \|\hspace{-0.7 pt} |\Phi | \hspace{-0.9 pt}  \|_{T,s}:= 
											\sup_{0\le t \le T}\big( \sum_{|\vec{I}|\le s}  \|\pr  \mathscr{Z}_{(Flat)}^{\vec{I}} 
											\Phi(t,\cdot)\|^2_{L^2(\RRR^n)}\big)^{1/2}.
                       \end{align}
    The norm \eqref{E:GENERALIZEDENERGYNORM} controls 
     $\partial\Phi$ not only in the standard $L^\infty(\RRR^n)$ norm  
		(through the Sobolev inequality as before), but also the
		weighted version $(1+t)^{\frac{n-1}{2}} \| \cdot
		\|_{L^{\infty}(\RRR^n)},$ which yields the expected uniform 
		$(1+t)^{-\frac{n-1}{2}}$ rate of decay of
	$\partial\Phi$ and its lower-order $Z_{(flat)}$ derivatives. 
	A standard way to obtain this control is to use Klainerman-Sobolev inequality 
                          \cite{sK1985}:  
													\begin{align} \label{E:KLAINERMANSOBOLEV}
														\sup_{0\le t\le  T}\sup_{x\in\RRR^n}
															(1+t+r)^{\frac{n-1}{2}} (1+ |t-r|)^{1/2} |\pr \Phi(t,x)|
															&\leq 
																C \|\hspace{-0.7 pt} |\Phi| \hspace{-0.9 pt}  \|_{T,(n+2)/2},
                          \end{align}
                           where $r = \sqrt{\sum_{a=1}^3(x^a)^2}.$
                           
In addition, from the boundedness of
$\|\hspace{-0.7 pt} |\Phi| \hspace{-0.9 pt}  \|_{T,s} ,$
we can derive a further refined
account of the dispersive properties of waves 
which shows that for $t \geq 0,$ the derivatives of $\pr \Phi$ in directions
tangent to the outgoing Minkowski cones $\lbrace t - r = \text{const} \rbrace$
have better decay properties than 
derivatives in a transversal direction.
To illustrate this fact, we first introduce the standard radial null pair in Minkowski space:
                         \begin{align} \label{E:STANDARDMINKOWSKINULLPAIR}
														\Lunit_{(Flat)} & := \pr_t + \pr_r, & & \uLunit_{(Flat)} := \pr_t-\pr_r,
                         \end{align}
                         with $\pr_r=\frac{x^a}{r}\pr_a$ the standard Euclidean radial derivative. 
                         Note that $\Lunit_{(Flat)}$ and	$\uLunit_{(Flat)}$ are
null vectorfields relative to the Minkowski metric, that is,
$m(\Lunit_{(Flat)},\Lunit_{(Flat)} )=m(\uLunit_{(Flat)},
\uLunit_{(Flat)})=0,$ and they satisfy
$m(\Lunit_{(Flat)},\uLunit_{(Flat)})=-2.$   
			The null pair can be completed to  
                          a null frame by choosing, 
													at every point in $ \RRR^{1+n},$  
													$n-1$ vectorfields     
													$e_1, \ldots, e_{n-1}$    
													orthogonal to $e_n:=\Lunit_{(Flat)}, e_{n+1}:=\uLunit_{(Flat)}$  
													such that $m(e_i, e_j)=\de_{ij}$ for $i,j=1,\ldots, n-1.$
                          As we sketch below, assuming that we have control
                          over $\|\hspace{-0.7 pt} |\Phi| \hspace{-0.9 pt}  \|_{T,(n + 4)/2},$
                          we are also able to obtain uniform control of the
													following derivatives ($a=1,\cdots,n-1$):
			  \bea
                           \label{eq:peeling}
                           \begin{cases}
                          &  \sup_{0\le t\le  T}\sup_{x\in\RRR^n}
															(1+t
			+r)^{\frac{n+1}{2}} (1+ |t-r|)^{1/2}   
                                 |e_a(\pr \Phi)(t,x)|, 
                                 \\
                              &   \sup_{0\le t\le  T}\sup_{x\in\RRR^n}   (1+ t +r)^{\frac{n+1}{2}} (1+ |t-r|)^{1/2}   
                                 |\Lunit_{(Flat)}(\pr \Phi)(t,x)|,\qquad
                                 \\
                             &        \sup_{0\le t\le  T}\sup_{x\in\RRR^n}   (1+ t +r)^{\frac{n-1}{2}} (1+ |t-r|)^{3/2}   
                                 |\uLunit_{(Flat)}(\pr \Phi)(t,x)|. 
                                 \end{cases}
                           \eea
        In other words, the derivatives of
				$\partial\Phi$ in the directions $e_1,\ldots,e_{n-1},\Lunit_{(Flat)},$
				which span the tangent space of the outgoing Minkowski cones 
				$\lbrace t - r = \text{const} \rbrace,$  
				have better uniform decay
				rates  than
                            $\uLunit_{(Flat)} \partial \Phi.$ 
                            The gain of decay rates can be obtained 
                            by expressing the vectorfields $e_1, \ldots, e_{n-1}, \Lunit_{(Flat)}, \uLunit_{(Flat)}$
                            in terms of the vectorfields in 
                            $\mathscr{Z}_{(Flat)}$ and estimating the
			    coefficients.
The decay estimates for $\Lunit_{(Flat)}(\partial\Phi)$ and
$\uLunit_{(Flat)}(\partial\Phi)$ come from combining
\eqref{E:KLAINERMANSOBOLEV} with the algebraic identities
\begin{subequations}
\begin{align}
(t + r) \Lunit_{(Flat)} &= S_{(Flat)} + \frac1{r}\sum_{i = 1}^n x^i
L_{(Flat;i)}, 
	\label{E:LINTERMSOFVECTORFEIDLS} \\
(t-r) \uLunit_{(Flat)} &= S_{(Flat)} - \frac1{r}\sum_{i = 1}^n x^i
L_{(Flat;i)}. \label{E:LBARINTERMSOFVECTORFIEDLS}
\end{align}
\end{subequations}
The decomposition for $e_{a}(\partial \Phi)$ is similar, but slightly
more involved.   

			The above discussion can also be used    
			to provide clear motivation
                            for the null condition and the corresponding small-data global existence results 
                            \cite{sK1986, dC1986a}
                            in $3D.$ The null condition (see
			Subsubsect.~\ref{SSS:CLASSICNULL}) is designed
			to capture the fact that
                            some quadratic terms exhibit better decay properties than
                            others.  For example, we can
                            consider the bilinear forms
                            \begin{align} \label{E:NULLFORMS}
                               \mathscr{Q}_0(\Phi,\Psi) &:=(m^{-1})^{\a\b}\pr_\a\Phi\pr_\b \Psi,
																& &
                               \mathscr{Q}_{\a\b}(\Phi,\Psi):=\pr_\a \Phi\pr_\b \Psi-\pr_\b \Phi\pr_\a\Psi.
                             \end{align}
                             
                             If $\mathscr{Q}$ is any of the bilinear forms \eqref{E:NULLFORMS},
                             then by using vectorfield algebra as in \eqref{E:LINTERMSOFVECTORFEIDLS}-\eqref{E:LBARINTERMSOFVECTORFIEDLS},
                             it is straightforward to derive the
pointwise estimate\footnote{
If we only used rectangular coordinates derivatives, then
we could only derive the weaker estimate
$\mathscr{Q}(\Phi,\Psi) \leq C |\partial\Phi||\partial\Psi|$.}
                             \begin{align} \label{E:NULLFORMGAIN}
                             		\left|
                             			\mathscr{Q}(\Phi, \Psi)
                             		\right|
                             		& \leq  
                             					\frac{C}{1 + t + r}
                             					\sum_{Z_{(Flat)},
Z'_{(Flat)} \in \mathscr{Z}_{(Flat)}}  
                             						\left|
                             							Z_{(Flat)} \Phi
                             						\right|
                             						\left|
                             							Z_{(Flat)}' \Psi
                             						\right|.
                             \end{align}
                            When $n=3,$ the gain of the critically important factor $(1 + t + r)^{-1}$ 
                            helps one avoid logarithmic divergences in 
                            $L^2$ estimates involving $\mathscr{Q}(\Phi, \Psi).$
  													In contrast, for a general quadratic form such as, for example,  
														$\pr_t\Phi\pr_t \Psi$   or
														$\nab\Phi\c\nab\Psi : =\sum_{a=1}^n \pr_a \Phi \pr_a \Psi,$
														the factor $(1 + t + r)^{-1}$ in \eqref{E:NULLFORMGAIN}
														must be replaced with $(1 + |t-r|)^{-1},$
														which
				yields no gain in the wave zone $\{t \sim r\}.$

                       \subsubsection{The classic null condition}     
                       	\label{SSS:CLASSICNULL}
                             The considerations described in Subsubsect.~\ref{SSS:BEYONDLOCALEXISTENCE}
                             lead to the classic null condition
                             for equations of type \eqref{modeleq:nongeo1}
                             and for \eqref{general-system} in the scalar case,
															which we will now discuss.
                            We first consider equation \eqref{modeleq:nongeo1}. 
														Since we are studying only the behavior of small solutions, 
                            we rewrite the equation as a perturbation of the linear wave equation $\square_m \Phi = 0,$
														that is, in the form
                            \begin{equation}\label{eq:qnlw}
- \partial_t^2 \Phi + \Delta \Phi +
  \AA^{\alpha\beta}(\partial\Phi)\partial^2_{\alpha\beta}\Phi 
= \NN(\Phi, \pr \Phi),
\end{equation}       
where $\AA^{\mu\nu} = \mathcal{O}(| \partial\Phi|)$  and
$\NN= \mathcal{O}(|\partial\Phi|^2)$  
for small $(\Phi, \pr\Phi).$   
Taylor expanding further $\AA$ and $\NN,$ we have
  \begin{align}
\AA^{\mu\nu}(\partial\Phi) &= \AA^{\mu\nu\sigma} \partial_\sigma \Phi +
\mathcal{O}(|\partial\Phi|^2), 
	\\
\NN(\Phi, \pr \Phi) 
&= \NN^{\mu\nu}\partial_\mu\Phi \partial_\nu\Phi 
+ \mathcal{O}(|\Phi||\partial \Phi|^2 + |\partial\Phi|^3),
\end{align}
where the constants
$\AA^{\mu\nu\sigma}$
and
$\NN^{\mu\nu}$ are 
\begin{align}
	\AA^{\mu\nu\sigma}
	& := \frac{\partial}{\partial (\partial_{\sigma} \Phi)} \AA^{\mu\nu}(\partial \Phi)|_{\partial \Phi = 0},
		\\
	\NN^{\mu\nu}
	& := \frac{\partial^2}{\partial (\partial_{\mu} \Phi) \partial (\partial_{\nu} \Phi)} 
		\NN(\Phi, \pr \Phi)|_{(\Phi,\partial \Phi) = (0,0)}.
		\label{E:NONLINEARITYPARTIALPHI}
\end{align}

Similarly, under the assumptions $(g^{-1})^{\mu \nu}(\Psi = 0) = (m^{-1})^{\mu \nu}$
and that 
$\NN(\Psi,\partial \Psi) = \mathcal{O}(|\partial \Psi|^2)$  
for small $(\Psi, \pr \Psi),$   
we can rewrite \eqref{general-system} as a perturbation of the linear wave equation,
where
$\AA^{\mu\nu}(\partial\Phi)$ in \eqref{eq:qnlw}
is replaced by $\AA^{\mu\nu}(\Psi),$
$\NN(\Phi,\partial\Phi)$ is replaced by 
$\NN(\Psi,\partial \Psi),$
$\AA^{\mu \nu \sigma}$ is replaced by
$\AA^{'\mu\nu} := \frac{d}{d \Psi} \AA^{\mu\nu}(\Psi)|_{\Psi = 0},$
and $\NN^{\mu\nu}$ is replaced by
$\NN^{\mu\nu}
:= \frac{\partial^2}{\partial (\partial_{\mu} \Psi) \partial (\partial_{\nu} \Psi)} 
		\NN(\Psi,\partial \Psi)|_{(\Psi,\partial \Psi) = (0,0)}.$

\begin{definition}[\textbf{Classic null condition}] \label{D:CLASSICNULL}
We say that the nonlinearities in equation 
\eqref{eq:qnlw} verify the classic null condition   
if for every 
covector $\ell = (\ell_0,\ell_1,\ell_2,\ell_3)$ satisfying 
$(m^{-1})^{\a\b}\ell_\a\ell_\b := -\ell_0^2 + \ell_1^2 +
\ell_2^2 + \ell_3^2 = 0,$ we have the
identities
\[  \AA^{\mu\nu\sigma}\ell_\mu\ell_\nu\ell_\sigma
	=  \NN^{\mu\nu}\ell_\mu\ell_\nu = 0.
\]

Similarly, in the case of \eqref{general-system} with $I = 1$
(the case of a single scalar equation),
we say that the nonlinearities verify the classic null condition
if for every Minkowski-null covector $\ell,$
we have the identities
\begin{align} \label{E:ALTERNATENULLCONDITION}
	 \AA'^{\mu\nu} \ell_\mu\ell_\nu
	=  \NN^{\mu\nu}\ell_\mu\ell_\nu = 0 ~.
\end{align}
\end{definition}

\begin{remark}
	Definition \ref{D:CLASSICNULL}
	can be extended for systems of wave equations; 
	see Remark \ref{Re:Aleph-systems} or \cite{sK1984}.
\end{remark}

We now provide two standard examples.

\begin{itemize}
	\item For the scalar equation \eqref{eq:qnlw},
		it is straightforward to check that 
		the quadratic semilinear term 
		$\NN^{\mu\nu}\partial_\mu\Phi \partial_\nu\Phi$
		verifies the classic null condition 
		if and only if it is a constant multiple
		of the null form $\mathscr{Q}_0(\Phi,\Phi)$ from \eqref{E:NULLFORMS}. 
	\item Similarly, 
		for equation \eqref{general-system} in the scalar case
		under the assumption $(g^{-1})^{\mu \nu}(\Psi = 0) = (m^{-1})^{\mu \nu},$
		one can show that the quadratic quasilinear terms verify the classic 
		null condition if and only if
		$\AA'^{\mu\nu}$ 
		is a  multiple of $(m^{-1})^{\mu\nu}.$ 
\end{itemize}

	The proof of global existence for equations of type
	\eqref{eq:qnlw} verifying the null condition follows a similar pattern as the proof of the almost  
  global existence in \cite{fJsK1984} by 
  taking into account 
  the favorable factor $(1 + t + r)^{-1}$ in \eqref{E:NULLFORMGAIN}. 
  Another important feature of the proof, 
			which is by now a familiar aspect of the literature, 
			is that the highest energy norm is not bounded but is instead allowed to grow          
			like a small power of $t$ as $t\to \infty.$ Despite the possible slow top-order energy growth,
        the resulting global solutions to equations verifying the classic null condition 
        in fact enjoy the same type of peeling properties \eqref{eq:peeling}   
        as solutions to the linear wave 
        equation in Minkowski spacetime
        (at least as far as the low-order derivatives of $\Phi$ are concerned).
        In the small-data shock-formation problem, 
				we also encounter a similar top-order growth phenomenon, but it is much more severe 
				when the characteristic hypersurfaces intersect
				(in fact, the top-order energies are allowed to blow-up);
				see Prop.~\ref{P:APRIORIENERGYESTIMATES}.

\subsubsection{John's conjecture and an overview of Alinhac's proof of it for non-degenerate small data}
\label{SSS:JOHNSCONJECTURE}
In $3D,$ when the quadratic nonlinearities fail the classic null condition,
we expect that small-data global existence fails to hold
(recall that John showed \cite{fJ1981} that in many cases,
one does have a breakdown, though the mechanism is not revealed by 
the proof). 
Nonetheless, we still have the almost global existence result of John and Klainerman
mentioned earlier and also a sharper version,
due to John and H{\"o}rmander, which we state as Theorem~\ref{T:JOHNHORMANDERLIFESPANLOWER}.
We first recall that the Radon transform of a function $f$ on $\mathbb{R}^3$ can be
defined for points 
	$q \in \mathbb{R},$ $\theta \in \mathbb{S}^2 \subset \mathbb{R}^3$ 
	as
	\begin{align} \label{E:RADONTRANSFORMOFF}
		\mathcal{R}[f](q,\theta)
		& := \int_{P_{q,\theta}} f(y) \, d \sigma_{q,\theta}(y),
	\end{align}
	where $P_{q,\theta} := \lbrace y \in \mathbb{R}^3 \ | \
	\Euct(\theta,y) = q \rbrace$ is the plane with unit normal $\theta$
	that passes through $q\theta\in \mathbb{R}^3,$ 
	$d \sigma(y)$ denotes the area form induced on the plane $P_{q,\theta}$ by the 
	Euclidean metric $\Euct$ on $\mathbb{R}^3,$ 
	and $\Euct(\theta,y)$ is the Euclidean inner
	product of $\theta$ and $y.$
	We also introduce the following function
	$\Fried[(\mathring{\Phi}, \mathring{\Phi}_0)]: \mathbb{R} \times \mathbb{S}^2 \rightarrow \mathbb{R},$
	which also depends on the initial data pair
	$(\Phi|_{t=0},\partial_t \Phi|_{t=0}) = (\mathring{\Phi}, \mathring{\Phi}_0):$
	\begin{align} \label{E:INTROFRIEDNALNDERRADIATIONFIELD}
		\Fried[(\mathring{\Phi}, \mathring{\Phi}_0)]
		(q,\theta)
		& := 
			-
			\frac{1}{4 \pi} 
				\frac{\partial}{\partial q}
				\mathcal{R}[\mathring{\Phi}](q,\theta)
			+
			\frac{1}{4\pi} 
			\mathcal{R}[\mathring{\Phi}_0](q,\theta).
	\end{align}
	
	\begin{remark}[\textbf{Friedlander's radiation field}]
		The function $\Fried[(\mathring{\Phi}, \mathring{\Phi}_0)]$
		from \eqref{E:INTROFRIEDNALNDERRADIATIONFIELD} is Friedlander's
		radiation field for the solution to the \emph{linear wave equation}
		corresponding to the data $(\mathring{\Phi}, \mathring{\Phi}_0).$
		See Subsect.~\ref{SS:DISCUSSIONOFSHOCKFORMINGDATA} for an extended
		discussion of the role that $\Fried[(\mathring{\Phi}, \mathring{\Phi}_0)]$ plays
		in determining when and where blow-up occurs.
	\end{remark}

\begin{theorem} \cite[\textbf{John and H{\"o}rmander}]{fJ1987,lH1987}
	\label{T:JOHNHORMANDERLIFESPANLOWER}
	Consider the initial value problem 
	\begin{gather*}
		(g^{-1})^{\alpha \beta}(\partial \Phi)
\partial_{\alpha} \partial_{\beta} \Phi = 0,\\
		(\Phi|_{t=0},\partial_t \Phi|_{t=0}) = \uplambda (\mathring{\Phi}, \mathring{\Phi}_0)
	\end{gather*}
	for a quasilinear wave equation in $\mathbb{R}^{1+3}$
	verifying \eqref{E:GINVERSEISMINKOWSKIFORPHIEQUALS0}
	with compactly supported smooth initial data, for which the classical null
	condition does \textbf{not} hold.
	Then the classical lifespan $T_{(Lifespan);\uplambda}$ of the solution
	verifies
	\begin{align} \label{E:JOHNHORMANDERLOWERBND}
		\liminf_{\uplambda \downarrow 0}
		\uplambda \ln T_{(Lifespan);\uplambda}
		\geq 
		\frac{1}
		{\sup_{(q,\theta) \in \mathbb{R} \times \mathbb{S}^2}
		\frac{1}{2} 
		\FutFailFac(\theta) 
		\frac{\partial^2}{\partial q^2} \Fried[(\mathring{\Phi}, \mathring{\Phi}_0)]
		(q,\theta)},
	\end{align}
	where $\FutFailFac$ is the future null condition failure factor for the equation;
	see \eqref{E:OTHERFAILUREFACTOR} for an explicit formula.
\end{theorem}

The natural conjecture, which was envisioned by F. John,\footnote{In 
\cite{fJ1989}, John also contemplated the possibility that away from
spherical symmetry, singularity formation might be avoided.}
is that Theorem~\ref{T:JOHNHORMANDERLIFESPANLOWER} is sharp and that 
small-data solutions in fact blow up at times near \eqref{E:JOHNHORMANDERLOWERBND}. 
Moreover, the blow-up should be due to the crossing of characteristics, similar
to the case of the $1D$ Burgers' equation. 
A restricted version of this conjecture, applicable to initial data
satisfying some non-degeneracy conditions, was first proved by
Alinhac; see Theorem~\ref{T:ALINHACSHOCKFORMATION}
and the discussion in Subsect.~\ref{SS:COMPARISON}.

It is easy to see that the right-hand side of \eqref{E:JOHNHORMANDERLOWERBND}
must be non-negative for compactly supported data.
The importance of Alinhac's work is further enhanced by the next proposition,
which shows that the right-hand side of \eqref{E:JOHNHORMANDERLOWERBND}
is strictly positive whenever the data are compactly supported and nontrivial.
Thus, Alihnac's work shows that in the $\uplambda \downarrow 0$ limit,
for nontrivial data verifying his non-degeneracy conditions,
\emph{shocks will always form}. 
Moreover, as we describe in Subsect.~\ref{SSS:SHARPNESSOFREFINED},
Alinhac's non-degeneracy conditions on the data turn out to be unnecessary. 
However, as we describe in Subsect.~\ref{SS:COMPARISON},
his proof cannot be extended to recover this fact;
the proof requires the full power of Christodoulou's framework.

\begin{proposition} \cite{fJ1987}*{pg. 98}
		\label{P:JOHNSCRITERIONISALWAYSSATISFIEDFORCOMPACTLYSUPPORTEDDATA}
		Let $\mathring{\Phi}, \mathring{\Phi}_0 \in C_c^{\infty}(\mathbb{R}^3).$ 
		Assume that $\FutFailFac \not \equiv 0$ and that
		\begin{align} \label{E:JOHNSQUANTITYVANISHES}
		\sup_{(q,\theta) \in \mathbb{R} \times \mathbb{S}^2}
		\FutFailFac(\theta) 
		\frac{\partial^2}{\partial q^2} \Fried[(\mathring{\Phi}, \mathring{\Phi}_0)]
		(q,\theta)
		& = 0.
		\end{align}
		Then $(\mathring{\Phi}, \mathring{\Phi}_0) = (0,0).$
	\end{proposition}

\subsection{A sharp description of small-data shock formation for spherically symmetric solutions in $3D$}
\label{subs:radial-blow}
  We are now ready to describe, in the simplified setting
	of spherical symmetry,
	how failure of the classic null condition can cause
	small-data solutions to equations of type \eqref{eq:qnlw}
  to form shock-type singularities. It turns out that in the
small-data regime, the main mechanism of shock formation is the same
both in and out of spherical symmetry. Hence, in the spherically symmetric case,
we provide a detailed proof of singularity formation in the higher-order derivatives and
regularity of the lower-order derivatives within an appropriate wave zone 
using the framework\footnote{The proof given here
is a bit sharper than the one given by John \cite{fJ1985} in that it
exhibits the precise blow-up mechanism due to the intersection of the
characteristic hypersurfaces, similar to that of Burgers' equation
(see Subsubsect.~\ref{SSS:SINGULARITIESIN1D}). The additional precision 
afforded by Christodoulou's framework is essential for
extending the result beyond spherical symmetry.} of
Christodoulou \cite{dC2007}. 

	\subsubsection{Geometric formulation of the problem}
	\label{SSS:GEOMETRICFORMULATION}
	Following F. John \cite{fJ1985},
	we examine the model equation \eqref{John-model}. In particular, we focus here
  on the simplest case\footnote{It is straightforward to see that this equation 
fails the classic null condition of Definition \ref{D:CLASSICNULL}.} 
$\square_m \Phi = - \partial_t \Phi \Delta \Phi,$
which takes the following form relative to standard spherical coordinates on Minkowski spacetime:
   \begin{align}
     \pr_t^2 (r\Phi)=\left(1+\partial_t \Phi\right) \pr_r^2(r\Phi). \label{John:radial}
   \end{align}
In \eqref{John:radial}, $\Phi(t,x)=\Phi(t,r)$ and $r:=\sqrt{\sum_{a=1}^3 (x^a)^2}.$
We expect that shock formation corresponds to the blow-up of some second
derivatives of $\Phi,$ while $\Phi$ itself and its first derivatives remain bounded.
Hence, we can equivalently consider the equation
\begin{equation} \label{boeq}
-\partial_t^2 \left(r \Psi \right) + \left(1+\Psi \right) \partial_r^2 \left(r \Psi\right)  = -r \frac{\left(\partial_t \Psi\right)^2}{1+\Psi} \, .
\end{equation}
for $\Psi := \partial_t \Phi$ induced from \eqref{John:radial}, 
and show that $\Psi$ remains bounded while some of its first
derivatives blow up. Our analysis takes place in a small strip
$\MM_{t,U_0}$ (contained in the ``wave zone'') defined just below;
see also Figure \ref{F:SSSPACETIMESUBSETS}. 

\begin{center}
\begin{overpic}[scale=.2]{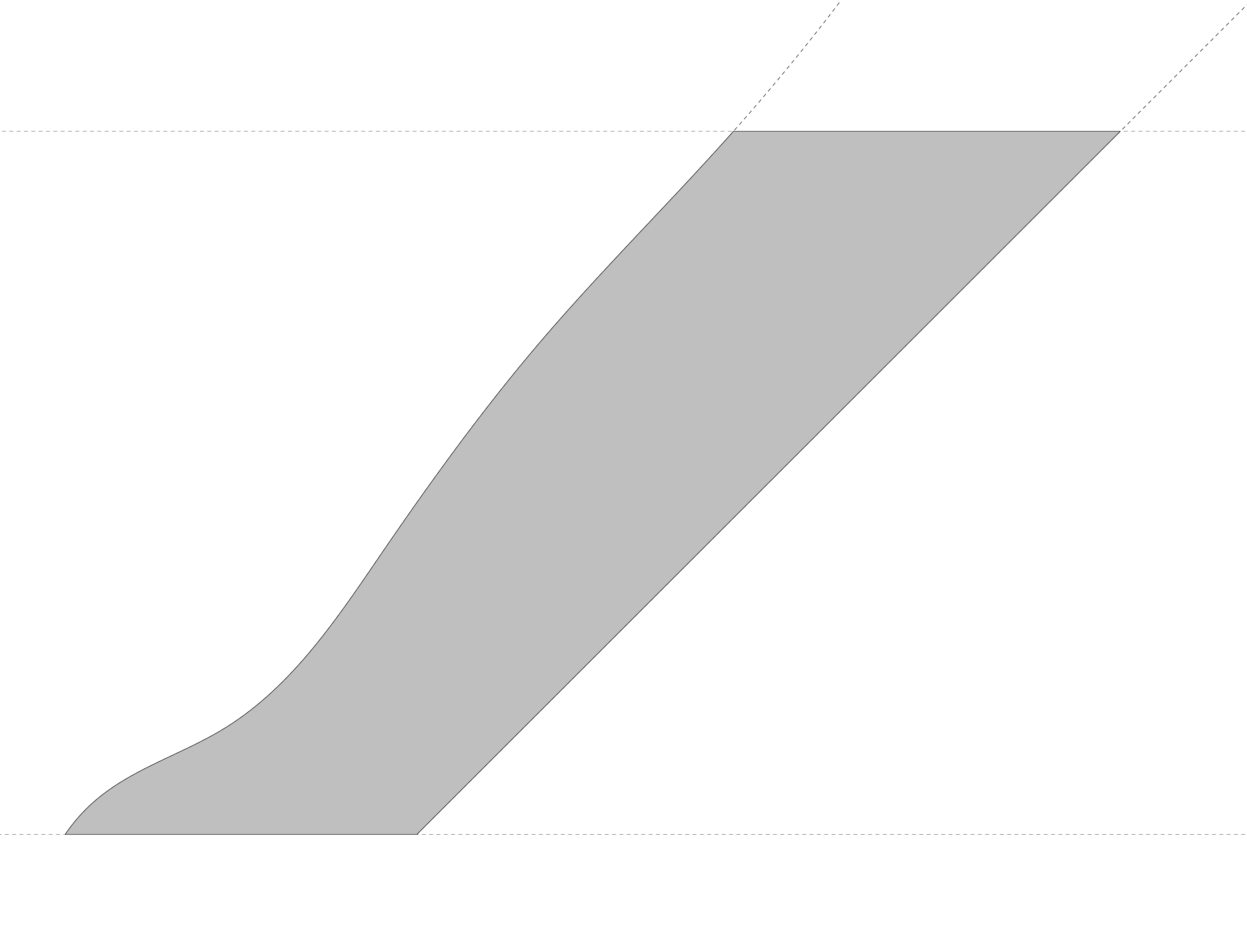}  
	\put (49,44) {\large$\displaystyle \MM_{t,u}$}
\put (95,68) {\large$\displaystyle \mathcal{C}_0$}
\put (54.5,68) {\large$\displaystyle \mathcal{C}_u$}
\put (29.5,40) {\large$\displaystyle \mathcal{C}_u^t$}
\put (67.5,40) {\large$\displaystyle \mathcal{C}_0^t$}
\put (35,68) {\large$\displaystyle \Sigma_t$}
\put (14,4) {\large$\displaystyle \Sigma_0^u$}
\put (70,68) {\large$\displaystyle \Sigma_t^u$}
\put (70,28) {\large$\displaystyle \Psi \equiv 0$}
\end{overpic}
\captionof{figure}{Regions of interest, where $u \in [0,U_0]$}
 \label{F:SSSPACETIMESUBSETS}
\end{center}

To define the region $\MM_{t,U_0}$, we note that the initial value
problem for \eqref{boeq} with initial data
\begin{align} \label{E:SSDATA}
	\Psi(0,r) 
	& := \mathring{\Psi}(r), 
	&& \partial_t \Psi(0,r) 
		:= \mathring{\Psi}_0(r)
\end{align}
can be solved
using the method of characteristics. The characteristic
vectorfields\footnote{These are the replacements for $\Lunit_{(Flat)}$
and $\uLunit_{(Flat)}$ defined in \eqref{E:STANDARDMINKOWSKINULLPAIR},
and are adapted to the true quasilinear geometry of the equation. Note
that we chose them to be normalized (see Footnote \ref{FN:XFNOTATION}
for notation) $\Lunit t = \uLunit t = 1.$} are
\begin{align} \label{E:CHARDIR}
\Lunit := \partial_t + \sqrt{1+\Psi} \partial_r \ \ \ \ \textrm{ and } \ \ \ \ \uLunit := \partial_t - \sqrt{1+\Psi} \partial_r \, .
\end{align}
Corresponding to the ``outgoing'' vectorfield $\Lunit,$ we can define
an \emph{eikonal function} $u(t,r)$
satisfying\footnote{\label{FN:XFNOTATION}If $X$ is a vectorfield and
$f$ a function, we write $X f = X^{\alpha} \partial_{\alpha} f$ for
the $X$-directional derivative of $f.$} 
\[
\Lunit u(t,r) = 0, 
\]
with $u=const$ defining the outgoing characteristics $\mathcal{C}_u.$
The function $u(t,r)$ is uniquely determined once we fix its value
along the hypersurface $\lbrace t = 0 \rbrace;$
we initialize $u$ by prescribing
\begin{align}
	u|_{t=0} := 1-r.
\end{align}
Finite speed of propagation \cite{fJ1981,fJ1990} for the wave equation implies that the
solution along $\mathcal{C}_u$ only depends on the data at points $r
\geq 1-u$.

We assume for convenience that $(\mathring{\Psi}, \mathring{\Psi}_0)$ are supported in 
$\lbrace r \leq 1 \rbrace.$ This implies that $\Psi \equiv 0$ when $u
\leq 0$. We thus define, in spherical coordinates, the region of
interest
\begin{align} \label{E:SSSTRIP}
	\MM_{t,U_0} := \lbrace (t',r) \ | \ 0 \leq t' < t \ \mbox{and} \ 0 \leq u(t',r) \leq U_0 \rbrace.
\end{align}
On $\MM_{t,U_0},$ the solution depends only on the data belonging to the
annulus $r\in [1-U_0,1]$. For convenience in notation, we also define
\begin{subequations}
\begin{equation}
\begin{aligned}
\Sigma_t & := \lbrace (t,r) \ | \ r \geq 0 \rbrace, &
\Sigma_t^{u'} & := \lbrace (t,r) \ | \ 0 \leq u(t,r) \leq u'
\rbrace,
\\
\mathcal{C}_{u'} & := \lbrace (t',r) \ | \ 0 \leq t' \ \mbox{and} \
u(t',r) = u' \rbrace, &
\mathcal{C}_{u'}^t &:= \lbrace (t',r) \ | \ 0 \leq t' \leq t \ \mbox{and} \ u(t',r) = u' \rbrace.
\end{aligned}
\end{equation}
\end{subequations}
 
\begin{definition}[\textbf{Inverse foliation density}]
The quantity $\upmu$ defined by
\begin{equation} \label{def:mu}
\upmu^{-1} := \partial_t u(t,r) = - \sqrt{1+\Psi} \partial_r u,
\end{equation}
is called the \emph{inverse foliation density}. It is also known as the null lapse. 
\end{definition}
By the choice of initial data for $u,$ on $\Sigma_0,$ we see $\upmu = 1 + \mathcal{O}(\Psi).$ 
The quantity $\upmu^{-1}$ plays a fundamental role in the analysis of
shock formation. It measures the density of the leaves\footnote{Later,
we will introduce the eikonal function and the inverse foliation
density in full generality, without the assumption of spherical
symmetry. Here it is sufficient to understand them in the context of
the method of characteristics.}  $\mathcal{C}_u$ with respect to the
time coordinate $t,$ and shock formation (intersection of
characteristics) corresponds to $\upmu \to 0.$
As long as $\upmu$ remains positive, the two functions $t,u$ are 
independent and form a coordinate system of $\MM_{t,U_0}.$ As we will see, there are advantages to using the
 ``geometric'' coordinates $(t,u)$ in place of $(t,r).$ We note that
\begin{align}
	\Lunit & = \frac{\partial}{\partial t}|_u,
	&&
	\upmu \uLunit u = 2.
\end{align}
Straightforward computations reveal that \eqref{boeq} can be expressed in the two equivalent forms
\begin{subequations}
\begin{align}
\Lunit \uLunit \left(r \Psi \right) &= \frac{1}{4} \frac{r}{(1+\Psi)}
\left[\left( \uLunit \Psi\right)^2 +3 (\Lunit \Psi) (\uLunit\Psi)
\right] - \frac{1}{2} \frac{1}{\sqrt{1+\Psi}} \Psi (\Lunit \Psi),  
	\label{E:LOUTSIDE} \\
\uLunit \Lunit \left(r \Psi \right) &=\frac{1}{4} \frac{r}{(1+\Psi)}
\left[ \left(\Lunit \Psi\right)^2 +3 (\Lunit \Psi) (\uLunit\Psi)
\right]+ \frac{1}{2} \frac{1}{\sqrt{1+\Psi}} \Psi (\uLunit \Psi),
	\label{E:LINSIDE}
\end{align}
\end{subequations}
which can be used to derive estimates along the characteristic directions.  Note that \eqref{E:LINSIDE} follows from
 \eqref{E:LOUTSIDE} and the commutator relations
 \beaa
 [\Lunit, \Lb]=-\frac{\pr_t\Psi}{\sqrt{1+\Psi}} \pr_r, \qquad  [\Lunit, \Lb]\Psi= \frac{(\uLunit \Psi)^2 - (\Lunit \Psi)^2}{4(1+\Psi)}.
 \eeaa

\begin{remark} \label{R:SSRICCATI}
Examining the semilinear terms in equations
\eqref{E:LOUTSIDE}-\eqref{E:LINSIDE}, 
we see that some of the nonlinearities in equation \eqref{John:radial} fail the classic null condition
of Definition \ref{D:CLASSICNULL}.
In particular, both\footnote{Note that in the limit $\Psi \to 0$ we
	have $\Lunit \to \Lunit_{(Flat)}$ and similarly for $\uLunit.$}  
$(\Lunit \Psi)^2$ and $(\uLunit \Psi)^2$ fail\footnote{As 
we will see, the remaining terms in \eqref{E:LOUTSIDE}-\eqref{E:LINSIDE}
can be treated as negligible errors.}
the classic null condition of
Definition \ref{D:CLASSICNULL}. 
However, in view of the forward peeling properties
\eqref{eq:peeling}, we expect that in the relevant future region
$\MM_{t,U_0},$
$\Lunit \Psi$ decays faster than $\uLunit \Psi.$
Hence, the only term that behaves poorly, 
from the point of view of
\emph{linear decay}, is the term 
$\frac{1}{4} \frac{r}{1 + \Psi} (\uLunit \Psi)^2$ on the right-hand side of \eqref{E:LOUTSIDE}. 
This term is in fact the source of the small-data shock formation:
we will use the estimate $r \approx t$ (within $\MM_{t,U_0}$)
to show that this term drives a
Riccati-type blow-up along the integral curves of $\Lunit.$
We rigorously prove a refined version of this claim in
Prop.~\ref{P:ge} and Cor.~\ref{C:ge}.
Furthermore, as we will see, the terms 
$\frac{1}{\sqrt{1+\Psi}} \Psi (\Lunit \Psi)$ and 
$\frac{1}{\sqrt{1+\Psi}} \Psi (\uLunit \Psi)$  
are negligible error terms.
\end{remark}

\subsubsection{Rescaling in the transversal
direction}\label{SSS:RESCALINGGIVESTHENULLCONDITION}
The crucial observation\footnote{F. John implicitly used a similar
strategy in his original argument \cite{fJ1985}
and in his earlier work \cite{fJ1974} in $1D.$ 
The idea of using Burgers' equation as a guide and studying the system in
characteristic coordinates is also used, in a somewhat different form, by
Alinhac \cite{sA1995,sA1999a,sA1999b,sA2001a,sA2001b}.} 
of Christodoulou is that we can derive an equivalent system of equations for
new $\upmu$-weighted quantities for which the problematic term
$(\uLunit \Psi)^2$ does not appear.  
As we show below in Prop.~\ref{P:ge},
the rescaled system can then be treated by straightforward dispersive-type methods 
(reminiscent of the peeling properties \eqref{eq:peeling}), 
in the spirit of small-data (spherically symmetric) global existence
results. This rescaling  by the  factor of $\upmu$  
takes place only in the $\mathcal{C}_u$-transversal
direction $\uLunit$ and has a simple interpretation, at least in
spherical symmetry, in terms of the method of characteristics. 
More precisely, it is straightforward to show that relative to the $(t,u)$ coordinates, 
we have the identity 
$\upmu \uLunit = \upmu \frac{\partial}{\partial t} + 2 \frac{\partial}{\partial u}.$ 
We note that expressing $\Psi$ as a function of $(t,u)$ 
is analogous to our earlier representation of a solution to the Burgers' equation
\eqref{Burger} in the characteristic (also known as Lagrangian) coordinates
$(t,\alpha)$ (see \eqref{eq:BurgerCharacterCoord}). Just as solutions
of Burgers' equation remain regular\footnote{Relative to the coordinates $(t,\alpha),$
the solution $\Psi$ to Burgers' equation \eqref{Burger}
verifies $\frac{\partial}{\partial t} \Psi = 0$
and $\frac{\partial}{\partial \alpha} \Psi = \mathring{\Psi}'(\a),$
where $\mathring{\Psi} := \Psi|_{t=0}.$} 
\emph{in the coordinates} $(t,\alpha),$ the shock-forming solutions of \eqref{boeq} remain
regular in the coordinates $(t,u)$; the singularity manifests 
itself only when we change variables back to the $(t,r)$ coordinates
because the Jacobian of the change of variables
map\footnote{Equivalently, if $\upmu \uLunit \Psi$ remains non-zero when
$\upmu\searrow 0$, we must have $|\uLunit\Psi| \nearrow \infty$.} 
contains factors of $\upmu^{-1}.$ 

To reveal the rescaled structure, we first note that from the definition \eqref{def:mu}, we have
\[
\Lunit \upmu^{-1} = \Lunit \partial_t u(t,r) = \partial_t \Lunit u(t,r) -\left[\partial_t, L\right] u(t,r).
\]
From this equation and the identities $\Lunit u(t,r)=0,$ $\uLunit u(t,r)=2\upmu^{-1}$ and $\left[\partial_t , L\right]=\frac{1}{2} \frac{1}{\sqrt{1+\Psi}} \partial_t \Psi \partial_r,$ we deduce that
\begin{equation} \label{mueq}
\Lunit \upmu = -\frac{1}{4} \frac{1}{(1+\Psi)} \upmu \left(\Lunit \Psi + \uLunit\Psi\right) \, .
\end{equation}
Hence, we can rewrite \eqref{E:LOUTSIDE}-\eqref{E:LINSIDE} as
\begin{subequations}
\begin{align}
\Lunit \left( \upmu \uLunit \left( r \Psi \right) \right) &= \frac{1}{2} \frac{r}{(1+\Psi)} \left[(\Lunit \Psi)\upmu\uLunit\Psi  \right] - \frac{1}{2} \frac{\upmu}{\sqrt{1+\Psi}} \Psi \Lunit \Psi,
\label{E:LOUTSIDEREWRITTEN} \\
\upmu \uLunit \Lunit \left( r \Psi \right) &=\frac{1}{4} \frac{r}{(1+\Psi)} \left[ \upmu \left(\Lunit \Psi\right)^2 +3 (\Lunit \Psi) \upmu \uLunit\Psi  \right]+ \frac{1}{2} \frac{1}{\sqrt{1+\Psi}} \Psi \upmu \uLunit \Psi.
\label{E:LINSIDEREWRITTEN}
\end{align}
\end{subequations}
A key point is that all products on the right-hand sides of 
\eqref{E:LOUTSIDEREWRITTEN}-\eqref{E:LINSIDEREWRITTEN}
are expected to decay at an integrable-in-time rate. 
In summary, we have formulated a system of equations
that on the one hand is expected to remain regular 
and exhibit dispersive properties,
and on the other hand
is tailored to see the blow-up of precisely the $\uLunit$ derivative
of the solution as $\upmu \to 0.$

\subsubsection{A sharp classical lifespan result and proof of shock formation}
\label{SSS:SSSHARPCLASSIALLIFESPAN}
The rescaling by $\upmu$ has introduced a partial decoupling of \eqref{boeq} into the wave
equations \eqref{E:LOUTSIDEREWRITTEN}-\eqref{E:LINSIDEREWRITTEN},
which we expect to remain regular, and a transport equation
\eqref{mueq} for the inverse foliation density $\upmu,$ which we
expect to drive the blow-up of $(t,r)$ coordinate derivatives of $\Psi.$ This
allows us to attack the problem of shock formation as a two-step
process, which we now outline.
\begin{enumerate}
\item First, 
we prove ``global-existence-type'' estimates and establish
a breakdown criterion for the system in the small data regime. In particular, 
we will show that classical solutions can be continued as long as $\upmu$ 
remains away from $0.$ 
Furthermore, we will show that when $\upmu \to 0,$  
some of the coordinate derivatives of $\Psi$ must blow-up.
We prove these claims in Prop.~\ref{P:ge}.
\item Next, using the global-existence-type estimates from Step (1),
we can rigorously justify our intuition that the 
$(\uLunit \Psi)^2$ terms in \eqref{E:LOUTSIDE} drives a Riccati-type blow-up.
To this end, we study the transport equation \eqref{mueq} and
show that the right hand side has enough positivity to drive
$\upmu$ to zero in finite time, provided that we sufficiently shrink the amplitude
of the initial data. See Figure
\ref{F:SHOCKFORMATION} for a picture illustrating the formation of the
shock, and Cor.~\ref{C:ge} for the statement.  
\end{enumerate}

\begin{remark}
Following Christodoulou \cite{dC2007}, we will also use this two-step process
in the non-spherically symmetric case. A related approach was also used by
Christodoulou to study the formation of trapped surfaces in general
relativity \cite{dC2009}. The main difficulty in the analysis of the
full problem is precisely establishing an analog to
the sharp classical lifespan Prop.~\ref{P:ge} outside of spherical
symmetry. Once the ``global-existence-type'' estimates 
(that is, analogs of \eqref{E:SSDISPERSIVE}-\eqref{E:SSTRANSVERSALDERIVATIVELARGEINMAGNITUDE} below) 
are established, it is relatively easy to prove a version of the shock-formation
results of Cor.~\ref{C:ge}.  
\end{remark}

We now provide the relevant definition of the solution's lifespan
in the shock formation problem.
\begin{definition}[\textbf{Outgoing classical lifespan}]
	\label{D:OUTGOINGCLASSICALLIFESPAN}
	We define $T_{(Lifespan);U_0},$ the \emph{outgoing classical
	lifespan} of the solution with parameter $U_0,$ 
	to be the supremum over all times $t > 0$
	such that $\Psi$ is a $C^2$ solution 
	(relative to the coordinates $(t,r)$)
	to equation \eqref{boeq} in the strip $\MM_{t,U_0}$ 
	(see Definition \eqref{E:SSSTRIP}).
\end{definition}

We now state the main sharp classical lifespan result for spherically symmetric solutions.

\begin{proposition}[\textbf{A sharp classical lifespan result for equation} \eqref{John:radial}] \label{P:ge}
Let $\mathring{\upepsilon} := \| \mathring{\Psi} \|_{C^2} + \| \mathring{\Psi}_0 \|_{C^1}$
denote the size of the spherically symmetric data \eqref{E:SSDATA}, supported in $\lbrace r \leq 1 \rbrace,$
for the wave equation \eqref{boeq}.
Let $0 \leq U_0 < 1$ be a fixed parameter. Then there exists a constant $\epsilon_0 > 0$ such that 
if $\mathring{\upepsilon} \leq \epsilon_0,$
then we have the following conclusions.
First, the outgoing classical lifespan $T_{(Lifespan);U_0}$
of Definition \ref{D:OUTGOINGCLASSICALLIFESPAN} is characterized by
\begin{align}
	T_{(Lifespan);U_0} = \sup \lbrace t > 0 \ | \ \inf_{\Sigma_t^{U_0}} \upmu > 0 \rbrace.
\end{align}
In addition, there exists a constant $C_{(Lower-Bound)} > 0$ such that
\begin{align}  \label{E:SSLIFESPANLOWERBOUND}
	T_{(Lifespan);U_0} > \exp\left(\frac{1}{C_{(Lower-Bound)}\mathring{\upepsilon}} \right).
\end{align}
Furthermore, there exists a constant $C>0$ 
such that on $\MM_{T_{(Lifespan);U_0}, U_0},$
we have
\begin{equation}\label{E:SSDISPERSIVE}
\begin{gathered}
|r^3 \Lunit^2 \Psi| \leq C \mathring{\upepsilon}, 
\qquad |r^2 \Lunit (\upmu \uLunit \Psi)| \leq C \mathring{\upepsilon}, 
\qquad |r^2 \Lunit \Psi| \leq C \mathring{\upepsilon}, 
\qquad |r \upmu \uLunit \Psi| \leq C \mathring{\upepsilon}, 
\qquad|r \Psi| \leq C \mathring{\upepsilon}, \\
\left|\upmu - 1 \right| \leq C \mathring{\upepsilon} \ln(e + t),
\qquad
\left|1 - r + t - u \right| \leq C \mathring{\upepsilon} \ln(e + t).
\end{gathered}
\end{equation}
Finally, there exists a constant $c > 0$ such that at any point with 
$\upmu < 1/4,$ we have
\begin{align} \label{E:SSLUNITUPMULARGEINMAGNITUDE}
	\Lunit \upmu \leq - c \frac{1}{(1 + t) \ln(e + t)}.
\end{align}
and 
\begin{align} \label{E:SSTRANSVERSALDERIVATIVELARGEINMAGNITUDE}
	\left|
		\upmu \uLunit \Psi
	\right|
	\geq 
		c \frac{1}{(1 + t) \ln(e + t)}.
\end{align}
In particular, it follows from 
\eqref{E:SSTRANSVERSALDERIVATIVELARGEINMAGNITUDE}
that $\uLunit \Psi$ blows up like $\upmu^{-1}$
at points where $\upmu$ vanishes.
\end{proposition}

\begin{remark}[\textbf{The sharp ``constant''}]
	As was first shown by John \cite{fJ1987}
	and H{\"o}rmander \cite{lH1987}
	in Theorem~\ref{T:JOHNHORMANDERLIFESPANLOWER},
	the sharp ``constant''
	$C_{(Lower-Bound)}$
	in \eqref{E:SSLIFESPANLOWERBOUND}
	depends on the profile of the data and
	the structure of the nonlinearities;
	see also equation \eqref{E:ALINHACLIMITINGLIFESPAN}.
\end{remark}

\begin{remark}
The estimate \eqref{E:SSLUNITUPMULARGEINMAGNITUDE} is a quantified version of the 
following rough idea: the only way $\upmu$ can shrink along the integral curves of $\Lunit$ 
is for $\Lunit \upmu$ to be significantly negative.
An interesting aspect is the ``point of no return'' nature of 
this estimate: once $\upmu < \frac14$ (recall that at $t = 0,$ its value is approximately $1$),   
\emph{$\upmu$ must continue to shrink until it eventually vanishes and
a shock forms}. The specific value $\frac{1}{4}$ is not significant: the
actual point of no return depends on $\epsilon_0$ and $\frac{1}{4}$ is just a
convenient number.
\end{remark}

\begin{center}
\begin{overpic}[scale=1.8]{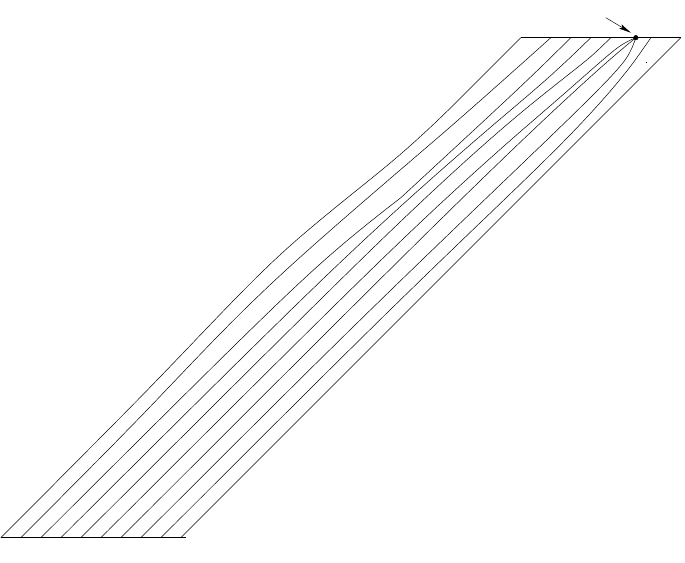}  
\put (87,80.5) {\large $\upmu = 0$}
\put (29,40) {\large $\displaystyle \mathcal{C}_{U_0}$}
\put (66,40) {\large $\displaystyle \mathcal{C}_0$}
\put (14,-.5) {\large $\displaystyle \Sigma_0$}
\put (63,79.3) {\large $\displaystyle \Sigma_{T_{(Lifespan);U_0}}$}
\end{overpic}
\captionof{figure}{Shock formation in spherically symmetric solutions to equation \eqref{boeq}}
 \label{F:SHOCKFORMATION}
\end{center}


\begin{remark}[\textbf{A preview on the Heuristic Principle}]
	Later, when investigating the general non-spherically symmetric case, we will
	encounter dispersive estimates in the spirit of \eqref{E:SSDISPERSIVE},
	complemented  with estimates for the angular derivatives. Such expected   
	estimates, which we refer to
	as the ``Heuristic Principle,'' provide the basic intuition behind 
	our approach in the non-symmetric case. 
\end{remark}

With the help of Prop.~\ref{P:ge}, we can easily derive the following shock-formation result
for spherically symmetric solutions.

\begin{remark}
	For technical reasons, in the corollary, we start with ``initial'' data at time $-1/2$
	supported in $\lbrace r \leq 1/2 \rbrace.$ We will explain this assumption in more detail 
	at the end of the proof of the corollary; see Footnote~\ref{F:WHYTIMEHALF}.
\end{remark}

\begin{corollary} [\textbf{Shock formation for rescaled spherically symmetric data}]  \label{C:ge}
Let $(\check{\Psi},\check{\Psi}_0) \in C^2 \times C^1$ 
be nontrivial spherically symmetric ``initial'' data on $\Sigma_{-1/2}$
that vanish for $r \geq 1/2.$
Let $(\mathring{\Psi} := \Psi|_{\Sigma_0},\mathring{\Psi}_0 := \partial_t \Psi|_{\Sigma_0})$ 
denote the data induced on $\Sigma_0$ by the solution\footnote{If that data on $\Sigma_{-1/2}$ are 
sufficiently small, then the solution will persist until time $0.$} $\Psi.$
Note that $(\mathring{\Psi},\mathring{\Psi}_0)$ vanish for $r \geq 1.$
Then we can chose a $U_0 \in (0,1)$ such that if we
rescale the initial data to be
$(\uplambda \check{\Psi}, \uplambda \check{\Psi}_0)$
for sufficiently small $\uplambda > 0,$
then $(\mathring{\Psi},\mathring{\Psi}_0)$ is small enough 
such that the results of Prop.~\ref{P:ge} apply 
and furthermore, $\Psi$ has a lifespan
$T_{(Lifespan);U_0} < \infty$
due to $\upmu$ vanishing in finite time. 
That is, a shock forms in finite time. 
\end{corollary}

\begin{remark}[\textbf{Maximal development of the data}]
	An important merit of the proofs of
	Prop.~\ref{P:ge}
	and Cor.~\ref{C:ge} is that with some additional
	effort, they can be extended 
	to reveal information beyond the hypersurface $\Sigma_{T_{(Lifespan);U_0}}.$
	That is, they can be extended to reveal 
	a portion of the maximal development of the data up to the boundary;
	see Remark \ref{R:MAXIMALDEVELOPMENT}
	and Figure \ref{F:MAXDEVBOUND}.
\end{remark}

\begin{remark}
	\label{R:MAXIMALDEVELOPMENTROUGHDEF}
	Roughly, the maximal development is the largest possible
  spacetime domain on which there exists a unique classical solution 
	determined by completely the data; see, for example, \cite{wW2013}.
\end{remark}

We now provide the proofs of the proposition and the corollary.

\begin{proof}[Proof of Prop.~\ref{P:ge}]
It suffices to prove \eqref{E:SSDISPERSIVE}-\eqref{E:SSTRANSVERSALDERIVATIVELARGEINMAGNITUDE}
on $\MM_{T_{(Lifespan)}, U_0}.$
For by the identity \eqref{E:CHARDIR}, if $\upmu$ remains uniformly bounded from above and 
from below away from $0,$ then
the estimates \eqref{E:SSDISPERSIVE} imply that
$|\Psi|,$
$|\partial_t \Psi|,$ 
and $|\partial_r \Psi|$ remain uniformly bounded;
it is a standard fact that such bounds allow us to 
extend the solution's lifespan (in a strip of $u$-width $U_0$).

Since our analysis is based on integrating along characteristics, 
we will work relative to the geometric coordinate system 
$(t,u),$ where $u$ is the eikonal function constructed above.\footnote{Note 
that this ``dynamic'' coordinate depends
on the solution itself; this is a feature of the quasilinear nature of
the equations.}
We use a continuity argument:
let $\mathcal{B} \subset \left[0, T_{(Lifespan);U_0} \right)$ be the
subset consisting of those times $T$ such that the estimates
\eqref{E:SSDISPERSIVE} of the proposition hold on $\MM_{T,U_0},$ but
with $C \mathring{\upepsilon}$ replaced by
$\sqrt{\mathring{\upepsilon}}.$ We remark that for $T\in
\mathcal{B},$ we have $r \approx 1 + t$ on $\MM_{T,U_0}.$
For $\upepsilon_0< 1$ sufficiently small, 
$\mathcal{B}$ is a connected, non-empty, relatively closed subset of
$\left[0,T_{(Lifespan)} \right).$  
To show that $\mathcal{B}$ is relatively open, we improve the bootstrap assumptions with a series of estimates that we now derive.

First, we insert the bootstrap assumptions into the right-hand side of
equation \eqref{E:LINSIDEREWRITTEN} to deduce that
\begin{align} \label{E:PREFIRSTESTIMATE}
	\left|
		\upmu \uLunit \Lunit \left(r \Psi \right)
	\right|
	& \leq C \mathring{\upepsilon} \frac{1}{(1 + t)^2}.
\end{align} 
We now integrate inequality \eqref{E:PREFIRSTESTIMATE} 
along the integral curves of $\upmu \uLunit$ relative to the affine parameter $u$
(note that  $\upmu \uLunit (u)= 2$), back to the initial cone $\mathcal{C}_0,$
along which the solution vanishes. Hence, since the strip of interest has eikonal function width $U_0 < 1,$ 
since $\frac{du}{dt}= \upmu^{-1},$
and since $\upmu \leq C \ln(e + t),$ we deduce that 
\begin{align} \label{E:FIRSTESTIMATE}
	\left|
		\Lunit \left( r \Psi \right)
	\right|
	& \leq C \mathring{\upepsilon} \frac{\ln(e + t)}{(1 + t)^2}.
\end{align}
Next, integrating inequality
\eqref{E:FIRSTESTIMATE} from $t=0$
along the integral curves of $\Lunit = \frac{\partial}{\partial t}$
and using the smallness of the data,  
we deduce that $\left|r \Psi \right| \leq C \mathring{\upepsilon}.$ 
In view of the bootstrap assumption corresponding to \eqref{E:SSDISPERSIVE}, we have $r\approx t$
inside our region, and therefore
\begin{align} \label{E:SECONDESTIMATE}
	\left|
		\Psi
	\right|
	& \leq C \mathring{\upepsilon} \frac{1}{1 + t}
\end{align}
as desired. Next, inserting the estimate \eqref{E:SECONDESTIMATE} into
\eqref{E:FIRSTESTIMATE} and using that 
$\left|\Lunit r \right| = \left|\sqrt{1+\Psi} \right| < 2,$ 
we find that
\begin{align} \label{E:THIRDESTIMATE}
	\left|
		\Lunit \Psi
	\right|
	& \leq C \mathring{\upepsilon} \frac{1}{(1 + t)^2}
\end{align}
as desired. 
Next, inserting the estimates
\eqref{E:SECONDESTIMATE}-\eqref{E:THIRDESTIMATE}
and the bootstrap assumptions for $\upmu$
into the right-hand side of \eqref{E:LOUTSIDEREWRITTEN}, 
we find that
\begin{align} \label{E:FOURTHESTIMATE}
	\left|
		\Lunit 
		\left(\upmu \uLunit (r \Psi) \right)
	\right|
	& \leq C \mathring{\upepsilon} \frac{1}{(1 + t)^2}.
\end{align}
Integrating \eqref{E:FOURTHESTIMATE}
along $\mathcal{C}_u$ from $t=0$ 
and using the small-data assumption,
we find that
$\left|\upmu \uLunit \left( r \Psi \right)\right| \leq C \mathring{\upepsilon}.$
Using the bootstrap assumptions, we deduce that $\upmu |\uLunit r| = |\upmu \sqrt{1+\Psi}| \leq C \ln(e + t)$
and hence, thanks to \eqref{E:SECONDESTIMATE}, that
\begin{align} \label{E:FIFTHESTIMATE}
	\left|
		\upmu \uLunit \Psi
	\right|
	& \leq C \mathring{\upepsilon} \frac{1}{1 + t}
\end{align}
as desired. Next, we insert the estimates 
\eqref{E:SECONDESTIMATE},
\eqref{E:THIRDESTIMATE},
and \eqref{E:FIFTHESTIMATE}
and the bootstrap assumption for $\upmu$
into the right-hand side of
equation \eqref{mueq}, 
thereby deducing that
\begin{align} \label{E:LMUBOUNDSS}
	\left|
		\Lunit \upmu
	\right|
	& \leq C \mathring{\upepsilon} \frac{1}{1 + t}.
\end{align}
Integrating \eqref{E:LMUBOUNDSS}
along $\mathcal{C}_u$ from $t=0$ where 
$\left|\upmu - 1 \right| \leq C \mathring{\upepsilon},$ we establish that
\begin{align} \label{E:SIXTHESTIMATE}
\left|
	\upmu - 1 
\right|
\leq C \mathring{\upepsilon} \ln(e + t)
\end{align}
as desired. Next, we note the identity
$\Lunit \left(1 - r + t - u(t,r) \right) = 1 - \sqrt{1 + \Psi}.$ 
Hence, by \eqref{E:SECONDESTIMATE}, we have
\begin{align} \label{E:ALMOSTSEVENTHESTIMATE}
	\left|
		\Lunit \left(1 - r + t - u \right)
	\right|
	& \leq C \mathring{\upepsilon} \frac{1}{1 + t}.
\end{align}
Integrating \eqref{E:ALMOSTSEVENTHESTIMATE} from $t=0,$ where
$u = 1 - r,$ we find that
\begin{align}
\left|
	1 - r + t - u 
\right|
& \leq C \mathring{\upepsilon} \ln(e + t)
\end{align}
as desired. 
Next, using the identity
\begin{align} \label{E:SECONDORDERDERIVATIVEID}
r \Lunit(\upmu \uLunit \Psi)
& = \Lunit \left(\upmu \uLunit (r \Psi) \right)
	- \upmu \uLunit \Psi
	+ \frac{1}{2} \upmu \frac{1}{\sqrt{1 + \Psi}} \Psi \Lunit \Psi
	+ \upmu \sqrt{1 + \Psi} \Lunit \Psi
	+ (\Lunit \upmu) \sqrt{1 + \Psi} \Psi
\end{align}
and the previously proven estimates,
we deduce that
\begin{align} \label{E:LUNITULGOODPSIESTIMATE}
	\left|
		\Lunit(\upmu \uLunit \Psi)
	\right|
	& \leq C \mathring{\upepsilon} \frac{1}{(1 + t)^2}
\end{align}
as desired.
We now show that 
\begin{align} \label{E:UPMUULLSQUAREDRPSIBOUND}
	\left| 
		\upmu \uLunit \Lunit^2 \left(r \Psi \right)
	\right|
	& \leq C \mathring{\upepsilon} \frac{1}{(1 + t)^3}.
\end{align}
To this end, we commute equation \eqref{E:LINSIDEREWRITTEN}
with $\Lunit$ to derive an equation of the form
$\upmu \uLunit \Lunit^2 \left(r \Psi \right) = \cdots.$
To bound the magnitude of $\Lunit$ applied to
the right-hand side of \eqref{E:LINSIDEREWRITTEN}
by $\leq$ the right-hand side of \eqref{E:UPMUULLSQUAREDRPSIBOUND},
we use the bootstrap assumptions and the previously proven estimates.
Similarly, to bound the commutator term 
$[\Lunit, \upmu \uLunit] \Lunit (r \Psi)
= - \upmu \frac{1}{4(1 + \Psi)} (\Lunit \Psi) \Lunit^2 (r \Psi)
- \frac{1}{4(1 + \Psi)} (\upmu \uLunit \Psi) \Lunit^2 (r \Psi)
$
by $\leq$ the right-hand side of \eqref{E:UPMUULLSQUAREDRPSIBOUND},
we use the bootstrap assumptions and the previously proven estimates.
We have thus proved \eqref{E:UPMUULLSQUAREDRPSIBOUND}.
Next, by arguing as in our proof of \eqref{E:FIRSTESTIMATE},
we deduce from \eqref{E:UPMUULLSQUAREDRPSIBOUND} that
\begin{align} \label{E:LUNITSQUARERPSIBOUND}
	\left|
		\Lunit^2 \left(r \Psi \right)
	\right|
	& \leq C \mathring{\upepsilon} \frac{\ln(e + t)}{(1 + t)^3}.
\end{align}
From the identity 
$r \Lunit^2 \Psi = \Lunit^2 \left(r \Psi \right) - 2 \Lunit \Psi$
and the estimates \eqref{E:THIRDESTIMATE}
and \eqref{E:LUNITSQUARERPSIBOUND}, 
we deduce that
\begin{align} \label{E:LUNITSQUAREDPSIESTIMATE}
	\left|
		\Lunit^2 \Psi 
	\right|
	& \leq C \mathring{\upepsilon} \frac{1}{(1 + t)^3}.
\end{align}
We have thus improved the bootstrap assumptions, having shown that
$\sqrt{\mathring{\upepsilon}}$ can be replaced with $C \mathring{\upepsilon},$ as stated in the estimates
\eqref{E:SSDISPERSIVE} of the proposition.

We now prove inequality \eqref{E:SSTRANSVERSALDERIVATIVELARGEINMAGNITUDE}.
First, we multiply the evolution equation \eqref{mueq} 
by $r,$ apply $\Lunit,$ and use
the previously proven estimates, 
including 
\eqref{E:LMUBOUNDSS},
\eqref{E:FOURTHESTIMATE},
and \eqref{E:LUNITSQUAREDPSIESTIMATE},
to deduce that
\begin{align} \label{E:SLOWACCELERATION}
	\left|
		\Lunit (r \Lunit \upmu)
	\right|
	& \leq C \mathring{\upepsilon}
		\frac{\ln(e + t)}{(1 + t)^2}.
\end{align}
Integrating \eqref{E:SLOWACCELERATION} 
from $s$ to $t$
along the integral curves of $\Lunit$
and using 
$r(s,u) = 1 - u + s + \mathcal{O}(\mathring{\upepsilon} \ln(e + s)),$ 
we find that
for $0 \leq s \leq t,$ we have
$
	\left|
		[r \Lunit \upmu](t,u)
		-
		[r \Lunit \upmu](s,u)
	\right|
	\leq C \mathring{\upepsilon} \ln(e + s) (1+s)^{-1}
$
and hence that
\begin{align} \label{E:SSKEYLUNITUPMUEST}
	\Lunit \upmu(s,u) & = \frac{1}{r(s,u)} [r \Lunit \upmu](t,u) + \mathcal{O}\left(\mathring{\upepsilon} \frac{\ln(e + s)}{(1 + s)^2}\right).
\end{align}
Integrating \eqref{E:SSKEYLUNITUPMUEST} from $s = 0$ to $s = t$ 
and using $\left| \upmu(0,u) - 1 \right| \leq C \mathring{\upepsilon},$
we find that for $0 \leq s \leq t,$ we have
\begin{align} \label{E:SSUPMUKEYEXPRESSION}
	\upmu(s,u) = 1 + \ln \left(\frac{1 - u + s}{1-u} \right) [r \Lunit \upmu](t,u) + \mathcal{O}(\mathring{\upepsilon}).
\end{align}
On the other hand, using the previously proven estimates 
and equation \eqref{mueq}, we deduce that
\begin{align} \label{E:SSRLMUKEYESTIMATE}
	[r \Lunit \upmu](s,u)
	& = - \frac{1}{4} [r \upmu \uLunit \Psi](s,u) 
		+ \mathcal{O}\left(\mathring{\upepsilon} \frac{\ln(e+s)}{1 + s} \right).
\end{align}
Combining \eqref{E:SSUPMUKEYEXPRESSION} and \eqref{E:SSRLMUKEYESTIMATE}, we see that
\begin{align} \label{E:SSUPMUKEYEXPRESSIONSLIGHTLYALTERED}
	\upmu(s,u) = 1 - \frac{1}{4} 
	\ln \left(\frac{1 - u + s}{1-u} \right) [r \upmu \uLunit \Psi](t,u) + \mathcal{O}(\mathring{\upepsilon}).
\end{align}
It follows from \eqref{E:SSUPMUKEYEXPRESSIONSLIGHTLYALTERED} that if
$\upmu(t,u) < 1/4$ and $\mathring{\upepsilon}$ is sufficiently small, then
\begin{align} \label{E:SSPROOFTRANSVERSALDERIVATIVELARGEINMAGNITUDE}
	\upmu \uLunit \Psi(t,u)
	& \geq c \frac{1}{(1 + t) \ln (e + t)},
\end{align}
which is the desired estimate \eqref{E:SSTRANSVERSALDERIVATIVELARGEINMAGNITUDE}.
The desired estimate \eqref{E:SSLUNITUPMULARGEINMAGNITUDE} then follows from
inserting the estimate \eqref{E:SSTRANSVERSALDERIVATIVELARGEINMAGNITUDE}
into \eqref{E:SSRLMUKEYESTIMATE}.

\end{proof}

\begin{proof}[Sketch of a proof of Cor.~\ref{C:ge}]
We must show that $\upmu$ vanishes in finite time along at least one of the characteristics $\mathcal{C}_u.$ 
Throughout most of the proof, we work with the data
$(\mathring{\Psi},\mathring{\Psi}_0)$ induced on $\Sigma_0^1$ by the solution.
The parameter $\mathring{\upepsilon}$ appearing throughout this proof is 
by definition $\mathring{\upepsilon} := \| \mathring{\Psi} \|_{C^2} + \| \mathring{\Psi}_0 \|_{C^1},$
as in the statement of Prop.~\ref{P:ge}. We assume that
$\mathring{\upepsilon}$ is small enough that the results of the proposition apply.
As a first step, we insert the estimates
\eqref{E:SSDISPERSIVE} into 
the evolution equation \eqref{E:LOUTSIDEREWRITTEN} and integrate along $\mathcal{C}_u$
to deduce that
\begin{align} \label{E:FIRSTBLOWUPESTIMATE}
	\left|
		\upmu \uLunit \left(r\Psi\right)(t,u)
		- \upmu \uLunit \left(r \Psi\right)(0,u)
	\right|
	& \leq C \mathring{\upepsilon}^2.
\end{align}
Using the estimate $|\upmu \uLunit r| \leq C \ln(e + t),$ 
the estimate
$|\Psi|(t,u) \leq C \mathring{\upepsilon} (1 + t)^{-1},$
the estimate $r(t,u) \approx 1 + t,$
and the fact that at $t=0,$
$\upmu - 1 = \mathcal{O}(\mathring{\upepsilon})$
and $\uLunit = \uLunit_{(Flat)} + \mathcal{O}(\mathring{\upepsilon}) \partial_r,$
we deduce from \eqref{E:FIRSTBLOWUPESTIMATE} that
\begin{align} \label{E:SECONDBLOWUPESTIMATE}
	[r \upmu \uLunit \Psi](t,u) 
	& =[\uLunit_{(Flat)} \left(r \Psi\right)](0,u) 
	+ \mathcal{O}(\mathring{\upepsilon}^2)
	+ \mathcal{O}\left(\mathring{\upepsilon} \frac{\ln(e + t)}{1 + t} \right).
\end{align}
Combining \eqref{E:SECONDBLOWUPESTIMATE} with equation \eqref{E:SSUPMUKEYEXPRESSION}, 
we deduce that
\begin{align} \label{E:BLOWUPSSUPMUKEYEXPRESSION}
	\upmu(t,u) = 1 
	+ \mathcal{O}(\mathring{\upepsilon})
	+ \ln \left(\frac{1 - u + t}{1-u} \right) [\uLunit_{(Flat)} \left(r \Psi\right)](0,u) 
	+ \mathcal{O}(\mathring{\upepsilon}^2) \ln(e + t)
	+ \mathcal{O}\left(\mathring{\upepsilon} \frac{\ln^2(e + t)}{1 + t} \right).
\end{align}
From \eqref{E:BLOWUPSSUPMUKEYEXPRESSION}, we conclude that if
$\uLunit_{(Flat)} \left(r \Psi\right)(0,u)$ is sufficiently negative 
for some $u \in (0,U_0]$ to overwhelm the
$\mathcal{O}(\mathring{\upepsilon}^2)$ term, then $\upmu$ will vanish
in finite time. 
The negativity of $\uLunit_{(Flat)} \left(r \Psi\right)(0,u)$
at some $u_* \in (0,1)$ is an easy consequence of the assumption that the data
given along $\Sigma_{-1/2}$ are nontrivial and compactly supported in the
Euclidean ball of radius $1/2$ centered at the origin.
This fact is roughly a spherically symmetric 
analog of Prop.~\ref{P:JOHNSCRITERIONISALWAYSSATISFIEDFORCOMPACTLYSUPPORTEDDATA};
see \cite{fJ1985} or \cite[Lemmas 22.2.1 and 22.2.2]{jS2014} for more details.\footnote{
As can be discerned from \cite{fJ1985} or \cite[Lemmas 22.2.1 and 22.2.2]{jS2014},
if we had started with data on $\Sigma_0$ supported in $\lbrace r \leq 1\rbrace,$
then the shock might ``want to form'' at a value of $u$ larger than $1,$
that is, in a region where our eikonal function is not defined.
It is for this reason that we started with the data on $\Sigma_{-1/2}$
supported in $\lbrace r \leq 1/2 \rbrace.$
\label{F:WHYTIMEHALF}}
We now run the above analysis with
$U_0$ less than one but greater than $u_*.$ By shrinking the amplitude of the
data as stated in the corollary, we can guarantee that $\mathcal{O}(\mathring{\upepsilon}^2)$ in \eqref{E:BLOWUPSSUPMUKEYEXPRESSION} 
is an ``error term'' compared to
$\uLunit_{(Flat)}\left(r\Psi\right)(0,u_*).$ This guarantees finite-time
shock formation in $\MM_{T_{(Lifespan);U_0},U_0}.$ 
\end{proof}

\subsection{Systems of equations of the form \eqref{general-system}}    
		\label{SS:GENERALSYSTEMEQUATIONS}
As we mentioned earlier, scalar equations of the form \eqref{modeleq:nongeo1}       
can be re-expressed in terms of systems of equations 
of type \eqref{general-system}, where $\Psi$ is the array $\Psi = \pr \Phi.$  
Given the fact that the shocks we are studying 
correspond to singularities of $\pr^2 \Phi$ while $\Psi=\pr \Phi$ remains bounded, 
it is easy to convince ourselves that
\emph{for this purpose},   
 the system of the type \eqref{general-system} is not more difficult to treat 
 than the simplified case of the scalar equation
 \begin{align}
  (g^{-1})^{\a\b}(\Psi) \pr_\a\pr_\b \Psi = \NN(\Psi, \pr \Psi)  \label{general-system-sc},
  \end{align}
  with $g(\Psi) = m + \mathcal{O}(|\Psi|)$ 
  and $\NN(\Psi, \pr \Psi) = \mathcal{O}(|\pr\Psi|^2)$
	in a neighborhood of $(\Psi, \pr \Psi) = (0,0)$
	(and as usual, $m$ is the Minkowski metric).
	We shall thus concentrate our attention on this scalar model, even though some important concepts,
	such as the ``weak  null condition'' (see just below), 
	are more broadly applicable to the full system.

	In Subsect.~\ref{subs:radial-blow}, for F. John's equation \eqref{boeq} in spherical symmetry, 
	we saw that the shock formation is essentially driven by some semilinear terms
	that lead to a Riccati-type blow-up.
  A remarkable fact about scalar equations of type \eqref{general-system-sc} is that      
	in $3D,$ small data blow-up cannot occur\footnote{In contrast,
we note that this statement is false for metrics $g =g(\partial \Phi);$ 
	see Remark \ref{R:BIGDIDFFERENCE}.} without semilinear terms. 
	Indeed, as was first pointed out by H. Lindblad,
  if we drop the nonlinear term on the right-hand side of \eqref{general-system-sc} (in the scalar case), 
  then the remaining quasilinear equation admits global solutions
  (even in 3D) for all sufficiently small initial conditions. In \cite{hL1992},  
  Lindblad proved this
  for spherically symmetric solutions of the model equation
  \begin{align}
  - \partial_t^2 \Psi
	+ c^2(\Psi) \Delta \Psi
		& = 0,   && c^2(0) = 1.
	\end{align}
  The result was later extended  
  by S. Alinhac \cite{sA2003} and H. Lindblad \cite{hL2008}
  to equations of the form\footnote{Lindblad considered the general
case \eqref{eq:Lind-gen} while Alinhac addressed the specific equation $-\partial_t^2 \Psi + (1 + \Psi)^2 \Delta \Psi = 0.$} 
  \bea
  \label{eq:Lind-gen}
  (g^{-1})^{\alpha \beta}(\Psi) \partial_{\alpha} \partial_{\beta} \Psi
	& = 0,
	\eea
	with $g(\Psi) = m + \mathcal{O}(|\Psi|).$
	This
 	result carries over to systems of the form \eqref{general-system}
 for which the nonlinear terms $\NN^I(\Psi,\pr\Psi)$ verify the classic null condition.
 In \cite{hLiR2003}, H. Lindblad and I. Rodnianski further  extended the result to    
 a larger class of nonlinearities $\NN$ that verify     
    the \emph{weak null condition}.\footnote{A system (in $3$ space dimensions) verifies the weak null condition
   	if the corresponding \emph{asymptotic system}
   	admits global solutions; see \cite{hLiR2003}
   	for an explanation of how to form the asymptotic system. The weak null condition manifests
as a structural condition on the nonlinearity $\NN$
for \emph{systems} of wave equations, allowing,
in addition to quadratic terms that verify the classic null condition,
also weakly coupled quadratic terms. For example, systems such as 
    $\square_m \Phi =\pr\Phi \cdot \pr\Psi,$ $\square_m \Psi =0$ 
    (or even $\square_m \Psi =\NN(\Psi, \pr\Psi)$ with $\NN$ verifying the classic null condition) 
    verify the weak null condition.}    Moreover, they showed that
    the weak null condition is verified by the Einstein vacuum equations 
    in the wave coordinate gauge \cite{hLiR2010}.    
   
			Hence, if we are interested      
		   in describing the phenomenon   
		   of small-data shock formation, 
       we must consider either scalar equations of  type \eqref{general-system-sc}    
       with nonlinearities $\NN$ which do not verify the 
       classic null condition, or more generally, systems of the type \eqref{general-system}   
       which do not verify the weak null condition. 
  						A convenient way to generate scalar equations that fail the classic null condition is to  
              rewrite \eqref{general-system-sc} in the geometric form
                \bea
                \label{E:GENERALQUASILINEARWAVE}
                \square_{g(\Psi)} \Psi = \NN(\Psi, \pr \Psi),
                \eea
                with $\square_g$ the standard covariant wave operator associated to the metric $g=g(\Psi).$
               	Note that the term $\NN(\Psi, \pr \Psi)$
								in \eqref{E:GENERALQUASILINEARWAVE}
								is of course different from the corresponding term   
								in \eqref{general-system-sc} in view of the difference between the operators 
								$(g^{-1})^{\a\b} \pr_\a  \pr_\b$ and $\square_g.$
								We already stress here that 
								in order to prove a shock formation result,
								we must make assumptions on the semilinear term $\NN(\Psi, \pr \Psi)$
								in equation \eqref{E:GENERALQUASILINEARWAVE}.
								For example, one could choose $\NN(\Psi, \pr \Psi)$ so that
								\eqref{E:GENERALQUASILINEARWAVE} is equivalent to 
								equation \eqref{eq:Lind-gen}, in which case there would be small-data global existence.
								In Subsubsect.~\ref{SSS:STRUCTURAL},
								we describe some sufficient assumptions 
								on the nonlinearities that lead to small-data shock formation.
								Note that in the particular case when the right hand side of \eqref{E:GENERALQUASILINEARWAVE}   
								is trivial, that is, in the case of the equation $\square_{g(\Psi)}\Psi=0,$
								the non-geometric form of the equation  
								(equation \eqref{general-system-sc})  
								is such that the corresponding right hand side does \emph{not} verify, 
								except in trivial cases, the null condition; see the discussion in Subsubsect.~\ref{SSS:FAILUREOFCLASSICNULL}.
								As a simple example to 
								keep in mind, consider the equation      
                 $\square_{g(\Psi)} \Psi = 0$ in the case 
                 of the metric
                  \begin{align} 
                  -dt^2 & + c^2(\Psi) \sum_{a=1}^3 (dx^a)^2. \label{Lindblad-metric}
                  \end{align} The non-geometric form of the equation is:
									\begin{align} \label{E:LINDBLADREWRITTEN}
										(g^{-1})^{\alpha \beta}(\Psi) \partial_{\alpha} \partial_{\beta} \Psi
										& = - (g^{-1})^{\alpha \beta}(\Psi) \partial_{\alpha} \ln c(\Psi) \partial_{\beta} \Psi
											+ 2 \partial_t \ln c(\Psi) \partial_t \Psi.
									\end{align}
									The first term on the right-hand side of \eqref{E:LINDBLADREWRITTEN}
									verifies the classic null condition, and if not for the second term,
									the methods of Alinhac \cite{sA2003} and Lindblad \cite{hL2008}
									would lead to small-data global existence. However, the term
									$2 \partial_t \ln c(\Psi) \partial_t \Psi$
									does not verify the classic null condition
									and causes the finite-time shock formation.
                  We remark that equation \eqref{E:LINDBLADREWRITTEN} admits 
									spherically symmetric solutions 
                  whose finite time blow-up can be analyzed 
                 	by employing essentially the same strategy that F. John used 
                 	to study equation \eqref{John-model},
                 	or by using the sharper strategy described in Subsect.~\ref{subs:radial-blow}.

							Finally, we note, 
              that although one can establish small-data global existence for     
              Lindblad's equation, and more generally for systems of type  
              \eqref{general-system} verifying the weak null condition,
              the resulting solutions sometimes verify weaker peeling
							properties than the ones \eqref{eq:peeling} corresponding to the linear wave equation.
              Alinhac refers to the distorted asymptotic behavior as
							``blow-up at infinity;''
              see, for example, \cite{sA2003}. 
               In the case of Lindblad's scalar equation, 
               this effect can only be generated by the quasilinear (principal)  
               part of the equation and is in fact due to the nontrivial asymptotic behavior of the     
               \emph{null (characteristic) hypersurfaces}, which are levels sets of a solution $u$    
               to the following \emph{eikonal equation}\footnote{In the next subsection,
               we discuss the eikonal equation in more detail.}            
								\begin{align}
                   \label{eikonal-psi}
                   (g^{-1})^{\a\b}(\Psi) \pr_\a u \pr_\b u=0.
                \end{align}
              	Solutions to \eqref{eikonal-psi} are analogs of the 
              	coordinate $u$ constructed in Subsect.~\ref{subs:radial-blow}
              	in spherical symmetry. They will play a major role
              	in all of the remaining discussion in this article.

	 \subsection{Why is the proof of shock formation so much harder in the general case?}        
					\label{SS:HARDER}
         	The short answer is simply this: because
          the spherically symmetric problem is truly $1+1$ dimensional and therefore
           	one can rely almost exclusively on the method of characteristics, 
           	a method which is in itself insufficient in higher dimensions.  
             After his blow-up work in spherical symmetry \cite{fJ1985}, 
             F. John tried to extend it
             by treating the general case as a perturbation        
             of the spherically symmetric one. In particular, 
						in treating the general case,
						he used radial characteristic curves corresponding to the truncated
						problem in which angular derivatives are set equal to $0,$
						which are not true characteristics for the actual  equation.
             At first glance, this seems
             reasonable since, in view of the peeling properties \eqref{eq:peeling},  
             we may expect that the angular derivatives decay faster and thus,  
             for large values of $t,$  
             the radial behavior dominates.    
             The problem with such a strategy is that it is not so easy to verify    
             that the angular derivatives  are indeed negligible. Actually, 
             Christodoulou's work \cite{dC2007} 
             and the third author's work \cite{jS2014}
             allow for the possibility 
             that the standard angular derivatives of the solution
             along the Euclidean spheres are non-negligible at late time
             and in fact \emph{they can blow up} when the shock forms! 
             The reason is that they can contain a small
             component that is transversal to the actual
						characteristic hypersurfaces, 
             and it is exactly this transversal
             derivative of the solution that blows up.               
             	\subsubsection{Eikonal functions in $3D$} \label{SSS:EIKONALIN3D}
             		In Subsect.~\ref{sect-Compr-dispers}, we outlined how to derive    
                   the decay properties of solutions to higher-dimensional nonlinear wave equations 
                   using a version of the vectorfield method that relies on the Killing and conformal Killing vectorfields 
                   $\mathscr{Z}_{(Flat)}$ of Minkowski spacetime. 
                    These vectorfields are well-adapted to
                    $u_{(Flat)} := t - r,$ which is an eikonal function of the Minkowski metric
                    (whose level sets are characteristics for the linear wave equation).
                    In general, these vectorfields are not suitable for
                    studying quasilinear wave equations, whose characteristics
										may be very different from those of solutions to the linear wave equation. 
										We were able to use the Minkowskian vectorfields 
										$\mathscr{Z}_{(Flat)}$
										in the proofs of
										small-data global and almost-global existence theorems, essentially
because we worked within spacetime regions where we can 
use a bootstrap procedure based on the peeling estimates \eqref{eq:peeling}
to control the difference between the actual characteristics and the Minkowskian ones.

In contrast, in the shock formation problem,
one is studying spacetime regions where the 
true characteristics are catastrophically diverging from
those of the linear wave equation. Therefore, there is no
reason to hope that we can derive good
peeling estimates by commuting with
vectorfields in $\mathscr{Z}_{(Flat)}.$
In the study of the linear wave equation
the peeling estimates \eqref{eq:peeling} are adapted to the outgoing Minkowskian null cones.
These cones are level sets of the eikonal function $u_{(Flat)},$
which solves the eikonal equation $(m^{-1})^{\a\b}\pr_\a u\pr_\b u=0.$
To study shock formation, it is natural then to replace $u_{(Flat)}$
by an appropriate outgoing solution of the eikonal 
equation of the dynamic metric $g(\pr\Phi)$ of equation \eqref{modeleq:nongeo1}:
           \bea
                   \label{eikonal-phi}
                   (g^{-1})^{\a\b}(\pr\Phi)\pr_\a u\pr_\b u=0
                   \eea
                   or \eqref{eikonal-psi}
                   in the case of equations of type \eqref{E:GENERALQUASILINEARWAVE}.
                   The hope is that $u$ will serve as a good coordinate,
                   as it did in Subsect.~\ref{subs:radial-blow}
              	   in spherical symmetry.
     
     				        Note that in the particular case of John's equation \eqref{boeq}, 
     				        or the equation 
									$\square_{g(\Psi)}
										\Psi = 0$ with Lindblad's metric \eqref{Lindblad-metric},
                   the eikonal equation takes the form
                   \bea\label{E:EIKONALLABELCONV}
                   (\pr_t u)^2= c^2(\Psi) |\nab  u|^2.
                   \eea
		  Suppose now that $u$ is a solution to the eikonal
			equation with $\partial_t u > 0$ and such that
			at each fixed time $t$, the level sets are
embedded $2$-spheres. We say that $u$ is outgoing if the spatial
gradient $\nabla u$ is inward-pointing, and incoming
		otherwise.\footnote{In Minkowski spacetime $t-r$ is outgoing and $t+r$ is
incoming; the terminology refers to the direction of travel of the
level sets as time flows forward. We typically denote outgoing solutions by $u$ and incoming
ones by $\ub$.}    Note that if $\Psi$ is spherically symmetric, 
                   then we can also choose a pair of eikonal functions
										$u$ and $\ub$, respectively outgoing and incoming, to be spherically symmetric, 
									that is, $(\pr_tu)^2 =c^2(\Psi) (\pr_r u)^2$ 
									and likewise for $\ub.$ 
									These symmetric eikonal functions are       
                  completely determined by the radial characteristics that
                  played a crucial role in F. John's work \cite{fJ1985} 
                  and in our argument in Subsect.~\ref{subs:radial-blow}. 
                  In the general non-spherically symmetric case, 
             			we will use a non-degenerate (i.e., $\partial u \neq \textbf{0}$) outgoing
									eikonal function $u$ to construct the adapted
									vectorfields needed to derive peeling estimates. 
									Starting in Subsubsect.~\ref{susub:prev-eikonal},
									we describe the many technical difficulties that accompany the use of an eikonal function
									in the general case.
                          
              \subsubsection{A preview on the vectorfield method tied to an eikonal function $u$}      
              \label{susub:prev-eikonal}          
              From now on,
              we shall primarily discuss equations of type                  
              \eqref{E:GENERALQUASILINEARWAVE}
              under assumptions\footnote{See
Subsubsect.~\ref{SSS:STRUCTURAL}.} that lead to small-data shock formation.
	 	Following the strategy described in
Subsubsect.~\ref{SSS:RESCALINGGIVESTHENULLCONDITION}, we aim to derive peeling
estimates, similar to those in \eqref{eq:peeling}, for a rescaled
problem, with the aid of vectorfields $Z$ adapted to an eikonal
function that have good commuting properties with the covariant wave operator  
               $\square_g.$ Of course,
                       we cannot expect vanishing commutators as in the flat case;    
                       we can only hope to control the error terms
generated by the commutation.      

             To begin, we note the following general formula for the commutator 
		between $\square_g $ and an arbitrary vectorfield $Z:$
   \bea
   \label{eq:comm-Z-square}
   \square_g(Z\Psi) &= Z(\square_g \Psi)- \piZ \c \D^2  \Psi  + (\D  \piZ)\c \D \Psi,
   \eea
   where $\D$ denotes the Levi-Civita connection corresponding to the metric $g,$ 
   $\piZ$ denotes the deformation tensor of the vectorfield $Z,$ that is,
   \begin{align} \label{E:DEFORMATIONTESNORDEF}
   \piZ_{\a\b}:=\Lie_Z g_{\a\b}= \D_\a Z_\b+\D_\b Z_\a,
   \end{align}
   where $\Lie_Z$ denotes Lie differentiation with respect to $Z.$
	The term $(\D \piZ) \c \D \Psi$ schematically denotes tensorial products between first covariant derivatives of $\piZ$ and  the first 
	derivatives of $\Psi,$ and similarly for the term $\piZ \c \D^2 \Psi.$ 

	In Subsect.~\ref{SS:STRATEGYFORGENERALIZEDENERGY}, 
	we will describe the commutator
	vectorfields $\mathscr{Z}$ needed in the shock-formation problem.
	For illustrative purposes, we discuss here
	a subset of the commutators that we use: the rotations $\Rot.$
	The simplest way to define good rotation vectorfields
  $\Rot$ tied to the eikonal function $u$
  is to use Christodoulou's strategy \cite{dC2007}
  by projecting, using the metric $g,$ the Euclidean rotation vectorfields\footnote{
  We recall from \eqref{E:MINKOWSKICONFORMALKILLING} 
  that the Euclidean rotations are defined relative to the 
  standard rectangular coordinates.} $O_{(Flat;ij)}$ 
  onto the intersection of the level sets of $u$ with $\Sigma_t$ (the hypersurfaces of constant $t$ in Minkowski space). 
  The projection operator can be constructed with the help of
	the null geodesic vectorfield $\Lgeo :=-(g^{-1})^{\a\b}(\Psi) \pr_\b u\pr_\a$
	corresponding to $u.$ 
		 Thus, the projection operator depends on $\Psi$ and the first rectangular derivatives of $u.$
     It is then easy to see that the deformation    
     tensor $\piO$  must depend on the first derivatives of $\Psi$ and the 
     Hessian $H := \D^2 u.$
		Therefore, the term $\D \piO$ appearing on the right hand side  
     of the equation
       \bea
       \label{eq:first-comm}
       \square_{g(\Psi)}(\Rot \Psi)&=& \Rot (\square_{g(\Psi)} \Psi) + \piO \c \D^2 \Psi +
      (\D \piO)  \c \D \Psi
       \eea 
        depends on the second derivatives of $\Psi$ and the
				\emph{third derivatives} of $u.$ 
				Hence, to close $L^2$ estimates at a consistent level
				of derivatives,
				we need to make sure that we can estimate     
        the third derivatives of $u$ in terms of two derivatives of $\Psi.$
				Note that by equation \eqref{eikonal-psi},    
				$u$ depends on $\Psi.$
				At first glance of equation \eqref{eikonal-psi}, one might
        believe in the heuristic relationship $\partial u \sim \Psi$ and hence
        $\partial^3 u \sim \partial^2 \Psi,$ which is the desired 
        degree of differentiability. 
        However, as we explain below, only a weakened version, 
        just barely sufficient for our purposes, of these
        relationships is true. Furthermore, the weakened version is 
        quite difficult to prove.

To flesh out the difficulty, we first note that one can derive
a Riccati-type matrix evolution equation for $H$ of the schematic form
                                           \bea
                                           \Lgeo H +  H^2=  \mathcal{R}, \label{eq:ricatti-H}
                                           \eea 
         where $\mathcal{R}$ depends\footnote{$\mathcal{R}$ is in fact the Riemann curvature tensor of the metric $g(\Psi)$ contracted twice  with the vectorfield $\Lgeo.$} on up-to-second-order derivatives  of $\Psi$
					and  up-to-second-order derivatives of $u.$  
				Ignoring for now the Riccati-type term $H^2,$ which actually plays a crucial role in the blow-up mechanism,  
				we note that the obvious way to estimate
         $H$ is by integrating the curvature term $\mathcal{R}$ along the integral curves of $\Lgeo.$
         The obstacle is that this argument only allows one to conclude that $H$ has the same degree of 
         differentiability, in directions transversal to $\Lunit,$ as $\mathcal{R}.$   
          In particular, using this argument,
          we can only estimate $H$ in terms of two derivatives of $\Psi,$
           which makes $\D \piO$ dependent on \emph{three} derivatives of $\Psi.$ 
           Thus, the term $(\D \piO) \c \D \Psi$ is far from being 
           a lower-order term as one would hope. It in fact 
           seems to be an \emph{above top-order term} that obstructs closure of the estimates. 
					This appears to make equation \eqref{eq:first-comm} useless
					and casts doubt on the desired differentiability $\partial^3 u \sim \partial^2 \Psi.$  
				This difficulty in deriving good differentiability properties of $u$ 
			may be the reason that F. John was not able to extend 
				the vectorfield method to study non-spherically symmetric blow-up.
   	As we explain in Subsubsect.~\ref{SSS:AVOIDINGTOPORDERDERIVATIVELOSS}, 
		this loss can be overcome by carefully exploiting 
		some special tensorial structures present in the components of $\D \piO$
		and the components of equation \eqref{eq:ricatti-H},
		and by using elliptic estimates. 
		Some of these special structures are closely tied to the
		fact that our commutators $Z$ are adapted to the eikonal function
		$u;$ see Remark \ref{R:ERRORTERMSNOTPRESENT}.

    \subsubsection{Connections between the proof of shock formation and the proof of the stability of the Minkowski space.}                
    The first successful use of null (characteristic) hypersurfaces     
in a global nonlinear evolution problem appeared in the proof  
of the nonlinear stability of the Minkowski space \cite{dCsK1993}.         
The properties of an exact, carefully constructed   
eikonal function $u$ were crucial for building
approximate Killing and conformal Killing vectorfields to
replace\footnote{In \cite{hLiR2005, hLiR2010}, Lindblad and Rodnianski
were able to prove a weaker (based on weaker peeling estimates)    
version of the stability of Minkowski space using the
Minkowskian vectorfields instead of ones adapted to the dynamic
geometry.}
those appearing in Subsubsect.~\ref{SSS:BEYONDLOCALEXISTENCE}. 
 These vectorfields were then used to derive       
 	generalized energy estimates and the peeling properties 
  of the Riemann curvature tensor $\mathcal{R}$ of the metric $g,$ 
  much like the linear peeling properties \eqref{eq:peeling}.     
  The non-vanishing nature of the commutation   
  of these carefully constructed vectorfields with Einstein's field equations   
  is measured by their deformation tensors \eqref{E:DEFORMATIONTESNORDEF},     
  which in turn depend on the properties of various higher-order derivatives 
  of the eikonal function $u.$ The Hessian $H$    
   of $u$ verifies an equation of type \eqref{eq:ricatti-H} and hence
its regularity and decay properties again depends on those of the
curvature tensor. The apparently loss of derivatives mentioned above
also appears and is overcome via a
renormalization procedure and elliptic estimates.
As we shall see, a similar procedure allows one to avoid derivative loss
in the shock formation problem, but it is more difficult to implement.
Although the basic ideas in the proof of \cite{dCsK1993} are simple and compelling, the proof 
   required a complicated and laborious bootstrap argument in which one uses\footnote{The description given here is
    vastly  simplified. There is another layer of complexity connected  to the choice of the time function $t,$ which like $u,$ is dynamically constructed. 
	Note that unlike the case of general relativity, in
\cite{dC2007} and in \cite{jS2014} there is a preferred physical
time function $t$ from the background Minkowski spacetime.}      the  expected properties of 
     the curvature tensor to derive estimates for various derivatives   
     of the eikonal function  and, based on them,   
     precise estimates for the deformation tensors of the adapted vectorfields mentioned above.  
		These vectorfields are then used  to derive generalized energy estimates for various components of the curvature tensor, 
		which are $L^2$ analogs of the peeling estimates.   
		The main error terms, which appear in these curvature estimates,  
		are controlled by a procedure similar to, but much more subtle,  
		than the one we have described in Subsubsect.~\ref{SSS:CLASSICNULL}
		for wave equations with nonlinearities that verify the classic null condition.   
		Just as in the case of these wave equations, in the Einstein equations,  
		the nonlinear terms are such that
    the most dangerous error terms that could in principle appear in the generalized energy estimates
    are \emph{not present} due to the special structure of the equations relative to the dynamic coordinates $t$ and $u.$

    \section{The main ideas behind the analysis of shock-forming solutions in $3D$}
      			\label{S:MAINIDEASIN3D}
				
				We now outline some of the new difficulties encountered
				and the key ingredients that Christodoulou used to overcome 
				them \cite{dC2007} when extending the proof of shock formation
				from the spherically symmetry case 
				(see Subsect.~\ref{subs:radial-blow})
				to the general case.

			\begin{enumerate}
      \item \textbf{(Dynamic geometric objects, dependent on the solution)} 
			As in the proof of the stability of Minkowski spacetime \cite{dCsK1993},
			the proof of shock formation uses a true outgoing eikonal
			function $u$ corresponding to the dynamic metric 
			and a collection of vectorfields dynamically adapted to it.

\item \textbf{(Inverse foliation density and shock formation)} 
As in the case of spherical symmetry,
shock formation is caused by the degeneracy of $u$ as measured
by the density of its level surfaces relative to the
Minkowskian time coordinate $t,$ captured by the inverse foliation density 
$\upmu$ going to $0$ in finite time 
(see Definition \ref{D:INVERSEFOLIATIONDENSITY} below). 

\item \textbf{(Peeling and sharp classical lifespan in rescaled frame)} 
At the heart of Christodoulou's entire approach lies a sharp 
classical lifespan result according to which solutions can be extended as long as   
$\upmu$ does not vanish.\footnote{The reader should keep in mind 
the simpler case of spherical symmetry discussed above (see Prop.~\ref{P:ge}).}   
To derive such a result, one needs to re-express the evolution equations as a coupled
system between the \emph{nonlinear wave equation}, expressed relative to a
$\upmu$-rescaled vectorfield frame, together with a \emph{nonlinear transport equation}
describing the evolution of the eikonal function $u$ (and hence, by extension, of $\upmu$).
In this formulation, 
the $\upmu$-rescaled wave equation
no longer exhibits the dangerous slow-decaying quadratic term
analogous to the term $(\uLunit \Psi)^2$ from spherical symmetry
(see Remark \ref{R:SSRICCATI}).
To prove the desired sharp classical lifespan result,
we need to show that the lower-order derivatives of 
the solution behave according to the linear peeling estimates
\eqref{eq:peeling}. To establish such peeling estimates,
one needs to rely on appropriate energy  estimates
for derivatives of the solution with respect to the $u$-adapted vectorfields.
The main technical difficulty one needs to overcome
is that the energy norms of the highest derivatives
can degenerate with respect to $\upmu^{-1},$ as we discuss in point (4).
\item \textbf{(Generalized energy estimates)}
To establish the desired energy-type estimates,
we need to commute the wave equation    
a large number of times with the adapted vectorfields,\footnote{\label{FN:NUMBEROFDERIVATIVES} Christodoulou 
did not give explicit bounds on the number of
commutations needed to close the estimates in \cite{dC2007}.
In \cite{jS2014}, the third author used
24 commutations. This may be further optimized.}
a procedure which not only generates a huge number of error
terms, but also seems to lead to a loss of derivatives. 
To overcome this apparent loss of derivatives
at the top order (see Subsubsect.~\ref{susub:prev-eikonal}), 
Christodoulou uses renormalizations and $2D$ elliptic
estimates, in the spirit\footnote{In his work \cite{dC2007}, Christodoulou
						recognizes that similar renormalization procedures can be done 
  					 in the context of nonlinear wave equations of type \eqref{modeleq:nongeo} and
						\eqref{E:SPECIFICSEMILINEARTERMSGENERALQUASILINEARWAVE}.
						A similar observation had previously been used in \cite{sKiR2003} in the context of quasilinear wave  
						equations similar to \eqref{E:SPECIFICSEMILINEARTERMSGENERALQUASILINEARWAVE}  
						to derive a low regularity local well-posedness result.}  of \cite{dCsK1993}. 
						The price one pays for renormalizing is the introduction of 
						a factor of $\upmu^{-1}$ at the top order. This 
						leads to $\upmu^{-1}$-degenerate high-order
						$L^2$ estimates; see 
						Prop.~\ref{P:APRIORIENERGYESTIMATES} and
						Subsubsect.~\ref{SSS:AVOIDINGTOPORDERDERIVATIVELOSS}.
						\emph{Establishing these degenerate high-order $L^2$ estimates
						and showing that the degeneracy does not propagate down to the lower levels 
						are the main new advances  of \cite{dC2007}.}
\end{enumerate}

In Sect.~\ref{S:MAINIDEASIN3D},
we describe the implementation
of points (1) and (2). In connection with point (3),
we also state the 
\emph{Heuristic Principle}, which is a collection of peeling estimates that play an
important role in controlling error terms in the proof. 
The proof of the Heuristic Principle is based on the generalized energy estimates mentioned in point (4)
and Sobolev embedding. Because the derivation of generalized energy estimates 
is the most difficult aspect of the
proof, we dedicate all of Sect.~\ref{S:GENERALIZEDENERGY} to
outlining the central ideas. This step is where the proof
deviates the most from the spherically symmetric case. 
In Sect.~\ref{S:SHARPLIFESPAN}, 
we summarize the sharp classical lifespan theorem,\footnote{We state the version of the theorem from
					 	\cite{jS2014}, which applies to the scalar equations
					 $\square_{g(\Psi)} \Psi = 0.$}
						which is the main ingredient needed to 
						show that a shock actually forms.
						We also outline its proof and indicate the role of
						the estimates described in the previous sections. 
						In Sect.~\ref{S:SHOCKFORMATIONANDCOMPARISON},
					we compare the results of Christodoulou to those of Alinhac and discuss some of the 
					new results in \cite{jS2014}.

      \subsection{Basic geometric notions and set-up of the problem without symmetry assumptions}
      \label{SS:BASICGEOMETRICNOTIONS}
Motivated by the discussion at the beginning of Subsect.\
\ref{SS:GENERALSYSTEMEQUATIONS},
from now until Subsect.~\ref{SS:EXTENSIONSOFTHESHARPCLASSICALLIFESPANTHEOREMTOALINHACSEQUATION},
we consider the model \emph{scalar} wave equation
of the form \eqref{E:GENERALQUASILINEARWAVE}
under the assumption 
 \begin{align}
 	g_{\alpha \beta}(\Psi = 0) = m_{\alpha \beta},
 \end{align}
 where $m_{\alpha \beta} = \mbox{diag}(-1,1,1,1)$ denotes the Minkowski metric.
 Furthermore, we assume that the semilinear terms on the right-hand side are quadratic in 
 $\partial \Psi$ with coefficients depending on $\Psi:$
 \begin{align} \label{E:SPECIFICSEMILINEARTERMSGENERALQUASILINEARWAVE}
  \square_{g(\Psi)}\Psi &= \NN(\Psi)(\partial \Psi, \partial \Psi).
 \end{align}
 Above,
	$\square_{g(\Psi)} :=(g^{-1})^{\a\b} \D_\a \pr_\b$
denotes the covariant wave operator of $g(\Psi)$
and $\D$ denotes the Levi-Civita connection\footnote{Throughout,
$(g^{-1})^{\a\b}(\Psi)$ denotes the inverse metric. That is,
$g_{\a\b}(g^{-1})^{\b\ga}=\de_\a^\ga.$}    of $g(\Psi),$ given in
coordinates by
\begin{align}
\D_\a \pr_\b\Psi 
	&= \pr_\a\pr_\b \Psi
	- \Gapsi_{\a \ \b}^{\ \la} \pr_\la \Psi.
\end{align}
Above, $\Ga = \Gapsi$ denotes a Christoffel symbol of $g(\Psi),$
\begin{align}
\Gapsi_{\a \ \b}^{\ \la}
& := \frac{1}{2} 
			(g^{-1})^{\la \si}(\Psi)
			\left\lbrace
				\pr_\a (g_{\si\b}(\Psi))
				+ \pr_\b( g_{\a\si}(\Psi))-\pr_\si (g_{\a\b}(\Psi))
			\right\rbrace\label{eq:Gapsi}\\
& = \frac{1}{2} (g^{-1})^{\la \si} 
	\left\lbrace  
		G_{\si\b} \pr_\a \Psi 
		+ G_{\a\si} \pr_\b\Psi  
		- G_{\a\b}\pr_\si \Psi 
	\right\rbrace
	\nn,
\end{align}
where
\begin{align}
\label{def:G}
G_{\mu\nu} & = G_{\mu\nu}(\Psi) 
:=\frac{d}{d\Psi} g_{\mu\nu}(\Psi).
\end{align}
Thus, the left-hand side of \eqref{E:SPECIFICSEMILINEARTERMSGENERALQUASILINEARWAVE} can be written as
\begin{align} \label{E:COVWAVEOPERATORINRECTANGULAR}
 \square_{g(\Psi)} \Psi 
 & =(g^{-1})^{\alpha \beta} \partial_{\alpha} \partial_{\beta} \Psi 
 		- \frac{1}{2} 
 			(g^{-1})^{\alpha \beta} 
 			(g^{-1})^{\lambda \sigma}
 		\left\lbrace  
 			G_{\sigma \beta} \partial_{\alpha} \Psi 
 			+ G_{\alpha \sigma} \partial_{\beta} \Psi  
 			- G_{\alpha \beta} \partial_{\sigma} \Psi 
 		\right\rbrace
 		\partial_{\lambda} \Psi.
\end{align}
Without loss of generality\footnote{It is straightforward to see that one component of the metric can
always be fixed by a conformal rescaling. This conformal rescaling generates an additional
semilinear term that verifies the future strong null condition of Subsubsect.~\ref{SSS:STRUCTURAL}; 
as we later explain, such terms have negligible effect on the dynamics.}
we make the following assumption,
which simplifies some of the calculations:
\begin{align} \label{E:METRICNORMALIZATION}
	(g^{-1})^{00}(\Psi) \equiv - 1.
\end{align}

To state and prove the main theorems, we assume for convenience that the 
initial data are 
supported in the Euclidean unit ball.
As in the case of spherical symmetry, we fix a constant $U_0 \in (0,1).$
We will study the solution in a spacetime region that is 
evolutionarily determined by the 
portion of the nontrivial part of the data lying 
in the region $\Sigma_0^{U_0},$
which is the annular subset of $\Sigma_0$ 
bounded between the inner sphere of Euclidean radius $1 - U_0$ and 
the outer sphere of Euclidean radius $1.$
The spacetime region of interest is bounded by
the inner null cone $\mathcal{C}_{U_0}$ and the outer 
null cone $\mathcal{C}_0,$ where $\mathcal{C}_0$ is ``flat'' (i.e.\
Minkowskian)
because the solution $\Psi$ completely vanishes in 
its exterior; see Figure \ref{F:REGION}. 
The region is an analog of the  
spherically symmetric region $\MM_{t,U_0}$ encountered in Subsect.\
\ref{subs:radial-blow}.

\begin{center}
\begin{overpic}[scale=.2]{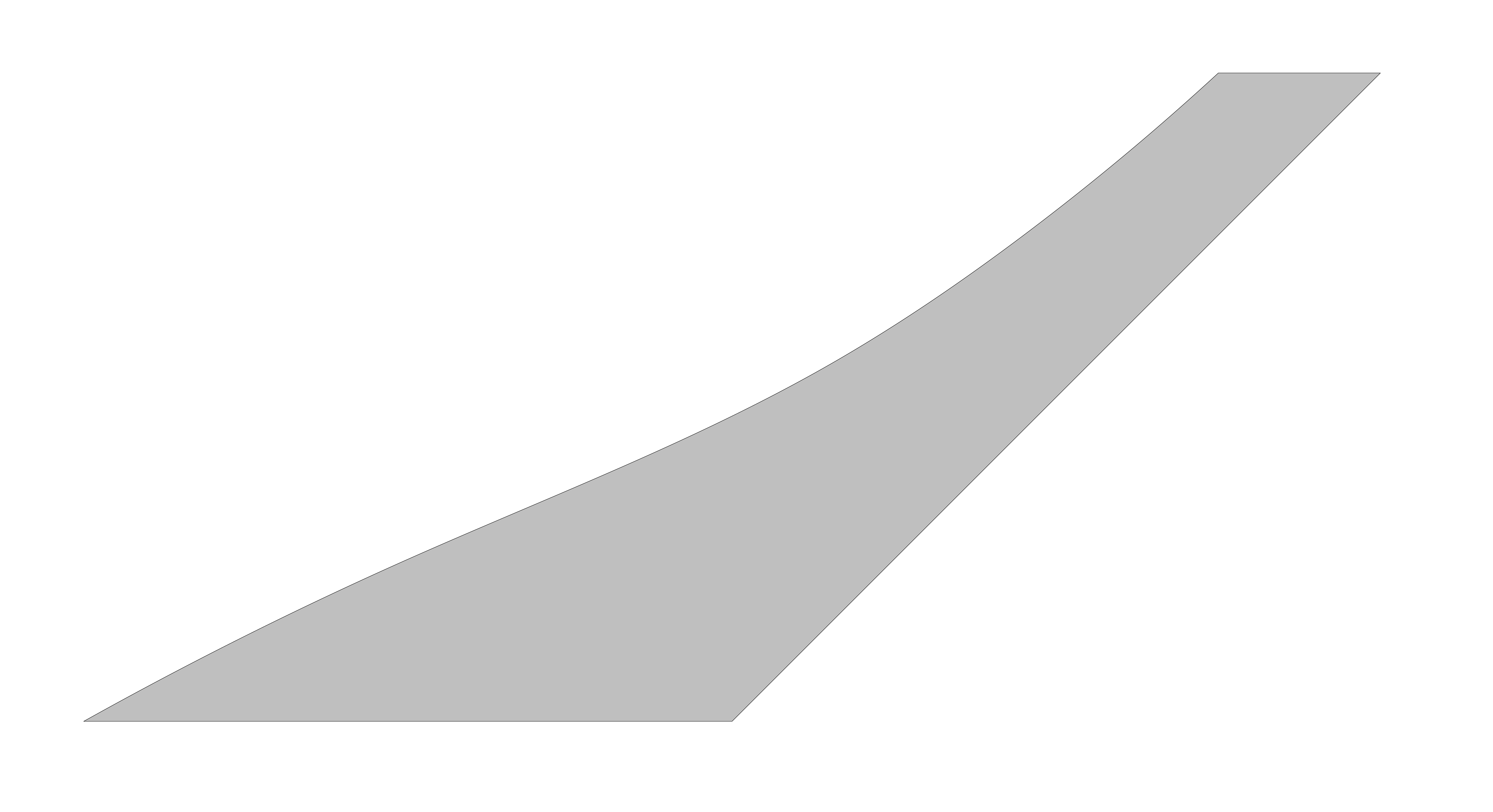}  
\put (75,15) {\large$\Psi \equiv 0$}
\put (30,2) {\large $\displaystyle \Sigma_0^{U_0}$}
\put (46.5,29) {\large $\displaystyle \mathcal{C}_{U_0}$}
\put (74.5,29) {\large $\displaystyle \mathcal{C}_0$}
\end{overpic}
\captionof{figure}{The region of study at a fixed angle.}
 \label{F:REGION}
\end{center}

To prove that small-data shock formation occurs in solutions 
to \eqref{E:SPECIFICSEMILINEARTERMSGENERALQUASILINEARWAVE}, we must make
assumptions on the structure of 
$G_{\mu\nu}$ as well as the ($g-$null) structure of $\NN(\Psi)(\partial \Psi, \partial \Psi).$
We shall give precise conditions in Subsubsect.~\ref{SSS:STRUCTURAL}, after we introduce 
some basic notions.

\subsubsection{The eikonal function, adapted frames, and the Heuristic Principle} 
	\label{SSS:EIKONALFUNCTION}
We start with an outgoing eikonal function, that is,
a solution $u$ of the eikonal equation
\begin{equation} \label{eikonal}
(g^{-1})^{\a\b}(\Psi) \partial_{\alpha} u \partial_\beta u = 0,
\end{equation}
with $\partial_t u > 0,$ subject to the initial condition
\begin{align} \label{E:EIKONALINITIALCONDITIONS}
	u|_{t=0} = 1 - r,
\end{align}
where $r$ denotes the Euclidean radial coordinate on $\mathbb{R}^3.$
We stress that \eqref{eikonal} can be viewed as a \emph{nonlinear
transport} equation to be solved in
conjunction with the wave equation \eqref{E:SPECIFICSEMILINEARTERMSGENERALQUASILINEARWAVE}.
This uniquely-defined $u$ 
is a perturbation of the flat eikonal function $u_{(Flat)} = 1 - r + t,$
where $t$ is the Minkowski time coordinate.
We associate to $u$ its  gradient vectorfield  
\begin{align} \label{E:NULLGEODESIC}
\Lgeo^{\nu} :=- (g^{-1})^{\nu \alpha} \partial_{\alpha} u.
\end{align}
Since $\D g = 0,$ it follows that $\Lgeo$ is null and geodesic, that is,
$g(\Lgeo,\Lgeo)=0$ and
\bea \label{E:INTROGEODESIC}
\D_{\Lgeo}\Lgeo=0.
\eea
Relative to the rectangular coordinates $x^\nu,$ \eqref{E:INTROGEODESIC} can be expressed as
\begin{align} \label{E:INTROGEODESICRELATIVETORECTANGULAR}
	\Lgeo^{\alpha} \partial_{\alpha} \Lgeo^{\nu}
	& = - \Gapsi_{\alpha \ \beta}^{\ \nu} \Lgeo^{\alpha} \Lgeo^{\beta}.
\end{align}

The level sets of $u,$ which we denote by $\mathcal{C}_u,$
are outgoing $g-$null hypersurfaces.
They intersect the flat hypersurfaces $\Sigma_t$
of constant Minkowskian time in topological spheres, which we denote by $S_{t,u}.$
We denote  by $\gsphere$  the Riemannian metric on $S_{t,u}$ induced
by $g.$ 

As in the case of spherical symmetry, shock formation is intimately
tied to the degeneration of
the inverse foliation density $\upmu$ of the null hypersurfaces $\mathcal{C}_u$ 
(as measured with respect to $\Sigma_t$). 
\begin{definition}[\textbf{Inverse foliation density}]
	\label{D:INVERSEFOLIATIONDENSITY}
We define the inverse foliation density $\upmu$ by
\bea
\frac{1}{\upmu} := -(g^{-1})^{\alpha \beta} \partial_{\alpha} t \partial_{\beta} u.
\label{E:INTROUPMUDEF}
\eea  
\end{definition}

\begin{remark}
For the background solution $\Psi \equiv 0$ we have $\upmu
\equiv 1.$ 
\end{remark}

The proof of shock formation outside of spherical symmetry will follow
the same general strategy as implemented in the proof of Proposition
\ref{P:ge}. In particular, we will prove a sharp lifespan theorem
along with ``global-existence-type'' estimates. We will derive these latter estimates
relative to a frame in which $\mathcal{C}_u-$transversal directional derivatives are 
rescaled by $\upmu$ and the $\mathcal{C}_u-$tangential directional derivatives are near their Minkowskian counterparts.
 We summarize this strategy
 in the following rough statement.
 
 \bigskip
 
\noindent  {\bf  Heuristic Principle I.}\, \emph{If we work with properly   $ \upmu$-rescaled   quantities
we can effectively
transform the shock formation problem into 
a sharp long-time existence problem 
in which various rescaled quantities 
exhibit dispersive behavior and decay similarly to
the peeling properties \eqref{eq:peeling}.
See also Remark \ref{R:GENERALCASERESCALINGGIVESNULLCONDITION}.}

We remark again here that, as in Christodoulou's work \cite{dC2007},
the Heuristic Principle is only expected to hold strictly for
lower-order derivatives of the solution; it turns out that the
argument requires accommodating possible degeneracies, in terms of the
$L^2$-norm control, of the higher-order $\upmu$-rescaled derivatives, near
the time of first shock formation. 
 \medskip
     
We introduce the following null vectorfield, which is a rescaled version of \eqref{E:NULLGEODESIC}:  
    \bea \label{E:INTROLUNITDEFINED}
		\Lunit := \upmu \Lgeo.
		\eea
		It follows from definitions 
		\eqref{E:NULLGEODESIC}
		and
		\eqref{E:INTROUPMUDEF} that $\Lunit t = 1$ and hence $\Lunit^0 = 1.$
		The vectorfield $\Lunit$ is the replacement of the
		one of \eqref{E:CHARDIR} encountered in spherical
symmetry: in the particular case of the metric associated to John's
equation $-\pr_t^2\Phi+(1+\pr_t\Phi)\lap \Phi,$
for spherically symmetric solutions  
$\Psi:=\pr_t \Phi$
(see Subsect.~\ref{subs:radial-blow}),  
the vectorfield $\upmu \Lgeo$ coincides with the vectorfield $\Lunit =
\pr_t +\sqrt{1+\Psi}\pr_r.$
     Consistent with the Heuristic Principle mentioned above
     and with our experience in spherical symmetry, 
     we expect that
     $\Lunit$ remains close   
     to the vectorfield $\Lunit_{(Flat)} = \pr_t + \pr_r.$

   We are now in a position to define a good set of coordinates,
   in analogy with the coordinates $(t,u)$
   that we used in proving Prop.~\ref{P:ge} in spherical symmetry.
   Specifically, to obtain a sharp picture of the dynamics,
	 we use \emph{geometric coordinates} 
		\begin{align} \label{E:GEOMETRICCOORDINATES}
			(t,u,\vartheta^1,\vartheta^2).
		\end{align}
		In \eqref{E:GEOMETRICCOORDINATES}, $t$ is the Minkowski time coordinate, $u$ is the eikonal function,
and $(\vartheta^1,\vartheta^2)$ are local angular
coordinates on the spheres $S_{t,u},$  
propagated from the initial Euclidean sphere $S_{0,0}$
by first solving the transport equation
\begin{align}
- \partial_r \vartheta^A & = 0, && (A=1,2)
\end{align}
to propagate them to the $S_{0,u}$ and then solving the 
transport equation
\begin{align}
\Lunit \vartheta^A = 0, && (A=1,2)
\end{align}
to propagate them to the $S_{t,u}.$
In particular, relative to geometric coordinates, we have 
\begin{align}
\Lunit = \frac{\partial}{\partial t},
\end{align}
a relation that we use throughout our analysis.

\begin{remark}[$\upmu$ \textbf{is connected to the Jacobian determinant}]
\label{R:MUISCONNECTEDTOTHEJACOBIANDETERMINANT}
We note here another important role played by $\upmu:$
it is not too difficult to show that
the Jacobian determinant of the change of variables map 
$(t,u,\vartheta^1,\vartheta^2) \overset{\Upsilon}{\rightarrow} (t,x^1,x^2,x^3)$
from geometric to rectangular coordinates is proportional to $\upmu;$
see \cite[Lemma 2.17.1]{jS2014}.

In particular, for small-data solutions, one can show that the Jacobian determinant
$\det d \Upsilon$
vanishes precisely at the points where $\upmu$ vanishes.
Hence, solutions that are regular relative to the geometric 
coordinates can in fact have rectangular derivatives that blow-up
at the locations where $\upmu$ vanishes because of the degeneracy of
the change of variables map. 
\end{remark}

In addition to the geometric coordinates, we will also use     
a vectorfield frame adapted to the shock-forming solutions. 
Three of the frame vectors are $\Lunit, X_1, X_2,$ where the $X_A$ are
the angular coordinate vectorfields along the $S_{t,u},$ that is, 
$X_1 = \frac{\partial}{\partial \vartheta^1}|_{t,u,\vartheta^2},$ and similarly for $X_2.$ 
These three vectors are tangent to the $\mathcal{C}_u.$
To complete the frame, we introduce             
the transversal vectorfields 
\begin{align}
	&\Radunit, && \Rad := \upmu \Radunit,
\end{align}
where 
$\Radunit$ is uniquely defined by requiring it to be tangent to $\Sigma_t,$ 
$g-$orthogonal to $S_{t,u},$
inward pointing, and normalized by 
$g(\Radunit, \Radunit) = 1$.
The analog of $\Radunit$ in Subsect.~\ref{subs:radial-blow}
is $- \sqrt{1+ \Psi} \partial_r.$ We are using $\Radunit$
as a convenient replacement for the null vectorfield $\uLunit$ from spherical symmetry (see \eqref{E:CHARDIR}).
Even though $\Radunit$ is not null, it is transversal to the $\mathcal{C}_u,$ which is
the property of greatest relevance.
Under the assumption \eqref{E:METRICNORMALIZATION} we have
\beaa
g(\Rad, \Rad) = \upmu^2,\quad g(\Lunit, \Radunit) = - 1, \quad g(\Lunit, \Rad) = - \upmu.
\eeaa
Consistent with the Heuristic Principle, one can show  
that 
\beaa
\Radunit = - \partial_r + \err,
\eeaa       
where $\err$ is small and decaying in time.
Hence, the vectorfield $\Rad$ vanishes exactly at the 
first shock singularity point, where $\upmu$ vanishes.
Observe further that $g(\Lunit,\Rad) = -\upmu \implies \Rad u = 1$; 
hence relative to the geometric coordinates $t,u,\vartheta,$ we have 
\beaa
\Rad = \frac{\partial}{\partial u}  + \mbox{a small $S_{t,u}-$tangent angular deviation}.
\eeaa
 
Having defined the above vectorfields, we can now define the frame
that we use to analyze solutions; see Figure \ref{F:FRAME}.
\begin{definition}[\textbf{Rescaled frame}]
We define the rescaled frame as follows:
\begin{align} \label{E:LUNITRADGOODGFRAME}
	\left\lbrace 
		\Lunit, \Rad, X_1=\frac{\partial}{\partial \vartheta^1}, X_2= \frac{\partial}{\partial \vartheta^2} 
	\right\rbrace.
\end{align}
\end{definition}

\begin{center}
\begin{overpic}[scale=.2]{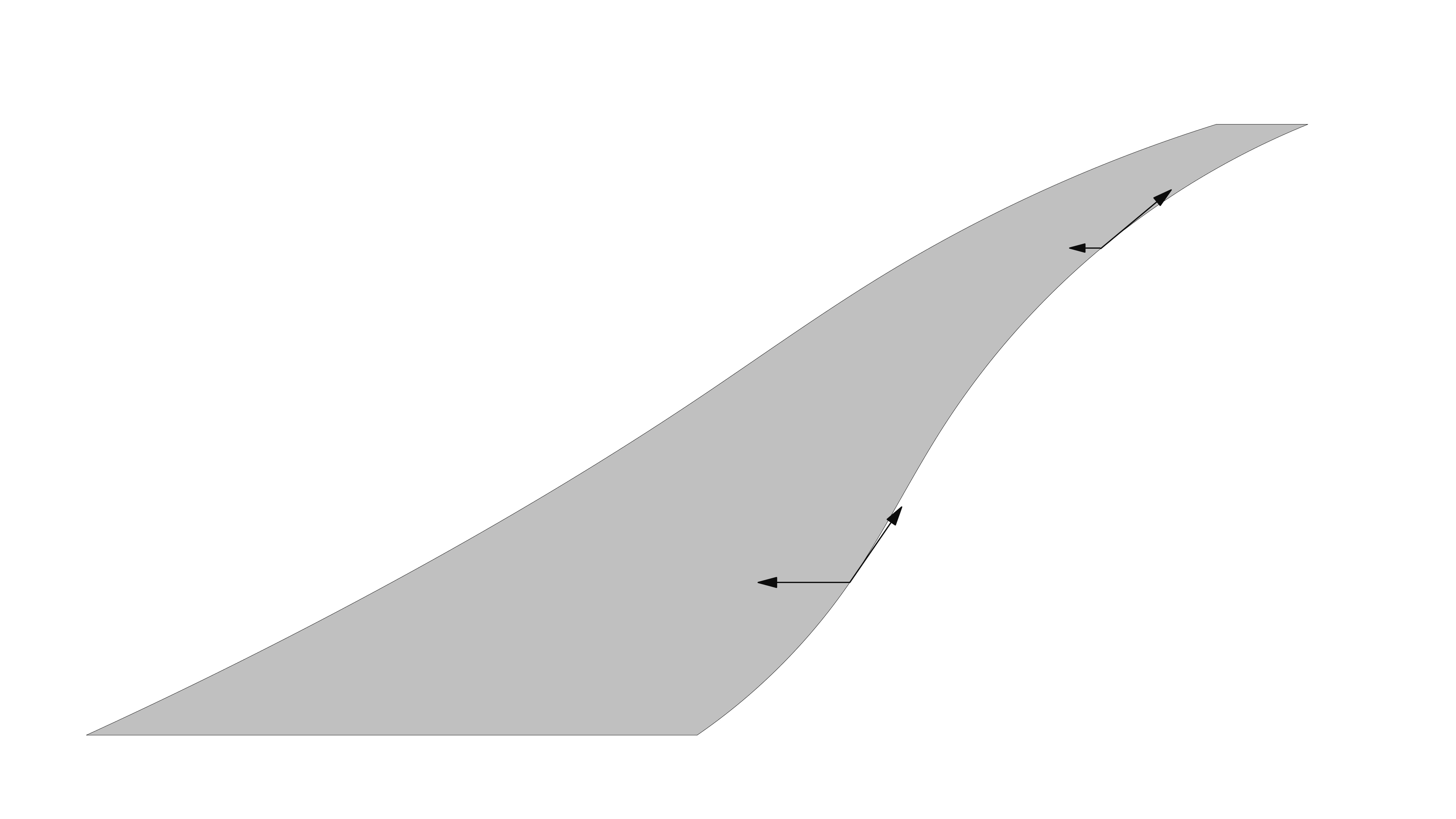}  
\put (54,17) {\large$\displaystyle \Rad$}
\put (60.5,17) {\large$\displaystyle \Lunit$}
\put (73.3,40) {\large$\displaystyle \Rad$}
\put (80,40) {\large$\displaystyle \Lunit$}
\put (57.8,15.5) {$\bullet$}
\put (57.5,13) {\large$\displaystyle X_1,X_2$}
\put (75.1,38.6) {$\bullet$}
\put (75,36.1) {\large$X_1,X_2$}
\put (42,29) {\large $\displaystyle \mathcal{C}_{u_2}$}
\put (68.5,29) {\large $\displaystyle \mathcal{C}_{u_1}$}
\end{overpic}
\captionof{figure}{The rescaled frame at two distinct points. $X_1,
X_2$ are orthogonal to the plane of the page and omitted.}
 \label{F:FRAME}
\end{center}

At each point where $\upmu > 0$, the frame \eqref{E:LUNITRADGOODGFRAME} has span equal to
$\mbox{span} \lbrace \frac{\partial}{\partial x^{\alpha}}
\rbrace_{\alpha=0,1,2,3}$.
In fact, relative to an arbitrary coordinate system, 
we have the following decompositions:
\begin{align} \label{E:GINVERSERELATIVETOFRAME}
	(g^{-1})^{\alpha \beta}
	& = - \Lunit^{\alpha} \Lunit^{\beta}
		- (\Lunit^{\alpha} \Radunit^{\beta} 
				+ \Radunit^{\alpha} \Lunit^{\beta})
		+ (\ginversesphere)^{AB} X_A^{\alpha} X_B^{\beta}
			\\
	& = - \Lunit^{\alpha} \Lunit^{\beta}
		- \upmu^{-1}
			(\Lunit^{\alpha} \Rad^{\beta} 
				+ \Rad^{\alpha} \Lunit^{\beta})
		+ (\ginversesphere)^{AB} X_A^{\alpha} X_B^{\beta},
		\notag
\end{align}
where $\gsphere_{AB} := g(X_A, X_B),$ $A,B = 1,2.$

In most of our analysis, we find it convenient to work with $\Rad,$ but
in some of our analysis, it is better to use instead the following vectorfield:
\begin{align} \label{E:ULRESCALED}
	\uLgood := \upmu \Lunit + 2 \Rad.
\end{align}
The analog of $\uLgood$ in Subsect.~\ref{subs:radial-blow}
is $\upmu \uLunit,$ where $\uLunit$ is defined in \eqref{E:CHARDIR}.
Note that $\uLgood$ is the uniquely defined null vectorfield that is orthogonal to the spheres $S_{t,u}$ and verifies
 \beaa
 g(\Lunit,\uLgood)= -2\upmu.
 \eeaa

\subsubsection{The wave operator relative to the rescaled frame and the $S_{t,u}$ tensorfield $\upchi$}
Now that we have a good frame for analyzing solutions, it
is important to understand how the covariant wave operator looks when expressed 
relative to it. Some rather tedious but straightforward calculations reveal
that we can decompose
\begin{align} \label{E:WAVEOPERATORDECOMPOSED}
	\upmu \square_{g(\Psi)} \Psi
	& = 
		- \Lunit \overbrace{(\upmu \Lunit \Psi + 2 \Rad \Psi)}^{\uLgood \Psi}
		+ \upmu \angLap \Psi
		- \mytr \upchi \Rad \Psi
		+ \err,
\end{align}
where $\angLap$ is the Laplacian of the metric 
$\gsphere$ on the spheres $S_{t,u}$
and the error terms are small and decaying according to the Heuristic Principle;
see \cite[Proposition 4.3.1]{jS2014} for more details.

In equation \eqref{E:WAVEOPERATORDECOMPOSED}, 
$\upchi$ is a symmetric type $\binom{0}{2}$ tensorfield on $S_{t,u}$ that verifies
\begin{align} \label{E:CHIDEF}
	\upchi_{AB} = g(\D_A \Lunit, X_B),
\end{align}
and $\mytr \upchi := (\ginversesphere)^{AB} \upchi_{AB}.$
Equivalently, $\upchi_{AB} = \Lie_{\Lunit} \gsphere_{AB} = \frac{\partial}{\partial t}|_{u,\vartheta} \gsphere_{AB},$
where $\Lie_{\Lunit}$ denotes Lie differentiation with respect to $\Lunit.$
In the case of the background solution $\Psi \equiv 0,$ 
we have $\mytr \upchi = 2 r^{-1},$
where $r$ is the Euclidean radial coordinate on $\Sigma_t.$
For the perturbed solutions under consideration, 
we have $\mytr \upchi = 2 \rgeo^{-1} + \err,$ where
\begin{align} \label{E:GEOMETRICRADIAL}
	\rgeo := 1 - u + t.
\end{align}
Since $\Lunit \rgeo = 1,$ we therefore deduce from \eqref{E:WAVEOPERATORDECOMPOSED} that
\begin{align} \label{E:ALTWAVEOPERATORDECOMPOSED}
	\rgeo \upmu \square_{g(\Psi)} \Psi
	& = 
		- \Lunit \left\lbrace
				\upmu \Lunit (\rgeo \Psi) + 2 \Rad (\rgeo \Psi)
			\right\rbrace
		+ \rgeo \upmu \angLap \Psi
		+ \err.
\end{align}
As we will see, the form of the equation \eqref{E:ALTWAVEOPERATORDECOMPOSED} is important
for showing that $\upmu$ can go to $0$ in finite time.
\begin{remark}
Note that \eqref{E:ALTWAVEOPERATORDECOMPOSED} corresponds,
roughly, to equation \eqref{E:LOUTSIDEREWRITTEN} in the context of John's spherically symmetric wave equation.
\end{remark}

\begin{remark}[\textbf{Rescaling by $\upmu$ ``removes'' the dangerous semilinear term}]
	\label{R:GENERALCASERESCALINGGIVESNULLCONDITION}
	Note that we have brought a factor of $\upmu$ under the outer $\Lunit$ differentiation
	in equations \eqref{E:WAVEOPERATORDECOMPOSED} and \eqref{E:ALTWAVEOPERATORDECOMPOSED}.
	We have already seen the importance of this ``rescaling by
$\upmu$'' in spherical symmetry:
	it removes the dangerous quadratic semilinear term that decays slowly;
	see equation \eqref{E:LOUTSIDEREWRITTEN}. Although a similar remark applies 
	away from spherical symmetry, it is more difficult to see that the
	remaining error terms $\err$ are  indeed        such that,  
	according to the Heuristic Principle, they enjoy better time
	decay. 	The reason is that they involve, 
	in addition to the first derivatives of $\Psi,$ the second
	derivatives of the eikonal function. Hence, to show that these 
	error terms are negligible  with respect to time  decay, one
	must control the asymptotic behavior of eikonal function.
	For example, one of the error terms in equation
	\eqref{E:ALTWAVEOPERATORDECOMPOSED} is
	$\rgeo \left\lbrace \mytr \upchi - \frac{2}{\rgeo} \right\rbrace \Rad \Psi.$
	Using the Heuristic Principle decay estimates, one can show that the sup-norm
	 of the product $\rgeo \left\lbrace \mytr \upchi - \frac{2}{\rgeo} \right\rbrace \Rad \Psi$
	is $\leq C \varepsilon \ln(\myexp + t) (1 + t)^{-2},$
	which is the same decay rate\footnote{If there were a
        $(\uLunit \Psi)^2$ term on the right-hand side of \eqref{E:LOUTSIDEREWRITTEN}, 
        then, because of the presence of the factor $r,$ 
	its decay rate would not be integrable in $t.$} as that of the
	two error term products on the right-hand side of 
	equation \eqref{E:LOUTSIDEREWRITTEN} in spherical symmetry.
\end{remark}

We also note that, from the definition of the covariant derivative $\D,$ we have
\begin{align} \label{E:CHIALT}
	\upchi_{AB}
	& = g_{ab} (X_A^c \partial_c \Lunit^a) X_B^b
		+ X_A^a X_B^b \Lunit^{\gamma} \Gamma_{a b \gamma},
\end{align}
where the lowercase Latin indices are relative to spatial rectangular spatial coordinates,
the lowercase Greek indices are relative to rectangular spacetime coordinates,
and the uppercase Latin indices correspond to the two $S_{t,u}$ frame vectors.
Hence, \eqref{E:CHIALT} shows that $\upchi$ is an auxiliary quantity 
expressible in terms of the frame derivatives of 
$\Psi$ and the frame derivatives of the rectangular components $\Lunit^i.$
However, because of its importance, to be clarified below,
it is convenient to think of $\upchi$ as an independent quantity.

\subsubsection{The evolution equation for $\upmu$}
We now derive a transport equation for $\upmu,$ analogous to 
the equation \eqref{mueq} in our analysis in spherical symmetry.

\begin{lemma}[\textbf{Transport equation for $\upmu$}]
The quantity $\upmu$ defined in \eqref{E:INTROUPMUDEF} verifies the following transport equation:
\begin{align} \label{E:UPMUSCHEMATICTRANSPORT}
	\Lunit \upmu
	& = \frac{1}{2} G_{\Lunit \Lunit}  \Rad \Psi 
		+ \upmu \err, \qquad  G_{\Lunit \Lunit} := G_{\alpha \beta}(\Psi) \Lunit^{\alpha} \Lunit^{\beta},
\end{align}
where the term $\err$ is an error term involving $\mathcal{C}_u-$tangential derivatives of $\Psi.$
\end{lemma}

\begin{proof}
Recalling that $\upmu = (\Lgeo^0)^{-1},$          
relative to the rectangular coordinates,
we consider the $0$ component of equation \eqref{E:INTROGEODESICRELATIVETORECTANGULAR}:
\begin{align} \label{E:LGEOGEODESIC0COMPONENT}
	\Lgeo \Lgeo^0 
	& = - (g^{-1})^{0 \gamma} \Gapsi_{\alpha \gamma \beta} \, \Lgeo^{\alpha} \, \Lgeo^{\beta}.
\end{align}
Multiplying \eqref{E:LGEOGEODESIC0COMPONENT} by $\upmu^3,$ 
using the definition \eqref{E:INTROLUNITDEFINED} of $ \Lunit,$
the decomposition \eqref{E:GINVERSERELATIVETOFRAME},
the identities $(\ginversesphere)^{0 \gamma} = 0,$ $\Lunit^0 = 1,$ $\Radunit^0=0,$ $\Rad = \upmu \Radunit,$ 
and equation \eqref{eq:Gapsi}, 
we deduce that
\begin{align} \label{E:SECONDVERSIONLGEOGEODESIC0COMPONENT}
\upmu^2 \Lunit \Lgeo^0 
& = \frac{1}{2}
		\left\lbrace
			\upmu \Lunit^{\gamma}
			+ \Rad^{\gamma} 
		\right\rbrace
	  \left(
			G_{\gamma \beta} \partial_{\alpha} \Psi
			+ G_{\alpha \gamma} \partial_{\beta} \Psi
			- G_{\alpha \beta}\partial_{\gamma} \Psi
		\right) 
		\Lunit^{\alpha} \Lunit^{\beta}
			\\
& = \frac{1}{2} \upmu G_{\Lunit \Lunit} \Lunit \Psi
		+ \upmu G_{\Lunit \Radunit} \Lunit \Psi
		- \frac{1}{2} G_{\Lunit \Lunit} \Rad \Psi.
		\notag
\end{align}
Equation \eqref{E:UPMUSCHEMATICTRANSPORT} now follows from equation \eqref{E:SECONDVERSIONLGEOGEODESIC0COMPONENT},
the relation $\upmu^2 \Lunit \Lgeo^0 = \upmu^2 \Lunit (\frac{1}{\upmu}) = - \Lunit \upmu,$
and incorporating the $\mathcal{C}_u-$tangent derivatives into the term $\err.$
\end{proof}

   \begin{remark}[\textbf{The coupled system}]
   \label{R:COUPLEDSYSTEM}
 		We stress the following important point:
 		the basic equations that we need to study 
		are the wave equation \eqref{E:SPECIFICSEMILINEARTERMSGENERALQUASILINEARWAVE}
		coupled to the evolution equations \eqref{E:INTROGEODESICRELATIVETORECTANGULAR}
 		for the rectangular components of $\Lgeo.$    
 		Though the standard form \eqref{E:SPECIFICSEMILINEARTERMSGENERALQUASILINEARWAVE} is useful for     
 		deriving energy estimates,   
 		the equivalent form \eqref{E:WAVEOPERATORDECOMPOSED} is fundamental
		for understanding the behavior of the rescaled quantities. 
	\end{remark}

\subsubsection{Initial conditions and relevant regions}
\label{SSS:DATAANDREGIONS}
We start by recalling the basic setup described in
Subsect.~\ref{SS:BASICGEOMETRICNOTIONS}, and in particular, gather our
notations in one place. 
Our setup here is closely related to the one we used 
in Subsubsect. \eqref{SSS:GEOMETRICFORMULATION}
in spherical symmetry.
We recall that we are studying solutions to the wave equation
\eqref{E:SPECIFICSEMILINEARTERMSGENERALQUASILINEARWAVE} subject to the following small initial conditions on 
$\Si_0=\lbrace t=0 \rbrace:$
\begin{align}
	\mathring{\Psi} := \Psi|_{\Sigma_0}, \qquad \mathring{\Psi}_0 := \partial_t \Psi|_{\Sigma_0},
\end{align}
where $(\mathring{\Psi}, \mathring{\Psi}_0)$ are supported in the Euclidean unit ball 
$\lbrace r \leq 1 \rbrace,$  
with $r = \sqrt{\sum_{a=1}^3 (x^a)^2}$ 
the standard Euclidean distance to the origin on $ \Si_0.$
As we mentioned before, $U_0$ is a real number verifying
\begin{align} \label{E:U0}
	0 \leq U_0 < 1.
\end{align}
We study the future-behavior of the solution in the region that corresponds to
the portion of the data lying in an annular region of inner Euclidean spherical radius
$1 - U_0$ and outer Euclidean spherical radius $1$ (that is, the thickness of the region is $U_0$):
\begin{align} \label{E:INTROSIGMA0U0DEF}
	\Sigma_0^{U_0} := \lbrace x \in \Sigma_0 \ | \ 1 - U_0 \leq r(x) \leq 1 \rbrace;
\end{align}
see Figure \ref{F:REGION}.
The reason that we assume $U_0 < 1$ is simply to avoid potential problems with degeneration of our coordinates at the origin.
We define the size of the data as follows:
\begin{align} \label{E:INTROSMALLDATA}
	\mathring{\upepsilon} 
	= \mathring{\upepsilon}[(\mathring{\Psi},\mathring{\Psi}_0)]
	:= \| \mathring{\Psi} \|_{H_{}^{25}(\Sigma_0^1)}
		+ \| \mathring{\Psi}_0 \|_{H_{}^{24}(\Sigma_0^1)}.
\end{align}
In \eqref{E:INTROSMALLDATA}, $H^N$ is the standard Euclidean Sobolev space involving 
order $\leq N$ rectangular spatial derivatives along $\Sigma_0^1.$ 
To prove small-data shock formation, we 
assume that $\mathring{\upepsilon}$ is sufficiently small 
(see Footnote \ref{FN:NUMBEROFDERIVATIVES} on pg.\ \pageref{FN:NUMBEROFDERIVATIVES}) 
together with some
other open conditions that we explain below. 
 
The following regions of spacetime
depend on our eikonal function $u$
and are analogs of regions
that we encountered in Subsubsect. \eqref{SSS:GEOMETRICFORMULATION}
in spherical symmetry.
\begin{definition} [\textbf{Subsets of spacetime}]
\label{D:HYPERSURFACESANDCONICALREGIONS}
We define the following spacetime subsets:
\begin{subequations}
\begin{align}
	\Sigma_{t'} & := \lbrace (t,x^1,x^2,x^3) \in \mathbb{R}^4  \ | \ t = t' \rbrace, 
		\\
	\Sigma_{t'}^{u'} & := \lbrace (t,x^1,x^2,x^3) \in \mathbb{R}^4  \ | \ t = t', \ 0 \leq u(t,x^1,x^2,x^3) \leq u' \rbrace, 
		\label{E:SIGMATU} 
		\\
	\mathcal{C}_{u'}^{t'} & := \lbrace (t,x^1,x^2,x^3) \in \mathbb{R}^4 \ | \ u(t,x^1,x^2,x^3) = u' \rbrace \cap \lbrace (t,x^1,x^2,x^3) 
		\in \mathbb{R}^4  \ | \ 0 \leq t \leq t' 
		\rbrace, 
		\\
	S_{t',u'} 
		&:= \mathcal{C}_{u'}^{t'} \cap \Sigma_{t'}^{u'}
		= \lbrace (t,x^1,x^2,x^3) \in \mathbb{R}^4 \ | \ t = t', \ u(t,x^1,x^2,x^3) = u' \rbrace, 
			\\
	\mathcal{M}_{t',u'} & := \cup_{u \in [0,u']} \mathcal{C}_u^{t'} \cap \lbrace (t,x^1,x^2,x^3) \in \mathbb{R}^4  \ | \ t < t' \rbrace.
		\label{E:MTUDEF}
\end{align}
\end{subequations}
We refer to the $\Sigma_t$ and $\Sigma_t^u$ as ``constant time
slices,'' the $\mathcal{C}_u^t$  as ``outgoing null cones,''
and the $S_{t,u}$ as ``spheres.'' 
We sometimes use the notation $\mathcal{C}_u$ in place of $\mathcal{C}_u^t$ 
when we are not concerned with the truncation time $t.$
\end{definition}

\begin{remark}
Note that $\mathcal{M}_{t,u}$ is, by definition, ``open at the top.'' 
\end{remark}

\begin{center}
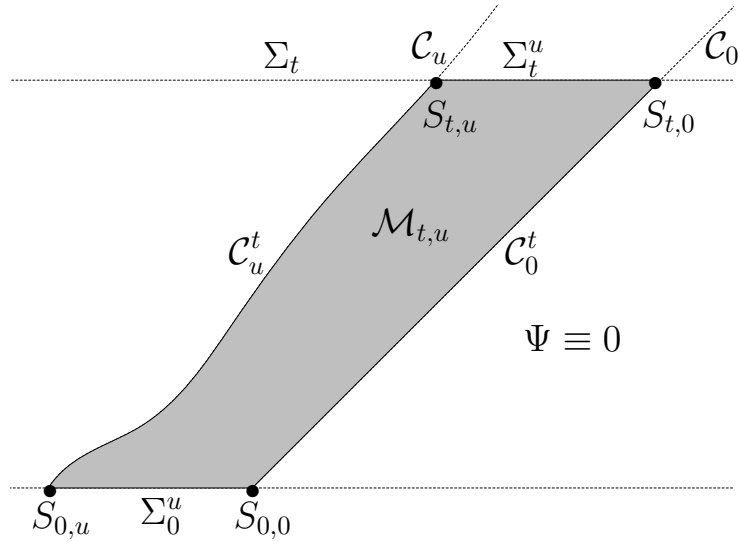

\begin{overpic}[scale=.2]{Spacetimesubsets.pdf}  
	\put (49.9,44.4) {\large$\displaystyle \MM_{t,u}$}
\put (96.1,68.9) {\large$\displaystyle \mathcal{C}_0$}
\put (55.5,68.9) {\large$\displaystyle \mathcal{C}_u$}
\put (30.2,40.4) {\large$\displaystyle \mathcal{C}_u^t$}
\put (68.4,40.4) {\large$\displaystyle \mathcal{C}_0^t$}
\put (34.9,67.6) {\large$\displaystyle \Sigma_t$}
\put (17.9,4.4) {\large$\displaystyle \Sigma_0^u$}
\put (67.9,67.6) {\large$\displaystyle \Sigma_t^u$}
\put (70.9,28.4) {\large$\displaystyle \Psi \equiv 0$}
\put (57.8,64.4) {$\displaystyle \bullet$}
\put (56.9,59.4) {\large$\displaystyle S_{t,u}$}
\put (88.1,64.4) {$\displaystyle \bullet$}
\put (87.3,59.4) {\large$\displaystyle S_{t,0}$}
\put (4.4,7.9) {$\displaystyle \bullet$}
\put (2.9,3.9) {\large$\displaystyle S_{0,u}$}
\put (32.4,7.9) {$\displaystyle \bullet$}
\put (30.9,3.9) {\large$\displaystyle S_{0,0}$}
\end{overpic}
\captionof{figure}{Spacetime subsets}
 \label{F:SPACETIMESUBSETS}
\end{center}

Just as in the case of the spherically symmetric shock formation
discussed in Subsect.~\ref{subs:radial-blow},
all of the interesting dynamics takes place in the region $\MM_{t,u}$ for 
$t$ sufficiently large.

\subsubsection{A more precise version of the Heuristic Principle.}
\label{SSS:HPMOREPRECISE}
We are now ready to give a more precise version of the Heuristic Principle of Subsubsect.~\ref{SSS:EIKONALFUNCTION};
see \cite[Section 11.4]{jS2014} for more details.

{\bf Heuristic Principle II.}  \emph{
Consider a small-data solution $\Psi$ to equation
\eqref{E:SPECIFICSEMILINEARTERMSGENERALQUASILINEARWAVE},
where the size of the data is defined in \eqref{E:INTROSMALLDATA}.
Then under the structural conditions on the nonlinearities
explained below in Subsubsect.~\ref{SSS:STRUCTURAL},
the directional derivatives of a solution $\Psi$ 
with respect to the rescaled frame
$\lbrace \Lunit, \Rad, X_1, X_2 \rbrace$
(see \eqref{E:LUNITRADGOODGFRAME})
decay, in the relevant $\MM_{t,u}$ region,    
in a manner analogous to
the linear peeling properties \eqref{eq:peeling} all the way to 
the formation of the first shock. More precisely,\footnote{Note that
terms of the form $(1+|u|)^{1/2}$ are omitted in these estimates. 
These factors are in fact $\mathcal{O}(1)$ inside the relevant $\MM_{t,u}$ region.}
the following estimates hold:
\begin{subequations}
\begin{align}
	|\Lunit \Psi|, 
	\, |\angD \Psi|
	& \leq \frac{\varepsilon}{(1 + t)^2},
	 \label{E:TANGENTIALFASTDECAY} \\
	|\Psi|,
		\,
	|\Rad \Psi|
	& \leq \frac{\varepsilon}{1 + t},
	\label{E:RESCALEDRADDISPERSIVEESTIMATE}
\end{align}
\end{subequations}
where $\varepsilon$ is a small constant that is controlled by the size of the data,
and $|\angD \Psi|$ is the size of the angular gradient\footnote{Note that
the frame vectorfields 
$X_1$ and $X_2$ span the tangent space
of the $S_{t,u}$ and hence a bound for
$|\angD \Psi|$ also implies a bound for the directional derivatives
$|X_1 \Psi|$ and $|X_2 \Psi|.$} of $\Psi$ as measured by $\gsphere,$
that is, the size of the gradient of $\Psi$ viewed as a function on the $S_{t,u}.$ A similar
statement holds for
a limited number of higher directional derivatives, where 
 each additional $\Lunit$ and $\angD$ differentiation
 leads to a gain in decay of $(1 + t)^{-1}.$ However, unlike in the linear case,
 the very high derivatives are allowed to have degenerate behavior in $\upmu$; see Prop.~\ref{P:APRIORIENERGYESTIMATES}.
}

Note that \eqref{E:RESCALEDRADDISPERSIVEESTIMATE} 
is equivalent to the following estimate for 
$\Radunit \Psi = \upmu^{-1} \Rad \Psi \sim - \partial_r \Psi:$
\begin{align}  \label{E:NOTRESCALEDRADDISPERSIVEESTIMATE}
	|\Radunit \Psi|
	& \leq \frac{1}{\upmu} \frac{\varepsilon}{(1 + t)}.
\end{align}
Actually, for the shock-forming
solutions of interest, a key ingredient in the proof is showing that
for an open set of data, we 
have a \emph{lower bound} of the form\footnote{\label{F:LESS}
	We often write $A \lesssim B$ whenever there exists a uniform constant $C > 0$
 such that $A \leq C B.$
	Similarly, we often write $A \gtrsim B$ whenever there exists a uniform constant $C > 0$
 such that $A \geq C B.$}  
$|\Rad \Psi| \gtrsim \varepsilon (1 + t)^{-1}$
(see inequality \eqref{E:KEYLOWER}), so that
\begin{align}  \label{E:LOWEROBUNDNOTRESCALEDRADDISPERSIVEESTIMATE}
	|\Radunit \Psi|
	& \gtrsim \frac{1}{\upmu} \frac{\varepsilon}{(1 + t)}.
\end{align} 
The inequalities
\eqref{E:NOTRESCALEDRADDISPERSIVEESTIMATE}
and
\eqref{E:LOWEROBUNDNOTRESCALEDRADDISPERSIVEESTIMATE} imply that
$\Radunit \Psi$ blows up \emph{exactly at the points where $\upmu$ vanishes.}
Note that we have already encountered an analog of the lower bound \eqref{E:LOWEROBUNDNOTRESCALEDRADDISPERSIVEESTIMATE}
in spherical symmetry; see inequalities 
\eqref{E:SSTRANSVERSALDERIVATIVELARGEINMAGNITUDE}
and
\eqref{E:SECONDBLOWUPESTIMATE}.

\begin{remark}
The decay estimates \eqref{E:TANGENTIALFASTDECAY}, \eqref{E:RESCALEDRADDISPERSIVEESTIMATE} 
can be used to derive estimates for the components of the covariant Hessian $H=\D^2 u$ of the eikonal function relative  
to the rescaled frame.\footnote{
Since $\upmu$ and $\Lunit^i$ are first derivatives of $u,$
this is essentially equivalent to deriving estimates for the first derivatives of $\upmu$ and $\Lunit^i.$
In practice, we directly estimate the derivatives of
$\upmu$ and $\Lunit^i$ by studying the transport equations that they verify; see Remark \ref{R:COUPLEDSYSTEM}.
}
Indeed, recall that $H$ verifies the 
  transport equation  \eqref{eq:ricatti-H}, and hence, ignoring for now
	factors of $\upmu^{-1},$ we have the schematic equation
  \begin{align}
     \Lunit H + H^2 = \mathcal{R}
   \end{align}    
                where as before, $\mathcal{R} = \mathcal{R}(\Psi)$ is a curvature component that can be       
                algebraically expressed in terms of the up-to-second-order derivatives of $\Psi$
								and the up-to-second-order derivatives of $u.$
								Using
the precise decay estimates for various components
			of $\mathcal{R}(\Psi)$ relative to the rescaled frame, which follow        
								from \eqref{E:TANGENTIALFASTDECAY} and \eqref{E:RESCALEDRADDISPERSIVEESTIMATE},
								together with some bootstrap assumptions on the first derivatives of $u,$
								we can derive precise decay estimates for the rescaled frame components of $H.$   
							  The behavior of some of these components 
								(more
precisely those involving $\Rad$) differs significantly from that of the corresponding components of the     
								Hessian of the flat eikonal function $1 + t - r.$ 
                         
\end{remark}

\subsubsection{The failure of the classic null condition}
\label{SSS:FAILUREOFCLASSICNULL}
The classic null condition
is defined for scalar equations of the form 
\eqref{general-system} (with $I=1,$ see Definition \ref{D:CLASSICNULL}),
scalar wave equations of the form
\eqref{modeleq:nongeo1},
and for general systems of wave equations 
of the form \eqref{general-system}
(see Remark \ref{Re:Aleph-systems} or \cite{sK1984}).
Here we study the particular case of scalar equations of the form 
\eqref{E:SPECIFICSEMILINEARTERMSGENERALQUASILINEARWAVE} in detail.
For simplicity, we assume in the present subsubsection that
the semilinear term $\NN(\Psi)(\partial \Psi, \partial \Psi)$
on the right-hand side of \eqref{E:SPECIFICSEMILINEARTERMSGENERALQUASILINEARWAVE}
is equal to $0.$ 
That is, we discuss here only the equation
\[
	\square_{g(\Psi)}\Psi = 0.
\]
In Subsubsect.~\ref{SSS:STRUCTURAL},
we will address the case in which $\NN(\Psi)(\partial \Psi, \partial \Psi)$
is non-zero. 
We now show that when 
$\NN(\Psi)(\partial \Psi, \partial \Psi) = 0$
in \eqref{E:SPECIFICSEMILINEARTERMSGENERALQUASILINEARWAVE},
the classic null condition 
holds if and only if the following
scalar-valued function $\FutFailFac$ 
completely vanishes. 

\begin{definition}
\label{D:INTROFAILUREFACTOR}
	We define the \emph{future null condition failure factor} $\FutFailFac$ by
	\begin{align} \label{E:INTROFAILUREFACTOR}
		\FutFailFac 
		:= \underbrace{G_{\alpha \beta}(\Psi = 0)}_{\mbox{constants}}
			\Lunit_{(Flat)}^{\alpha} \Lunit_{(Flat)}^{\beta},
	\end{align}
	where $\Lunit_{(Flat)} = \partial_t + \partial_r$
	and $G_{\a\b}(\Psi) =\frac{d}{d\Psi} g_{\a\b}(\Psi)$
	(see \eqref{def:G}).
\end{definition}

\begin{remark}
	\label{R:ALEPHDEPENDSONLYONEUCLIDEANANGLE}
	Note that relative to standard spherical coordinates $(t,r,\theta)$ on Minkowski spacetime,
	$\FutFailFac$ can be viewed as a function that depends only on $\theta.$
\end{remark}

\begin{remark}
	In the region $\lbrace t \geq 0 \rbrace,$
	$\FutFailFac$ is the \emph{coefficient} of the most dangerous
	(in terms of \emph{linear} decay rate) quadratic terms
	in the wave equation $\square_{g(\Psi)}\Psi = 0,$
	when the equation is expressed relative to the Minkowskian frame 
	\eqref{E:MINKOWSKIFRAME} introduced below.
	Roughly, as in the case of F. John's equation in spherical symmetry,
	these dangerous terms are the ones that drive future shock formation.
	However, when carrying out detailed analysis, 
	the correct frame to use is the dynamic one given in \eqref{E:LUNITRADGOODGFRAME}.
\end{remark}

\begin{remark}[\textbf{Connection between $\FutFailFac$ and small-data shock formation}]
	When $\NN(\Psi)(\partial \Psi, \partial \Psi) \equiv 0,$
	small-data shock formation occurs
	whenever $\FutFailFac$ is nontrivial; see
	Theorem~\ref{T:STABLESHOCKFORMATION}.
\end{remark}

\begin{remark}[\textbf{Past null condition failure factor}]
	\label{R:PASTFAILUREFAC}
	We could also study shock formation in the region $\lbrace t \leq 0 \rbrace.$
	In this case, the relevant function is not 
	$\FutFailFac,$ but is instead the past null condition failure factor
	$\PastFailFac,$ defined by replacing the vectorfield
	$\Lunit_{(Flat)}$ in equation \eqref{E:INTROFAILUREFACTOR}
	with $-\partial_t + \partial_r.$ 
	Note that $-\partial_t + \partial_r$ is outward pointing
	as we head to the past.
	The point is that quadratic terms that have a slow decay rate as
	$t \to \infty$ can have a faster decay rate as $t \to - \infty$
	and vice versa; the function $\PastFailFac$
	is the coefficient of the slowest decaying quadratic terms
	as $t \to - \infty.$
	Note also that $\FutFailFac$ completely vanishes
	if and only if $\PastFailFac$ completely vanishes.
	In fact, the functions $\FutFailFac$ and $-\PastFailFac$ 
	have the same range.
\end{remark}

To show that the complete vanishing of $\FutFailFac$ is equivalent to the classic null condition 
being verified (see Definition \ref{D:CLASSICNULL}),
we first Taylor expand the right-hand side of \eqref{E:COVWAVEOPERATORINRECTANGULAR}
around $(\Psi, \partial \Psi, \partial^2 \Psi) = (0, \mathbf{0}, \mathbf{0})$ and find that the quadratic
nonlinear terms are, up to constant factors,
\begin{align}
	& G_{\alpha \beta}(0) 
	(m^{-1})^{\alpha \kappa} 
	(m^{-1})^{\beta \lambda}
	\Psi 
	\partial_{\kappa} \partial_{\lambda} \Psi,
		\label{E:QUASILINEARFAILSNULL} \\
	& G_{\kappa \lambda}(0)
		(m^{-1})^{\kappa \lambda}
		(m^{-1})^{\alpha \beta}
		\partial_{\alpha} \Psi \partial_{\beta} \Psi,
		\label{E:VERIFIESNULL} \\
	& G_{\kappa \lambda}(0)
		(m^{-1})^{\alpha \kappa}
		(m^{-1})^{\beta \lambda}
		\partial_{\alpha} \Psi \partial_{\beta} \Psi.
		\label{E:SEMILINEARFAILSNULL}
\end{align}
Clearly the term \eqref{E:VERIFIESNULL} always verifies
the classic null condition. 
By definition, the term \eqref{E:QUASILINEARFAILSNULL}
verifies the classic null condition if and only if
$G_{\alpha \beta}(0) (m^{-1})^{\alpha \kappa} (m^{-1})^{\beta \lambda}
\ell_{\kappa} \ell_{\lambda} = 0$ for all Minkowski-null
covectors $\ell.$ It is straightforward to see that 
equivalently,\footnote{Note that given any future-directed Minkowski-null vector $\ell^{\alpha},$ 
there exists a spacetime point such that the Minkowski-null vector
$\Lunit_{(Flat)}^{\alpha}$ in \eqref{E:INTROFAILUREFACTOR}
is parallel to $\ell^{\alpha}.$} 
the term \eqref{E:QUASILINEARFAILSNULL}
verifies the classic null condition if and only if
$\FutFailFac \equiv 0.$
Similarly, the term \eqref{E:SEMILINEARFAILSNULL}
verifies the classic null condition if and only if
$\FutFailFac$ is trivial. 

We now discuss three relevant examples.
\begin{itemize}
		\item It is easy to see that $\FutFailFac$ completely vanishes 
			if and only if the constant tensorfield
			$G_{\a\b}(\Psi=0)$ is proportional to the Minkowski metric $m_{\alpha \beta} = \mbox{\upshape diag}(-1,1,1,1).$
			Thus, for the equations $\square_{g(\Psi)}\Psi = 0,$ 
			a necessary and sufficient condition for the nonlinearities to verify the classic null condition
			is that up to cubic terms, 
			$g(\Psi) = (1 + f(\Psi)) m,$
			where $f(0) = 0.$
			Consequently, for the equations
			$\square_{g(\Psi)}\Psi = 0,$ 
			the classic null condition is very restrictive and is satisfied only in trivial cases.
	\item Consider the equation  
		$\square_{g(\Psi)}\Psi = 0$ in the case of 
		F. John's metric $- (dt)^2 + (1 + \Psi)^{-1} \sum_{a=1}^3 (dx^a)^2,$
		as in equation\footnote{It is easy to show that
		equation \eqref{boeq} is equivalent to the covariant wave equation 
		$\square_{g(\Psi)} \Psi = \NN(\Psi)(\partial \Psi, \partial \Psi),$
		where $\NN(\Psi)(\partial \Psi, \partial \Psi) = 
		- \frac{1}{2} (1 + \Psi)^{-1} (g^{-1})^{\alpha \beta} \partial_{\alpha} \Psi \partial_{\beta} \Psi.$ 
		Note that $\NN(\Psi)(\partial \Psi, \partial \Psi)$ verifies the classic null condition.
		Hence, from the point of view of investigating failure of the classic null condition,
		we can study the equation $\square_{g(\Psi)}\Psi = 0$ instead of \eqref{boeq}.}  
		\eqref{boeq}.
		We compute that $G_{ij}(\Psi = 0) = -1$ if $i=j \in \lbrace 1, 2, 3 \rbrace,$
		and all other rectangular components of $G(\Psi = 0)$ vanish. Using also that
		$\Lunit_{(Flat)}^{\alpha} = (1,x^1/r,x^2/r, x^3/r)$ relative to the rectangular coordinates, 
		where $r = \sqrt{\sum_{a=1}^3 (x^a)^2},$ 
		we find that $\FutFailFac \equiv - 1.$ 
		This example is a good model of the kinds of equations
		that Christodoulou studied in \cite{dC2007}, where the analog of $\FutFailFac$
		is constant.  
	\item If $g_{\alpha \beta}(\Psi) = m_{\alpha \beta} + \Psi (\delta_{\alpha}^1 \delta_{\beta}^2 + \delta_{\alpha}^2 \delta_{\beta}^1),$
		then\footnote{Here, 
		$\delta_{\bullet}^{\bullet}$ denotes the standard Kronecker delta.} 
		$G_{12}(\Psi = 0) = G_{21}(\Psi = 0) = 1,$ and all other rectangular components of $G(\Psi = 0)$ vanish.
		Hence, we find that $\FutFailFac = 2 \Lunit_{(Flat)}^1 \Lunit_{(Flat)}^2 = 2 x^1 x^2/r^2$ in this case.
		Note that $\FutFailFac$ can be viewed as a function on $\mathbb{S}^2 \subset \mathbb{R}^3.$
\end{itemize}

We now discuss the classic null condition
for equations $\square_{g(\Psi)}\Psi = 0$ 
from a slightly different point of view,
one which explains the connection between the non-vanishing
of $\FutFailFac$ and the presence of dangerous
quadratic terms and which is connected to our analysis 
of shock formation
outside of spherically symmetry.
Specifically, we will write the equation
relative to rectangular coordinates and then
decompose the quadratic parts of the nonlinearities relative to    
the following Minkowskian frame:   
\begin{align} \label{E:MINKOWSKIFRAME}
	\lbrace \Lunit_{(Flat)}, \Radunit_{(Flat)}, X_{(Flat);1}, X_{(Flat);2} \rbrace,
\end{align}
where 
$\Lunit_{(Flat)} = \partial_t + \partial_r,$
$\Radunit_{(Flat)}= - \partial_r,$
and the $X_{(Flat);A}$ are angular vectorfields 
tangent to the Euclidean spheres of constant $r-$value in $\Sigma_t.$
We stress that we use the frame \eqref{E:MINKOWSKIFRAME}
for illustrative purposes only.
It is not suitable for studying solutions near the shock,
where we should instead use the dynamic frame \eqref{E:LUNITRADGOODGFRAME}.
Nonetheless, the main idea to keep in mind is that the
\emph{forward} linear peeling properties \eqref{eq:peeling}
suggest that relative to the frame \eqref{E:MINKOWSKIFRAME},
the most dangerous quadratic terms
in equation $\square_{g(\Psi)}\Psi = 0$
in the region $\lbrace t \geq 0 \rbrace$
are the ones proportional to $\Psi \Radunit_{(Flat)} (\Radunit_{(Flat)} \Psi)$
and $(\Radunit_{(Flat)} \Psi)^2;$
these terms have the slowest $t$ decay rates.
Note that this assertion can be relevant only in the region
$\lbrace t \geq 0 \rbrace$ and should be altered
if it is to apply to the region $\lbrace t \leq 0 \rbrace.$

To carry out the decomposition, we first note that in analogy with \eqref{E:GINVERSERELATIVETOFRAME},
relative to the frame \eqref{E:MINKOWSKIFRAME}, 
the inverse Minkowski metric can be decomposed as
\begin{align} \label{E:MINKINFLATFRAME}
	(m^{-1})^{\alpha \beta}
	& = - \Lunit_{(Flat)}^{\alpha} \Lunit_{(Flat)}^{\beta}
		- (\Lunit_{(Flat)}^{\alpha} \Radunit_{(Flat)}^{\beta} 
				+ \Radunit_{(Flat)}^{\alpha} \Lunit_{(Flat)}^{\beta})
		+ (\minversesphere)^{AB} X_{(Flat);A}^{\alpha} X_{(Flat);B}^{\beta}.
\end{align}
Next, decomposing the quadratic part \eqref{E:QUASILINEARFAILSNULL} 
of the quasilinear term relative to the frame \eqref{E:MINKOWSKIFRAME}
and using in particular \eqref{E:MINKINFLATFRAME},
we find that the component proportional to $\Psi \Radunit_{(Flat)} (\Radunit_{(Flat)} \Psi),$ 
is, up to constant factors,
\begin{align} \label{E:PRINCIPALBADTERM}
	\FutFailFac \Psi \Radunit_{(Flat)} (\Radunit_{(Flat)} \Psi).
\end{align}
Similarly, decomposing \eqref{E:VERIFIESNULL} and \eqref{E:SEMILINEARFAILSNULL},
we find that the term proportional to $(\Radunit_{(Flat)} \Psi)^2$ 
is, up to constant factors,
\begin{align} \label{E:SEMILINEARBADTERM}
	\FutFailFac (\Radunit_{(Flat)} \Psi)^2.
\end{align}
Hence, for the equation
$\square_{g(\Psi)}\Psi = 0,$   
$\FutFailFac \equiv 0$ is 
equivalent to the absence, relative to the frame \eqref{E:MINKOWSKIFRAME}, 
of the dangerous quadratic 
terms $\Psi \Radunit_{(Flat)} (\Radunit_{(Flat)} \Psi)$
and $(\Radunit_{(Flat)} \Psi)^2.$

\begin{remark}
\label{Re:Aleph-systems}
In the case of 
\emph{general systems} of the form \eqref{E:SPECIFICSEMILINEARTERMSGENERALQUASILINEARWAVE}  
with $\Psi=\lbrace \Psi^I \rbrace_{I=1,\ldots, N},$ 
the correct definition of 
$\FutFailFac$ has to be changed, in view of possible cancellations between
components. For example, if 
$\Phi$ verifies the scalar equation $g^{\a\b}(\pr \Phi)\pr_\a\pr_\b \Phi=0$
and $\Psi: = (\Psi_0, \Psi_1, \Psi_2, \Psi_3),$
where $\Psi_{\lambda} := \partial_{\lambda} \Phi,$
then the relevant definition of $\FutFailFac$ is as follows:
\begin{align} \label{E:ALEPH-SYSTEMS}
\FutFailFac
:= m_{\kappa \lambda} G_{\alpha \beta}^{\kappa}(\Psi = 0) 
\Lunit_{(Flat)}^{\alpha}\Lunit_{(Flat)}^{\beta} \Lunit_{(Flat)}^{\lambda},
\end{align}
where 
\begin{align} \label{E: GSYSTEMS}
	G_{\alpha \beta}^{\lambda}
	= G_{\alpha \beta}^{\lambda}(\Psi) 
	&:=\frac{\partial}{\partial \Psi_\la} g_{\a\b}(\Psi).
\end{align}
If, for the equation $g^{\a\b}(\pr \Phi)\pr_\a\pr_\b \Phi=0,$
we repeat the Minkowskian frame decomposition carried out above
for the equation $\Box_{g(\Psi)}\Psi = 0$,
we find that up to constant factors, 
$\FutFailFac$ as defined in \eqref{E:ALEPH-SYSTEMS}
is precisely the coefficient of the dangerous quadratic term
$(\Radunit_{(Flat)}\Phi)\cdot\Radunit_{(Flat)}(\Radunit_{(Flat)}\Phi)$. 
We shall return to this issue in Subsect.~\ref{SS:EXTENSIONSOFTHESHARPCLASSICALLIFESPANTHEOREMTOALINHACSEQUATION}.
\end{remark}
 
\subsubsection{Structural assumptions on the nonlinearities in equation \eqref{E:GENERALQUASILINEARWAVE}}
\label{SSS:STRUCTURAL}
We are now ready to make assumptions on the metric $g$ 
and the semilinear term $\NN(\Psi)(\partial \Psi, \partial \Psi)$ 
from equation \eqref{E:SPECIFICSEMILINEARTERMSGENERALQUASILINEARWAVE}
for which we can derive a small-data shock-formation result
in the region $\lbrace t \geq 0 \rbrace.$

\begin{enumerate}
\item  To produce a shock, we assume that the metric $g$ verifies the condition $\FutFailFac \not \equiv 0$
	(see Definition \ref{D:INTROFAILUREFACTOR}).
	\item We assume that for $\Psi$ sufficiently small,
		the semilinear term $\NN(\Psi)(\partial \Psi, \partial \Psi)$
		on the right-hand side of \eqref{E:SPECIFICSEMILINEARTERMSGENERALQUASILINEARWAVE} has
		the following structure when it is decomposed relative to the non-rescaled dynamic frame 
		$\lbrace \Lunit, \Radunit, X_1, X_2 \rbrace:$
		\begin{equation} \label{E:SEMILINEARSTRONGNULL}
			\text{No terms in the expansion of }
				\NN(\Psi)(\partial \Psi, \partial \Psi) 
			\text{ involve the factor  }
				(\Radunit \Psi)^2.
		\end{equation}
		\end{enumerate}

		\begin{remark} 
		\label{R:HARMLESSSEMILINEAR}
			A term $\NN(\Psi)(\partial \Psi, \partial \Psi)$
			verifying the above assumptions should be viewed as a negligible
			error term that does not interfere with the shock formation processes.
		\end{remark}

	\begin{remark}
			To further explain the relevance of the condition $\FutFailFac \not \equiv 0$
			in the shock-formation problem, it pays to redo,
			relative to 
		 	the non-rescaled dynamic frame $\lbrace \Lunit, \Radunit, X_1, X_2 \rbrace,$ 
			the analysis 
			that identified the dangerous terms 
			\eqref{E:PRINCIPALBADTERM} and \eqref{E:SEMILINEARBADTERM}.
			In doing so, we find that the dangerous terms are,
			up to constant factors,
		 	$G_{\Lunit \Lunit} \Psi \Radunit (\Radunit \Psi)$
		 	and $G_{\Lunit \Lunit} (\Radunit \Psi)^2,$ 
		 	where
		 	\begin{align} \label{E:NONLINEARNULLCONDITIONFAILUREFACTOR}
	    	G_{\Lunit \Lunit} := G_{\alpha \beta}(\Psi) \Lunit^{\alpha} \Lunit^{\beta}.
	    \end{align}
	    The connection with $\FutFailFac$ is: 
	    the decay estimates of the Heuristic Principle 
	    can be used to show that
	    $G_{\Lunit \Lunit}$ is well-approximated by $\FutFailFac$ along the integral curves of $\Lunit.$    
	    Hence, if $\FutFailFac \not \equiv 0,$ then the dangerous terms 
	    have the strength needed to drive shock formation.
	 \end{remark}

\begin{remark}[\textbf{Future strong null condition}] \label{R:STRONGNULL}
	Note that the condition \eqref{E:SEMILINEARSTRONGNULL}
	for $\NN(\Psi)(\partial \Psi, \partial \Psi)$ 
	\emph{cannot be extended to include arbitrary cubic or higher order terms} in $\pr \Psi.$
	 Though such terms are harmless in the context of proving small-data global existence 
	(for example, when the classic null condition is verified),
 	this is no longer the case
 	if we expect $\Radunit \Psi$ to become singular,  
 	since in that case cubic terms can become dominant whenever $\Radunit \Psi$ 
 	is large. A useful version of the null condition for all higher order terms    
 	in $\pr\Psi$ can only allow terms which are \emph{linear} with respect to the directional derivative $\Radunit \Psi.$  
 	Such a condition may be called the \emph{future strong null condition},
 	where we explain the ``future'' aspect of it in Remark \ref{R:ASYMMETRY}.
 	We stress that 
 	\emph{the future strong null condition is a true nonlinear condition tied to the dynamic metric} $g,$
 	as opposed to the classic null condition, 
 	which is based on Taylor expanding a nonlinearity around $0$ and keeping
 	only the quadratic part.
 \end{remark}
 
 \begin{remark}[\textbf{Asymmetry between the future and the past}]
 \label{R:ASYMMETRY}
 	Note that 
 	$\NN(\Psi)(\partial \Psi, \partial \Psi) := (\Lunit \Psi)^2$
 	verifies the future strong null condition even though 
 	the flat analog term $(\Lunit_{(Flat)} \Psi)^2$ fails the classic null condition.
 	Hence, strictly speaking, it is not correct to view the
 	future strong null condition as more restrictive
 	than the classic null condition.
 	The relevant point is that as $t \to \infty,$ 
 	$(\Lunit \Psi)^2$ is expected to decay sufficiently quickly, while 
 	the same behavior for $(\Lunit \Psi)^2$ 
 	is not expected as $t \to - \infty.$
 	We could also formulate a ``past strong null condition.''
 	We would simply need to replace the dynamic frame 
 	$\lbrace \Lunit, \Radunit, X_1, X_2 \rbrace$
 	used in the statement \eqref{E:SEMILINEARSTRONGNULL}
 	with an analogous dynamic frame whose first vector
 	is $g-$null and outgoing as $t \downarrow - \infty.$
 \end{remark}

\subsection{The role of $ \upmu$} 
\label{SS:ROLEOFINVERSEFOLIATIONDENSITY}
Here, we continue our rough description of 
shock formation and show that $\upmu \to 0$ 
precisely corresponds to the formation of a shock and the blow-up of 
the directional derivative $\Radunit \Psi.$
For simplicity, we focus on
solutions that are nearly spherically symmetric, 
at least in the sense of lower-order derivatives.
The argument we give here closely follows the argument
given in spherical symmetry in Subsect.~\ref{subs:radial-blow}
(see in particular the proof of Cor.~\ref{C:ge}).

As a first step, we use the wave equation \eqref{E:ALTWAVEOPERATORDECOMPOSED} to infer the existence of an open set of initial data such that,
\emph{for sufficiently large} $t,$ 
a lower bound of the form
\begin{align}  \label{E:KEYLOWER}
	\Rad \Psi(t,u,\vartheta)  \gtrsim \mathring{\upepsilon} \frac{1}{1 + t}
\end{align}
holds along some integral curve of $\Lunit$ (that is, at fixed $u$ and $\vartheta$), 
 with $\mathring{\upepsilon}$ the size of the data. 
 Alternatively, for a different open set of data,
 we could derive the bound
 $\Rad \Psi(t,u,\vartheta) \lesssim - \mathring{\upepsilon} (1 +
t)^{-1}$ (again for fixed $u,\vartheta$).
 To derive these bounds, we use the fact  
 that the last two terms on the right-hand side of \eqref{E:ALTWAVEOPERATORDECOMPOSED} are
 small error terms that decay at an integrable-in-time rate,
 thereby deducing that
  \begin{align} \label{E:INTEGRATEDREADYALTWAVEOPERATORDECOMPOSED}
	\Lunit \left\lbrace
						\upmu \Lunit (\rgeo \Psi) + 2 \Rad (\rgeo \Psi)
		\right\rbrace
		& = \err.
\end{align}
Hence, we can integrate \eqref{E:INTEGRATEDREADYALTWAVEOPERATORDECOMPOSED} along the integral curves of $\Lunit$ to deduce
\begin{align} \label{E:NEARLYCONSTANTTRANSPORTTERM}
	\left\lbrace
		\upmu \Lunit (\rgeo \Psi) + 2 \Rad (\rgeo \Psi) 
	\right\rbrace(t,u,\vartheta)
	\approx f_{data}(u,\vartheta),
\end{align}
where $f_{data}(u,\vartheta)$ is equal to $\upmu \Lunit(\rgeo \Psi) + 2 \Rad(\rgeo \Psi)$
evaluated at $(0,u,\vartheta).$ 
Expanding the left-hand side \eqref{E:NEARLYCONSTANTTRANSPORTTERM} via the Leibniz rule and
appealing to the Heuristic Principle decay estimates, 
we see that all terms except for $2 \rgeo \Rad \Psi$ decay.
 Hence, we find that for suitably large times, we have
\begin{align} \label{E:KEYLOWERBOUNDFORRADPSI}
	\Rad \Psi(t,u,\vartheta) 
	\approx \frac{1}{2} \frac{1}{\rgeo(t,u)} f_{data}(u,\vartheta)
	\approx \frac{1}{1 + t} f_{data}(u,\vartheta).
\end{align}
We have therefore derived the desired bounds.

\begin{remark}[\textbf{Remarks on the linear term} $\upmu \rgeo \angLap \Psi$]
	\label{R:ANGULARDERIVATIVESEVENSMALLER}
 	The linearly small product $\upmu \rgeo \angLap \Psi$
 	is present in the term $\err$ in equation
	\eqref{E:INTEGRATEDREADYALTWAVEOPERATORDECOMPOSED}
	(see equation \eqref{E:WAVEOPERATORDECOMPOSED}).
	 At $t=0,$ the term $\upmu \rgeo \angLap \Psi$ can be large
	compared to $f_{data}(u,\vartheta).$
	Hence, in order for the above proof of 
	\eqref{E:KEYLOWERBOUNDFORRADPSI} to work,
	we must assume that the initial angular derivatives of
	$\Psi$ are even smaller than the other derivatives.
	However,
	using a more refined argument based on Friedlander's radiation field,
	one can significantly enlarge the set of small data for
	which it is possible to prove a lower bound
	of the form \eqref{E:KEYLOWER}; see Subsect.~\ref{SS:DISCUSSIONOFSHOCKFORMINGDATA}.
\end{remark}

Next, we insert the bound \eqref{E:KEYLOWERBOUNDFORRADPSI} into
the evolution equation \eqref{E:UPMUSCHEMATICTRANSPORT} for $\upmu$ and ignore the error terms,
which are small and decaying sufficiently fast
by the Heuristic Principle. 
Although the factor $G_{\Lunit \Lunit}$ in equation \eqref{E:UPMUSCHEMATICTRANSPORT}
is not constant along the integral curves of $\Lunit,$
the Heuristic Principle decay estimates can be used to show that
$G_{\Lunit \Lunit}$ is well-approximated
(relative to the geometric coordinates) by 
\begin{align} \label{E:DATAFAILREFACTOR}
	\InitialFutFailFac(t,u,\vartheta) 
	= \InitialFutFailFac(\vartheta) 
	:= \FutFailFac(t=0,u=0,\vartheta).
\end{align}
The good feature of $\InitialFutFailFac$ is that it 
(by definition)
depends only\footnote{Recall that $\FutFailFac$ was defined in \eqref{E:INTROFAILUREFACTOR}
and that at $t=0,$ $u=1-r$ and the geometric angular coordinates coincide with the 
standard Euclidean angular coordinates. Since $\FutFailFac$ can be viewed as a function 
depending only on the standard Euclidean angular coordinates,
it follows that indeed, the right-hand side of \eqref{E:DATAFAILREFACTOR}
depends only on $\vartheta.$} on the geometric angular coordinates $\vartheta$
and hence is constant along the integral curves of $\Lunit.$
Hence, for suitably large times, we have
\begin{align} \label{E:LUPMUKEYUPPERBOUND}
	\Lunit \upmu(t,u,\vartheta) 
	& \approx \frac{1}{2} \InitialFutFailFac(\vartheta) \Rad \Psi(t,u,\vartheta) 
	\approx \frac{1}{2} \InitialFutFailFac(\vartheta) \frac{1}{1 + t} f_{data}(u,\vartheta).
\end{align}

Integrating \eqref{E:LUPMUKEYUPPERBOUND} along the integral curves of $\Lunit$ and
using the small-data assumption that $\upmu$ is initially near $1,$ we deduce that
\begin{align} \label{E:UPMUKEYAPPROXIMATION}
	\upmu(t,u,\vartheta) \approx 1 + \frac{1}{2} \InitialFutFailFac(\vartheta) \ln(1 + t) f_{data}(u,\vartheta).
\end{align}
Clearly, 
if the data are such that for some angle $\vartheta,$
$\InitialFutFailFac(\vartheta) f_{data}(u,\vartheta)$
is negative, then
\eqref{E:UPMUKEYAPPROXIMATION}
implies that $\upmu$ will become $0$ 
at a time of order $\exp(C \mathring{\upepsilon}^{-1}).$

We now remind the reader of the following simple consequence of the above discussion:
in view of lower bound \eqref{E:KEYLOWERBOUNDFORRADPSI} and the relation
$\Radunit \Psi = \upmu^{-1} \Rad \Psi,$ where 
$\Radunit \sim - \partial_r$ has close to Euclidean-unit-length, 
it follows that some rectangular spatial derivative of 
$\Psi$ blows up precisely when $\upmu$ vanishes.

In the work \cite{dC2007}, Christodoulou studied quasilinear wave equations for which 
the analog of $\FutFailFac$ was constant-valued, as in the case of
John's equation, which we discussed in the first example given just below 
Definition \ref{D:INTROFAILUREFACTOR}.
This property simplified some of his analysis and, 
as we describe in Subsect.~\ref{SS:DISCUSSIONOFSHOCKFORMINGDATA},
it played a central role in his identification of a class of small data that
lead to shock formation for his equations.
In general, $\FutFailFac$ can be highly angularly dependent 
and in particular, there can be angular directions along which
the function $\InitialFutFailFac$
from \eqref{E:DATAFAILREFACTOR} vanishes.
Along the integral curves of $\Lunit$ corresponding to such angular directions,
$\upmu$ is not expected to change very much 
during the solution's classical lifespan.

\section{Generalized energy estimates}  
	\label{S:GENERALIZEDENERGY} 
        In this section, we discuss the most difficult aspect 
        of proving small-data shock formation away from spherical symmetry:
        the derivation of generalized energy estimates that hold up to top order.
        Our discussion in this section applies to
        the nonlinear wave equation 
				$\square_{g(\Psi)}\Psi = \NN(\Psi)(\partial \Psi, \partial \Psi)$
				(that is, \eqref{E:SPECIFICSEMILINEARTERMSGENERALQUASILINEARWAVE})
				in the region $\lbrace t \geq 0 \rbrace$
       	under the structural conditions on $\NN$
       	stated in Subsubsect.~\ref{SSS:STRUCTURAL}.
    
    \subsection{The basic strategy for deriving generalized energy estimates}    
      \label{SS:STRATEGYFORGENERALIZEDENERGY}
      The discussion in the previous sections suggests 
      the following strategy for proving shock formation
      in solutions to equation \eqref{E:SPECIFICSEMILINEARTERMSGENERALQUASILINEARWAVE}.
      
   	\begin{enumerate}
      \item  With the help of the eikonal function $u,$ one should 
				construct \emph{commutator vectorfields} $Z$     
      	that have good commuting properties
       	with our nonlinear wave equation \eqref{E:SPECIFICSEMILINEARTERMSGENERALQUASILINEARWAVE}.   
       	It turns out that a suitable collection of commutators is the set
       \begin{align} 
        \mathscr{Z} :=
        \lbrace \rgeo \Lunit, \Rad, \Rot_{(1)}, \Rot_{(2)}, \Rot_{(3)} \rbrace,\label{eq:commuting-vfs}
        \end{align}
				which has span equal to $\mbox{span}\lbrace \partial_{\alpha} \rbrace_{\alpha = 0,1,2,3}$ at each
				spacetime point where $\upmu > 0.$
        The rotational vectorfields $\Rot_{(l)}$ are constructed by projecting the standard  
        Euclidean rotation vectorfields $\Rot_{(Flat,l)}^j := \epsilon_{laj}x^a$ 
        onto the spheres $S_{t,u},$ where $\epsilon_{ijk}$ is the fully antisymmetric
        symbol normalized by $\epsilon_{123} = 1;$
        see \cite[Chapter 5]{jS2014} for more details.
        Note that all vectorfields in $\mathscr{Z}$ depend on the first derivatives of $u.$
       \item Derive generalized energy estimates for a sufficiently large number of $Z$-derivatives 
       of $\Psi.$ The work of the third author showed (see Prop.~\ref{P:APRIORIENERGYESTIMATES}) that
       it suffices to commute the nonlinear wave equation 
			\eqref{E:SPECIFICSEMILINEARTERMSGENERALQUASILINEARWAVE}
			up to $24$ times with the commutation vectorfields belonging to 
			$\mathscr{Z}.$
       Typically, in the flat case, such estimates are derived by contracting the energy-momentum
			tensorfield (see \eqref{E:INTROENERGYMOMENTUMTENSOR}) against the following two
       \textit{multiplier vectorfields}:
       \beaa
         \pr_t = \frac 1 2\left\lbrace \Lunit_{(Flat)} + \uLunit_{(Flat)}\right\rbrace,
				\qquad \Mor_{(Flat)} =\frac 1 2 \left\lbrace(t+r)^2 \Lunit_{(Flat)} + (t-r)^2 \uLunit_{(Flat)} \right\rbrace,
       \eeaa 
     	 where $\Lunit_{(Flat)} :=\pr_t+\pr_r$ and  $\uLunit_{(Flat)}: =\pr_t-\pr_r $ are the standard 
     	 radial null pair, as in \eqref{E:STANDARDMINKOWSKINULLPAIR}. We remark that $\partial_t$ is Killing while $\Mor_{(Flat)}$ is conformally Killing in Minkowski space.
       If one is interested only in the region\footnote{Recall that
$u_{(Flat)} := 1 - r + t$ is an eikonal function corresponding to the
Minkowski metric, so in this regime $r\approx 1+t.$} where $0 < u_{(Flat)} < 1,$
       we can replace  the Morawetz vectorfield $\Mor_{(Flat)}$ by $r^2 \Lunit_{(Flat)}.$   
       The vectorfields that we use in the shock-formation problem
       in the region $\MM_{t,u}$ 
       (see \eqref{E:MTUDEF})
       are the following dynamic versions,
			which are essentially the same as the vectorfields used in \cite{dC2007}:
       \begin{subequations}
	\begin{align}
		\Mult 
		& := (1 + \upmu) \Lunit + \uLgood
			= (1 + 2 \upmu) \Lunit + 2 \Rad, 
			\label{E:INTRODEFINITIONMULT} \\
		\Mor
		& := \rgeo^2 \Lunit.
			\label{E:INTRODEFINITIONMOR}
	\end{align}
	\end{subequations}
	$\Mult$ is a $g-$timelike vectorfield that is designed to
	yield generalized energy quantities that are useful 
	both in regions where $\upmu$ is large and where it is small.
	$\Mor$ is a $g-$null vectorfield whose role we will explain below.
	These vectorfields are neither Killing nor conformal Killing,\footnote{That is, their deformations tensors 
	\eqref{E:DEFORMATIONTESNORDEF}
	neither vanish nor are proportion to the metric.}
	even when $\Psi \equiv 0.$
	Nonetheless, the energy estimate error terms corresponding to their deformation tensors
	(see the right-hand side of \eqref{E:MTUDIVERGENCETHM})
	are controllable in the region $\MM_{t,u}.$  
	Actually, as we will see in Lemma \ref{L:QUANTIFIEDMORAWETZCOERCIVENESS},
	one of the deformation tensor terms corresponding to $\Mor$
	has a favorable sign and is important for controlling
	other error terms.
	\item As long as we can suitably bound the generalized energy quantities, 
	 based on the commutation vectorfields \eqref{eq:commuting-vfs} and multiplier vectorfields 
	 \eqref{E:INTRODEFINITIONMULT}, \eqref{E:INTRODEFINITIONMOR}, we can also derive, via
	Sobolev embedding, decay estimates  
	for the low-order derivatives of $\Psi,$ consistent with our Heuristic Principle;
	see \cite[Corollary 17.2.2]{jS2014} for the details.
 \item The deformation tensors 
(see \eqref{E:DEFORMATIONTESNORDEF})
of the commutator vectorfields $Z$ can be expressed 
 in terms of the covariant Hessian $H= \D^2 u,$ which verifies a transport equation 
	of the form 
 \bea
 \Lunit H + H^2 = \mathcal{R} \label{eq:ricatti-H-1},
 \eea
	where, as we have mentioned, $\mathcal{R}$ depends on the up-to-second-order derivatives of $\Psi$
	and the up-to-second-order derivatives of $u.$
  As we explained in Subsubsect.~\ref{susub:prev-eikonal}, 
  every time we commute $\square_{g(\Psi)}$ with a 
   vectorfield $Z,$ we generate terms of the form $(\D \piZ)\c \D \Psi$
    which can be traced back, via the transport equation \eqref{eq:ricatti-H-1},
    to one more derivative of $\Psi$ than we are able to control
    by an energy estimate at the same level. However, it is essential to note 
    that even though we lose a derivative,
    we do not introduce any factors of $\upmu^{-1},$ 
    which would blow-up as we approach
    the expected singularity.\footnote{To see this in detail, one must decompose \eqref{eq:ricatti-H-1}
    relative to the rescaled frame $\lbrace \Lunit,\Rad, X_1, X_2 \rbrace.$ 
    At one derivative level lower, the $\upmu$-regular behavior can be seen in the transport equation 
    \eqref{E:UPMUSCHEMATICTRANSPORT}  
     for $\upmu,$ where there are no factors of $\upmu^{-1}$ present.
             }  
             In other words, we can derive estimates for        
             the components of the derivatives of 
             $H$ relative to the rescaled frame
             $\lbrace \Lunit, \Rad, X_1, X_2 \rbrace$
             that are regular with respect to $\upmu,$
             but only at the expense of losing a derivative. 
       This is key to understanding Christodoulou's strategy: at the
top level we combat derivative loss through renormalization 
(see the next item), which has
as a trade-off the introduction of a factor of $\upmu^{-1};$ at the lower derivative levels we
can avoid this factor since the derivative loss can be absorbed.
This trade-off is where understanding the dynamic geometry is
most important.

              \item  To control the top derivatives        
              of $\Psi$ when the loss of a derivative, 
              due to \eqref{eq:ricatti-H-1}, can no longer be 
							ignored, we use a renormalization procedure,\footnote{The procedure involves combining
		\eqref{eq:ricatti-H-1} 
		with the  wave equation $\square_{g(\Psi)} \Psi = 0$ and using elliptic estimates on the surfaces
	$S_{t,u}.$ This is similar to the approach taken in
	\cite{dCsK1993} and \cite{sKiR2003}.}      
               which recovers the   
               loss of derivatives mentioned above    
               at the expense of introducing a          
							dangerous factor of $\upmu^{-1}$;	see Subsubsect.~\ref{SSS:AVOIDINGTOPORDERDERIVATIVELOSS}.
              This new difficulty of having to derive a priori estimates 
              in the presence of the singular factor
							$\upmu^{-1}$ is handled by Christodoulou 
							with the help of a subtle Gronwall-type inequality, which we provide as
              Lemma \ref{L:INTROKEYINTEGRATINGFACTORGRONWALLESTIMATE}. 
      \end{enumerate}

\subsection{Energy estimates via the multiplier method}
Before specializing to equation \eqref{E:ROTCOMMUTEDWAVE}, we first recall
the multiplier method for deriving generalized energy estimates
for solutions to 
\begin{align} \label{E:WAVEINHOM}
	\upmu \square_g \Psi = \waveinhom.
\end{align}

\subsubsection{A version of the divergence theorem via the multiplier method}
One key ingredient is the \emph{energy-momentum tensorfield}
\begin{align} \label{E:INTROENERGYMOMENTUMTENSOR}
	\enmomtensor_{\mu \nu}[\Psi]
	= \enmomtensor_{\mu \nu}
	& := \D_{\mu} \Psi  \D_{\nu} \Psi
	- \frac{1}{2} g_{\mu \nu} \D^{\alpha} \Psi  \D_{\alpha} \Psi.
\end{align}
It is straightforward to check that for solutions to \eqref{E:WAVEINHOM}, we have
\begin{align} \label{E:DIVOFENERGYMOMENTUM}
\upmu \, \D_{\alpha} \enmomtensor^{\alpha \nu}= \waveinhom \D^{\nu} \Psi.
\end{align}
Furthermore, for any pair of future-directed vectorfields $V$ and $W$
verifying $g(V,V), g(W,W) \leq 0,$ 
we have the well-known inequality\footnote{In general relativity, inequality \eqref{E:DOMINANTENERGY}
is often referred to as the \emph{dominant energy condition}.}
which plays a role in the construction of coercive $L^2$ quantities:
\begin{align}	\label{E:DOMINANTENERGY}
	\enmomtensor_{\alpha \beta} V^{\alpha} W^{\beta} \geq 0.
\end{align}

The following currents are useful for bookkeeping during integration by parts.
Specifically, to any auxiliary ``multiplier'' vectorfield $X,$ we associate the following 
\emph{compatible current} vectorfield.
\begin{definition}[\textbf{Compatible current}]
\begin{align} \label{E:ENERGYCURRENT}
	\JX^{\nu}[ \Psi]
		& := \enmomtensor_{\ \alpha}^{\nu}[\Psi] X^{\alpha}.
\end{align}
\end{definition}
By \eqref{E:DIVOFENERGYMOMENTUM}, for solutions $\Psi$ to \eqref{E:WAVEINHOM}, we have
\begin{align}
	\upmu \D_{\alpha} \JX^{\alpha}
	& = \frac{1}{2}
			\upmu \enmomtensor^{\alpha \beta}[\Psi] \piX _{\alpha \beta}
			+ (X \Psi) \waveinhom,
\end{align}
where $\piX_{\alpha \beta} = \D_{\alpha} X_{\beta} + \D_{\beta} X_{\alpha}$
is the deformation tensor of $X,$ as in \eqref{E:DEFORMATIONTESNORDEF}

To derive generalized energy estimates, we apply 
the divergence theorem on the region $\MM_{t,u}$ 
(see Figure \ref{F:DIVTHM})
to obtain the following energy-flux identity for solutions to 
\eqref{E:WAVEINHOM}.
\begin{lemma}\cite[\textbf{Lemma 9.2.1; Divergence theorem}]{jS2014}
For solutions $\Psi$ to $\upmu \square_g \Psi = \waveinhom$
that vanish along\footnote{Recall that the vanishing of $\Psi$ along $\mathcal{C}_0$ 
is an easy consequence of our assumptions on the support of the data.} $\mathcal{C}_0,$ we have
\begin{align} \label{E:MTUDIVERGENCETHM}
	\int_{\Sigma_t^u} 
		\upmu  \enmomtensor[\Psi](X,\Timenormal) 
	\, d \tvol
	+ 
	\int_{\mathcal{C}_u^t}  
		\enmomtensor[\Psi](X,\Lunit) 
	\, d \conevol
	& = \int_{\Sigma_0^u} 
				\upmu \enmomtensor[\Psi](X,\Timenormal) 
			\, d \tvol
				\\
	& \ \ - \int_{\MM_{t,u}}
					(X \Psi) \waveinhom
				\, d \vol \nn
		- 
			\frac{1}{2} 
			\int_{\MM_{t,u}}
				\upmu \enmomtensor[\Psi] \cdot \piX   
			\, d \vol,
			\notag
\end{align}
where $\enmomtensor[\Psi] \cdot \piX := \enmomtensor^{\alpha \beta} \piX_{\alpha \beta}.$
\end{lemma}

In \eqref{E:MTUDIVERGENCETHM}, $\Timenormal = \Lunit + \Radunit$
is the future-directed unit-normal to $\Sigma_t^u,$
$\enmomtensor(X,\Timenormal) = g(\JX,\Timenormal)$
and
$\enmomtensor(X,\Lunit) = g(\JX,\Lunit).$
Furthermore, 
\begin{align} \label{E:RESCALEDFORMS}
	d \tvol := \sqrt{\mbox{det} \gsphere} \, d \vartheta du',
		\qquad
	d \conevol := \sqrt{\mbox{det} \gsphere} \, d \vartheta dt',
		\qquad
	d \vol := \sqrt{\mbox{det} \gsphere} \, d \vartheta du' dt'
\end{align}
are rescaled volume forms on
$\Sigma_t^u,$
$\mathcal{C}_u^t,$
and $\MM_{t,u}.$ 
As before, $\gsphere$ is the Riemannian metric induced by $g$
on the spheres $S_{t,u}$ and the determinant is taken
relative to the geometric angular coordinates $(\vartheta^1,\vartheta^2).$
We call the above volume forms ``rescaled'' because 
the canonical volume forms 
induced by $g$
on $\Sigma_t^u$ and $\MM_{t,u}$
are
$\upmu \, d \tvol$
and $\upmu \, d \vol.$

Note that by the property \eqref{E:DOMINANTENERGY},
the first two integrands on the left-hand 
side of \eqref{E:MTUDIVERGENCETHM}
are non-negative for both of the multiplier vectorfields 
$X=\Mult$ and $X=\Mor;$
see Prop.~\ref{P:COERCIVEENERGIESANDFLUXES}
for a more precise account of the coerciveness of these terms.

\begin{center}
\begin{overpic}[scale=.2]{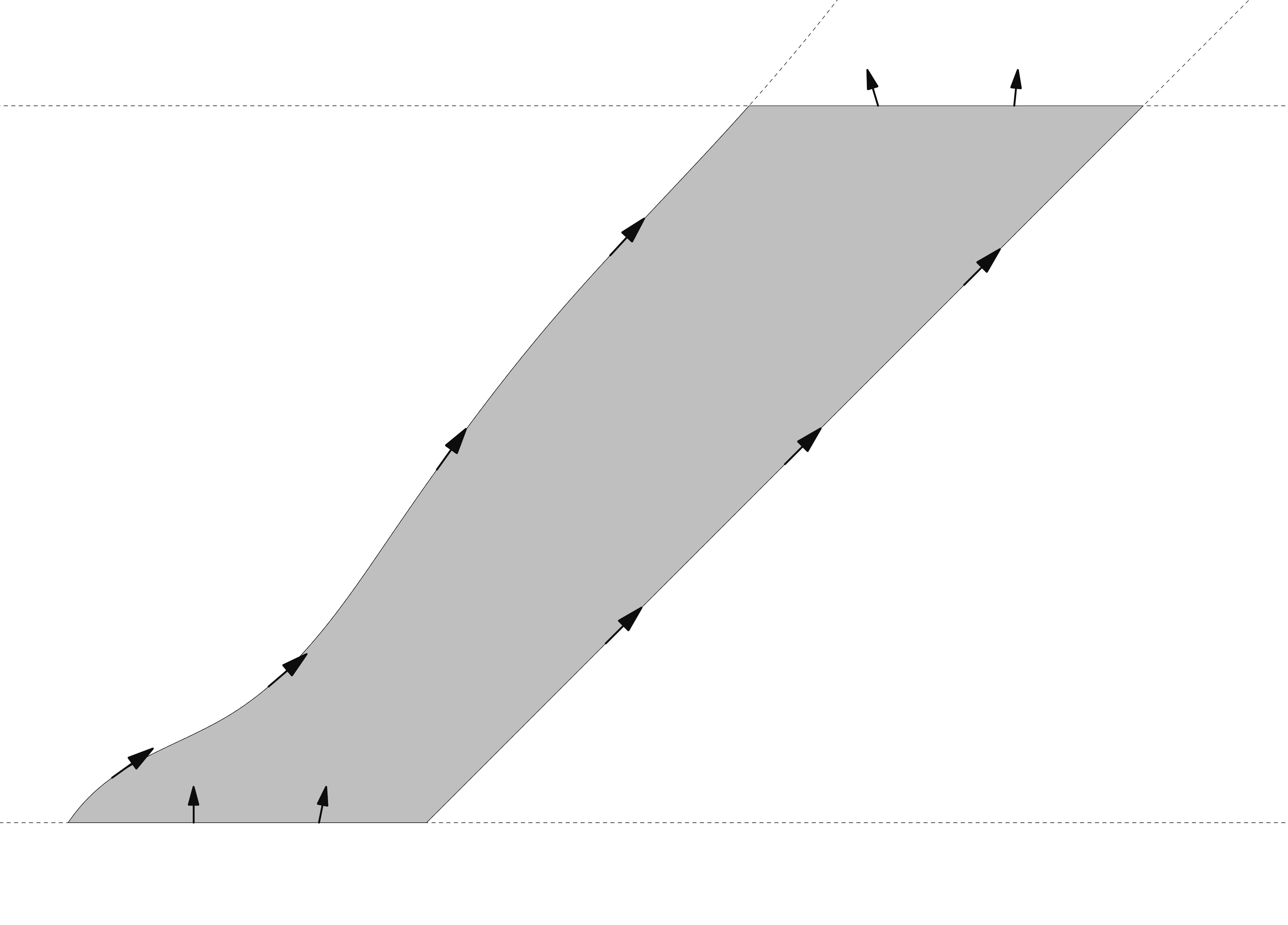}  
	\put (51,42) {\large$\displaystyle \MM_{t,u}$}
\put (77,67) {\large$\displaystyle N$}
\put (33.5,32) {\large$\displaystyle \Lunit$}
\put (95.9,67.2) {\large$\displaystyle \mathcal{C}_0$}
\put (55.8,67.2) {\large$\displaystyle \mathcal{C}_u$}
\put (30,39.7) {\large$\displaystyle \mathcal{C}_u^t$}
\put (68.2,39.7) {\large$\displaystyle \mathcal{C}_0^t$}
\put (34.9,65.9) {\large$\displaystyle \Sigma_t$}
\put (17.7,3) {\large$\displaystyle \Sigma_0^u$}
\put (70.3,65.9) {\large$\displaystyle \Sigma_t^u$}
\put (70.7,27.7) {\large$\displaystyle \Psi \equiv 0$}
\put (57.4,62.6) {$\displaystyle \bullet$}
\put (56.2,58.5) {\large$\displaystyle S_{t,u}$}
\put (87.9,62.6) {$\displaystyle \bullet$}
\put (87.1,58.5) {\large$\displaystyle S_{t,0}$}
\put (4.2,6.7) {$\displaystyle \bullet$}
\put (2.7,2.8) {\large$\displaystyle S_{0,u}$}
\put (32.2,6.7) {$\displaystyle \bullet$}
\put (30.7,2.8) {\large$\displaystyle S_{0,0}$}\end{overpic}
\captionof{figure}{The divergence theorem on $\MM_{t,u}$}
 \label{F:DIVTHM}
\end{center}

\begin{remark}[\textbf{Lower-order correction term}]
\label{R:CORRECTIONCURRENT}
Actually, in the case $X=\Mor,$ we need to modify the 
current \eqref{E:ENERGYCURRENT} by adding to it the
lower-order current
$\Jenergycurrent{Correction}^{\nu}[\Psi] := \frac{1}{2}
	\left\lbrace
		\rgeo^2 \mytr \upchi \Psi \D^{\nu} \Psi
		- \frac{1}{2} \Psi^2 \D^{\nu} [\rgeo^2 \mytr \upchi]
	\right\rbrace.$
We need this correction current because
 $\piMor_{\alpha \beta}$
fails to vanish even
in the case of Minkowski spacetime, and it turns out that the corresponding error terms are not controllable.
The use of lower-order corrections (``Lagrangian term") is 
standard and is often used 
even in the case of semilinear wave equations;
see, for example, \cite{sK2001}.
\end{remark}
\begin{remark}\label{R:MODCURRENTNOTINDIVTHM}
The divergence theorem identity \eqref{E:MTUDIVERGENCETHM} 
does not account for the effect of adding
the correction current 
$\Jenergycurrent{Correction}^{\nu}[\Psi]$ 
described in Remark \ref{R:CORRECTIONCURRENT}.
To adjust \eqref{E:MTUDIVERGENCETHM} so that it is correct
after the modification, one needs to include some additional integrals in the identity \eqref{E:MTUDIVERGENCETHM}, and in particular,
the last integral needs to be replaced with
$	- 
			\frac{1}{2} 
			\int_{\MM_{t,u}}
				\upmu \enmomtensor^{\alpha \beta} \left\lbrace \piMor_{\alpha \beta} - \rgeo^2 \mytr \upchi g_{\alpha \beta} \right\rbrace
			\, d \vol.$
\end{remark}

\subsubsection{Energies and fluxes}
With the help of the above currents and the corresponding divergence identities,  
we can now define the energies and fluxes, which are the main quantities 
that we use to control $\Psi$ and its derivatives in $L^2.$

\begin{definition}[\textbf{Energies and fluxes}]
\label{D:INTROENERGIESANDFLUXES}
	Let $\Timenormal := \Lunit + \Radunit$
	denote the future-directed unit normal to $\Sigma_t.$
	We define the energy $\enzero[\Psi](t,u)$ 
	and the cone flux $\flzero[\Psi](t,u)$
	corresponding to the multiplier vectorfield $\Mult$
	(see \eqref{E:INTRODEFINITIONMULT})
	in terms of the \emph{rescaled} volume forms 
	\eqref{E:RESCALEDFORMS} as follows:
	\begin{subequations}
	\begin{align}
		\enzero[\Psi](t,u)
		& := \int_{\Sigma_t^u} \upmu \enmomtensor[\Psi](\Mult,\Timenormal) \, d \tvol,
			\label{E:INTROE0DEF} 
	\end{align}
	\begin{align}
		\flzero[\Psi](t,u)
		& := \int_{\mathcal{C}_u^t} \enmomtensor[\Psi](\Mult,\Lunit) \, d \conevol.
			\label{E:INTROF0DEF} 
	\end{align}
	\end{subequations}
	We can also define similar quantities 
	$\enone[\Psi](t,u),$ $\flone[\Psi](t,u)$
	corresponding to the Morawetz multiplier $\Mor$
	(see \eqref{E:INTRODEFINITIONMULT}),
	but we have to take into account the lower-order terms mentioned 
	in Remark \ref{R:MODCURRENTNOTINDIVTHM}.
	\end{definition}

The following proposition reveals the coercive nature of the energies and fluxes.
Roughly speaking, its proof is based on carefully decomposing the energy-momentum tensor
\eqref{E:INTROENERGYMOMENTUMTENSOR} relative to\footnote{Actually, in the proof,
it is convenient to decompose relative to the rescaled null frame
$\lbrace \Lunit, \uLgood, X_1, X_2 \rbrace,$
where $\uLgood := \upmu \Lunit + 2 \Rad.$} the rescaled frame
$\lbrace \Lunit, \Rad, X_1, X_2 \rbrace.$

\begin{proposition}\cite[\textbf{Lemma 13.1.1; Coerciveness of the energies and fluxes}]{jS2014}
		\label{P:COERCIVEENERGIESANDFLUXES}
		Under suitable smallness bootstrap assumptions, 
		the energies and fluxes from Definition \ref{D:INTROENERGIESANDFLUXES}
		have the following coerciveness properties:
		\begin{subequations}
		\begin{align} \label{E:INTROMULTENERGYCOERCIVITY}
				\enzero[\Psi](t,u)
				& \geq 
					\| \Rad \Psi \|_{L^2(\Sigma_t^u)}^2
					+ C^{-1} \| \upmu \angD \Psi \|_{L^2(\Sigma_t^u)}^2
						\\
				& \ \
						+  
							C^{-1} \| \Psi \|_{L^2(S_{t,u})}^2,
						+
							C^{-1} \| \Psi \|_{L^2(\Sigma_t^u)}^2
						+
							C^{-1} \| \sqrt{\upmu} \Lunit \Psi \|_{L^2(\Sigma_t^u)}^2
						+
						  C^{-1} \| \upmu \Lunit \Psi \|_{L^2(\Sigma_t^u)}^2,
					\notag \\
				\flzero[\Psi](t,u)
					& \geq
						C^{-1} \| \Lunit \Psi \|_{L^2(\mathcal{C}_u^t)}^2
						+ C^{-1} \| \sqrt{\upmu} \Lunit \Psi \|_{L^2(\mathcal{C}_u^t)}^2
						+ C^{-1} \| \sqrt{\upmu} \angD \Psi \|_{L^2(\mathcal{C}_u^t)}^2,
						\label{E:INTROMULTCONEFLUXCOERCIVITY} 
		\end{align}
		\end{subequations}

		\begin{subequations}
		\begin{align}
				\enone[\Psi](t,u)
				& \geq    C^{-1}
									(1 + t)^2 
									\left\| \sqrt{\upmu} 
										\left(
											\Lunit \Psi 
											+ \frac{1}{2} \mytr \upchi \Psi 
										\right)
									\right\|_{L^2(\Sigma_t^u)}^2
							+ \frac{1}{2}
								\| \rgeo \sqrt{\upmu} 
									\angD \Psi
								\|_{L^2(\Sigma_t^u)}^2,
					 \label{E:INTROENONECOERCIVENESS} \\
				\flone[\Psi](t,u)
				& \geq
					C^{-1}
					\left\| 
						(1 + t') 
						\left(
							\Lunit \Psi 
							+ \frac{1}{2} \mytr \upchi \Psi 
						\right)
					\right\|_{L^2(\mathcal{C}_u^t)}^2.
					\label{E:INTROFLUXONECOERCIVENESS}
		\end{align}
		\end{subequations}
		The $L^2$ norms above are relative to 
		the rescaled volume forms
		$d \tvol,$
		and
		$d \conevol$
		(see \eqref{E:RESCALEDFORMS}),
		which do not degenerate as $\upmu \to 0.$
		Furthermore, $\rgeo(t,u) = 1 - u + t.$
	\end{proposition}

\begin{remark}
	Note that we have provided the explicit constant ``1'' in the term
	$\| \Rad \Psi \|_{L^2(\Sigma_t^u)}^2$
	on the right-hand side of
	\eqref{E:INTROMULTENERGYCOERCIVITY}
	and similarly for the second term on the right-hand side of \eqref{E:INTROENONECOERCIVENESS}.
	These constants are important because they affect the number of derivatives
	we need to close the estimates; see, for example, 
	the derivation of inequality \eqref{E:SHARPCONSTANTNEEDED}
	from inequality \eqref{E:REWRITINGOFKEYINTEGRAL}.
\end{remark}

\subsubsection{The role of $\upmu$ weights in the energies and fluxes}
Observe that the energies $\enzero$ and $\enone$ 
from Prop.~\ref{P:COERCIVEENERGIESANDFLUXES} control
only $\upmu-$weighted versions of $\Lunit \Psi$ and $\angD \Psi.$
Hence, for $\upmu$ near $0,$ they provide only very weak control 
over $\Lunit \Psi$ and $\angD \Psi.$  However,  when bounding various
error integrals on the right-hand side of \eqref{E:MTUDIVERGENCETHM}, 
we encounter \emph{non $\upmu-$weighted factors of} $\Lunit \Psi$
and $\angD \Psi,$ which cannot be controlled directly by $\enzero$ and $\enone.$ 
We give an example of such an error term 
and describe how to handle it
in Subsubsect.~\ref{SSS:IGNOREDERIVATIVELOSS}.
To handle the
non $\upmu-$weighted factors of $\Lunit \Psi$
when $\upmu$ is small,
we will  need to rely  on  the null-fluxes 
$\flzero$ and $\flone$ 
from Prop.~\ref{P:COERCIVEENERGIESANDFLUXES},
which provide control over
$\Lunit \Psi$ \emph{without any $\upmu$ weights}.

\subsubsection{The need for the Morawetz spacetime integral}
Note that Prop.~\ref{P:COERCIVEENERGIESANDFLUXES}
does not provide any quantity that yields control
of the non $\upmu-$weighted factors of $\angD \Psi$
when $\upmu$ is small. To obtain such control, 
we use a much more interesting and subtle estimate, 
first derived by Christodoulou in \cite{dC2007},
which we now discuss.

The main idea is that in the case of the Morawetz multiplier $\Mor,$
there is a subtly coercive term hiding in the
last integral on the right-hand side of \eqref{E:MTUDIVERGENCETHM}.
That is, a careful decomposition of the integrand
$- \frac{1}{2} \upmu \enmomtensor^{\alpha \beta} \left\lbrace \piMor_{\alpha \beta} - \rgeo^2 \mytr \upchi g_{\alpha \beta} \right\rbrace$
(see Remark \ref{R:MODCURRENTNOTINDIVTHM})
reveals the presence of an important \emph{negative} spacetime integral $- \Morint[\Psi]$
on the right-hand side of \eqref{E:MTUDIVERGENCETHM}.
The corresponding positive spacetime integral has the following structure.

\begin{definition}[\textbf{Coercive Morawetz spacetime integral}]
\label{D:MORINTEGRALDEF}
\begin{align} \label{E:MORINTEGRALDEF}
	\Morint[\Psi](t,u)
	& :=
	\int_{\MM_{t,u}}
			\rgeo^2
			[\Lunit \upmu]_{-}
			|\angD \Psi|^2
		\, d \vol.
\end{align}
 Here  $[\Lunit \upmu]_{-} = |\Lunit \upmu|$ when $\Lunit \upmu < 0$ and 
$[\Lunit \upmu]_{-} = 0$ otherwise. 

\end{definition}
The coerciveness of the Morawetz integral is provided by the following lemma.

\begin{lemma}\cite[\textbf{Lemma 13.2.1; Quantified coerciveness of the Morawetz spacetime integral}]{jS2014}
\label{L:QUANTIFIEDMORAWETZCOERCIVENESS}
The Morawetz integral $\Morint[\Psi]$ from Definition \ref{D:MORINTEGRALDEF}
verifies the following lower bound:
\begin{align} \label{E:MORINTEGRALCOERCIVE}
	\Morint[\Psi](t,u)
	\geq 
		\frac{1}{C}
		\int_{\MM_{t,u}}
			\mathbf{1}_{\lbrace \upmu \leq 1/4 \rbrace}
			\frac{1 + t'}{\ln(e + t')}
			|\angD \Psi|^2(t',u',\vartheta)
		\, d \vol.
\end{align}
\end{lemma}

The main idea behind the proof of Lemma \ref{L:QUANTIFIEDMORAWETZCOERCIVENESS}
is simple: just insert  an estimate very similar\footnote{ Indeed
in the small-data regime,
the same estimate
\eqref{E:SSLUNITUPMULARGEINMAGNITUDE} holds  even in the  non-spherical symmetric case (see \eqref{E:LUPMUNEGATIVEQUANTIFIED}).
}  to  \eqref{E:SSLUNITUPMULARGEINMAGNITUDE}
(derived in spherical symmetry) into \eqref{E:MORINTEGRALDEF}.
The important points concerning $\Morint[\Psi]$ are:
\begin{itemize}
 \item $-\Morint[\Psi](t,u)$ appears on the right-hand side of \eqref{E:MTUDIVERGENCETHM} (see Remark \ref{R:MODCURRENTNOTINDIVTHM})
		and hence we can bring $\Morint[\Psi](t,u)$ to the left and obtain additional
		spacetime control over $|\angD \Psi|^2.$
	\item It \emph{contains no $\upmu$ weights}, so it is significantly coercive even in regions
		where $\upmu$ is near $0.$
	\item The integrand features favorable factors of $t'.$ 
\end{itemize}

\subsubsection{Overview of the $L^2$ hierarchy and the $\upmu_{\star}^{-1}$ degeneracy}
We are almost ready to provide an overview of the main a priori energy-flux estimates. The estimates
involve the following important quantity, which captures the ``worst-case'' scenario
for $\upmu$ being small along $\Sigma_t^u.$

\begin{definition}[\textbf{A modified minimum value of $\upmu$}] \label{D:UPMUSTAR}
	We define the function $\upmu_{\star}(t,u)$ as follows:
	\begin{align} \label{E:UPMUSTAR}
		\upmu_{\star}(t,u)
		:= \min\lbrace 1, \min_{\Sigma_t^u} \upmu \rbrace.
	\end{align}
\end{definition}

Now that we have defined all of the quantities of interest,
we can now state a proposition that provides
the a priori energy-flux-Morawetz estimates that hold
on spacetime domains of the form $\mathcal{M}_{t,u}.$    
There is no analog of this proposition in spherical
symmetry because in the symmetric setting, we did not 
need to derive $L^2$ estimates.

\begin{proposition}\cite[\textbf{Lemma 19.2.3; Rough statement of the hierarchy of a priori energy-flux-Morawetz estimates}]{jS2014}
	\label{P:APRIORIENERGYESTIMATES}
	Assume that $\square_{g(\Psi)} \Psi = 0.$
	Assume that the data are of size $\mathring{\upepsilon},$ 
	defined by \eqref{E:INTROSMALLDATA}.
	Then there exist large constants $C > 0$ and $\Cononestar > 0$ such
	that if $\mathring{\upepsilon}$ is sufficiently small,
	then the following energy-flux-Morawetz estimates
	hold for the quantities 
	from Definitions \ref{D:INTROENERGIESANDFLUXES} and \ref{D:MORINTEGRALDEF}
	for $0 \leq M \leq 7:$
	\begin{subequations}
	\begin{align}
		\enzero^{1/2}[\mathscr{Z}^{\leq 15} \Psi](t,u)
		+ \flzero^{1/2}[\mathscr{Z}^{\leq 15} \Psi](t,u)
		& \leq C \mathring{\upepsilon},
		 	\label{E:LOWESTLEVELMULTNONDEGENERATE} \\
		\enone^{1/2}[\mathscr{Z}^{\leq 15} \Psi](t,u)
		+ \flone^{1/2}[\mathscr{Z}^{\leq 15} \Psi](t,u)
		+ \Morint^{1/2}[\mathscr{Z}^{\leq 15} \Psi](t,u)
		& \leq C \mathring{\upepsilon} \ln^2(\myexp + t),
			\label{E:LOWESTLEVELMORNONDEGENERATE}\\
		\enzero^{1/2}[\mathscr{Z}^{16 + M} \Psi](t,u)
		+ \flzero^{1/2}[\mathscr{Z}^{16 + M} \Psi](t,u)
		& \leq C \mathring{\upepsilon} \upmu_{\star}^{-.75-M},
			\label{E:MIDLEVELMULTESTIMATE} \\
		\enone^{1/2}[\mathscr{Z}^{16 + M} \Psi](t,u)
		+ \flone^{1/2}[\mathscr{Z}^{16 + M} \Psi](t,u)
		+ \Morint^{1/2}[\mathscr{Z}^{16 + M} \Psi](t,u)
		& \leq C \mathring{\upepsilon} \ln^2(\myexp + t) \upmu_{\star}^{-.75-M}(t,u),
			\label{E:MIDLEVELMORESTIMATE}  \\
		\enzero^{1/2}[\mathscr{Z}^{24} \Psi](t,u)
		+ \flzero^{1/2}[\mathscr{Z}^{24} \Psi](t,u)
		& \leq C \mathring{\upepsilon} \ln^{\Cononestar}(\myexp + t)\upmu_{\star}^{-8.75}(t,u),
			\label{E:TOPORDERMULTESTIMATE} 
			\\
		\enone^{1/2}[\mathscr{Z}^{24} \Psi](t,u)
		+ \flone^{1/2}[\mathscr{Z}^{24} \Psi](t,u)
		+ \Morint^{1/2}[\mathscr{Z}^{24} \Psi](t,u)
		& \leq C \mathring{\upepsilon} \ln^{\Cononestar + 2}(\myexp + t) \upmu_{\star}^{-8.75}(t,u).
			&& 
			\label{E:TOPORDERMORESTIMATE}
	\end{align}
	\end{subequations}
	In the above estimates, $\mathscr{Z}^{\leq k}$
	denotes an arbitrary differential operator of order $\leq k$
	corresponding to repeated differentiation with respect
	to the commutation vectorfields in $\mathscr{Z}$
	(see \eqref{eq:commuting-vfs}).
\end{proposition}

\begin{remark}[\textbf{The $\upmu_{\star}^{-1}$ hierarchy}]
An important feature of Prop.~\ref{P:APRIORIENERGYESTIMATES} 
to notice is that the top-order quantities are allowed to blow up
like $\upmu_{\star}^{-8.75}$
as $\upmu_{\star}$ tends to $0.$ 
The power $-8.75$ is a consequence of some delicate structural
features of the equations. We explain this below
(see in particular Remark \ref{R:STRUCTURAL}).
Another important feature is 
that as we descend below the top-order, we see improvements in 
the $\upmu_{\star}^{-1}$ blow-up rate until we reach a level in which
the quantities no longer blow-up. The non-degenerate 
estimates can be used to show
that the lower-order derivatives of $\Psi$ extend as
continuous functions, relative to the geometric
coordinates $(t,u,\vartheta),$ to the constant-time hypersurface of
first shock formation.
The precise features of this hierarchy 
play a fundamental role in guiding the analysis.

\end{remark}

\subsection{Details on the behavior of $\upmu$}
\label{SSS:MUSHARPER}
In order to explain how to derive the energy estimate hierarchy
of Prop.~\ref{P:APRIORIENERGYESTIMATES},
we first need to provide some sharp information on the behavior of $\upmu.$
In the next three lemmas, we state the most relevant properties of $\upmu$
and sketch some of their proofs. See \cite[Chapter 12]{jS2014} for more details. 
We emphasize once more that one needs very detailed 
control on the blow-up behavior of $\upmu^{-1}$ to close the energy estimates and that this 
is a major difference from the spherically symmetric case.

The first lemma provides the main Gronwall estimate that leads to
the degeneracy of the top-order energy estimates
\eqref{E:TOPORDERMULTESTIMATE}-\eqref{E:TOPORDERMORESTIMATE}. The reader should think that (\ref{E:INTROKEYINTEGRATINGFACTORGRONWALLREADYESTIMATE}) is the type of inequality appearing when trying to close the energy estimates at the top order.
The lemma is a drastically simplified version of 
\cite[Lemma 19.2.3]{jS2014}.

\begin{lemma}[\textbf{A Gronwall estimate used at top order}] \label{L:INTROKEYINTEGRATINGFACTORGRONWALLESTIMATE} 
	Let $\Contwo > 0$ be a constant.
	There exist a small
	constant $0 < \upsigma \ll 1$ 
	and large constants 
	$C > 0$ and
	$\Conone > 0$
	such that 
	for $u \in [0,U_0],$
	solutions $f(t)$ to the inequality
	\begin{align} \label{E:INTROKEYINTEGRATINGFACTORGRONWALLREADYESTIMATE}
	f(t)
	& \leq 
		C \mathring{\upepsilon}
		+ 
		\Contwo 
		\int_{t'=0}^t
			\left(
				\sup_{\Sigma_{t'}^u}
				\left|\frac{\Lunit \upmu}{\upmu}\right|
			\right)
			f(t')
		\, dt'
\end{align}
verify the estimate
\begin{align} \label{E:INTROKEYINTEGRATINGFACTORGRONWALLESTIMATE}
	f(t)
	& \leq 
		C \mathring{\upepsilon}
		\ln^{\Conone}(\myexp + t)
		\upmu_{\star}^{-(\Contwo + \sigma)}(t,u).
\end{align}
\end{lemma}

The second lemma is used to show that the below-top-order
energy estimates are less degenerate than the top-order ones.
The main idea is that we can gain powers of 
$\upmu_{\star}$ by integrating in time.

\begin{lemma}\cite[\textbf{Proposition 12.3.1; Gaining powers of $\upmu_{\star}$ by time integration}]{jS2014} \label{L:GAININGMU}
Let $\Contwo > 1$ be a constant. Then for $u \in [0,U_0],$ we have
\begin{align} \label{E:GAININGPOWERS}
	\int_{t'=0}^t 
		\frac{1}{(1 + t')^{3/2}} \upmu_{\star}^{-\Contwo}(t',u) 
	\, dt'
	& \leq C \upmu_{\star}^{1-\Contwo}(t,u).
\end{align}

Furthermore,
\begin{align} \label{E:LASTGAIN}
	\int_{t'=0}^t 
		\frac{1}{(1 + t')^{3/2}} \upmu_{\star}^{-3/4}(t',u) 
	\, dt'
	& \leq C.
\end{align}
\end{lemma}

The third lemma plays a supporting role in establishing the previous
two lemmas. In addition, the estimate \eqref{E:LUPMUNEGATIVEQUANTIFIED}
is the ingredient used to show that the Morawetz spacetime
integral is coercive in the regions where $\upmu$ is small;
see Lemma \ref{L:QUANTIFIEDMORAWETZCOERCIVENESS}.

\begin{lemma}\cite[\textbf{Sections 12.1 and 12.2; Some key properties of $\upmu.$}]{jS2014}
\label{L:KEYMUPROPERTIES}
Consider a fixed point $(t,u,\vartheta)$ and let $\updelta_{t,u,\vartheta} := \rgeo\left(t,u\right) \Lunit \upmu(t,u,\vartheta).$
Then for $0 \leq s \leq t,$ we have\footnote{The notation $A \sim B$ indicates, in an imprecise fashion, that 
$A$ is well-approximated by $B.$}
\begin{align} \label{E:LUNITUPMUAPPROXIMATEDALONGFIXEDINTEGRALCURVE}
	\Lunit \upmu(s,u,\vartheta)
	& \sim \frac{1}{\rgeo(s,u)} \updelta_{t,u,\vartheta},
		\\
	\upmu(s,u,\vartheta)
	& \sim 1 + \updelta_{t,u,\vartheta} \ln \left(\frac{\rgeo(s,u)}{\rgeo(0,u)} \right).
		\label{E:UPMUAPPROXIMATEDALONGFIXEDINTEGRALCURVE}
\end{align}

Let $[\Lunit \upmu]_{-} = |\Lunit \upmu|$ when $\Lunit \upmu < 0$
and $[\Lunit \upmu]_{-} = 0$ otherwise.
Then at any point $(t,u,\vartheta)$
with $\upmu(t,u,\vartheta) < 1/4,$ we have
\begin{align} \label{E:LUPMUNEGATIVEQUANTIFIED}
[\Lunit \upmu]_{-}(t,u,\vartheta) 
&\geq 
	c
	\frac
	{1+t}{\ln(e+t)}.
\end{align}

\end{lemma}

\begin{proof}[Discussion of the proof of Lemma \ref{L:KEYMUPROPERTIES}]
Thanks to the decay estimates of the Heuristic Principle (see Subsubsect.~\ref{SSS:HPMOREPRECISE}), 
Lemma \ref{L:KEYMUPROPERTIES} can be proved
by using essentially the same arguments that we used above in spherical symmetry;
see Prop.~\ref{P:ge} and its proof. The additional terms
present away from spherical symmetry involve $\mathcal{C}_u-$tangential
derivatives of $\Psi,$ and hence they decay very rapidly and make only
a negligible contribution to the inequalities.
\end{proof}

\begin{proof}[Discussion of the proof of Lemma \ref{L:INTROKEYINTEGRATINGFACTORGRONWALLESTIMATE}]
By the standard Gronwall inequality, we deduce
\begin{align}
\label{eq:INTERMEDIATE}
f(t)
& \leq 
		C \mathring{\upepsilon} 
		\exp{
			\left( B\int_{s=0}^t
			\sup_{\Sigma_s^u}
			\left|
				\frac{\Lunit \upmu}{\upmu}
			\right| \, \right)  
		}
	ds.
\end{align}
We now need to pass from \eqref{eq:INTERMEDIATE} to \eqref{E:INTROKEYINTEGRATINGFACTORGRONWALLESTIMATE}.
The detailed proof is somewhat difficult because of the presence of the $\sup$ on the right.  

To reveal the main ideas behind the proof, we first use 
\eqref{E:LUNITUPMUAPPROXIMATEDALONGFIXEDINTEGRALCURVE}
and \eqref{E:UPMUAPPROXIMATEDALONGFIXEDINTEGRALCURVE} to deduce
that for $0 \leq s \leq t,$ we have
\begin{align} \label{E:LUNITUPMUOVERMUAPPROXIMATEDALONGFIXEDINTEGRALCURVE}
	\frac{\Lunit \upmu}{\upmu}(s,u,\vartheta)
	& \sim \frac{\updelta_{t,u,\vartheta}}
		{\rgeo(s,u) \left\lbrace 1 + \updelta_{t,u,\vartheta} \ln \left(\frac{\rgeo(s,u)}{\rgeo(0,u)} \right)\right\rbrace}.
\end{align}
The important point in \eqref{E:LUNITUPMUOVERMUAPPROXIMATEDALONGFIXEDINTEGRALCURVE} is that
\emph{the same constant $\updelta_{t,u,\vartheta}$} appears in the numerator and denominator.
For the sake of illustration, let us   simplify  the analysis  by assuming that
$\updelta_{t,u',\vartheta} \leq 0$ for $u' \in [0,u], \vartheta \in \mathbb{S}^2$.\footnote{This is indeed what holds for \emph{some} $(u',\vartheta)$ close to the formation of the shock. The fact that in reality it does not hold for all 
$(u',\vartheta),$ even close to the time of shock formation, leads to additional technical complications which we suppress here.} Using the fact that for a fixed $a > 0,$ the function
$f(x) = \frac{x}{1 + a x}$ is increasing on the domain $x \in (-a^{-1},0],$ we deduce (recall $\updelta_{t,u',\vartheta}$ is non-positive)
\begin{align} \label{E:LMUOVERMUSIMPLIFIEDINEQUALITY}
\frac{\Lunit \upmu}{\upmu}(s,u,\vartheta)
& \geq
		\frac{ \min_{u' \in [0,u], \vartheta \in \mathbb{S}^2} \updelta_{t,u',\vartheta}}
		{\rgeo(s,u) \left\lbrace 1 +\min_{u' \in [0,u], \vartheta \in \mathbb{S}^2}  \updelta_{t,u',\vartheta} \ln \left(\frac{\rgeo(s,u)}{\rgeo(0,u)} \right)\right\rbrace} + \err.
\end{align}
We finally set 
\[
\updelta_t := 
\left|
	\min_{u' \in [0,u], \vartheta \in \mathbb{S}^2} \updelta_{t,u',\vartheta} 
\right| \, 
\]
and conclude
\begin{align} \label{E:LUPMUOVERUPMUCARICATURESUPESTIMATE}
	\sup_{\Sigma_s^u}
		\left|
			\frac{\Lunit \upmu}{\upmu}
		\right|
		\le
		\frac{\updelta_t}
		{\rgeo(s,u) \left\lbrace 1 - \updelta_t \ln \left(\frac{\rgeo(s,u)}{\rgeo(0,u)} \right)\right\rbrace}
		+ \err.
\end{align}
Note also that,  
in view of Definition \ref{D:UPMUSTAR} and 
\eqref{E:UPMUAPPROXIMATEDALONGFIXEDINTEGRALCURVE}, 
we deduce that for $0 \leq s \leq t,$ 
we have
\begin{align} \label{E:UPMUSTARCARICATUREESTIMATE}
	\upmu_{\star}(s,u) \sim 1 - \updelta_t \ln \left(\frac{\rgeo(s,u)}{\rgeo(0,u)} \right).
\end{align}
We now integrate \eqref{E:LUPMUOVERUPMUCARICATURESUPESTIMATE} $ds$
from $s = 0$ to $t$
and use \eqref{E:UPMUSTARCARICATUREESTIMATE}
to deduce that
\begin{align} \label{E:HEURISTICESTIMATEFORTHEKEYINTEGRATINGNFACTOR}
	\int_{s=0}^t
		\sup_{\Sigma_s^u}
		\left|
			\frac{\Lunit \upmu}{\upmu}
		\right|
	\, ds
	& \sim
		\int_{s=0}^t
			\frac{\updelta_t}
			{\rgeo(s,u) \left\lbrace 1 - \updelta_t \ln \left(\frac{\rgeo(s,u)}{\rgeo(0,u)} \right)\right\rbrace}
		\, ds
			\\
	& = \ln \left| 1 - \updelta_t \ln \left(\frac{\rgeo(t,u)}{\rgeo(0,u)} \right) \right|
		\sim \ln \left | \upmu_{\star}^{-1}(t,u) \right|
		\notag
\end{align}
(recall that $\upmu_{\star}(t,u) \leq 1$ by definition).
The desired estimate\footnote{Up to the correction factors $\upsigma$ and $\ln^{\Conone}(\myexp + t).$}
\eqref{E:INTROKEYINTEGRATINGFACTORGRONWALLESTIMATE}
now easily follows from \eqref{eq:INTERMEDIATE}. 

\end{proof}

\begin{proof}[Discussion of the proof of Lemma \ref{L:GAININGMU}]
The main idea is that the integrals in \eqref{E:GAININGPOWERS} and 
\eqref{E:LASTGAIN} are easy to estimate 
once we have obtained sharp information about
the behavior of $\upmu_{\star}.$ For example,
if we assume for simplicity 
that $\updelta_{t,u',\vartheta} \leq 0$ for $u' \in [0,u], \vartheta \in \mathbb{S}^2,$
then the estimate \eqref{E:UPMUSTARCARICATUREESTIMATE} holds.
We can then estimate the integrals by 
using \eqref{E:UPMUSTARCARICATUREESTIMATE},
splitting them into a small-time portion
and a large-time portion, and optimizing the splitting time.

\end{proof}

\subsection{Details on the top-order energy estimates}
\label{SS:DETAILSONTOPORDERESTIMATES}
We now explain some of the main ideas behind the proof of 
Prop.~\ref{P:APRIORIENERGYESTIMATES}. 
Throughout Subsect.~\ref{SS:DETAILSONTOPORDERESTIMATES}, 
$\mathring{\upepsilon}$ denotes the small size of the data and
$\varepsilon$ denotes the small amplitude size corresponding to the Heuristic Principle estimates of Subsubsect.~\ref{SSS:HPMOREPRECISE}, which we use as 
bootstrap assumptions.
For convenience, we focus on only the difficult top-order energy estimate
\eqref{E:TOPORDERMULTESTIMATE}.
To illustrate the main ideas,  
we might as well commute the equation with a single rotational vectorfield 
$\Rot,$ pretend that we are at the highest level of derivatives, 
and show how to avoid the derivative loss.
We remark that we must also avoid, using similar arguments, 
the derivative loss when
we commute with $\Rad$ and $\rgeo \Lunit,$ though the difficulties are somewhat
less severe in the case of $\rgeo \Lunit.$
To proceed, we consider the wave equation verified by $\Rot \Psi.$
That is, we commute the equation 
$\upmu \square_{g(\Psi)} \Psi = 0$ with $\Rot \in \lbrace \Rot_{(1)}, \Rot_{(2)}, \Rot_{(3)} \rbrace$ 
to deduce the equation
\begin{align} \label{E:ROTCOMMUTEDWAVE}
	\upmu \square_{g(\Psi)} \Rot \Psi 
	& = (\Rad \Psi) \Rot \mytr \upchi 
		- (\Rot \upmu) \angLap \Psi
		+ \cdots.
\end{align}

\begin{remark}
On the right hand side of \eqref{E:ROTCOMMUTEDWAVE} $\cdots$ denotes a long list of additional
error terms that turn out to be much easier to control than 
the first one that is explicitly listed;
see \cite[Proposition 6.2.2, Lemma 8.1.2, Proposition 8.2.1]{jS2014}
for more details.
A rigorous derivation of \eqref{E:ROTCOMMUTEDWAVE} would involve lengthy
computations; in an effort to avoid distracting the reader, we will
simply take \eqref{E:ROTCOMMUTEDWAVE} for granted.
We furthermore remark that related but distinct difficulties arise
when we commute with $\Rad$ or $\rgeo \Lunit,$ but for simplicity, we discuss
only the case of $\Rot.$
\end{remark}

\begin{remark}
Note that in \eqref{E:ROTCOMMUTEDWAVE}, 
we are working with the $\upmu-$weighted wave operator
$\upmu \square_{g(\Psi)}.$ It turns out that $\upmu \square_{g(\Psi)}$ 
has better commutation properties 
with the vectorfields in $\mathscr{Z}$ (see \eqref{eq:commuting-vfs})
than the unweighted operator $\square_{g(\Psi)}.$ 
The important property of
$\upmu \square_{g(\Psi)}$
is that we do not introduce any 
factors of $\upmu^{-1}$ when we repeatedly 
commute it with vectorfields in $\mathscr{Z};$
generally, we would not be able to control such factors.
The moral reason behind the good properties of $\upmu \square_{g(\Psi)}$
can be discerned from
the decomposition \eqref{E:WAVEOPERATORDECOMPOSED}.
To see this, we recall that, 
relative to the geometric coordinates, 
we have
$\Lunit = \frac{\partial}{\partial t}$
and 
$\Rad = \frac{\partial}{\partial u} + \mbox{angular error term}.$
We can therefore rewrite \eqref{E:WAVEOPERATORDECOMPOSED} as
\begin{align} \label{E:MORALWAVEOPERATORDECOMPOSED}
	\upmu \square_{g(\Psi)} \Psi
	& = 
		- \frac{\partial}{\partial t}
			\left\lbrace
				\upmu \frac{\partial}{\partial t} \Psi + 2 \frac{\partial}{\partial u} \Psi
			\right\rbrace
		+ \upmu \angLap \Psi
		+ \err.
\end{align}
The right-hand side of \eqref{E:MORALWAVEOPERATORDECOMPOSED}
now suggests that,
for example, the differential operator $\Rad \in \mathscr{Z}$ can be commuted
through the equation without introducing any dangerous factors of 
$\upmu^{-1}.$ 

\end{remark}

\begin{remark}[\textbf{On the importance of terms that are not present}]
	\label{R:ERRORTERMSNOTPRESENT}
	One crucial property of the commutation vectorfield set $\mathscr{Z}$
	is that after commuting the through the operator $\upmu \square_{g(\Psi)}$ 
	one time, we never produce terms
	of the form 
	$\angD \Rad \upmu$ or $\Rad \Rad \upmu.$ 
	This is important because we have no means to bound the top-order derivatives of these terms.
	In contrast, as we will see,
	there is a procedure based on modified quantities and elliptic estimates that 
	allows us to bound the top-order derivatives of the term $\Rot \mytr \upchi$
	on the right-hand side of \eqref{E:ROTCOMMUTEDWAVE}
	(see Subsubsect.~\ref{SSS:AVOIDINGTOPORDERDERIVATIVELOSS}).
	This discrepancy occurs even though
	 $\angD \Rad \upmu,$
	 $\Rot \mytr \upchi,$
	 and $\Rad \Rad \upmu$
	 are all third-order derivatives of the eikonal function $u.$
\end{remark}

Our goal is to show how to estimate solutions to 
\eqref{E:ROTCOMMUTEDWAVE} without losing derivatives.
In particular, we sketch a proof of how to derive a ``top-order'' 
estimate for the rotation commutation vectorfields $\Rot$ 
(described at the beginning of Sect.~\ref{S:GENERALIZEDENERGY})
of the form
\[
\enzero^{1/2}[\Rot \Psi](t,u)
+ \flzero^{1/2}[\Rot \Psi](t,u)
\leq
C \mathring{\upepsilon} \ln^{\Conone}(\myexp + t) \upmu_{\star}^{-\Contwo}(t,u),
\]
in the spirit of \eqref{E:TOPORDERMULTESTIMATE},
and we highlight the role      
played by Lemma \ref{L:INTROKEYINTEGRATINGFACTORGRONWALLESTIMATE}.
To begin, we use 
\eqref{E:ROTCOMMUTEDWAVE},
\eqref{E:MTUDIVERGENCETHM},
\eqref{E:INTROE0DEF}, and \eqref{E:INTROF0DEF}
to deduce that\footnote{The remaining error integrals $\cdots$ 
on the right-hand side of \eqref{E:TOPORDERENERGYIDCARICATURE}
are easier to estimate than the explicitly indicated ones, so we ignore them here.}
\begin{align} \label{E:TOPORDERENERGYIDCARICATURE}
	\enzero[\Rot \Psi](t,u)
	+ \flzero[\Rot \Psi](t,u)
	& \leq C \mathring{\upepsilon}
		- \int_{\MM_{t,u}}
				(2 \Rad \Psi)
				(\Rot \mytr \upchi) 
				\Rad \Rot \Psi
			\, d \vol
		+ \int_{\MM_{t,u}}
				(\Rot \upmu) 
				(\Lunit \Rot \Psi)
				\angLap \Psi
			\, d \vol
		+ \cdots.
\end{align}
To deduce \eqref{E:TOPORDERENERGYIDCARICATURE}, 
we have used the divergence  identity 
\eqref{E:MTUDIVERGENCETHM} with $X = \Mult$
(see \eqref{E:INTRODEFINITIONMULT}),
$\Rot \Psi$ in the role of $\Psi,$
and $\waveinhom= (\Rad \Psi) \Rot \mytr \upchi 
		- (\Rot \upmu) \angLap \Psi
		+ \cdots $ from  the right-hand side of \eqref{E:ROTCOMMUTEDWAVE}.
Furthermore, we have replaced the integrand 
$(\Mult \Rot \Psi) \waveinhom$ 
from \eqref{E:MTUDIVERGENCETHM} 
with the expression 
\[
	2 (\Rad \Psi) (\Rad \Rot \Psi) \Rot \mytr \upchi
		-  (\Rot \upmu)
			(\Lunit \Rot \Psi)
			\angLap \Psi+\ldots
\]
 with $\dots $ denoting terms that are easier to treat.\footnote{The remaining terms in $\Mult$ involve
$\mathcal{C}_u-$tangential derivatives of $\Psi.$}

\begin{remark}
	We have suppressed the 
	error-term  
	$
	\int_{\MM_{t,u}}
				(\Rot \upmu) 
				(\Rad \Rot \Psi)
				\angLap \Psi
	\, d \vol
	$
	by relegating it to the $\ldots$ term on the 
	right-hand side of \eqref{E:TOPORDERENERGYIDCARICATURE}.
	One might expect that this integral is
	more difficult to estimate than the second one
	written on the right-hand side of 
	\eqref{E:TOPORDERENERGYIDCARICATURE}
	because it involves the transversal derivative
	factor $\Rad \Rot \Psi$ in place of
	$\Lunit \Rot \Psi.$
	However, the
	$\Lunit \Rot \Psi-$involving error integral is actually slightly more difficult to estimate
	because we have to use the cone fluxes and the Morawetz spacetime integral to bound it; see inequality
	\eqref{E:BOUNDSINSERTEDTOPORDERENERGYIDCARICATURE}.
	In contrast, the arguments given in Subsubsect.~\ref{SSS:IGNOREDERIVATIVELOSS} 
	can easily be modified to show that 
	the $\Rad \Rot \Psi-$involving error integral
	can be bounded in magnitude by
	\begin{align} \label{E:RADROTPSIERRORINTEGRAL}
	 & 
	 	\lesssim
		\mathring{\upepsilon}
		\int_{t'=0}^t
			\int_{\Sigma_{t'}^u}
				\frac{\ln(\myexp + t')}{1 + t'}
				|\Rad \Rot \Psi|
				|\Rot \Rot \Psi|
			\, d \tvol
		\, dt'
			\\
	& \lesssim
		\mathring{\upepsilon}
		\int_{t'=0}^t
			\frac{1}{\upmu_{\star}^{1/2}(t',u)}
			\frac{\ln(\myexp + t')}{(1 + t')^2}
			\enzero[\Rot \Psi]
			\enone[\Rot \Psi]
		\, dt',
		\notag
	\end{align}
	where we have used Cauchy-Schwarz on $\Sigma_{t'}^u$ and Prop.~\ref{P:COERCIVEENERGIESANDFLUXES}
	to pass to the final inequality.
	Thanks to favorable powers of $t'$ present in the integrand
	on the right-hand side of \eqref{E:RADROTPSIERRORINTEGRAL},
	we can handle the singular factor $\upmu_{\star}^{-1/2}$
	with inequality \eqref{E:LASTGAIN},
	and it is therefore easy to show
	that the right-hand side of \eqref{E:RADROTPSIERRORINTEGRAL}
	is a cubic error term. In contrast, if we tried 
	to handle
	the $\Lunit \Rot \Psi-$involving error integral
	in the same way, we would find the worse 
	factor $\upmu_{\star}^{-1}(t',u)$ in the integrand.
	This strategy will fail because $\upmu_{\star}^{-1}(t',u)$ 
	is too singular to be handled by
	inequality \eqref{E:LASTGAIN}.
\end{remark}

\subsubsection{Raychaudhuri-type identity} 
\label{SSS:RAYCHAUDURIID}
We now highlight the main technical hurdle 
in proving Prop.~\ref{P:APRIORIENERGYESTIMATES}, 
which we already mentioned in Subsubsect.~\ref{susub:prev-eikonal}:
the only way by which we can estimate 
the factor $\Rot \mytr \upchi$ on the right-hand side of \eqref{E:TOPORDERENERGYIDCARICATURE}
is by exploiting an important transport equation which is the exact analog of the
well-known \emph{Raychaudhuri equation} \cite{aR1955} in General Relativity.
The Raychaudhuri-type equation satisfied by 
$\mytr \upchi$ is (see, for example, the proof of \cite[Corollary 10.2.1]{jS2014})
	\begin{align}
       \label{E:RAYCHAUDURI}
        \Lunit \mytr \upchi 
                + \frac{1}{2} (\mytr \upchi)^2 
                + |\hat{\upchi}|^2 
        & = - \mbox{Ric}_{\Lunit \Lunit}
                + \frac{\Lunit \upmu}{\upmu} \mytr \upchi,
       \end{align}
 where $\mbox{Ric}$ is the Ricci curvature of $g$ and $\hat{\upchi}$ is the trace-free part of $\upchi.$  
The Ricci tensor (see \cite[Corollary 10.1.3]{jS2014}) can be decomposed 
through a tedious but straightforward calculation, which yields
    $\mbox{Ric}_{\Lunit \Lunit} := \mbox{Ric}_{\a\b} \Lunit^\a\Lunit^\b = - \frac{1}{2} G_{\Lunit \Lunit} \angLap \Psi
		   + \cdots.$ 
  Since the term $\frac{1}{2} (\mytr \upchi)^2 + |\hat{\upchi}|^2 $ is also lower-order,\footnote{In reality, the analysis is somewhat more complicated.
		   Specifically, we need to use elliptic estimates to bound the top-order derivatives of $|\hat{\upchi}|^2;$ 
		   see Remark \ref{R:NEEDFORELLIPTIC}.} 
		   we arrive at the transport equation
\begin{align} \label{E:TRCHIDERIVATIVELOSS}
	\Lunit \mytr \upchi
	&= \frac{1}{2} G_{\Lunit \Lunit} \angLap \Psi
		+ \cdots,
\end{align}
where $\cdots$ denotes easier terms which can be ignored.

 The main difficulty is that after we commute
\eqref{E:TRCHIDERIVATIVELOSS} with $\Rot,$ we obtain the equation
\[
\Lunit \Rot \mytr \upchi = \frac{1}{2} G_{\Lunit \Lunit} \angLap \Rot \Psi + \cdots, 
\]
which depends on \emph{three derivatives} of $\Psi,$ 
whereas the left-hand side of
\eqref{E:TOPORDERENERGYIDCARICATURE} only yields control over \emph{two derivatives} of $\Psi$
(see Prop.~\ref{P:COERCIVEENERGIESANDFLUXES}).
Hence, it seems that we are \emph{losing derivatives} in our estimates for $\Rot \mytr \upchi.$ 
In Subsubsect.~\ref{SSS:AVOIDINGTOPORDERDERIVATIVELOSS}, we explain how to overcome this difficulty.

\subsubsection{The energy estimates ignoring derivative loss}
\label{SSS:IGNOREDERIVATIVELOSS}
Before we address how to circumvent the loss in derivatives mentioned above, 
we first address how the proof 
of Prop.~\ref{P:APRIORIENERGYESTIMATES}
would work if we did not have to worry about it. 
Our discussion will highlight the role of the Morawetz integral \eqref{E:MORINTEGRALCOERCIVE}
in the proof.
To begin, we imagine that \eqref{E:ROTCOMMUTEDWAVE},
that is, the equation
\begin{align} \label{E:AGAINROTCOMMUTEDWAVE}
	\upmu \square_{g(\Psi)} \Rot \Psi 
	& = (\Rad \Psi) \Rot \mytr \upchi 
		- (\Rot \upmu) \angLap \Psi
		+ \cdots,
\end{align} 
is the top-order equation
and that we are trying to bound the right-hand side of \eqref{E:TOPORDERENERGYIDCARICATURE}
back in terms of the left so that we can apply Gronwall's inequality.
We will use Prop.~\ref{P:COERCIVEENERGIESANDFLUXES} to connect various $L^2$
norms back to $\enzero,$ $\flzero,$ etc.
For the time being,
we ignore the difficult error integral on the right-hand side of \eqref{E:TOPORDERENERGYIDCARICATURE}
and instead focus on the second one 
\begin{align} \label{E:SECONDERRORINTEGRAL}
\int_{\MM_{t,u}}
				(\Rot \upmu) 
				(\Lunit \Rot \Psi)
				\angLap \Psi
			\, d \vol,
\end{align}
in which we do not have to worry about derivative loss.
To bound this integral, we use the following pointwise estimate: 
\begin{align} \label{E:ROTUPMUC0}
	\left|
		\Rot \upmu
	\right|
	& \lesssim
		 \varepsilon 
		 \ln(\myexp + t).
\end{align}
The estimate \eqref{E:ROTUPMUC0} is easy to
derive by commuting the evolution equation \eqref{E:UPMUSCHEMATICTRANSPORT} for $\upmu$ with $\Rot,$ using
the Heuristic Principle estimates 
(see Subsubsect.~\ref{SSS:HPMOREPRECISE})
to bound the right-hand side, and then integrating
the resulting inequality along the integral curves of $\Lunit = \frac{\partial}{\partial t};$
see \cite[Proposition 11.27.1]{jS2014} for the details.
We also use the following property of our rotation vectorfields, which is familiar
from the case of Minkowski spacetime
(see \cite[Lemma 11.12.1]{jS2014} for a proof):
\begin{align} \label{E:ANGLAPINTERMSOFANGDIFFROTANDGOODTFACTOR}
	\left|
		\angLap \Psi
	\right|
	& \lesssim \frac{1}{1 + t} \sum_{l=1}^3 |\angD \Rot_{(l)} \Psi|.
\end{align}
In view of \eqref{E:ROTUPMUC0} and \eqref{E:ANGLAPINTERMSOFANGDIFFROTANDGOODTFACTOR}, 
we see that the error integral \eqref{E:SECONDERRORINTEGRAL}
can be bounded as follows,
where we split it into the region where $\upmu \leq 1/4$
and the region where $\upmu > 1/4:$
\begin{align} \label{E:REPRESENTATIVEMORAWETZNEEDEDERRORINTEGRAL}
	\lesssim
	\varepsilon
	\int_{\MM_{t,u}}
			\mathbf{1}_{\lbrace \upmu \leq 1/4 \rbrace}
			\frac{\ln(\myexp + t')}{1 + t'}
			|\Lunit \Rot \Psi|
			|\angD \Rot \Psi|
		\, d \vol
		+ \varepsilon
		\int_{\MM_{t,u}}
			\mathbf{1}_{\lbrace \upmu > 1/4 \rbrace}
			\frac{\ln(\myexp + t')}{1 + t'}
			|\Lunit \Rot \Psi|
			|\angD \Rot \Psi|
		\, d \vol.
\end{align}
The main difficulty is present in the first integral in \eqref{E:REPRESENTATIVEMORAWETZNEEDEDERRORINTEGRAL}.
Indeed, the first integral lacks a $\upmu$ weight and involves the angular derivative term
$|\angD \Rot \Psi|.$
Hence, the angular derivative coerciveness of the energy-flux quantities,
which is provided by Prop.~\ref{P:COERCIVEENERGIESANDFLUXES},
is not sufficient to control it.
As in \cite{dC2007}, to overcome the difficulty, 
we use the strength of the Morawetz integral; see inequality \eqref{E:MORINTEGRALCOERCIVE}.
More precisely, by dividing the time interval $[0,t]$ into suitable subintervals
and using Cauchy-Schwarz, 
it is not difficult 
to show (see the proof of \cite[Lemma 19.3.3]{jS2014})
that 
the first integral on the
right-hand side of \eqref{E:REPRESENTATIVEMORAWETZNEEDEDERRORINTEGRAL} is
\begin{align}  \label{E:FIRSTBOUNDREPRESENTATIVEMORAWETZNEEDEDERRORINTEGRAL}
	 & \lesssim
	 \varepsilon
		\int_{\MM_{t,u}}
			\mathbf{1}_{\lbrace \upmu \leq 1/4 \rbrace}
			|\Lunit \Rot \Psi|^2
		\, d \vol
		+ 
		\varepsilon
		\sup_{\tau \in [0,t)}
		\frac{1}{(1 + \tau)^{1/2}}
	 	\int_{\MM_{\tau,u}}
	 		\mathbf{1}_{\lbrace \upmu \leq 1/4 \rbrace}
			\frac{1 + t'}{\ln(\myexp + t')}
			|\angD \Rot \Psi|^2
		\, d \vol.
\end{align}
Using \eqref{E:INTROMULTCONEFLUXCOERCIVITY}, 
we deduce that the first term on the right-hand side of 
\eqref{E:FIRSTBOUNDREPRESENTATIVEMORAWETZNEEDEDERRORINTEGRAL}
is bounded by
\begin{align} \label{E:SECONDBOUNDREPRESENTATIVEMORAWETZNEEDEDERRORINTEGRAL}
		\varepsilon
		\int_{\MM_{t,u}}
			\mathbf{1}_{\lbrace \upmu \leq 1/4 \rbrace}
			|\Lunit \Rot \Psi|^2
		\, d \vol
		& \lesssim
			\varepsilon
	 		\int_{u'=0}^u
	 			\flzero[\Rot \Psi](t,u')
	 		\, du'.
\end{align}
In addition, the second term on the right-hand side of 
\eqref{E:FIRSTBOUNDREPRESENTATIVEMORAWETZNEEDEDERRORINTEGRAL}
is bounded by
\begin{align} \label{E:THIRDBOUNDREPRESENTATIVEMORAWETZNEEDEDERRORINTEGRAL}
		\varepsilon
		\sup_{\tau \in [0,t)}
		\frac{1}{(1 + \tau)^{1/2}}
	 	\int_{\MM_{\tau,u}}
	 		\mathbf{1}_{\lbrace \upmu \leq 1/4 \rbrace}
			\frac{1 + t'}{\ln(\myexp + t')}
			|\angD \Rot \Psi|^2
		\, d \vol
		& 
		\lesssim
		\varepsilon
		\sup_{\tau \in [0,t)}
		\frac{1}{(1 + \tau)^{1/2}}
	 	\Morint[\Rot \Psi](\tau,u),
\end{align}
where we have used the key Morawetz estimate \eqref{E:MORINTEGRALCOERCIVE}.

We then insert these estimates into the right-hand side of \eqref{E:TOPORDERENERGYIDCARICATURE},
ignore the (difficult) first error integral,
and find that
\begin{align} \label{E:BOUNDSINSERTEDTOPORDERENERGYIDCARICATURE}
	\enzero[\Rot \Psi](t,u)
	+ \flzero[\Rot \Psi](t,u)
	& \leq C \mathring{\upepsilon}
		+
		\varepsilon
	 		\int_{u'=0}^u
	 			\flzero[\Rot \Psi](t,u')
	 		\, du'
		+ \varepsilon
		\sup_{\tau \in [0,t)}
		\frac{1}{(1 + \tau)^{1/2}}
	 	\Morint[\Rot \Psi](\tau,u)
		+ \cdots.
\end{align}
Clearly, the first integral on the right-hand side of \eqref{E:BOUNDSINSERTEDTOPORDERENERGYIDCARICATURE}
is treatable with Gronwall's inequality (recall that $0 < u < 1$). Furthermore, the 
second integral $\varepsilon
		\sup_{\tau \in [0,t)}
		\frac{1}{(1 + \tau)^{1/2}}
	 	\Morint[\Rot \Psi](\tau,u)$
can be treated as a harmless cubic term,
even if the Morawetz integral $\Morint[\Rot \Psi](t,u)$ grows logarithmically in time,
consistent with \eqref{E:LOWESTLEVELMORNONDEGENERATE}.
Hence, assuming data of small size $\mathring{\upepsilon},$
we have provided some indication of 
how to derive an a priori estimate of the form
$\enzero^{1/2}[\Rot \Psi](t,u) + \flzero^{1/2}[\Rot \Psi](t,u) \lesssim \mathring{\upepsilon}$
if we did not have to worry about the 
dangerous error integral
$
	- \int_{\MM_{t,u}}
				(2 \Rad \Psi)
				(\Rot \mytr \upchi) 
				\Rad \Rot \Psi
			\, d \vol.
$
As we now discuss, this dangerous integral leads to 
a much worse a priori estimate.

\subsubsection{Avoiding top-order derivative loss via a Raychaudhuri-type identity}
\label{SSS:AVOIDINGTOPORDERDERIVATIVELOSS}
We now confront the main difficulty in deriving the top-order energy estimate
\eqref{E:TOPORDERMULTESTIMATE}: the potential derivative loss in the $\Rot \mytr \upchi$ term
in the error integral 
\[
		- \int_{\MM_{t,u}}
				(2 \Rad \Psi)
				(\Rot \mytr \upchi) 
				\Rad \Rot \Psi
			\, d \vol
\]
on the right-hand side of \eqref{E:TOPORDERENERGYIDCARICATURE}.
We are still imagining, for the sake of illustration, 
that the second-order derivatives of $\Psi$ are top-order.
The main point of the procedure outlined below          
is to replace 
this error integral with 
\begin{align} \label{E:REWRITINGOFKEYINTEGRAL}
	-4
	\int_{\MM_{t,u}}
			\frac{\Lunit \upmu}{\upmu}
			(\Rad \Rot \Psi)^2
	\, d \vol
	+ \cdots,
\end{align}
where the $\cdots$ integrals are similar in nature or easier.
We can then use the coerciveness property \eqref{E:INTROMULTENERGYCOERCIVITY}
and the co-area formula $\int_{\MM_{t,u}} \cdots \, d \vol 
= \int_{t' = 0}^t \int_{\Sigma_{t'}^u} \cdots \, d \tvol dt'$
to bound \eqref{E:REWRITINGOFKEYINTEGRAL} in magnitude by
\begin{align} \label{E:SHARPCONSTANTNEEDED}
	& \leq 4
		\int_{t'=0}^t
			\left(
				\sup_{\Sigma_{t'}^u}
				\left|\frac{\Lunit \upmu}{\upmu}\right|
			\right)
			\enzero[\Rot \Psi](t,u)
		\, dt'.
\end{align}
Thus, recalling \eqref{E:TOPORDERENERGYIDCARICATURE}, 
we find that
\begin{align} \label{E:TOPORDERENERGYCARICATUREGRONWALLREADY}
	\enzero[\Rot \Psi](t,u)
	+ \flzero[\Rot \Psi](t,u)
	& \leq C \mathring{\upepsilon}
		+ 
		4 
		\int_{t'=0}^t
			\left(
				\sup_{\Sigma_{t'}^u}
				\left|\frac{\Lunit \upmu}{\upmu}\right|
			\right)
			\enzero[\Rot \Psi](t',u)
		\, dt'
		+ \cdots,
\end{align}
where the constant ``$4$'' on the right-hand side of
\eqref{E:TOPORDERENERGYCARICATUREGRONWALLREADY} is a
``structural constant,'' the $\cdots$ terms
are similar and nature or easier,
and $\mathring{\upepsilon}$ is the size of the data.

We can now appeal to Lemma \ref{L:INTROKEYINTEGRATINGFACTORGRONWALLESTIMATE} to
derive an priori estimate for $\enzero[\Rot \Psi](t,u)
	+ \flzero[\Rot \Psi](t,u),$ that is, we have
\[
\enzero^{1/2}[\Rot \Psi](t,u)
+ \flzero^{1/2}[\Rot \Psi](t,u)
\leq
C \mathring{\upepsilon} \ln^{\Conone}(\myexp + t) \upmu_{\star}^{-\Contwo}(t,u).
\]

\begin{remark}[\textbf{The importance of the structural constants}] \label{R:STRUCTURAL}
	Note that the structural constant 
	``$4$'' that appears in \eqref{E:TOPORDERENERGYCARICATUREGRONWALLREADY}
	is \emph{independent of the number of times that we commute the wave
	equation with vectorfield operators.}
	This observation is important, for 
	the structural constant affects the power of $\upmu_{\star}^{-1}$ appearing in the top-order energy estimates
	and hence the number of derivatives we need to close the estimates.
\end{remark}

It remains for us to explain the procedure used above,
which allowed us to
replace the derivative-losing error integral 
\[
- \int_{\MM_{t,u}}
				(2 \Rad \Psi)
				(\Rot \mytr \upchi) 
				\Rad \Rot \Psi
			\, d \vol
\]
with \eqref{E:REWRITINGOFKEYINTEGRAL}.
The procedure is based on the following \emph{renormalized Raychaudhuri equation\footnote{The same idea was also used earlier, 
in a different context, in \cite{sKiR2003}.}}
which we explain below.
 
 \subsubsection{Renormalized Raychaudhuri equation} 
	\label{SSS:RENORMALIZEDRAYCHADHOURI}
We begin by recalling equation \eqref{E:TRCHIDERIVATIVELOSS}:
\begin{align} \label{E:AGAINTRCHIDERIVATIVELOSS}
	\Lunit \mytr \upchi
	&= \frac{1}{2} G_{\Lunit \Lunit} \angLap \Psi
		+ \cdots.
\end{align}
 To avoid the derivative loss, 
 we need to take advantage of the wave equation in the form
(see \eqref{E:ULRESCALED} and \eqref{E:WAVEOPERATORDECOMPOSED})
\[ 
	0 = \upmu \square_{g(\Psi)} \Psi = - \Lunit \uLgood \Psi + \upmu \angLap \Psi + l.o.t.
\]
Hence, using the wave equation, 
we can replace,
up to a crucially important factor of $\upmu^{-1}$ and $l.o.t.,$ 
the term $\frac{1}{2} G_{\Lunit \Lunit} \angLap \Psi$ in \eqref{E:AGAINTRCHIDERIVATIVELOSS}
with a \emph{perfect $\Lunit$ derivative of} $\frac{1}{2} G_{\Lunit \Lunit} \uLgood \Psi$ and then bring this perfect
$\Lunit$ derivative over to the \emph{left-hand side} of \eqref{E:AGAINTRCHIDERIVATIVELOSS}.
Furthermore, one can show that the remaining 
second derivatives of $\Psi$ in the $\cdots$ terms
on the right-hand side of \eqref{E:AGAINTRCHIDERIVATIVELOSS} are
also perfect $\Lunit$ derivatives, and thus we can bring those terms to the left as well.
In total, at the expense of a factor of $\upmu^{-1},$
we can renormalize away all of the terms in equation \eqref{E:AGAINTRCHIDERIVATIVELOSS}
that lose derivatives relative to $\Psi,$ thereby obtaining
an equation for a ``modified'' version of 
$\mytr \upchi$ of the form $\Lunit (Modified) = l.o.t.$

Hence, the important structure used by Christodoulou    
in \cite{dC2007} can be restated as follows:
for solutions to $\square_{g(\Psi)} \Psi = 0,$
the $\mbox{Ric}_{\Lunit \Lunit}$ term in the Raychaudhuri  
equation \eqref{E:RAYCHAUDURI} is,
up to lower-order terms, 
a perfect $\Lunit$ derivative
of the first derivatives of $\Psi.$
To close our estimates, what we really need are higher-order\footnote{In fact, we need only top-order versions of the identity.} versions 
of this identity.
In particular, we can commute the Raychaudhuri-type identity 
with $\Rot$ to obtain a transport equation equation for a ``modified'' version
of $\Rot \mytr \upchi$ that does not lose derivatives relative to $\Psi.$
We make this precise in the following definition, where 
$\chifullmodarg{\Rot}$ is the ``modified'' quantity.

\begin{definition}[\textbf{Modified version of} $\Rot \mytr \upchi$]
We define the modified quantity $\chifullmodarg{\Rot}$ as follows:
\begin{align}
	\chifullmodarg{\Rot}
	& := \upmu \Rot \mytr \upchi 
		+ \Rot \mathfrak{X},
			\label{E:ROTTRCHIMODIFIED} \\ 
	\mathfrak{X}
	& := - G_{\Lunit \Lunit} \Rad \Psi
			+ \upmu \angG_{\Lunit}^{\ A} \angD_A \Psi
			- \frac{1}{2} \upmu \angG_A^{\ A} \Lunit \Psi
			- \frac{1}{2} \upmu G_{\Lunit \Lunit} \Lunit \Psi.
			\label{E:ROTTRCHIMODIFIEDDISCREPANCY}
\end{align}
\end{definition}
In \eqref{E:ROTTRCHIMODIFIEDDISCREPANCY},
$\angG_{\Lunit}^{\ A}$ is the $S_{t,u}-$tangent vectorfield
formed by projecting the vectorfield
with rectangular components
 $G_{\alpha}^{\ \nu}\Lunit^{\alpha}$
onto the $S_{t,u}.$

In total, the strategy described above allows us to show that
$\chifullmodarg{\Rot}$ verifies a transport equation
of the following delicate form.
\begin{lemma} \cite[\textbf{Proposition 10.2.3; Transport equation for the modified quantity}]{jS2014}
The quantity $\chifullmodarg{\Rot}$ defined in \eqref{E:ROTTRCHIMODIFIED}
verifies the transport equation
\begin{align} \label{E:MODIFIEDTRANSPORTEQUATIONFORROTCHI}
	\Lunit \chifullmodarg{\Rot}
	- \left\lbrace
			2 \frac{\Lunit \upmu}{\upmu}
			- \mytr \upchi 
		\right\rbrace
		\chifullmodarg{\Rot}
	& = 
		\left\lbrace
					\frac{1}{2} \mytr \upchi
					- 2 \frac{\Lunit \upmu}{\upmu} 
		\right\rbrace
		\Rot \mathfrak{X}
		+ \err,
\end{align}
where $\err$ depends on at most two derivatives of $\Psi$, is regular in $\upmu,$ and decaying in $t$.
\end{lemma}

We stress again that the advantage of \eqref{E:MODIFIEDTRANSPORTEQUATIONFORROTCHI}
over the unmodified equation
$\Lunit \Rot \mytr \upchi
= \frac{1}{2} G_{\Lunit \Lunit} \angLap \Rot \Psi + \cdots$
is that the right-hand side of equation \eqref{E:MODIFIEDTRANSPORTEQUATIONFORROTCHI} 
\emph{does not depend on the third derivatives of $\Psi.$}
Hence, equation \eqref{E:MODIFIEDTRANSPORTEQUATIONFORROTCHI} can be used to derive 
$L^2$ estimates for $\chifullmodarg{\Rot}$ that 
\emph{do not lose derivatives relative to $\Psi.$} 

\begin{remark}[\textbf{The need for elliptic estimates}]
\label{R:NEEDFORELLIPTIC}
Hiding in the terms
$\err$ in \eqref{E:MODIFIEDTRANSPORTEQUATIONFORROTCHI}
lies another technical headache that we will briefly mention but not dwell on. 
Specifically, there is a quadratically small term, 
roughly of the form $\upmu \hat{\upchi} \cdot \angLie_{\Rot} \hat{\upchi},$ 
that formally involves the same
number of $\upchi$ derivatives as the modified quantity $\chifullmodarg{\Rot}$
(that is, one)
but that cannot be directly estimated back in terms of $\chifullmodarg{\Rot}.$
Here, $\hat{\upchi}$ is the trace-free part of the $S_{t,u}$ tensor \eqref{E:CHIDEF}
and $\angLie_{\Rot}$ denotes Lie differentiation with respect to $\Rot$
followed by projection onto the $S_{t,u}.$
The term $\angLie_{\Rot} \hat{\upchi}$ involves three derivatives of the eikonal function $u$ 
and as we have described, it will lead to derivative loss if not properly handled.
To derive suitable $L^2$ estimates for this term,
we have to derive a family of
elliptic estimates on the spheres $S_{t,u}.$ The main ideas behind this strategy
can be traced back to Christodoulou-Klainerman's proof
of the stability of Minkowski spacetime \cite{dCsK1993}. Similar strategies were also employed 
in \cite{sKiR2003} and \cite{dC2007}.
The main point is that the elliptic estimates 
allow us to estimate
$\| \upmu \angLie_{\Rot} \hat{\upchi}\|_{L^2(S_{t,u})}$ back in terms of 
$\| \upmu \Rot \mytr \upchi \|_{L^2(S_{t,u})}$ plus errors, 
and that $\upmu \Rot \mytr \upchi$ can be controlled in $L^2$
by using the $L^2$ estimates for $\chifullmodarg{\Rot}$ and 
the up-to-second-order $L^2$ estimates for $\Psi.$
\end{remark}

\begin{remark} \label{R:HARDTRANSPORTTERM}
In the detailed proof, we must invert the transport equation \eqref{E:MODIFIEDTRANSPORTEQUATIONFORROTCHI}
and obtain suitable $L^2$ estimates for $\chifullmodarg{\Rot}.$ However, this is not an easy task;
see the proof of \cite[Lemma 19.4.1]{jS2014} for the details.
The main reason is that the factors 
$\left\lbrace
			2 \frac{\Lunit \upmu}{\upmu}
			- \mytr \upchi 
		\right\rbrace$
and 
$\left\lbrace
					\frac{1}{2} \mytr \upchi
					- 2 \frac{\Lunit \upmu}{\upmu} 
\right\rbrace$
in \eqref{E:MODIFIEDTRANSPORTEQUATIONFORROTCHI} 
have a drastic effect on the behavior of $\chifullmodarg{\Rot}$
and require a careful analysis.
\end{remark}

We now return to the question of how to replace the derivative-losing error integral 
\[
- \int_{\MM_{t,u}}
				(2 \Rad \Psi)
				(\Rot \mytr \upchi) 
				\Rad \Rot \Psi
			\, d \vol
\]
with \eqref{E:REWRITINGOFKEYINTEGRAL}.
We first use the identity
$\Rot \mytr \upchi = \upmu^{-1}\chifullmodarg{\Rot} - \upmu^{-1} \Rot \mathfrak{X}$
to replace the derivative-losing term $\Rot \mytr \upchi$ with terms that do not lose derivatives.
As we described in Remark \ref{R:HARDTRANSPORTTERM}, the most difficult analysis corresponds to the error integral
generated by the piece $\upmu^{-1}\chifullmodarg{\Rot}.$
We do not want to burden the reader
with the large number of technical complications that arise in the analysis of this error integral.
Instead, we focus on the error integral generated by the other piece, namely
$	\int_{\MM_{t,u}}
			2 (\Rad \Psi) \upmu^{-1} (\Rot \mathfrak{X}) \Rad \Rot \Psi
			\, d \vol.
$
The difficult part of this error integral comes 
from the top-order part of $\Rot$ applied to the first term 
$- G_{\Lunit \Lunit} \Rad \Psi$
on the right-hand side of
\eqref{E:ROTTRCHIMODIFIEDDISCREPANCY}. 
That is, we focus on the 
following error integral:
\begin{align} \label{E:REPRESENTATIVEANNOYINGERRORINTEGRAL}
		-2
		\int_{\MM_{t,u}}
			\frac{1}{\upmu}
			(\Rad \Psi) 
			G_{\Lunit \Lunit}
			(\Rad \Rot \Psi)^2
		\, d \vol.
\end{align}
Though the integral \eqref{E:REPRESENTATIVEANNOYINGERRORINTEGRAL} does not lose derivatives,
it is nonetheless difficult to bound. If we were
to try to bound it \eqref{E:REPRESENTATIVEANNOYINGERRORINTEGRAL} 
by simply inserting the Heuristic Principle-type estimates $|\Rad \Psi| \lesssim \varepsilon (1 + t)^{-1}$
and $|G_{\Lunit \Lunit}| \lesssim 1,$
then we would not be able to derive the desired a priori energy estimate \eqref{E:TOPORDERMULTESTIMATE}; 
we would find that
there is a loss that spoils the estimates and allows for the 
power of $\upmu_{\star}^{-1}$ on the right-hand side of \eqref{E:TOPORDERMULTESTIMATE}
to grow like $C \varepsilon \ln(\myexp + t),$
thereby completely ruining the $L^2$ hierarchy of Prop.~\ref{P:APRIORIENERGYESTIMATES}.
Christodoulou overcame this difficulty by observing the following 
\emph{critically important structure}:
by using the transport equation 
$\Lunit \upmu = \frac{1}{2} G_{\Lunit \Lunit} \Rad \Psi + \err$
(see \eqref{E:UPMUSCHEMATICTRANSPORT}),
we can rewrite \eqref{E:REPRESENTATIVEANNOYINGERRORINTEGRAL} as
\begin{align} \label{E:AGAINREWRITINGOFKEYINTEGRAL}
	-4
	\int_{\MM_{t,u}}
			\frac{\Lunit \upmu}{\upmu}
			(\Rad \Rot \Psi)^2
	\, d \vol
	+ \cdots,
\end{align}
which is precisely the integral \eqref{E:REWRITINGOFKEYINTEGRAL} that we 
successfully treated above. We have thus sketched the main ideas behind the procedure
that allows us to avoid losing derivatives.

\begin{remark}[\textbf{Difficult top-order error integrals that arise during the Morawetz multiplier estimates}]
In order to derive the top-order estimate \eqref{E:TOPORDERMORESTIMATE}
corresponding to the Morawetz multiplier $\Mor = \rgeo^2 \Lunit,$
we use Christodoulou's strategy \cite{dC2007}, which is quite different
than the one we use to derive the estimate \eqref{E:TOPORDERMULTESTIMATE}
corresponding to the timelike multiplier $\Mult.$
The main idea is that since $\Mor$ is proportional to 
$\Lunit,$ we can integrate by parts in the divergence theorem identity
\eqref{E:MTUDIVERGENCETHM} (see Remark \ref{R:MODCURRENTNOTINDIVTHM})
in order to trade,
in the analog of the error integral \eqref{E:TOPORDERENERGYIDCARICATURE}, 
the $\Rot$ derivative on $\mytr \upchi$
for an $\Lunit$ derivative. The gain is that whenever the top-order derivative
of an eikonal function quantity such as $\mytr \upchi$
involves an $\Lunit$ derivative, 
we do not have to worry about losing derivatives because we have a ``direct expression''
for these quantities based on the fact that they verify a transport equation in the direction of $\Lunit.$
Hence, for the Morawetz multiplier estimates,
we can avoid working with fully modified
quantities such as \eqref{E:ROTTRCHIMODIFIED}-\eqref{E:ROTTRCHIMODIFIEDDISCREPANCY}, and we do not have
to invoke any elliptic estimates on $S_{t,u}$
(see Remark \ref{R:NEEDFORELLIPTIC}). 
However, moving the $\Lunit$ derivative 
generates some very difficult $\Sigma_t^u$ error integrals that
lead to top-order $\upmu_{\star}^{-1}$ degeneracy, 
similar to the degeneracy we encountered in Lemma \ref{L:INTROKEYINTEGRATINGFACTORGRONWALLESTIMATE}.
In deriving the Morawetz multiplier estimates,
although we do not need to use fully modified quantities of the form
\eqref{E:ROTTRCHIMODIFIED}-\eqref{E:ROTTRCHIMODIFIEDDISCREPANCY}, we
do need to define and use related partially modified versions of
both $\mytr \upchi$ and $\angD \upmu$
in order to avoid certain error integrals
that have unfavorable $t-$growth.
We do not want to further burden the reader with these technical details here,
so we do not pursue this issue further.
\end{remark}

\subsection{Descending below top order}
If we applied the above strategy of Subsubsect.~\ref{SSS:AVOIDINGTOPORDERDERIVATIVELOSS}
at all derivative levels,
then \emph{all of the energy-flux-Morawetz estimates would degenerate in the same way 
as \eqref{E:TOPORDERMULTESTIMATE}-\eqref{E:TOPORDERMORESTIMATE} with respect to}
$\upmu_{\star}^{-1}.$ In particular, we would not recover the non-degenerate estimates
\eqref{E:LOWESTLEVELMULTNONDEGENERATE}-\eqref{E:LOWESTLEVELMORNONDEGENERATE}.
This would in turn prevent us from recovering the 
decay estimates of the Heuristic Principle, which are based
on \eqref{E:LOWESTLEVELMULTNONDEGENERATE}-\eqref{E:LOWESTLEVELMORNONDEGENERATE} and Sobolev embedding.
Hence, we would not be able to show that 
the terms we have deemed small errors are in fact small, 
and the entire proof would break down.

To overcome this difficulty, we note that since we are below top order, we can allow
the loss in derivatives in the difficult error integral. 
In particular, there is no need to use the complicated procedure
that led to the difficult top-order integral \eqref{E:REWRITINGOFKEYINTEGRAL}.
In avoiding this procedure, 
we are rewarded with a less degenerate power of $\upmu_{\star}^{-1},$
which comes from Lemma \ref{L:GAININGMU} and the availability of favorable powers of $t.$

As before, in the following discussion, 
$\mathring{\upepsilon}$ denotes the small size of the data.
To illustrate our strategy in some detail, 
let us imagine that three derivatives of $\Psi$ in $L^2$ 
(which corresponds to $\enzero[\mathscr{Z}^2 \Psi],$ etc.)
represents the top-order.
We also imagine, consistent with \eqref{E:TOPORDERMULTESTIMATE} and \eqref{E:TOPORDERMORESTIMATE}, that
the top-order energy-flux-Morawetz quantities are bounded by
\begin{align} \label{E:TOPORDERILLUSTRATIONBOUND}
	\enzero^{1/2}[\mathscr{Z}^2 \Psi](t,u) 
	+ \flzero^{1/2}[\mathscr{Z}^2 \Psi](t,u) 
	\lesssim \mathring{\upepsilon} \ln^{\Conone}(\myexp + t)\upmu_{\star}^{-\Contwo}(t,u),
		\\
	\enone^{1/2}[\mathscr{Z}^2 \Psi](t,u) 
	+ \flone^{1/2}[\mathscr{Z}^2 \Psi](t,u)
	+ \Morint^{1/2}[\mathscr{Z}^2 \Psi](t,u)
	\lesssim \mathring{\upepsilon} \ln^{\Conone + 2}(\myexp + t)\upmu_{\star}^{-\Contwo}(t,u)
	\label{E:TOPORDERMORILLUSTRATIONBOUND}
\end{align}
for positive constants $A$ and $B.$
We will use these estimates to show how to derive a bound
for the just-below-top-order quantities
$\enzero^{1/2}[\Rot \Psi](t,u)
+ \flzero^{1/2}[\Rot \Psi](t,u)$
with a \emph{smaller power of} $\upmu_{\star}^{-1}.$

One important ingredient is that the weighted quantity 
$\rgeo^2 \Rot \mytr \upchi$ verifies a transport equation with a good structure,
and this allows us to recover \emph{good $t-$weighted estimates}\footnote{Recall that $\rgeo(t,u) \approx 1 + t$
in the region of interest.} for
$\| \Rot \mytr \upchi \|_{L^2(\Sigma_t^u)}$ at the expense of a loss of derivatives. 
More precisely, a careful analysis of equation \eqref{E:TRCHIDERIVATIVELOSS} 
reveals that we can commute it with
$\rgeo^2 \Rot$ and use \eqref{E:ANGLAPINTERMSOFANGDIFFROTANDGOODTFACTOR} to deduce
\begin{align} \label{E:TRCHIROTCOMMUTEDLOSESDERIVATIVES}
	\left|
		\Lunit(\rgeo^2 \Rot \trch)
	\right|
	& = \left| 
				\rgeo^2 \angLap \Rot \Psi
			\right|
			+ \cdots
			\lesssim 
			(1 + t) 
			\left| 
				\angD \Rot \Rot \Psi 
			\right|
			+ \cdots,
\end{align}
where $\angD \Rot \Rot \Psi$ is a top-order term. 
Using the coerciveness property \eqref{E:INTROENONECOERCIVENESS} 
(note the appearance of another factor of $\upmu_{\star}^{-1/2}(t,u)!$), 
\eqref{E:TRCHIROTCOMMUTEDLOSESDERIVATIVES},
and the top-order estimate \eqref{E:TOPORDERMORILLUSTRATIONBOUND}, 
and ignoring the $\cdots$ terms,
we deduce
\begin{align} \label{E:LROTTRCHIDERIVATIVELOSSL2ESTIMATE}
	\left\|
		\Lunit(\rgeo^2 \Rot \trch)
	\right\|_{L^2(\Sigma_t^u)}
	& \lesssim \upmu_{\star}^{- 1/2}(t,u) \enone[\Rot \Rot \Psi](t,u)
		\lesssim \mathring{\upepsilon} \ln^{\Conone+2}(\myexp + t) \upmu_{\star}^{-\Contwo - 1/2}(t,u).
\end{align} 
Recalling that $\Lunit = \frac{\partial}{\partial t}$
and taking into account the fact that the spherical area form inherent in 
the norm $\| \cdot \|_{L^2(\Sigma_t^u)}$
is, in a pointwise sense, $\sim \rgeo^2 \sim (1 + t)^2,$
it is not too difficult 
(see \cite[Lemma 11.30.6]{jS2014})
to integrate
\eqref{E:LROTTRCHIDERIVATIVELOSSL2ESTIMATE}
to deduce
\begin{align} \label{E:ROTTRCHIDERIVATIVELOSSL2ESTIMATE}
	\left\|
		\Rot \trch
	\right\|_{L^2(\Sigma_t^u)}
	& 	\lesssim
			\mathring{\upepsilon} 
			\frac{\ln^{\Conone + 2}(\myexp + t) }{1 + t}
			\int_{t'=0}^t
				\frac{1}{1 + t'}
				\upmu_{\star}^{-\Contwo - 1/2}(t',u)
			\, dt'
		+ \cdots.
\end{align}
Applying Lemma \ref{L:GAININGMU} to inequality 
\eqref{E:ROTTRCHIDERIVATIVELOSSL2ESTIMATE},
we gain a power of $\upmu_{\star}$ through the time integration:
\begin{align} \label{E:ROTTRCHIDERIVATIVELOSSL2DETAILEDESTIMATE}
	\left\|
		\Rot \trch
	\right\|_{L^2(\Sigma_t^u)}
	& 	\lesssim
			\mathring{\upepsilon} 
			\frac{\ln^{\Conone + 3}(\myexp + t)}{1 + t}
			\upmu_{\star}^{-\Contwo + 1/2}(t',u)
		+ \cdots.
\end{align}

We now 
bound the first integral on the right-hand side of \eqref{E:TOPORDERENERGYIDCARICATURE},
that is, the integral
\[
		- \int_{\MM_{t,u}}
				(2 \Rad \Psi)
				(\Rot \mytr \upchi) 
				\Rad \Rot \Psi
			\, d \vol,
\]
by using the estimate 
\eqref{E:ROTTRCHIDERIVATIVELOSSL2DETAILEDESTIMATE}, 
the Heuristic Principle estimate 
(see Subsubsect.~\ref{SSS:HPMOREPRECISE})
$|\Rad \Psi| \lesssim \mathring{\upepsilon} (1 + t)^{-1},$
the coerciveness property \eqref{E:INTROMULTENERGYCOERCIVITY},
and Cauchy-Schwarz. 
It therefore follows from \eqref{E:TOPORDERENERGYIDCARICATURE} that
\begin{align} \label{E:BELOWTOPORDERCARICATURE}
	\sup_{s \in [0,t]}
	\enzero[\Rot \Psi](s,u)
	+ \flzero[\Rot \Psi](s,u)
	& \leq C \mathring{\upepsilon}
		+ C \mathring{\upepsilon}^2
			\int_{t'=0}^t
				\frac{1}{(1 + t')^{3/2}}
				\upmu_{\star}^{-\Contwo + 1/2}(t',u)
				\enzero^{1/2}[\Rot \Psi](t',u)
			\, dt'
			+ \cdots \\
	& \ \ 
		\leq
		 C \mathring{\upepsilon}^2
			\sup_{s \in [0,t] }\enzero^{1/2}[\Rot \Psi](s,u)
			\int_{t'=0}^t
				\frac{1}{(1 + t')^{3/2}}
				\upmu_{\star}^{-\Contwo + 1/2}(t',u)
			\, dt'.
			\notag
\end{align}
Using Lemma \ref{L:GAININGMU} to
bound the time integral on the right-hand side of
\eqref{E:BELOWTOPORDERCARICATURE},
and in particular
taking advantage of the good time decay in the integrand \eqref{E:BELOWTOPORDERCARICATURE},
we deduce from \eqref{E:BELOWTOPORDERCARICATURE} that
\begin{align} \label{E:LOWERORDERPOWERGAINED}
	\enzero^{1/2}[\Rot \Psi](t,u)
	+ \flzero^{1/2}[\Rot \Psi](t,u)
	& \lesssim 
		\mathring{\upepsilon}
		\upmu_{\star}^{-\Contwo + 3/2}(t',u)
		+ \cdots.
\end{align}
The inequality
\eqref{E:LOWERORDERPOWERGAINED}
has thus yielded the desired
gain in $\upmu_{\star}$
compared to the top-order bound
\eqref{E:TOPORDERILLUSTRATIONBOUND}.

\begin{remark}
Inequality \eqref{E:LOWERORDERPOWERGAINED} is mildly misleading
in the sense that there are some worse
error terms that only allow us to gain a single power of
$\upmu_{\star},$ rather than the $3/2$
suggested by \eqref{E:LOWERORDERPOWERGAINED}.
\end{remark}

We have thus explained the main ideas of how to descend one level below the top order
in the energy estimate hierarchy of Prop.~\ref{P:APRIORIENERGYESTIMATES}.
One can continue the descent, each time 
using Lemma \ref{L:GAININGMU} to gain a power of $\upmu_{\star}.$
Furthermore, the estimate \eqref{E:LASTGAIN} 
explains why we can eventually descend to the estimates
\eqref{E:LOWESTLEVELMULTNONDEGENERATE}-\eqref{E:LOWESTLEVELMORNONDEGENERATE},
which no longer degenerate at all, \emph{even as a shock forms}!

\section{The Sharp Classical Lifespan Theorem in $3D$ and Generalizations}
\label{S:SHARPLIFESPAN}
In this section, we provide a detailed statement of the
general \emph{sharp classical lifespan result}  
from \cite{jS2014}, which applies to
equations of the type $\square_{g(\Psi)} \Psi =0.$   
The theorem is an analog of the main theorem from
Christodoulou's work, namely
\cite[{Theorem 13.1 on pg. 888}]{dC2007},
which applied to a related class of quasilinear wave equations
that arise in relativistic fluid mechanics;
see Subsect.~\ref{SS:CHRISTODOULOURESULTS}.
We also provide a brief overview of its proof,  
which complements our discussion of 
generalized energy estimates
from Sect.~\ref{S:GENERALIZEDENERGY}.
We then sketch how to extend the result to apply to
equations of the form
$(g^{-1})^{\alpha \beta}(\partial \Phi)
\partial_{\alpha} \partial_{\beta} \Phi = 0.$
As in the spherically symmetric case, 
the result is the main ingredient
used in proving that a shock actually forms in solutions launched by an open set of data
(see Sect.~\ref{S:SHOCKFORMATIONANDCOMPARISON}).

\subsection{The sharp classical lifespan theorem}
The sharp classical lifespan theorem below is a direct analog of Prop.~\ref{P:ge},
which applied to spherically symmetric solutions.

\begin{theorem}\cite[\textbf{Theorem 21.1.1; Sharp classical lifespan theorem}]{jS2014} 
\label{T:LONGTIMEPLUSESTIMATES}
Let $(\mathring{\Psi} := \Psi|_{\Sigma_0}, \mathring{\Psi}_0 := \partial_t \Psi|_{\Sigma_0})$ 
be initial data for the covariant scalar wave equation\footnote{The theorem extends without any significant alterations
to equations of the form
$\square_{g(\Psi)} \Psi = \NN(\Psi)(\partial \Psi, \partial \Psi)$
whenever $\NN$ verifies the future strong null condition of
Remark \ref{R:STRONGNULL}. For simplicity we also assume \eqref{E:METRICNORMALIZATION}.   
As we have noted earlier, this assumption is easy to eliminate.}       
(in $3$ space dimensions)
\[
	\square_{g(\Psi)} \Psi = 0.
\] 
Assume that the data are supported in the Euclidean unit ball $\Sigma_0^1.$
Let $\mathring{\upepsilon} = 
\mathring{\upepsilon}[(\mathring{\Psi},\mathring{\Psi}_0)]
:=
\| \mathring{\Psi} \|_{H^{25}(\Sigma_0^1)} + \| \mathring{\Psi}_0 \|_{H^{24}(\Sigma_0^1)}$ be the size of the data. 
Let $0 < U_0 < 1$ be a fixed constant,
and $\Psi$ the corresponding solution restricted to a nontrivial region of the form $\MM_{T,U_0}$
(see Definition \eqref{E:MTUDEF} and Figure \ref{F:DIVTHM}).
If $\mathring{\upepsilon}$ is sufficiently small, then the outgoing lifespan $T_{(Lifespan);U_0},$ 
as defined in Subsubsect.~\ref{SSS:SSSHARPCLASSIALLIFESPAN}, is determined as follows:
\begin{align} \label{E:TLIFESPAN}  
	T_{(Lifespan);U_0}:= \sup \lbrace t \ | \ \inf_{s \in [0,t)} \upmu_{\star}(s,U_0) > 0 \rbrace,
\end{align}
where $\upmu_{\star}(t,u) := \min\lbrace 1, \min_{\Sigma_t^u} \upmu \rbrace$
(see Definition \eqref{E:SIGMATU}).
Furthermore, there exists a constant $C_{(Lower-Bound)} > 0$ such that
\begin{align}  \label{E:LIFESPANLOWERBOUND}
	T_{(Lifespan);U_0} > \exp\left(\frac{1}{C_{(Lower-Bound)}\mathring{\upepsilon}} \right).
\end{align}
In addition, the following statements    
hold true in 
$\MM_{T_{(Lifespan);U_0},U_0}.$

\begin{enumerate}
\item \textbf{Energy estimates.}
The energy estimate hierarchy of Prop.~\ref{P:APRIORIENERGYESTIMATES} is verified
for $(t,u) \in [0,U_0] \times [0,T_{(Lifespan);U_0}).$
A similar $L^2$ hierarchy holds for the scalar-valued functions
$\upmu - 1,$
$\Lunit_{(Small)}^i := \Lunit^i - \frac{x^i}{\rgeo},$
and
$\Radunit_{(Small)}^i := \Radunit^i + \frac{x^i}{\rgeo},$
and for the $S_{t,u}$ tensorfield $\upchi^{(Small)} := \upchi - \frac{\gsphere}{\rgeo},$
where $\rgeo(t,u) := 1 - u + t.$ 

\item
\textbf{Heuristic Principle.}
The Heuristic Principle estimates stated in  
 Subsubsect.~\ref{SSS:HPMOREPRECISE} are valid for $\Psi$ and its  
 low-order derivatives with respect 
to the commutation set
$
	\mathscr{Z} :=
	 \lbrace \rgeo \Lunit, \Rad, \Rot_{(1)}, \Rot_{(2)}, \Rot_{(3)} \rbrace
$ (see \eqref{eq:commuting-vfs}).

Related $C^0$ estimates hold for the 
low-order derivatives
of the scalar-valued functions
$\upmu - 1,$
$\Lunit_{(Small)}^i := \Lunit^i - \frac{x^i}{\rgeo},$
and
$\Radunit_{(Small)}^i := \Radunit^i + \frac{x^i}{\rgeo},$
and for the $S_{t,u}$ tensorfield $\upchi^{(Small)} := \upchi - \frac{\gsphere}{\rgeo}.$
In particular, 
if $T_{(Lifespan);U_0}< \infty,$
then these quantities extend
to $\Sigma_{T_{(Lifespan);U_0}}^{U_0}$
as many-times classically differentiable functions of
the geometric coordinates $(t,u,\vartheta).$
\item\textbf{Rectangular coordinates.} If $T_{(Lifespan);U_0}< \infty,$
then the change of variables map 
$\Upsilon: [0,T_{(Lifespan);U_0}) \times [0,U_0] \times \mathbb{S}^2 \rightarrow \MM_{T_{(Lifespan);U_0},U_0}$
from geometric to rectangular coordinates extends continuously to
$[0,T_{(Lifespan);U_0}] \times [0,U_0] \times \mathbb{S}^2.$
Furthermore, 
$\Upsilon$ has a positive Jacobian determinant and is globally invertible
on $[0,T_{(Lifespan);U_0}) \times [0,U_0] \times \mathbb{S}^2.$
In addition, if $T_{(Lifespan);U_0}< \infty,$ then the Jacobian determinant
of $\Upsilon$ vanishes precisely on the set of points
$p \in \Sigma_{T_{(Lifespan);U_0}}^{U_0}$
with $\upmu(p) = 0.$
\item \textbf{Lower bound for $\Rad \Psi= \upmu \Radunit\Psi.$} There exists a constant $c > 0$ such that if
$\upmu(t,u,\vartheta) \leq 1/4,$
and 
$G_{\Lunit \Lunit}(t,u,\vartheta) = \frac{d}{d \Psi} g_{\alpha \beta}(\Psi) \Lunit^{\alpha} \Lunit^{\beta}(t,u,\vartheta) \neq 0$
then,
\begin{align} 
	\Lunit \upmu(t,u,\vartheta) 
	& \leq
		- \frac{c}{(1 + t) \ln(e + t)},
	  \label{E:MAINTHEOREMSMALLMUIMPLIESLMUISNEGATIVE} \\
	|\Radunit \Psi|(t,u,\vartheta) 
	& \geq \frac{c}{\upmu(t,u,\theta) (1 + t) \ln(e + t)} 
		\frac{1}{\left| G_{\Lunit \Lunit}(t,u,\vartheta) \right|}.
	\label{E:RADUNITPSIBLOWSUP}
\end{align}
Moreover, the vectorfield $\Radunit$ verifies the Euclidean estimate\footnote{Here, $|V|_{\Euct}^2 := \delta_{ab} V^a V^b$ 
and $\partial_r$ is the standard Euclidean radial derivative.}
\begin{align} \label{E:RADUNITBOUND}
\left| \Radunit - (- \partial_r) \right|_{\Euct} \lesssim \mathring{\upepsilon} \ln(\myexp + t) (1 + t)^{-1}.
\end{align}
At all points $p\in \Sigma_{T_{(Lifespan);U_0}}^{U_0}$ where $\upmu(p)
= 0$, the derivative $\Radunit \Psi$ blows up like $\upmu^{-1}.$
\end{enumerate}
\medskip
\end{theorem}

\begin{remark}[\textbf{Maximal development of the data}] \label{R:MAXIMALDEVELOPMENT}
	We stress the following important feature, made possible
	by Christodoulou's framework:
	with some additional effort, the results of
	Theorem~\ref{S:SHARPLIFESPAN} can be extended 
	to a larger region, beyond the hypersurface $\Sigma_{T_{(Lifespan);U_0}},$
	to reveal a portion of the maximal development of the data;
	see Subsect.~\ref{SS:CHRISTODOULOURESULTS} and
	in particular Christodoulou's Theorem~\ref{T:CHRISTODOULOUSHOCKFORMATION}.
	This extra information can be obtained because the results
	of Theorem~\ref{S:SHARPLIFESPAN} are sufficiently sharp.
\end{remark}

\begin{proof}[Discussion of the proof of Theorem~\ref{T:LONGTIMEPLUSESTIMATES}]
The basic strategy begins with assuming, as bootstrap assumptions, 
that the Heuristic Principle $C^0$ decay estimates 
(see Subsubsect.~\ref{SSS:HPMOREPRECISE})
hold for $\Psi$ and its low-order derivatives
on a region of the form $\MM_{T,U_0}$ for which   $\upmu > 0.$
By ``derivatives,'' we mean derivatives
with respect to the commutation vectorfields
$
	\mathscr{Z} :=
	 \lbrace \rgeo \Lunit, \Rad, \Rot_{(1)}, \Rot_{(2)}, \Rot_{(3)} \rbrace
$ (see \eqref{eq:commuting-vfs}).
This mirrors the start to our proof of Proposition 
\ref{P:ge},  in spherical symmetry.
Using these  bootstrap assumptions for $\Psi$
and the smallness of the initial data,
we derive analogous $C^0$ estimates for
$\upmu - 1$ and its low-order derivatives 
by using the transport equation \eqref{E:UPMUSCHEMATICTRANSPORT}
(note that $\upmu - 1$ vanishes in the case $\Psi \equiv 0$).
We also derive $C^0$ estimates for the 
quantities $\Lunit^i - x^i/\rgeo$ as well as $\upchi$ 
from the simple transport equations which they satisfy; 
see \eqref{E:UPMUSCHEMATICTRANSPORT}, \eqref{E:CHIDEF}, and \eqref{E:CHIALT}.
It is essential to note that all of these low-order estimates are regular relative
to $\upmu.$
In particular, the estimate \eqref{E:RADUNITBOUND} 
can be proved during this stage of the argument.
Furthermore, 
assuming that one knows that the quantity in \eqref{E:TLIFESPAN}  
is the classical lifespan of the solution in the region of interest
(below, we describe how to establish this fact),
the estimate \eqref{E:LIFESPANLOWERBOUND}
can easily be derived by using the transport equation
\eqref{E:UPMUSCHEMATICTRANSPORT} to
prove that $\upmu$ must remain positive
up to a time of order
$\exp\left(\frac{1}{C_{(Lower-Bound)}\mathring{\upepsilon}} \right).$

\medskip

Next, we derive generalized energy estimates for $\Psi$ on the region $\MM_{T,U_0}.$
The main ideas behind these estimates  were discussed in 
Sect.~\ref{S:GENERALIZEDENERGY}. To control the
error terms, it is convenient to rely not only on the low derivative assumptions
discussed below, but on a full set of bootstrap assumptions,
including $L^2$ assumptions consistent with Prop.~\ref{P:APRIORIENERGYESTIMATES}.  
Clearly, in deriving the generalized energy estimates,
we must bound the norm 
$\| \cdot \|_{L^2(\Sigma_t^u)}$
of the high derivatives of the eikonal function
quantities such as
$\upmu,$ 
$\Lunit^i,$ 
and
$\upchi.$
Most of these estimates can be derived using the transport
equations mentioned in the previous paragraph. However, to bound the 
top derivatives of $\upchi$ in the norm $\| \cdot \|_{L^2(\Sigma_t^u)},$
we avoid derivative loss by using the modified quantities
described in Subsubsect.~\ref{SSS:AVOIDINGTOPORDERDERIVATIVELOSS}
and elliptic estimates (see Remark \ref{R:NEEDFORELLIPTIC}).
Similarly, one must carefully avoid top-order derivative loss 
stemming from the terms $\angLap \upmu,$
which appear upon commuting the wave equation with the transversal derivative $\Rad.$

After deriving the generalized energy estimates, 
we improve the Heuristic Principle bootstrap assumptions 
by assuming small-data and using Sobolev embedding on the spheres $S_{t,u}.$ 
In particular,
we use the low-order $L^2$ estimates \eqref{E:LOWESTLEVELMULTNONDEGENERATE} of 
Prop.~\ref{P:APRIORIENERGYESTIMATES}, which do not degenerate
at all relative to $\upmu^{-1}.$

We now give the main idea explaining why
the classical lifespan of $\Psi$ in regions of the form $\MM_{T,U_0}$
is given by \eqref{E:TLIFESPAN}.
The main point is that if $\inf_{\MM_{T,U_0}} \upmu > 0,$
then the rescaled frame $\lbrace \Lunit, \Rad, X_1, X_2 \rbrace$
is uniformly comparable to the rectangular coordinate vectorfield frame
$\lbrace \frac{\partial}{\partial x^{\alpha}} \rbrace_{\alpha = 0,1,2,3}$
on $\MM_{T,U_0}.$ Hence, the above $C^0$ bounds, which show 
that $\Psi$ and its derivatives relative to the rescaled frame  
remain uniformly bounded on $\MM_{T,U_0},$
imply that the first rectangular derivatives of $\Psi$ also
remain uniformly bounded on $\MM_{T,U_0}.$ 
Therefore, by standard techniques,\footnote{By ``standard techniques," we mean an adapted
version of the continuation criterion of Proposition
\ref{PROP:LOCALEXISTENCE}. Note that since in the present context, the metric depends only on
$\Psi,$ it suffices to control $\Psi$ in $W^{1,\infty}$, instead of $W^{2,\infty}$ as in the 
proposition.} we can 
extend the solution to a larger region of 
the form $\MM_{T + \Delta,U_0}.$ 

Next, as we noted in 
Remark \ref{R:MUISCONNECTEDTOTHEJACOBIANDETERMINANT},
the Jacobian determinant of the change of variables map $\Upsilon$
is proportional to $\upmu.$ This is the main observation needed
to prove the statements concerning $\Upsilon.$
  
Finally, 
the estimates
\eqref{E:MAINTHEOREMSMALLMUIMPLIESLMUISNEGATIVE} 
and \eqref{E:RADUNITPSIBLOWSUP}
are analogs of the estimates
\eqref{E:SSLUNITUPMULARGEINMAGNITUDE}
and \eqref{E:SSTRANSVERSALDERIVATIVELARGEINMAGNITUDE}
proved in spherical symmetry.
The additional terms
present away from spherical symmetry involve $\mathcal{C}_u-$tangential
derivatives of $\Psi.$ Hence, they decay very rapidly and make only
a negligible contribution to the estimates.

\end{proof}

\subsection{Extending the sharp classical lifespan theorem to a related class of equations}
	\label{SS:EXTENSIONSOFTHESHARPCLASSICALLIFESPANTHEOREMTOALINHACSEQUATION}
Below we sketch how to extend Theorem~\ref{T:LONGTIMEPLUSESTIMATES} to apply to 
non-covariant quasilinear equations of the form
\begin{align} \label{E:NONCOVARIANTHOMOGENEOUSQUASILINEARWAVEEQUATION}
	(g^{-1})^{\alpha \beta}(\partial \Phi)
	\partial_{\alpha} \partial_{\beta}
	\Phi & = 0.
\end{align}
Analogously, 
the small-data shock-formation theorem (Theorem~\ref{T:STABLESHOCKFORMATION} below)
can be extended to apply to equations of type 
\eqref{E:NONCOVARIANTHOMOGENEOUSQUASILINEARWAVEEQUATION},
provided the nonlinearities fail the classic null condition (we
assume, of course, that $g_{\alpha\beta} = m_{\alpha\beta} +
\mathcal{O}(|\partial\Phi|)$ is a perturbation of the Minkowski
metric). In particular, recall from Remark \ref{Re:Aleph-systems} that the correct
analog of the future null condition failure factor $\FutFailFac$ is
\begin{align} \label{E:OTHERFAILUREFACTOR}
\FutFailFac
:= m_{\kappa \lambda} G_{\alpha \beta}^{\kappa}(\partial \Phi = 0) 
\Lunit_{(Flat)}^{\alpha}\Lunit_{(Flat)}^{\beta} \Lunit_{(Flat)}^{\lambda},
\end{align}
where 
\begin{align} \label{E:BIGGFORTHEOTHERMETRIC}
	G_{\alpha \beta}^{\lambda}
	= G_{\alpha \beta}^{\lambda}(\partial \Phi)
	& := \frac{\partial}{\partial (\partial_{\lambda} \Phi)} g_{\alpha \beta}(\partial \Phi).
\end{align}
When $\FutFailFac \equiv 0,$ the nonlinearities verify Klainerman's classic null
condition \cite{sK1984}, and the methods of 
\cite{sK1986} and \cite{dC1986a} yield small-data global existence.
When $\FutFailFac$ is nontrivial,
Theorem~\ref{T:STABLESHOCKFORMATION} below
can be extended to show that
small-data future shock formation occurs.
See also Sect.~\ref{SS:ALINHACSHOCKFORMATIONRESULTS}
for a discussion of Alinhac's related
small-data shock formation theorem.

\begin{remark}
\label{R:BIGDIDFFERENCE}
	Note that there is a major difference between equations of type
	\eqref{E:NONCOVARIANTHOMOGENEOUSQUASILINEARWAVEEQUATION} and the  
	\emph{scalar equations} of the form
	$(g^{-1})^{\alpha \beta}(\Psi)
	\partial_{\alpha} \partial_{\beta}
	\Psi = 0.$
	As we described in Subsect.~\ref{SS:GENERALSYSTEMEQUATIONS},
	the latter type of equations exhibit small-data global existence
	even when the classic null condition fails.      
\end{remark}

{\bf Examples}
\begin{itemize}
	\item In the case of the equation
		$\square_m \Phi= \pr_t((m^{-1})^{\a\b}\pr_\a\Phi \pr_\b \Phi),$
		we compute that $G_{\alpha \beta}^{\lambda} = 2 \delta_{\alpha}^{\lambda} m_{\beta 0}.$
		Hence, $\FutFailFac = m_{\kappa \lambda} \delta_{\alpha}^{\kappa} 
		\Lunit_{(Flat)}^{\alpha}\Lunit_{(Flat)}^{\beta} \Lunit_{(Flat)}^{\lambda}
		= - m_{\alpha \lambda} \Lunit_{(Flat)}^{\alpha}\Lunit_{(Flat)}^0 \Lunit_{(Flat)}^{\lambda} = 0.$
		Therefore, the nonlinearities in this equation verify the classic null condition.
	\item In the case of the equation
		$\square_m \Phi= 2 \partial_t \Phi \partial_t^2 \Phi,$
		we compute that $G_{\alpha \beta}^{\lambda} = m_{\alpha 0} m_{\beta 0} \delta_0^{\lambda}.$
		Hence, $\FutFailFac = m_{\alpha 0} m_{\beta 0} \delta_0^{\lambda} \Lunit_{(Flat)}^{\alpha}\Lunit_{(Flat)}^{\beta} \Lunit_{(Flat)}^{\lambda}
		\equiv 1.$ Therefore, the nonlinearities in this equation fail the classic null condition.
\end{itemize}

\subsubsection{Connections to equations of the form $\square_g \Psi = \NN$}
The main idea of extending the theorem is to differentiate
\eqref{E:NONCOVARIANTHOMOGENEOUSQUASILINEARWAVEEQUATION} with
rectangular coordinate derivatives $\partial_{\nu}$ and to set
\begin{align}
	\Psi_{\nu} := \partial_{\nu} \Phi,
		\\
	\vec{\Psi} := (\Psi_0,\Psi_1,\Psi_2,\Psi_3),
\end{align}
thereby arriving at a coupled system that can be put into the form
\begin{align} \label{E:RECTDIFFERENTIATEDCOUPLEDSYSTEM}
	\square_{g(\vec{\Psi})} \Psi_{\nu}
	& = \NN(\vec{\Psi})(\partial \vec{\Psi},\partial \Psi_{\nu}),
\end{align}
where $\square_{g(\vec{\Psi})}$ is the \emph{covariant wave operator}
corresponding to $g(\vec{\Psi}).$
The semilinear term
$\NN(\vec{\Psi})(\partial \vec{\Psi},\partial \Psi_{\nu})$ generated
from the commutation
verifies the future strong null condition\footnote{More precisely, we have
\[       \NN(\vec{\Psi})(\partial \vec{\Psi},\partial \Psi)
= (g^{-1})^{\alpha \alpha'} (g^{-1})^{\beta \beta'} G_{\alpha'
\beta'}^{\mu} 
                                \left\lbrace
                                        \partial_{\beta} \Psi_{\alpha}
\partial_{\mu} \Psi
                                        - \partial_{\mu} \Psi_{\alpha}
                                          \partial_{\beta} \Psi
                                \right\rbrace 
+ (g^{-1})^{\alpha \beta} 
                        \frac{1}{\sqrt{|\mbox{det} g|}} 
                        \frac{\partial \sqrt{|\mbox{det} g|}}{\partial
\Psi_{\lambda}}
                        \partial_{\alpha} \Psi_{\lambda}
\partial_{\beta} \Psi,\]
                        where $G_{\alpha \beta}^{\lambda}(\vec{\Psi})$
                        is defined in \eqref{E:BIGGFORTHEOTHERMETRIC}
                        and the determinant is taken relative to the
rectangular coordinates; see \cite[Lemma A.1.2]{jS2014} for the details.}
of Remark \ref{R:STRONGNULL}. 
Hence, as in our study of the scalar equation \eqref{E:SPECIFICSEMILINEARTERMSGENERALQUASILINEARWAVE}
under the structural assumptions of Subsubsect.~\ref{SSS:STRUCTURAL},
the dangerous quadratic terms, whose presence is heralded by
$\FutFailFac\not\equiv 0$, can only hide in the operator
$\square_{g(\vec{\Psi})}.$ 
With the term $\NN(\vec{\Psi})(\partial \vec{\Psi},\partial \Psi_{\nu})$
having little effect on the dynamics, we can effectively 
analyze the system \eqref{E:RECTDIFFERENTIATEDCOUPLEDSYSTEM}
by studying each scalar equation for $\Psi_{\nu}$
using methods similar to the ones we used to analyze 
the scalar equation $\square_{g(\Psi)} \Psi = 0.$
In particular, the Heuristic Principle estimates
\eqref{E:TANGENTIALFASTDECAY}-\eqref{E:RESCALEDRADDISPERSIVEESTIMATE}
hold for each scalar component $\Psi_{\nu}.$

\subsubsection{The evolution equation for $\upmu$ and its connection to the top-order $L^2$ estimates}
It is instructive to examine the dangerous quadratic terms
present in the system \eqref{E:RECTDIFFERENTIATEDCOUPLEDSYSTEM}
from a different point of view by deriving
the evolution equation for $\upmu$ 
(that is, an analog of equation \eqref{E:UPMUSCHEMATICTRANSPORT})
in the present case of the metric $g(\vec{\Psi}).$ Specifically, 
arguing as in our proof of \eqref{E:UPMUSCHEMATICTRANSPORT} 
and exploiting the identity $\partial_{\alpha} \Psi_{\beta} = \partial_{\beta} \Psi_{\alpha},$ 
we find that
 \begin{align} \label{E:NEWUPMUSCHEMATICTRANSPORT}
	\Lunit \upmu(t,u,\vartheta)
	& = - \frac{1}{2} [G_{\Lunit \Lunit}^{\Lunit} \Radunit^a \Rad \Psi_a](t,u,\vartheta)
		+ \upmu \err	\\
	& = - \frac{1}{2} \InitialFutFailFac(\vartheta) [\Radunit^a \Rad \Psi_a](t,u,\vartheta)
		+ \upmu \err,
		\notag
\end{align}
where\footnote{As we described in Subsect.~\ref{SS:ROLEOFINVERSEFOLIATIONDENSITY}, 
$\InitialFutFailFac$ is a good approximation to $\FutFailFac$ that has the advantage of being
constant along the integral curves of $\Lunit.$} 
\begin{align} \label{E:OTHERDATAFAILUREFACTOR}
	\InitialFutFailFac(\vartheta)
	:= \FutFailFac(t=0,u=0,\vartheta),
\end{align} 
$G_{\Lunit \Lunit}^{\Lunit} 
:= G_{\alpha \beta}^{\lambda}(\vec{\Psi}) \Lunit^{\alpha} \Lunit^{\beta} \Lunit_{\lambda}$
and the error terms $\err$
in \eqref{E:NEWUPMUSCHEMATICTRANSPORT}
are small and decaying according to \eqref{E:TANGENTIALFASTDECAY}-\eqref{E:RESCALEDRADDISPERSIVEESTIMATE}.
In deriving the second line in \eqref{E:NEWUPMUSCHEMATICTRANSPORT}, we have used
the fact that we can prove an estimate of the form
$G_{\Lunit \Lunit}^{\Lunit}(t,u,\vartheta) = \InitialFutFailFac(\vartheta) + \mathcal{O}(\mathring{\upepsilon}),$
where $\mathring{\upepsilon}$ is the size of the data.
Hence, the term
$- \frac{1}{2} \InitialFutFailFac(\vartheta) [\Radunit^a \Rad \Psi_a](t,u,\vartheta)$
is the dangerous one that can cause $\upmu$ to vanish in finite time.

To analyze solutions $\vec{\Psi},$
we can derive energy estimates for each scalar component
$\Psi_{\nu}$ by commuting each equation
\eqref{E:RECTDIFFERENTIATEDCOUPLEDSYSTEM} with vectorfield operators
$\mathscr{Z}^N$ and deriving energy identities of the form
\eqref{E:MTUDIVERGENCETHM} for $\mathscr{Z}^N \Psi_{\nu}.$ 
The energy identities for the $\Psi_{\nu}$ are of course
coupled, but the analysis of each component is essentially the same as it is for the 
scalar equation $\square_{g(\Psi)} \Psi = 0.$
The only new ingredients that we need are estimates of the following form,
\emph{which must hold for each} $\nu = 0,1,2,3:$
\begin{align} \label{E:FIRSTSTATEMENTHARDSHARPRADPSIPOINTWISEESTIMATE}
	\left|G_{\Lunit \Lunit}^{\Lunit} \Rad \Psi_{\nu} \right| 
	& \leq 2 \left|\Lunit \upmu \right| + \err.
\end{align}
The estimates \eqref{E:FIRSTSTATEMENTHARDSHARPRADPSIPOINTWISEESTIMATE} 
are are less straightforward to derive compared to 
the case of the scalar equation $\square_{g(\Psi)} \Psi = 0$; their
derivation uses the symmetry condition $\partial_{\alpha} \Psi_{\beta} = \partial_{\beta} \Psi_{\alpha};$
see Appendix $A$ of \cite{jS2014}. These estimates are used to replace
inequality \eqref{E:TOPORDERENERGYCARICATUREGRONWALLREADY}
for each scalar component $\Psi_{\nu}$ of our system.
More precisely, the estimates \eqref{E:FIRSTSTATEMENTHARDSHARPRADPSIPOINTWISEESTIMATE}
are analogs of the algebraic replacement 
$\Lunit \upmu = \frac{1}{2} G_{\Lunit \Lunit} \Rad \Psi + \err$
that we used to derive \eqref{E:AGAINREWRITINGOFKEYINTEGRAL} from
\eqref{E:REPRESENTATIVEANNOYINGERRORINTEGRAL}. As such, they play
essential roles in allowing us to close the top-order $L^2$ estimates 
for the $\Psi_{\nu}$.

\section{The Shock-Formation Theorems and Comparisons}
\label{S:SHOCKFORMATIONANDCOMPARISON}

In this final section, we first state 
Alinhac's and Christodoulou's shock formation theorems.
We then compare and contrast their approaches
and explain the advantages of Christodoulou's framework.
In particular, we highlight the conceptual 
and technical gains that stem from
using a true eikonal function throughout the proof 
and working with quantities that are properly rescaled by $\upmu:$
relative to the rescaled quantities, the problem becomes
a traditional one in which one establishes long-time
well- posedness. We then state the shock formation theorem
of \cite{jS2014}. Finally, we compare and contrast the various results.

\subsection{Alinhac's shock formation theorem}
	\label{SS:ALINHACSHOCKFORMATIONRESULTS}
	In this section, we state Alinhac's shock formation results
	in $3$ space dimensions. We summarize the most important aspects of his shock formation results in the following
theorem.\footnote{Despite the title of the article \cite{sA2001b}, it
addresses both the cases of $2$ and $3$ space dimensions.} We try to stay true to the original formulation when stating his theorems (some of the coordinate systems appearing in the theorem could actually be eliminated).
The results are a partial summary of Theorems $2$ and $3$ of
\cite{sA2001b} in the case of $3$ space dimensions.

\begin{theorem} [\textbf{Alinhac}]
	\label{T:ALINHACSHOCKFORMATION}
	Consider the following initial value problem expressed relative to Minkowski-rectangular coordinates:
	\begin{align} \label{E:ALINHACWAVE}
		(g^{-1})^{\alpha \beta}(\partial \Phi) \partial_{\alpha} \partial_{\beta} \Phi
		& = 0, 
			\\
		(\Phi|_{t=0},\partial_t \Phi|_{t=0}) & = \uplambda (\mathring{\Phi}, \mathring{\Phi}_0),
			\label{E:ALINHACDATA}
	\end{align}
	where $\uplambda (\mathring{\Phi}, \mathring{\Phi}_0)$
	is a one-parameter family of smooth, compactly supported 
	initial data indexed by $\uplambda > 0.$ Let $\Phi_{\uplambda}$
	be the solution corresponding to the data.
	Assume that \eqref{E:GINVERSEISMINKOWSKIFORPHIEQUALS0} holds
	and that Klainerman's classic null condition fails for the nonlinearities in \eqref{E:ALINHACWAVE}, 
	that is, that the function $\FutFailFac(\theta)$ from \eqref{E:OTHERFAILUREFACTOR} is non-vanishing at some 
	Euclidean angle $\theta \in \mathbb{S}^2.$ 
	Recall that Friedlander's radiation field is the function
	$\Fried[(\mathring{\Phi}, \mathring{\Phi}_0)]: \mathbb{R} \times \mathbb{S}^2 \rightarrow \mathbb{R}$
	defined by
	\begin{align} \label{E:FRIEDNALNDERRADIATIONFIELD}
		\Fried[(\mathring{\Phi}, \mathring{\Phi}_0)]
		(q,\theta)
		& := 
			-
			\frac{1}{4 \pi} 
				\frac{\partial}{\partial q}
				\mathcal{R}[\mathring{\Phi}](q,\theta)
			+
			\frac{1}{4\pi} 
			\mathcal{R}[\mathring{\Phi}_0](q,\theta),
	\end{align}
	where the Radon transform $\mathcal{R}$ is defined in \eqref{E:RADONTRANSFORMOFF}.
	Assume that the function
	\begin{align} \label{E:ALINHACDATABLOWUPFUNCTION}
		\frac{1}{2} 
		\FutFailFac(\theta) 
		\frac{\partial^2}{\partial q^2} \Fried[(\mathring{\Phi}, \mathring{\Phi}_0)]
		(q,\theta)
	\end{align}
	has a unique, strictly positive, non-degenerate maximum at $(q_*,\theta_*).$
	If $\uplambda$ is sufficiently small and positive,
	then the classical lifespan $T_{(Lifespan);\uplambda}$ of the solution is finite and verifies
	\begin{align} \label{E:ALINHACLIMITINGLIFESPAN}
		\lim_{\uplambda \downarrow 0}
			\uplambda \ln T_{(Lifespan);\uplambda}
			& = \frac{1}
					{\frac{1}{2} 
						\FutFailFac(\theta_*) 
						\frac{\partial^2}{\partial q^2} \Fried[(\mathring{\Phi}, \mathring{\Phi}_0)](q_*,\theta_*)
					}.
	\end{align}

	In addition, there exists a first blow-up point $p_{(Blow-up);\uplambda}$
	with rectangular coordinates
	$p_{(Blow-up);\uplambda} 
	= (T_{(Lifespan);\uplambda},x_{\uplambda}^1,x_{\uplambda}^2,x_{\uplambda}^3)$
	and a constant $C > 0$ depending on $(\mathring{\Phi}, \mathring{\Phi}_0)$
	such that whenever $\uplambda$ is sufficiently small and positive, 
	the following statements hold true.

\medskip

\noindent \underline{\textbf{$C^1$ behavior relative to rectangular coordinates}.}
	$\Phi_{\uplambda}$ is a $C^1$ function of the rectangular coordinates 
	$\lbrace x^{\alpha} \rbrace$ and for $t \leq T_{(Lifespan);\uplambda},$
	we have
	\begin{align}	 \label{E:ALINHACSOLUTIONREMAINSC1BOUND}
		|\Phi_{\uplambda}| + \sum_{\alpha=0}^3 |\partial_{\alpha} \Phi_{\uplambda}| 
		\leq C \uplambda \frac{1}{1 + t}.
	\end{align}

	\medskip

	\noindent \underline{\textbf{Blow-up of second rectangular derivatives}.}
			One can obtain the behavior of the second rectangular 
			derivatives of $\Phi_{\uplambda}$ in the past domain of dependence
			of a neighborhood of $p_{(Blow-up);\uplambda}$ in $\Sigma_{T_{(Lifespan);\uplambda}}.$
			More precisely, strictly away from $p_{(Blow-up);\uplambda},$ $\Phi_{\uplambda}$ is a $C^2$ function
			of the rectangular coordinates with second-order derivatives that
			verify a bound of the form \eqref{E:ALINHACSOLUTIONREMAINSC1BOUND},
			where the constant $C$ depends on the distance to $p_{(Blow-up);\uplambda}.$
			In contrast to the regular behavior \eqref{E:ALINHACSOLUTIONREMAINSC1BOUND}, 
			the following blow-up behavior occurs:
			\begin{align} \label{E:ALINHACSOLUTIONREMAINSC2BLOWUP}
				C^{-1} \left(t \ln \frac{T_{(Lifespan);\uplambda}}{t} \right)^{-1}
				\leq
				\sum_{\alpha,\beta=0}^3 
				\left\|
					\partial_{\alpha} \partial_{\beta} \Phi_{\uplambda}
				\right\|_{C^0(\Sigma_t)}
				& \leq C \left(t \ln \frac{T_{(Lifespan);\uplambda}}{t} \right)^{-1}.
			\end{align}

	\medskip

	\noindent \underline{\textbf{Detailed description near the first blow-up point}.}
	We define the rescaled time variable 
	$\tau := \uplambda \ln t$ and in particular set
	$\tau_{(Lifespan);\uplambda} = \uplambda \ln T_{(Lifespan);\uplambda}.$
	Let $u_{(Flat)} = 1 + t - r$ be a flat eikonal function of
	the Minkowski metric. There exists a true eikonal function $u$ for the dynamic
	metric $g(\partial \Phi_{\uplambda})$ defined near $p_{(Blow-up);\uplambda}.$ 
	$u_{(Flat)}$ and $u$ respectively induce time-rescaled flat coordinates $(\tau,u_{(Flat)},\theta)$
	and geometric coordinates $(\tau, u,\theta),$ where
	$\theta$ is the Euclidean angle. The first blow-up point can be
	written uniquely in the time-rescaled flat coordinates as
	$p_{(Blow-up);\uplambda} = (\tau_{(Lifespan);\uplambda},
	u_{(Flat);\uplambda}, \theta_{\uplambda}).$ 

      Relative to the time-rescaled geometric coordinates, we have 
			the following conclusions. There exists a value $u_{\uplambda},$ a neighborhood 
			$\Omega 
				\subset 
				\lbrace (\tau,u,\theta) \ | \ 
					\tau \leq \tau_{(Lifespan);\uplambda},
					u \in \mathbb{R}, 
					\, \theta \in \mathbb{S}^2 
				\rbrace$ of
				$(\tau_{(Lifespan);\uplambda}, u_{\uplambda},\theta_{\uplambda}),$ and 
			functions 
$v,w,\zeta\in C^3(\Omega)$ with the following properties.
\begin{enumerate}
\item The functions $v,w,\zeta$ can be related to the solution
		$\Phi_{\uplambda}$ by interpreting $\zeta$ as the change of variables
		from $(\tau,u,\theta)$ to $u_{(Flat)},$ 
		$v$ as the solution $\Phi_{\uplambda}$ expressed
		in the time-rescaled geometric coordinates, 
		and $w$ as the rescaled first transversal derivative of $v.$ 
		More precisely, we have
\begin{subequations}
\begin{gather}
\zeta(\tau_{(Lifespan);\lambda},u_\lambda,\theta_\lambda) = u_{(Flat);\lambda}, 
	 \label{E:ALINHACSZETA} \\
v(\tau,u,\theta) = \lambda^{-1} \underbrace{(1 + e^{\tau/\lambda} -
\zeta(\tau,u,\theta))}_{r} \Phi_{\lambda}(\tau,u_{(Flat)} = \zeta(\tau,u,\theta), \theta),
  \label{E:ALINHACSLITTLEV} \\
\frac{\partial}{\partial u} v = w \frac{\partial}{\partial u}\zeta.
	\label{E:ALINHACSLITTLEW}
\end{gather}
\end{subequations}
\item The change-of-variables function $\zeta$ satisfies
	\begin{itemize}
		\item $\frac{\partial}{\partial u} \zeta \geq 0,$ 
			with equality exactly at $(\tau_{(Lifespan);\uplambda}, u_{\uplambda},\theta_{\uplambda})$ and
			nowhere else.
		\item At the point 
		$(\tau_{(Lifespan);\uplambda}, u_{\uplambda},\theta_{\uplambda}),
		$ we have $\frac{\partial^2}{\partial \tau \partial u} \zeta < 0,$
$\frac{\partial^2}{\partial \theta \partial u} 	\zeta = \frac{\partial^2}{\partial u^2}
\zeta = 0,$ and the Hessian with respect to $u,\theta$
			of $\frac{\partial}{\partial u} \zeta$
			is positive definite.
	\end{itemize}
\item The derivative $\frac{\partial}{\partial u} w$ does not vanish
at $(\tau_{(Lifespan);\uplambda}, u_{\uplambda}, \theta_{\uplambda}).$
\end{enumerate}
\end{theorem}

We make the following clarifying remarks concerning Alinhac's theorem.
\begin{itemize}
	\item Consider the inverse change of variables to $\zeta.$ That is, let
         $\eta$ be defined by
			$\eta(\tau,\zeta(\tau,u,\theta),\theta) = u.$
		Then $\eta$ is, relative to rectangular coordinates,
		a solution to the eikonal equation:
		$(g^{-1})^{\alpha \beta}(\partial \Phi_{\uplambda}) \partial_{\alpha} \eta \partial_{\beta} \eta = 0.$
	\item	Note that by the chain rule and the change of
		variables $u_{(Flat)} = \zeta(\tau,u,\theta)$ we have, with
		$\tau,\theta$ fixed, 
\[ 
	\frac{\partial\zeta}{\partial u} \frac{\partial}{\partial u_{(Flat)}}
	= \frac{\partial}{\partial u}. 
\]
Hence, by \eqref{E:ALINHACSLITTLEV}-\eqref{E:ALINHACSLITTLEW}, 
we have
\begin{subequations}
\begin{align} \label{E:ALINHACNONBLOWUPRELATION}
		\frac{\partial}{\partial u_{(Flat)}}(r
\Phi_{\uplambda})(\tau,\zeta(\tau,u,\theta),\theta) 
		& =\uplambda w(\tau,u,\theta),
				\\
		\frac{\partial^2}{\partial u_{(Flat)}^2}(r
\Phi_{\uplambda})(\tau,\zeta(\tau,u,\theta),\theta)
		& =\uplambda 
			\frac{\frac{\partial}{\partial u} w}
				{\frac{\partial}{\partial
u}\zeta}(\tau,u,\theta).
				\label{E:ALINHACKEYBLOWUPRELATION}
	\end{align}
\end{subequations}
	Hence, from \eqref{E:ALINHACNONBLOWUPRELATION}-\eqref{E:ALINHACKEYBLOWUPRELATION}
	and the conclusions of the theorem, it follows that the
	transversal second derivative 
	$\frac{\partial^2}{\partial u_{(Flat)}^2}(r \Phi_{\uplambda})$ blows up
	at $p_{(Blow-up);\uplambda}$ thanks to the vanishing of
	$\frac{\partial}{\partial u} \zeta,$ while the first
	derivative
	$\frac{\partial}{\partial u_{(Flat)}}(r \Phi_{\uplambda})$ does not blow-up.
	\item The quantity $\frac{\partial}{\partial u} \zeta$ should be compared to
	the quantity $\upmu$ discussed throughout this paper. The
	statements concerning the first derivatives of 
	$\frac{\partial}{\partial u} \zeta$ given in the theorem above are natural: the
	non-degeneracy condition $\frac{\partial^2}{\partial \tau \partial u} \zeta < 0$ is the
exact analogue of \eqref{E:LUPMUNEGATIVEQUANTIFIED} (see also
\eqref{E:SSLUNITUPMULARGEINMAGNITUDE} in spherical symmetry); the
conditions concerning $\frac{\partial^2}{\partial u^2} \zeta$ and
$\frac{\partial^2}{\partial \tau \partial u} \zeta$ are in fact \emph{necessary} if
$\tau_{(Lifespan);\uplambda}$ is the first (rescaled) blow-up time and
$(\tau_{(Lifespan);\uplambda}, u_{\uplambda}, \theta_{\uplambda})$ is
the unique first blow-up point. 
\end{itemize}

\subsection{Christodoulou's results}
\label{SS:CHRISTODOULOURESULTS}
In \cite{dC2007}, Christodoulou proved, for a class of quasilinear wave equations arising in 
irrotational relativistic fluid mechanics
(see also \cite{dCsM2012} for a generalization to the non-relativistic Euler equations), 
theorems that are analogous
to the sharp classical lifespan theorem (Theorem~\ref{T:LONGTIMEPLUSESTIMATES}) 
and the small-data shock-formation theorem (Theorem~\ref{T:STABLESHOCKFORMATION})
of the third author.
Actually, Christodoulou's work went somewhat beyond these two results in the following two senses.
\begin{enumerate}
	\item His shock-formation theorem was extended to apply to 
a class of small fluid equation data for which there is non-zero vorticity. 
However, most of his main results, 
including the shock-formation aspect of his work, 
applied only to a region in which the fluid is irrotational (vorticity-free), 
in which case the fluid equations reduce to the aforementioned 
scalar quasilinear wave equation. 
Hence, we will not elaborate on Christodoulou's treatment 
of the full relativistic Euler equations, but instead focus only on 
describing his results for irrotational flows.
\item After identifying the constant-time hypersurface region
$\Sigma_{T_{(Lifespan)};U_0}^{U_0}$ 
where the first shock-point occurs, he goes further by characterizing the nature of
the maximal future development, including the boundary, of the data lying in the exterior of 
the sphere $S_{0,U_0} \subset \Sigma_0^{U_0}.$   
Christodoulou's full description of the maximal development is made possible by the 
sharp estimates he proved in his sharp classical lifespan theorem
\cite[Theorem 13.1 on pg.\ 888]{dC2007},
analogous to Theorem~\ref{T:LONGTIMEPLUSESTIMATES} stated above, 
and which forms the most difficult part of the analysis. 
\end{enumerate}

We now describe Christodoulou's results \cite{dC2007} in more detail.
There are some inessential complications that arise in the formulation of the problem
compared to our study of the equations $\square_{g(\Psi)} \Psi = 0$ 
and that of Alinhac because 
Christodoulou's background solutions are not $\Phi = 0,$ but rather $\Phi = kt,$ where
$k$ is a non-zero constant. These are the solutions that correspond to the 
nontrivial constant states in relativistic fluid mechanics in Minkowski spacetime, and 
the resulting complications are simply issues of normalization and not serious ones.
To avoid impeding the flow of the paper,
we describe Christodoulou's equations in detail and address the normalization issue in
Appendix \ref{A:CHRISTODOULOUSEQUATIONS}.
Here, we summarize the most important aspects of his work.
The results stated below as Theorem~\ref{T:CHRISTODOULOUSHOCKFORMATION} are a conglomeration of
\cite[Theorem 13.1 on pg. 888, Theorem 14.1 on pg. 903, Proposition 15.3 on pg. 974,
and the Epilogue on pg. 977]{dC2007}.
The quantities that appear in the theorem
are essentially the same as the quantities
we have studied in Sects.~\ref{S:MAINIDEASIN3D}-\ref{S:SHARPLIFESPAN},
up to the differences in normalization we describe in
Appendix \ref{A:CHRISTODOULOUSEQUATIONS}.

\begin{theorem} [\textbf{Christodoulou}]
	\label{T:CHRISTODOULOUSHOCKFORMATION}
	Let $\upsigma = - (m^{-1})^{\alpha \beta} \partial_{\alpha} \Phi \partial_{\beta} \Phi$ 
	be as defined in \eqref{E:ENTHALPHYSQUARED}, where $m$ is the Minkowski metric.
	Assume that the Lagrangian $\mathcal{L}(\upsigma)$ verifies 
	the positivity conditions
	\eqref{E:FLUIDSINTERPRETATIONPOSITIVITY}
	in a neighborhood of $\upsigma = k^2,$ where $k$ is a non-zero constant, but that
	$\mathcal{L}(\upsigma)$ is not the exceptional Lagrangian
	\eqref{E:EXCEPTIONALLAGRANGIAN}.
	Consider the following Cauchy problem for the quasilinear (Euler-Lagrange) wave equation
	corresponding to $\mathcal{L}(\upsigma),$ expressed relative to rectangular coordinates:
	\begin{align} \label{E:CHRISTODOULOUWAVE}
		\partial_{\alpha} \left( \frac{\partial \mathcal{L}(\upsigma)}{\partial (\partial_{\alpha} \Phi)} \right)
		& = 0, 
			\\
		(\Phi|_{t=0},\partial_t \Phi|_{t=0}) & = (\mathring{\Phi}, \mathring{\Phi}_0).
	\end{align}
	Assume that the data are small perturbations of the 
	data corresponding to the non-zero constant-state solution
	$\Phi = k t$ and that the perturbations are compactly
	supported in the Euclidean unit ball.
	Let $U_0 \in (0,1/2)$ and let
	\begin{align}
		\mathring{\upepsilon}
		= \mathring{\upepsilon}[(\mathring{\Phi}, \mathring{\Phi}_0)]
		:=
		\| \mathring{\Phi}_0 - k \|_{H^N(\Sigma_0^{U_0})}
		+ \sum_{i=1}^3 \| \partial_i \mathring{\Phi} \|_{H^N(\Sigma_0^{U_0})}
	\end{align}
	denote the size of the data,
	where $N$ is a sufficiently large integer.\footnote{A numerical value of $N$ was not provided in \cite{dC2007}.}

	\medskip

	\noindent \underline{\textbf{Sharp classical lifespan}.}
	If $\mathring{\upepsilon}$ is sufficiently small, 
	then a sharp classical lifespan theorem in analogy with
	Theorem~\ref{T:LONGTIMEPLUSESTIMATES} holds. 

	\medskip

	\noindent \underline{\textbf{Small-data shock formation}.}
	We define the following data-dependent functions of $u|_{\Sigma_0} = 1 - r$
	(see Appendix \ref{A:CHRISTODOULOUSEQUATIONS} for definitions of
	$\upalpha, \upeta,$ etc.):
	\begin{align}
	& \mathcal{E}[(\mathring{\Phi}, \mathring{\Phi}_0)](u)
		\\
	& \ \ 
		:= \sum_{\Psi \in \lbrace \partial_t \Phi - k, \partial_1 \Phi, \partial_2 \Phi, \partial_3 \Phi \rbrace} 
		\int_{\Sigma_0^u}
					\left\lbrace
						\upalpha^{-2} \upmu
						(\upeta_0^{-1} + \upalpha^{-2} \upmu)
						(\Lunit \Psi)^2
						+ (\uLgood \Psi)^2
						+ (\upeta_0^{-1} + 2 \upalpha^{-2} \upmu) \upmu |\angD \Psi|^2
					\right\rbrace
			\, d \tvol,
			\notag
	\end{align}
	\begin{align} \label{E:CHRISTODOULOUDATAFUNDTION}
		\mathcal{S}[(\mathring{\Phi}, \mathring{\Phi}_0)](u)
		& :=
		\int_{S_{0,u}}
			r 
			\left\lbrace
				(\mathring{\Phi}_0 - k)
				- \upeta_0 \partial_r \mathring{\Phi}
			\right\rbrace
		 \, d \Eucspherevol
			+
			\int_{\Sigma_0^u}
			 \left\lbrace
				2 (\mathring{\Phi}_0 - k)
				- \upeta_0 \partial_r \mathring{\Phi}
			\right\rbrace
		 \, d^3 x,
	\end{align}
	where 
	$d \tvol$ is defined in \eqref{E:RESCALEDFORMS},
	$d \Eucspherevol$ denotes the Euclidean area form on the sphere $S_{0,u}$ of 
	Euclidean radius $r = 1 - u,$ 
	and $d^3 x$ denotes the standard flat volume form on $\mathbb{R}^3.$
	Assume that (see \eqref{E:CHRISTODOULOUSH} for the definition of $H$)
	\begin{align} \label{E:CHRISTNULLCONDITIONFAIL}
		\ell := \frac{d H}{d \upsigma}(\upsigma = k^2) > 0. 
	\end{align}
	There exist constants $C > 0$ and $C' > 0,$ 
	independent of $U \in (0,U_0],$
	such that if $\mathring{\upepsilon}$ is sufficiently small and
	if for some $U \in (0,U_0]$ we have
	\begin{align} \label{E:SHOCKFUNCTIONMUSTBESUFFICIENTLYLARGE}
		\mathcal{S}[(\mathring{\Phi}, \mathring{\Phi}_0)](U)
		& \leq - C \mathring{\upepsilon} \mathcal{E}^{1/2}[(\mathring{\Phi}, \mathring{\Phi}_0)](U) < 0,
	\end{align}
	then a shock forms in the solution\footnote{That is, $\Phi$ and its first rectangular derivatives remain bounded, while
	some second-order rectangular derivative blows up due to the vanishing of $\upmu.$} 
	$\Phi$ 
	and the first shock in the maximal development of the portion of the data
	in the exterior of $S_{0,U} \subset \Sigma_0^U$
	originates in the hypersurface region $\Sigma_{T_{(Lifespan)};U}^U$ 
	(see Definition \ref{D:HYPERSURFACESANDCONICALREGIONS}), where 
	\begin{align}
		T_{(Lifespan);U} < \exp\left( C' \frac{U}{\left|k^3 \ell \mathcal{S}[(\mathring{\Phi}, \mathring{\Phi}_0)](U)\right|} \right).
	\end{align}
	A similar result holds if $\ell < 0;$ in this case, we delete the 
	``$-$'' sign in \eqref{E:SHOCKFUNCTIONMUSTBESUFFICIENTLYLARGE} and change 
	``$\leq$'' and ``$<$'' to ``$\geq$'' and ``$>.$''

	\medskip

	\noindent \underline{\textbf{Description of the boundary of the maximal development}.}
	For shock-forming solutions,\footnote{Some
	of the results stated here depend on some non-degeneracy assumptions 
	on the solution that are expected to hold
	generically, such as $\frac{\partial^2}{\partial u^2}
	\upmu > 0$ at the shock points.} 
	the boundary $\mathcal{B}$ of the maximal development of the data
	in the exterior of $S_{0,U} \subset \Sigma_0^U$
	is a disjoint union $\mathcal{B} = (\partial_- \mathcal{H} \cup \mathcal{H}) \cup \underline{\mathcal{C}},$
	where $\partial_- \mathcal{H} \cup \mathcal{H}$ is the
	singular part (where $\upmu$ vanishes) and 
	$\underline{\mathcal{C}}$ is the regular part (where $\upmu$
	extends continuously to a positive value).
	The solution and its rectangular derivatives extend continuously 
	in rectangular coordinates to the regular part.
	Each component of
	$\partial_- \mathcal{H}$ is a smooth $2-$dimensional embedded submanifold of Minkowski spacetime,
	spacelike with respect to the dynamic metric\footnote{We follow the conventions of \cite{dC2007}
	and denote the dynamic metric by $h=h(\partial \Phi)$ in this section.}
	 $h$ (see \eqref{E:CHRISTODOULOUWAVEEQNEXPANDED}). The corresponding component of 
	$\mathcal{H}$ is a smooth, embedded, $3-$dimensional
	submanifold in Minkowski spacetime ruled by curves that are null relative to $h$
	and with past endpoints on $\partial_- \mathcal{H}.$ The corresponding component 
	$\underline{\mathcal{C}}$ is the incoming null hypersurface
	corresponding to $\partial_- \mathcal{H},$ and it is ruled by incoming $h-$null geodesics 
	with past endpoints on $\partial_- \mathcal{H}.$
\end{theorem}

\begin{center}
\begin{overpic}[scale=1]{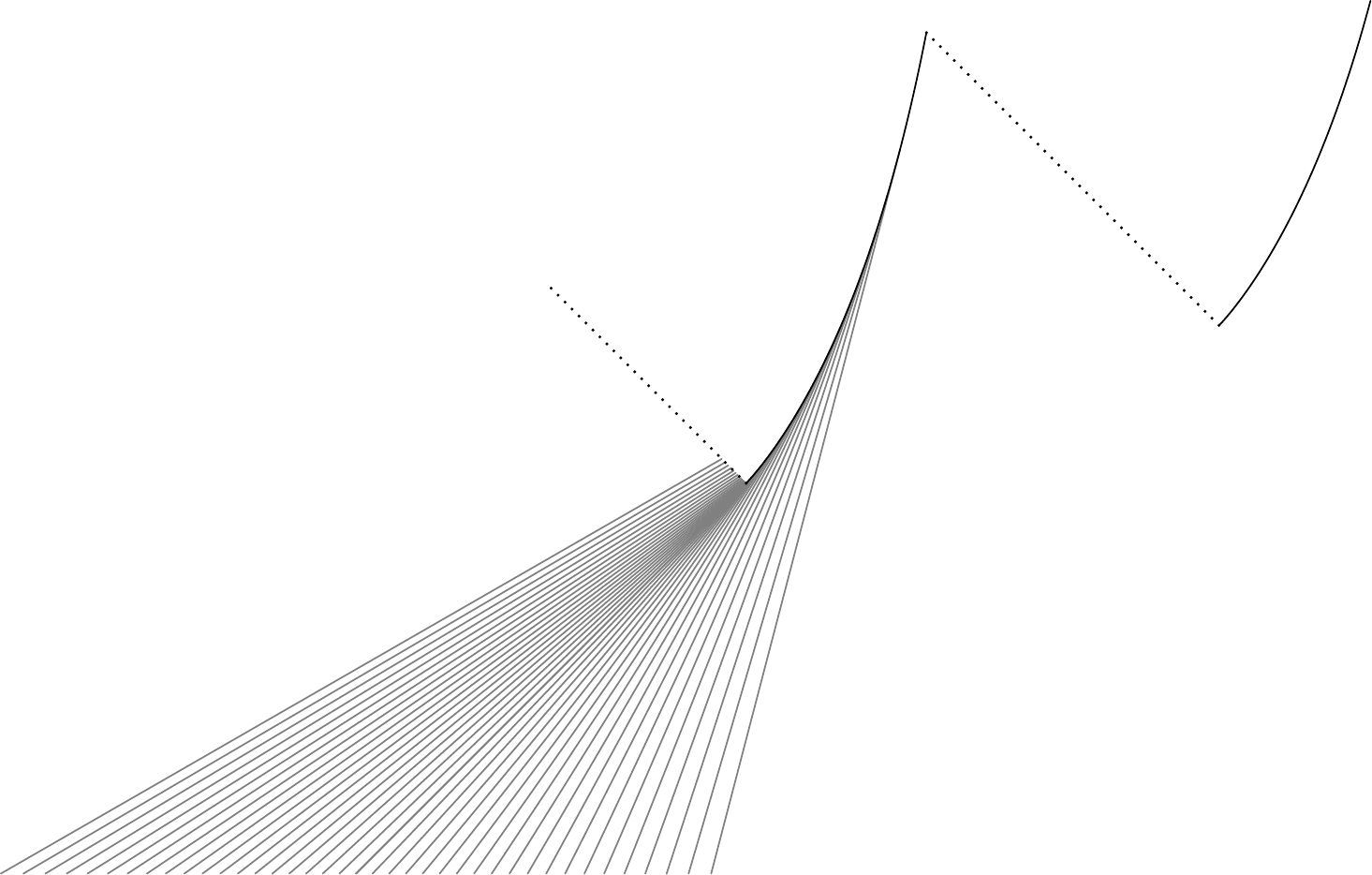}
        \put (59,44) {\large$\mathcal{H}$}
	\put (93.5,44) {\large$\mathcal{H}$}
	\put (79,44) {\large$\underline{\mathcal{C}}$}
\end{overpic}
\captionof{figure}{The geometry of the maximum development in
cross-sectional view. The gray lines
indicate the level sets of our eikonal function $u$, near the first
blow-up point. The dotted lines form the regular
boundary $\underline{\mathcal{C}}$. The black curves indicate the singular
boundary $\mathcal{H}$, whose lower endpoints are
$\partial_-\mathcal{H}$.} 
 \label{F:MAXDEVBOUND}
\end{center}

We make the following remarks concerning Christodoulou's theorem.

\begin{itemize}
	\item 
		Most aspects of
		Theorem~\ref{T:LONGTIMEPLUSESTIMATES}
		can be proved by using the strategy 
		outlined in the discussion of the proof of
		Theorem~\ref{T:CHRISTODOULOUSHOCKFORMATION}
		(see also Subsect.~\ref{SS:EXTENSIONSOFTHESHARPCLASSICALLIFESPANTHEOREMTOALINHACSEQUATION}).
	\item The full description of the boundary of the maximal
		development, especially in view of the goal of
		extending the solution past the shock front, involves
		discussions both relative to Minkowski spacetime and
		relative to the eikonal foliation corresponding to $u,$ 
		which degenerates along
		$\partial_-\mathcal{H}\cup\mathcal{H}.$
		We invite interested readers to
		consult \cite[Ch.15]{dC2007} and will not discuss these issues
		further except to note that the full description requires
		studying the solution at times $t$ beyond the time of first blow-up
		and studying the blow-up sets $\upmu \to 0$ along the $\Sigma_t,$ which
		have positive dimension. 
	\item The quantity \eqref{E:CHRISTNULLCONDITIONFAIL}
		is the exact analog of the future null condition failure factor $\FutFailFac$ from \eqref{E:INTROFAILUREFACTOR}.
		Note that unlike the general classes of equations considered in
		Theorems \ref{T:LONGTIMEPLUSESTIMATES}, \ref{T:ALINHACSHOCKFORMATION},
		and \ref{T:STABLESHOCKFORMATION},
		the quantity \eqref{E:CHRISTNULLCONDITIONFAIL} is not angularly dependent.
	\item As we make clear in Subsect.~\ref{SS:COMPARISON},
		Christodoulou's condition for shock formation, though compelling,
		is not sharp. On the other hand, Alinhac's condition for shock formation, 
		based on John's conjecture, is sharp in a sense that we make precise.
		For example, it is easy to see 
		that there exist spherically symmetric data 
		for which Christodoulou's quantity 
		\eqref{E:CHRISTODOULOUDATAFUNDTION} verifies\footnote{Simply take data
		with $\mathring{\Phi}_0 - k \geq 0$ and $\partial_r \mathring{\Phi} \equiv 0.$} 
		$\mathcal{S}[(\mathring{\Phi}, \mathring{\Phi}_0)](U) \geq 0$ 
		for all $U.$
		For such data, the shock formation condition \eqref{E:SHOCKFUNCTIONMUSTBESUFFICIENTLYLARGE}
		cannot be satisfied. However, Cor.~\ref{C:ge} can be extended to show that such data, when nontrivial,
		lead to finite-time shock formation, and Theorem 
		\ref{T:STABLESHOCKFORMATION} shows that this shock formation
		is in fact stable under general small perturbations.
		Hence, the condition \eqref{E:SHOCKFUNCTIONMUSTBESUFFICIENTLYLARGE}
		does not detect all shock forming data.
\end{itemize}

\subsection{Comparison of Alinhac's and Christodoulou's
frameworks}\label{SS:COMPARISON}
The frameworks of Alinhac and Christodoulou share some fundamental features,
including the following:
\begin{itemize}
	\item Shock formation is caused by the crossing of characteristics,
		as in Burgers' equations.
	\item Shock-forming solutions remain regular relative to adapted coordinates constructed out of a true eikonal function.
	\item Establishing good peeling properties plays an important role
		in the analysis.
\end{itemize}
However, they also differ in one significant way.
The main advantage of Christodoulou's framework 
is that it allows one to extend the solution beyond the hypersurface 
$\Sigma_{T_{(Lifespan)}}$ where the first singularity occurs.
In fact, his methods reveal a large portion of the
maximal development
of the data (see Remark \ref{R:MAXIMALDEVELOPMENTROUGHDEF}
and Figure \ref{F:MAXDEVBOUND}). 
The extension is made possible 
by the precise form of the dispersive estimates 
and the formulation of the well-posedness theorem 
(see Theorems \ref{T:LONGTIMEPLUSESTIMATES} and \ref{T:CHRISTODOULOUSHOCKFORMATION})
in terms of the sharp breakdown criterion $\upmu \to 0.$

In contrast, Alinhac's results are valid only up
the hypersurface $\Sigma_{T_{(Lifespan)}}$
where first singularity occurs, 
and only for data for which there is a unique first singularity point;
see his non-degeneracy assumptions on the data stated just below
\eqref{E:ALINHACDATABLOWUPFUNCTION}.
In particular, his results do not apply to the spherically symmetric data
that we treated in Subsubsect.~\ref{SSS:GEOMETRICFORMULATION}.
This should be further contrasted with another strength of
Christodoulou's framework, which is that it 
can be extended to show
the \emph{stability} (under general small perturbations) 
of John's spherically symmetric shock-formation result;
see Theorem~\ref{T:STABLESHOCKFORMATION} by the third author.
It is natural to wonder whether or not Alinhac's approach can
be easily modified to recover all of the detailed
features revealed by Christodoulou's framework.
Unfortunately, as we describe below, the answer seems to be ``no.''
In total, only Christodoulou's framework
is suitable for setting up the important problem
that we discussed in the Introduction: extending our understanding of $1D$
conservation laws to higher dimensions, including extending the
solution beyond the shock. 

We now highlight two merits of Alinhac's results.
First, his proofs are relatively short
and he was the first to show that indeed, failure of the null condition in equation
\eqref{E:ALINHACWAVE} leads, for a set of small data, to finite-time shock formation.
A second merit is that his condition on the data for shock-formation,
stated just below \eqref{E:ALINHACDATABLOWUPFUNCTION},
is explicitly connected to the limiting lifespan of the solution 
via equation \eqref{E:ALINHACLIMITINGLIFESPAN}.
That is, he proved a restricted version of John's conjecture,
limited only by his non-degeneracy assumptions on the data.
We also note that in \cite{fJ1989} (see also \cite{fJ1990}),
John made notable progress towards proving his conjecture
by showing that the second derivatives of
$\Phi$ start to grow near the limiting time. 
However, he never proved actual blow-up. 
This discussion suggests that the John-H{\"o}rmander lifespan lower bound
is essentially sharp and that if the John-H{\"o}rmander quantity \eqref{E:ALINHACDATABLOWUPFUNCTION} is
\emph{non-positive} in a region, then the solution should exist 
beyond the standard almost global existence time in a related spacetime region.
In Subsect.~\ref{SS:DISCUSSIONOFSHOCKFORMINGDATA}, we will in fact sketch a proof
of this statement.

We now describe a few aspects of Alinhac's proof and
explain the origin of its limitations.
His proof is relatively short, primarily 
because he was able to disregard
many of the intricate geometric structures 
present in Christodoulou's framework. 
As we have seen in 
Sect.~\ref{S:GENERALIZEDENERGY}, 
Christodoulou's framework leads to a complicated interplay between
derivative loss and $\upmu$-degeneration of the generalized energy estimates. 
Having disregarded these features,
Alinhac's approach led to linearized equations that
lose derivatives relative to the background.
More precisely, he set up an iteration scheme\footnote{The initial
guess is ``$\Phi_{\uplambda = 0}$'', which formally solves
a Burgers-type equation along each outgoing null geodesic.}
to construct the blow-up solution $\Phi_\uplambda$ together
with the smooth functions $v,$ $w,$ and $\zeta$, as well as the
coordinates of the first blow-up time (in particular, $\tau_{(Lifespan);\uplambda)}$) 
of Theorem~\ref{T:ALINHACSHOCKFORMATION}.
At each step in the iteration, 
his effective eikonal function corresponds to the current iterate of
$\zeta.$ Hence, the $\zeta$ iterate does not correspond to 
a true eikonal function of the nonlinear solution. 
For similar reasons, his adapted vectorfields
(which also vary from iterate to iterate) 
have small components that are transversal to the true characteristics,
which led to derivative loss in the estimates relative to the previous iterate. 
These derivative losses turns out to be sufficiently tame,\footnote{Interestingly, 
although Alinhac did not use the elliptic estimates
of Remark \ref{R:NEEDFORELLIPTIC} in his work, he \emph{did} need to use an analog
of the renormalized Raychaudhuri equation of Subsubsect.~\ref{SSS:RENORMALIZEDRAYCHADHOURI}
in his derivation of tame $L^2$ estimates for his linearized
equations.} and Alinhac was therefore able to handle them with a
Nash-Moser scheme.

Alinhac's iteration scheme, however, fundamentally depends on a
condition that he calls ``(H)'';
see \cite[pg. 15]{sA1999b}. Roughly speaking, condition (H)
demands that each iterate 
has a corresponding $\upmu$ that vanishes at exactly one
point on its constant-time hypersurface of first blow-up; this turns
out to be guaranteed when his non-degeneracy assumption on the data, stated 
immediately after \eqref{E:ALINHACDATABLOWUPFUNCTION}, hold. 
On the other hand, when the maximum of the John-H{\"o}rmander quantity
\eqref{E:ALINHACDATABLOWUPFUNCTION} is attained at multiple points, or
perhaps even along a submanifold, condition (H) fails 
for the zeroth iterate and the scheme cannot continue. 
It is for this reason that Alinhac's framework does not recover
the stability of spherically symmetric blow-up; compare with
Remark \ref{R:STABILITYSPHSYMBLOWUP}. Furthermore, the condition (H)
also poses a barrier to recovering the geometry of the maximal
development, as Christodoulou did in his Theorem
\ref{T:CHRISTODOULOUSHOCKFORMATION}: Christodoulou showed that to the future of the first
blow-up point, the subset of $\Sigma_t$ where $\upmu \to 0$
generically has dimension at least two and thus falls beyond the scope 
of Alinhac's iteration scheme.

\subsection{The shock-formation theorem of \cite{jS2014}}
We now state the small-data shock formation theorem 
from \cite{jS2014}
for solutions to the covariant wave equation $\square_{g(\Psi)} \Psi = 0$
in $3$ space dimensions. 
We also briefly discuss its proof.
As we have described above, the
theorem extends without any significant alterations
to equations of the form
$\square_{g(\Psi)} \Psi = \NN(\Psi)(\partial \Psi, \partial \Psi)$
whenever the semilinear term $\NN(\Psi)(\partial \Psi, \partial \Psi)$ verifies the 
future strong null condition of
Remark \ref{R:STRONGNULL} 
(or, if we are studying shock formation to the past, 
the past strong null condition of Remark \ref{R:ASYMMETRY});
see also Remark \ref{R:HARMLESSSEMILINEAR}.

\begin{theorem}\cite[\textbf{Theorem 22.3.1; Shock formation for nearly radial data}]{jS2014}
\label{T:STABLESHOCKFORMATION}
Let $(\check{\Psi} := \Psi|_{\Sigma_{-1/2}}, \check{\Psi}_0 := \partial_t \Psi|_{\Sigma_{-1/2}})$
be ``initial'' data (at time $-1/2$) for the covariant scalar wave equation 
\[
	\square_{g(\Psi)} \Psi = 0.
\]
Assume that Klainerman's classic null condition fails for the nonlinearities,
that is, that the future null condition failure factor $\FutFailFac$ from Definition 
\ref{D:INTROFAILUREFACTOR}
does not completely vanish.
Assume that the data are nontrivial,
spherically symmetric,\footnote{Note that we are not assuming that the equation itself is invariant under Euclidean rotations.
Hence, spherically symmetric data do not generally launch spherically symmetric solutions.}
supported in the Euclidean ball of radius $1/2$ centered at the origin
and that $(\check{\Psi}, \check{\Psi}_0) \in H^{25} \times H^{24}.$
Then (perhaps shrinking the amplitude of the data if necessary), 
a shock-formation result in analogy with Cor.~\ref{C:ge}
holds for the corresponding solution.
Furthermore, for each shock-forming spherically symmetric (small) data pair,
the shock-formation processes are \textbf{stable under general small perturbations}
(without symmetry assumptions)
of the data belonging to $H^{25} \times H^{24}$
and the Euclidean ball of radius $1/2.$ 

Furthermore, all of the conclusions of Theorem~\ref{T:LONGTIMEPLUSESTIMATES} hold
for the solution. In particular, its lifespan is finite precisely because 
$\upmu$ vanishes at one or more points
and at such points, some rectangular derivative 
$\partial_{\nu} \Psi$ blows up. 

\end{theorem}

\begin{remark}[\textbf{The stability of spherically symmetric blow up}]
	\label{R:STABILITYSPHSYMBLOWUP}
	An immediate corollary is that F. John's blow-up result in
	spherical symmetry 
	(see Subsubsect.~\ref{SSS:GEOMETRICFORMULATION})
	is stable under small arbitrary
	perturbations. It turns out, however, 
	that for technical reasons, 
	it is easier to prove that shock formation
	occurs for spherically symmetric initial data,
	\emph{even for equations that are not invariant under the Euclidean rotations.}
	Theorem~\ref{T:STABLESHOCKFORMATION} asserts that these shock
	formation processes are also stable under general small perturbations.
\end{remark}

\begin{proof}[\textbf{Discussion of the proof}]
Thanks to the difficult estimates of
Theorem~\ref{T:LONGTIMEPLUSESTIMATES},
Theorem~\ref{T:STABLESHOCKFORMATION}
can be proved without much difficulty.
We need only to show that $\upmu$ vanishes 
in finite time.
In fact, 
for the nearly spherically data under consideration, 
Theorem~\ref{T:STABLESHOCKFORMATION} 
can be proved by using arguments very similar
to the ones we used 
in proving Cor.~\ref{C:ge},
given in spherical symmetry.
Although there are additional terms present away from
spherical symmetry, the low-order Heuristic Principle
estimates of Theorem~\ref{T:LONGTIMEPLUSESTIMATES}
can be used to show that they decay sufficiently fast
and do not affect the shock formation processes
in a substantial manner.
See Subsect.~\ref{SS:ROLEOFINVERSEFOLIATIONDENSITY}
for some additional details.

\end{proof}

\subsection{Additional connections between the results}
\label{SS:DISCUSSIONOFSHOCKFORMINGDATA}

We now discuss some additional connections between the
shock-formation results of Christodoulou, 
those of Alinhac, 
and those of \cite{jS2014}.
Throughout this subsection $\mathring{\upepsilon}$ denotes the small size of the data.

\subsubsection{Only one term can drive $\upmu$ to $0$}
The sufficient conditions on the initial data
from Theorems 
\ref{T:ALINHACSHOCKFORMATION}, 
\ref{T:CHRISTODOULOUSHOCKFORMATION},
and \ref{T:STABLESHOCKFORMATION} 
that lead to finite-time shock formation
are not obviously related.
However, as we have noted in the previous subsections as well as our
discussion of Theorem~\ref{T:LONGTIMEPLUSESTIMATES},
shock formation is essentially driven by one term and one term only,
at least in the context the three theorems mentioned above. 
In the case of Theorem~\ref{T:STABLESHOCKFORMATION}, 
to analyze the behavior of $\upmu,$
one uses the following estimate for solutions to the equations 
$\square_{g(\Psi)} \Psi = 0$ 
(see \eqref{E:LUPMUKEYUPPERBOUND}):
\begin{align} \label{E:AGAINLUPMUKEYUPPERBOUND}
        \Lunit \upmu(t,u,\vartheta) 
        & = \frac{1}{2} \InitialFutFailFac(\vartheta) \Rad \Psi(t,u,\vartheta) 
        + \cdots,
        && t \geq \mathring{\upepsilon}^{-1}.
\end{align}
In the case of the equations treated in Alinhac's
Theorem~\ref{T:ALINHACSHOCKFORMATION}
or in Christodoulou's Theorem~\ref{T:CHRISTODOULOUSHOCKFORMATION},
one uses equation \eqref{E:NEWUPMUSCHEMATICTRANSPORT}.
Thus, to guarantee shock formation, one must
carry out the following two steps.
\begin{enumerate}
\item Show that the term $\frac{1}{2}
\InitialFutFailFac(\vartheta) \Rad \Psi(t,u,\vartheta)$ 
from \eqref{E:AGAINLUPMUKEYUPPERBOUND}
(or its analog in the case of the other equations) becomes
negative with a sufficiently strong lower bound on its absolute value.
\item Derive upper bounds for the remaining terms showing that they are dominated by the negative term.
\end{enumerate}
The various conditions on the initial data stated in the three
shock-formation theorems are all included for these two purposes. 

In the case of Theorem~\ref{T:STABLESHOCKFORMATION},
which applies to nearly spherically symmetric data,
we explained the claim made in the previous sentence
in Subsect.~\ref{SS:ROLEOFINVERSEFOLIATIONDENSITY}.
In the case of Christodoulou's Theorem
\ref{T:CHRISTODOULOUSHOCKFORMATION}, his arguments are explained on 
\cite[pgs. 893-903]{dC2007}. In 
Subsubsect.~\ref{SSS:EXTENDINGCHSHOCKCONDITION} we provide additional details 
on Christodoulou's arguments and explain how his conditions on the data can be modified to apply to some other equations not
studied in his monograph. In Subsubsect.~\ref{SSS:SHARPNESSOFREFINED},
we flesh out the connection between Alinhac's condition on the data for shock formation
and the two steps described above. Furthermore, we
show how to use Christodoulou's framework to relax 
Alinhac's non-degeneracy assumptions on the data, 
thus yielding a full resolution of John's conjecture;
see Subsubsect.~\ref{SSS:CLASSICNULL}.
We finish in Subsubsect.~\ref{SSS:UNIFIEDPERSPECTIVE}
by describing the various shock formation results from a unified perspective. 

\subsubsection{Extending Christodoulou's shock-formation condition to other equations}
\label{SSS:EXTENDINGCHSHOCKCONDITION}
Under some structural assumptions on the nonlinearities, it is possible to modify 
Christodoulou's condition \eqref{E:SHOCKFUNCTIONMUSTBESUFFICIENTLYLARGE}
so that it applies to 
the scalar equations $\square_{g(\Psi)} \Psi = 0$
from Theorems \ref{T:LONGTIMEPLUSESTIMATES} and \ref{T:STABLESHOCKFORMATION}.
Such a condition provides a set of shock-generating data
that differs from the nearly spherically symmetric data of
Theorem~\ref{T:STABLESHOCKFORMATION}.
Specifically, his condition can be modified without
difficulty to apply whenever the future null condition failure factor\footnote{Recall that relative to standard spherical coordinates 
$(t,r,\theta)$ on Minkowski spacetime, we have $\FutFailFac = \FutFailFac(\theta).$} 
$\FutFailFac(\theta)$ from \eqref{E:INTROFAILUREFACTOR}
takes on a strictly positive or negative sign for 
$\theta \in \mathbb{S}^2.$
The reason is that Christodoulou's analysis 
is based on averaging over the spheres $S_{t,u},$
and his condition guarantees that
the analog of $\pm \Rad \Psi$ eventually verifies a lower bound
of the form $\gtrsim \mathring{\upepsilon} (1 + t)^{-1}$
(as in \eqref{E:KEYLOWER}) 
along \emph{some unknown}
integral curve of $\Lunit.$ Hence, when $\FutFailFac$ has a definite sign,
the analog of Christodoulou's condition,
with the correct sign,
ensures that the product
$\frac{1}{2} \InitialFutFailFac(\vartheta) \Rad \Psi(t,u,\vartheta) $
from equation \eqref{E:AGAINLUPMUKEYUPPERBOUND}
becomes sufficiently negative along the unknown
integral curve of $\Lunit;$ this is sufficient
to guarantee that $\upmu$ vanishes in finite time.
Clearly, because nothing is known about the integral curve,
the definite sign\footnote{Recall that at $t=0,$ the Euclidean angular coordinate $\theta$ is equal to the geometric angular coordinate
$\vartheta,$ and hence $\InitialFutFailFac(\vartheta) = \FutFailFac(\theta = \vartheta)$
(see \eqref{E:DATAFAILREFACTOR} and \eqref{E:OTHERDATAFAILUREFACTOR}).} of $\FutFailFac(\theta)$ 
for all $\theta \in \mathbb{S}^2$
plays an essential role in this argument.
Similarly, thanks to the observations of Subsect.\ 
\ref{SS:EXTENSIONSOFTHESHARPCLASSICALLIFESPANTHEOREMTOALINHACSEQUATION},
Christodoulou's condition \eqref{E:SHOCKFUNCTIONMUSTBESUFFICIENTLYLARGE} can be modified without
difficulty to apply to the non-covariant equation
$(g^{-1})^{\alpha \beta}(\partial \Phi) \partial_{\alpha} \partial_{\beta} \Phi = 0$
whenever the future null condition failure factor
$\FutFailFac = \FutFailFac(\theta)$ from \eqref{E:OTHERFAILUREFACTOR}
takes on a strictly positive or negative sign for 
$\theta \in \mathbb{S}^2.$

\subsubsection{Eliminating Alinhac's non-degeneracy assumptions on the data}
\label{SSS:SHARPNESSOFREFINED}
With the more precise estimates from Christodoulou's framework,
we can eliminate the non-degeneracy conditions
on the data that Alinhac used to prove shock formation
(see just below equation \eqref{E:ALINHACDATABLOWUPFUNCTION}).
Here, we sketch a proof that small-data finite-time shock formation occurs 
in solutions to equation \eqref{E:ALINHACWAVE} if we sufficiently shrink
the amplitude of the data and if John's condition holds:
\begin{align}  \label{E:ALINHACRELAXED}
        \mbox{the John-H{\"o}rmander quantity \eqref{E:ALINHACDATABLOWUPFUNCTION}
        is positive at one point \ } 
        (q_*,\theta_*).
\end{align}
This shows that the lifespan
lower-bound of Theorem~\ref{T:JOHNHORMANDERLIFESPANLOWER} is sharp in the small-data limit.
An analogous sharp condition can also be stated in the cases 
of Christodoulou's equations \eqref{E:CHRISTODOULOUWAVE}
and the equations $\square_{g(\Psi)} \Psi = 0.$
Furthermore, we recall that by Prop.~\ref{P:JOHNSCRITERIONISALWAYSSATISFIEDFORCOMPACTLYSUPPORTEDDATA},
\emph{the condition \eqref{E:ALINHACRELAXED} always holds for nontrivial compactly supported data}.
We begin our sketch by first studying Alinhac's equations using 
Christodoulou's framework and showing how the condition 
\eqref{E:ALINHACRELAXED} can be exploited. 
For convenience, we assume here that the data
for Alinhac's equations 
are supported in the Euclidean unit ball $\Sigma_0^1,$
and we study the solution only in regions of the form
$\MM_{t,U_0}$ 
(see \eqref{E:MTUDEF}),
where $0 < U_0 < 1$ is a fixed constant.

We first explain how the behavior of the term 
$-\frac{1}{2} \InitialFutFailFac(\vartheta) [\Radunit^a \Rad \Psi_a](t,u,\vartheta)$
on the right-hand side of the relevant evolution equation \eqref{E:NEWUPMUSCHEMATICTRANSPORT} for $\upmu$
is connected to the John-H{\"o}rmander quantity
\begin{align} \label{E:ALINHACDATAFUNCT}
	\frac{1}{2} 
	\FutFailFac
	\frac{\partial^2}{\partial q^2} \Fried[(\mathring{\Phi}, \mathring{\Phi}_0)],
\end{align}
the (data-dependent) function 
appearing in Theorems \ref{T:JOHNHORMANDERLIFESPANLOWER} and 
\ref{T:ALINHACSHOCKFORMATION}.
The first important observation is that at time $\mathring{\upepsilon}^{-1},$ 
long before any singularity can form,
we have the following estimate (whose proof we will sketch below)
relative to standard spherical coordinates $(t,r,\theta)$ on Minkowski spacetime,
valid in the constant-time hypersurface subset $\Sigma_{\mathring{\upepsilon}^{-1}}^{U_0}:$
\begin{align} \label{E:KEYCOMPARISONTOFRIEDLANDERSRADIATIONFIELD}
	\left|
		- \frac{1}{2} 
			\rgeo
			\FutFailFac
			\Radunit^a \Rad \Psi_a
			(t = \frac{1}{\mathring{\upepsilon}},r,\theta)
		+ 
			\frac{1}{2} 
			r
			\FutFailFac
			\frac{\partial^2}{\partial q^2} \Fried[(\mathring{\Phi}, \mathring{\Phi}_0)]
			(q=r - \frac{1}{\mathring{\upepsilon}} ,r,\theta)
	\right|
	& \leq C \mathring{\upepsilon}^2 \ln\left( \frac{1}{\mathring{\upepsilon}} \right).
\end{align}
Hence, 
switching to geometric coordinates $(t,u,\vartheta),$
using \eqref{E:KEYCOMPARISONTOFRIEDLANDERSRADIATIONFIELD},
the estimate $\FutFailFac(t,u,\vartheta) \approx \InitialFutFailFac(\vartheta)$
mentioned in Subsect.~\ref{SS:ROLEOFINVERSEFOLIATIONDENSITY},
and assuming that $\mathring{\upepsilon}$ is sufficiently small, 
we see that the assumption \eqref{E:ALINHACRELAXED} implies that
there exists a point 
$p$ belonging to a region\footnote{It could happen that $p$ does not belong to a subset
$\Sigma_{\mathring{\upepsilon}^{-1}}^{U_0}$ with $0 < U_0 < 1.$
In this case, we would have to rework some of our constructions 
in order to allow us to study regions with $U_0 > 1.$
Alternatively, we could start with data given at time $-1/2$
and supported in the Euclidean ball of radius $1/2$ centered
at the origin, as in Theorem~\ref{T:STABLESHOCKFORMATION}.
} $\Sigma_{\mathring{\upepsilon}^{-1}}^{U_0}$
such that at $p,$ the term 
$-\frac{1}{2} \InitialFutFailFac(\vartheta) [\Radunit^a \Rad \Psi_a](t,u,\vartheta)$
from equation \eqref{E:NEWUPMUSCHEMATICTRANSPORT}
is dominant, negative, and of order 
$c \mathring{\upepsilon} \rgeo^{-1}(t = \mathring{\upepsilon}^{-1},u) \approx c \mathring{\upepsilon} (1 + t)^{-1}.$ 
Hence, at time $t=\mathring{\upepsilon}^{-1},$ this term causes $\upmu$
to begin decaying
along the integral curve of $\Lunit$
emanating from $p,$
at the rate $- c \mathring{\upepsilon} \ln(1 + t).$
Furthermore, since 
the Heuristic Principle estimates
(see \eqref{E:TANGENTIALFASTDECAY})
imply that the geometric angular derivatives of the $\Psi_{\nu}$ 
have significantly died off by time
$\mathring{\upepsilon}^{-1},$
we can use ideas similar to the ones used to prove \eqref{E:KEYLOWERBOUNDFORRADPSI}
to deduce that the product
$- \frac{1}{2} \InitialFutFailFac \rgeo \Radunit^a \Rad \Psi_a$
(note the factor of $\rgeo$)
is approximately constant along the integral curves of $\Lunit$ for times beyond $\mathring{\upepsilon}^{-1}.$
In particular, along the integral curve emanating from $p,$ the product
$- \frac{1}{2} \InitialFutFailFac \rgeo \Radunit^a \Rad \Psi_a$
remains order $- c \mathring{\upepsilon}$ for times beyond
$\mathring{\upepsilon}^{-1}.$ Hence, by
equation \eqref{E:NEWUPMUSCHEMATICTRANSPORT},
we see that along that integral curve
(which corresponds to fixed $u$ and $\vartheta$), we have\footnote{
A careful proof of \eqref{E:ALINHACLMUNEGATIVE} would involve possibly
shrinking the amplitude of the data by $\mbox{data} \rightarrow \uplambda \cdot \mbox{data}$
(for $\uplambda$ sufficiently small)
to ensure that the term
$-\frac{1}{2} \InitialFutFailFac(\vartheta) [\Radunit^a \Rad \Psi_a](t,u,\vartheta)$
from equation \eqref{E:NEWUPMUSCHEMATICTRANSPORT}
dominates all of the other terms.
}
\begin{align} \label{E:ALINHACLMUNEGATIVE}
	\Lunit \upmu(t,u,\vartheta)
	\approx - c \mathring{\upepsilon} (1 + t)^{-1}.
\end{align}
Since 
$\Lunit = \frac{\partial}{\partial t},$
it easily follows from \eqref{E:ALINHACLMUNEGATIVE} that $\upmu$ must vanish in finite time
and a shock forms.
We stress that in contrast to our proof of \eqref{E:KEYLOWERBOUNDFORRADPSI},
we did not assume here that the angular derivatives of the data
are even smaller than the small radial derivatives. 
Previously, we had made this assumption 
(see Remark \ref{R:ANGULARDERIVATIVESEVENSMALLER})
so that we could treat the linear term
$\rgeo \upmu \angLap \Psi$ on the right-hand side of
\eqref{E:ALTWAVEOPERATORDECOMPOSED} as negligible starting from $t=0;$
in general, we have to wait for the angular derivatives of the solution 
to die off before this term becomes negligible and in this sense, inequality
\eqref{E:KEYCOMPARISONTOFRIEDLANDERSRADIATIONFIELD} is important
because it accounts for the nontrivial influence of the angular derivatives of the \emph{data}
on the product $\rgeo \Radunit^a \Rad \Psi_a$ at time $\mathring{\upepsilon}^{-1}.$

We now provide arguments that lead to a
sketch of a proof of \eqref{E:KEYCOMPARISONTOFRIEDLANDERSRADIATIONFIELD}
and more. We begin by considering data $(\mathring{\Phi},\mathring{\Phi}_0)$
for Alinhac's wave equation \eqref{E:ALINHACWAVE}, 
but we now solve the Cauchy problem for the \emph{linear} wave equation with that data:
\begin{align}
	\square_m \Phi_{(Linear)} & = 0,
		\label{E:LINEARWAVE} \\
	\Phi_{(Linear)}|_{t=0}
	& = \mathring{\Phi},
	\qquad
	\partial_t \Phi_{(Linear)}|_{t=0}
	= \mathring{\Phi}_0.
\end{align}
We now recall that the function $\Fried[(\mathring{\Phi}, \mathring{\Phi}_0)]$
from \eqref{E:FRIEDNALNDERRADIATIONFIELD} is Friedlander's radiation field
for the linear wave equation \eqref{E:LINEARWAVE}. 
That is, the $r-$weighted solution
$r \Phi_{(Linear)}$ to \eqref{E:LINEARWAVE} is, 
relative to standard
spherical coordinates $(t,r,\theta)$ on Minkowski spacetime, 
asymptotic to
$\Fried[(\mathring{\Phi}, \mathring{\Phi}_0)] (q=r-t,r,\theta).$
Related statements hold for various derivatives of $\Phi_{(Linear)}.$ In particular, with
\begin{align}
	\Sigma_t' := \Sigma_t \cap \left\lbrace r > \frac{t}{2} > 1 \right\rbrace,
\end{align}
we have the following standard estimate (see, for example, \cite{lH1997}):
\begin{align} \label{E:RADIATIONFIELDESTIMATE}
	\left|
		r \partial_r^2 \Phi_{(Linear)}(t,r,\theta)
		-
		\frac{\partial^2}{\partial q^2}|_{t,\theta} \Fried[(\mathring{\Phi}, \mathring{\Phi}_0)]
		(q=r-t,r,\theta)
	\right|
	& \leq C \frac{\mathring{\upepsilon}}{1 + t},
		\qquad 
		\mbox{along \ } \Sigma_t',
\end{align}
where $\mathring{\upepsilon}$ is the (small) size of $(\mathring{\Phi},\mathring{\Phi}_0).$

To deduce \eqref{E:KEYCOMPARISONTOFRIEDLANDERSRADIATIONFIELD}, 
we must connect the estimate \eqref{E:RADIATIONFIELDESTIMATE}
back to the nonlinear problem \eqref{E:ALINHACWAVE}.
To this end, we solve both the nonlinear wave equation \eqref{E:ALINHACWAVE} and the linear
wave equation \eqref{E:LINEARWAVE}
with the same data $(\mathring{\Phi},\mathring{\Phi}_0).$
The difference $\Phi - \Phi_{(Linear)}$ solves the inhomogeneous linear wave equation
with trivial data
and with a source equal to the quadratic term
$\left\lbrace (m^{-1})^{\alpha \beta} - (g^{-1})^{\alpha \beta}(\partial \Phi) \right\rbrace
\partial_{\alpha} \partial_{\beta} \Phi.$
It therefore follows from the standard Minkowskian vectorfield method, as developed in \cite{sK1985}, 
that $\left\| r \Phi_{(Linear)} \right \|_{C^0(\Sigma_{\mathring{\upepsilon}^{-1}}')} 
\leq C \mathring{\upepsilon},$
$\left\| r \Phi \right \|_{C^0(\Sigma_{\mathring{\upepsilon}^{-1}}')}
\leq C \mathring{\upepsilon},$
and $\left\| r (\Phi - \Phi_{(Linear)}) \right \|_{C^0(\Sigma_{\mathring{\upepsilon}^{-1}}')}
\leq C \mathring{\upepsilon}^2 \ln \mathring{\upepsilon}^{-1}.$
Furthermore, when the data are
sufficiently regular, similar estimates hold for 
a limited number of higher $(t,r,\theta)$ coordinate derivatives of $\Phi$ and $\Phi_{(Linear)}.$
In addition, it is not difficult to show using \eqref{E:NEWUPMUSCHEMATICTRANSPORT}
that along $\Sigma_{\mathring{\upepsilon}^{-1}}^{U_0},$ 
$|\upmu - 1|$ is no larger than $C \mathring{\upepsilon} \ln \mathring{\upepsilon}^{-1}.$
One can also show that along $\Sigma_{\mathring{\upepsilon}^{-1}}^{U_0},$ 
$\Radunit^a$ is equal to $- x^a/r$ plus an error term that is 
no larger than $C \mathring{\upepsilon}^2 \ln \mathring{\upepsilon}^{-1},$
and similarly, $\rgeo - r$ is no larger than 
$C \mathring{\upepsilon} \ln \mathring{\upepsilon}^{-1}.$
It follows that along $\Sigma_{\mathring{\upepsilon}^{-1}}^{U_0},$
the nonlinear solution
$r \partial_r^2 \Phi$ is equal to
$\rgeo \Radunit^a \Rad \Psi_a$
(recall that $\Psi_a = \partial_a \Phi$)
up to an error term of size
$\leq C \mathring{\upepsilon}^2 \ln \mathring{\upepsilon}^{-1}.$
In total, we have the following estimate:
\begin{align}	 \label{E:NONLINEARCLOSETOLINEAR}
	\left\|
		r \partial_r^2 \Phi_{(Linear)}
		- \rgeo \Radunit^a \Rad \Psi_a
	\right\|_{C^0(\Sigma_{\mathring{\upepsilon}^{-1}}^{U_0})}
	& \leq C \mathring{\upepsilon}^2 \ln\left( \frac{1}{\mathring{\upepsilon}} \right).
\end{align}
Combining \eqref{E:RADIATIONFIELDESTIMATE}
and \eqref{E:NONLINEARCLOSETOLINEAR},
we arrive at \eqref{E:KEYCOMPARISONTOFRIEDLANDERSRADIATIONFIELD}.

The above discussion suggests that it should be possible to show that 
in the relevant region,
John's condition \eqref{E:ALINHACRELAXED} is automatically implied 
by the shock formation criteria 
of Theorem~\ref{T:STABLESHOCKFORMATION}
or Christodoulou's criteria; we do not
investigate this possibility here. It would be interesting
to know whether or not all (nontrivial) compactly supported data 
that are small enough for Theorem~\ref{T:LONGTIMEPLUSESTIMATES} to apply
must necessarily lead to shock formation. 
The proof outlined above 
is limited in the sense that the argument requires one to perhaps shrink the 
amplitude of the data in order to 
deduce shock formation.
A hint that such a result might hold true,
at least for some nonlinearities, 
lies in John's results \cite{fJ1981}:
for the class of equations that he addressed,
finite-time breakdown of an unknown nature occurs for all such data,
even without the smallness assumption.

The above discussion can also easily be extended to prove the following interesting consequence:
if the John-H{\"o}rmander quantity
$\frac{1}{2} 
		\FutFailFac(\theta) 
		\frac{\partial^2}{\partial q^2} \Fried[(\mathring{\Phi}, \mathring{\Phi}_0)]
		(q,\theta)$
appearing in \eqref{E:ALINHACDATABLOWUPFUNCTION}
is negative on a Minkowskian annular region 
$(q,\theta) \in [q_1, q_2 := 0] \times \mathbb{S}^2 \subset \Sigma_0^{U_0},$
and if $\mathring{\upepsilon}$ is sufficiently small,
then the corresponding solution to equation 
\eqref{E:ALINHACWAVE} exists 
beyond the standard lifespan lower bound
$\exp
	\left(
	\frac{1}
	{C \mathring{\upepsilon}}
	\right),
$
in a spacetime region bounded by an inner $g-$null cone
$\mathcal{C}_{u_1}$ and the outer $g-$null cone 
$\mathcal{C}_0,$ where $u_1 \approx 1 - q_1.$ 
The reason is simple: 
under the assumptions, up to small errors, 
the term 
$
- \frac{1}{2} \InitialFutFailFac(\vartheta) [\Radunit^a \Rad \Psi_a](t,u,\vartheta)
$
from the evolution equation \eqref{E:NEWUPMUSCHEMATICTRANSPORT} for $\upmu$
is positive and thus works \emph{against} shock formation.
More precisely, an argument
similar to the one outlined above, 
based on the estimate
\eqref{E:KEYCOMPARISONTOFRIEDLANDERSRADIATIONFIELD}
and equation
\eqref{E:NEWUPMUSCHEMATICTRANSPORT},
leads to the conclusion that 
$\upmu \geq 1 - C \mathring{\upepsilon}^2 \ln\left( \frac{1}{\mathring{\upepsilon}} \right) \ln(e + t),$
where the factor $C \mathring{\upepsilon}^2 \ln\left( \frac{1}{\mathring{\upepsilon}} \right)$
is from the right-hand side of \eqref{E:KEYCOMPARISONTOFRIEDLANDERSRADIATIONFIELD}.
Hence, since $\upmu$ cannot vanish in the region before the time
\begin{align} \label{E:LONGLIFE}
	\sim
	\exp
	\left(
		\frac{1}
		{C \mathring{\upepsilon}^2 \ln\left( \frac{1}{\mathring{\upepsilon}} \right)}
	\right),
\end{align}
we see from the analog of Theorem
\ref{T:LONGTIMEPLUSESTIMATES} for equation \eqref{E:ALINHACWAVE}
(as outlined in Subsect.~\ref{SS:EXTENSIONSOFTHESHARPCLASSICALLIFESPANTHEOREMTOALINHACSEQUATION})
that blow-up cannot occur in the region of interest before the time \eqref{E:LONGLIFE}.
Similar results hold for Christodoulou's equations \eqref{E:CHRISTODOULOUWAVE}
and for the equations $\square_{g(\Psi)} \Psi = 0.$

\subsubsection{A unified perspective on shock-formation involving three phases}
\label{SSS:UNIFIEDPERSPECTIVE}
A convenient way to summarize the above results for shock-forming data, 
combining the methods of Alinhac \cite{sA2001b} and Christodoulou \cite{dC2007}
and incorporating the perspective of John \cite{fJ1989},
is to divide the shock-formation evolution 
(for sufficiently small data)
into the following three phases.
Our use of the terminology ``phases'' 
is motivated by John's work \cite{fJ1989},
in which he was able to follow the solution
nearly to the singularity, long enough to see some growth in the
higher derivatives of the solution (John's ``third phase''), but not long enough to see
the actual singularity form.
\begin{description}
	\item[Phase (i)] On the time interval $[0,\mathring{\upepsilon}^{-1}],$
	the linear evolution dominates, the $\mathcal{C}_u-$tangential
	derivatives die off, and the important transversal derivative
	term behaves according to Friedlander's radiation field, as in
	\eqref{E:KEYCOMPARISONTOFRIEDLANDERSRADIATIONFIELD}.
	\item[Phase (ii)] This is the period after time
$\mathring{\upepsilon}^{-1}$ where $\upmu$ remains bounded away from
$0$ by a fixed amount.  Starting at around time
	$\mathring{\upepsilon}^{-1},$ 
		if the $\mathcal{C}_u-$transversal derivative term has, at least
		at some points, the ``right'' sign and 
		is large enough in magnitude,
		then $\upmu$ begins to decay along the corresponding
		integral curves of $\Lunit.$ Once the decay starts, 
		it does not stop. As long as $\upmu$ stays some fixed
		distance away from $0,$ the geometric $L^2$ estimates, such as those of
		Prop.~\ref{P:APRIORIENERGYESTIMATES}, are essentially equivalent
		to standard $L^2$ estimates that could be derived via the vectorfield commutator and multiplier 
		methods applied with Minkowski conformal Killing fields,
		as outlined in Subsect.~\ref{sect-Compr-dispers}. 
		In particular, we do not need the precision
		of a true eikonal function in order to understand the behavior of
		the solution in this phase.
	\item[Phase (iii)] $\upmu$ is now very close to $0$ at some
		spacetime point. We need the full precision of the eikonal
		function to follow the dynamics all the way to shock formation.
		To close the geometric $L^2$ estimates without derivative loss, 
		we need to use Christodoulou's strategy as described in 
		Subsubsect.~\ref{SSS:AVOIDINGTOPORDERDERIVATIVELOSS}.  
		For the shock-generating data that verify Alinhac's assumptions
		(which in particular ensure that at the time of first breakdown, there is only one shock point),
		if we are willing to allow the linearized equations to lose derivatives relative to 
		the previous iterate, 
		then we can also close the $L^2$ estimates using his Nash-Moser scheme.
		However, Alinhac's methods are not designed to reveal the complete structure of 
		the maximal development of the data, including the boundary.
	\end{description}

\section*{Acknowledgments}
We would like to thank the American Institute of Mathematics for funding three SQuaRE workshops
on the formation of shocks, which greatly aided the development of many of the ideas
presented in this article.
We would also like to thank Jonathan Luk and Shiwu Yang for participating in the workshops and for their helpful 
contributions, as well as Hans Lindblad for sharing his insight on
Alinhac's work. 
GH is grateful for the support offered by a grant of the European Research Council.
SK is grateful for the support offered by NSF grant \# DMS-1362872.
JS is grateful for the support offered by NSF grant \# DMS-1162211 
and by a Solomon Buchsbaum grant administered by the
Massachusetts Institute of Technology. 
WW is grateful for the support offered by the Swiss National Science Foundation 
through a grant to Joachim Krieger.

\appendix

\section{Some details on the wave equations studied in \cite{dC2007}}
\label{A:CHRISTODOULOUSEQUATIONS}

In \cite{dC2007}, Christodoulou considered Lagrangians of the form
$\mathcal{L}(\upsigma),$ where
\begin{align} \label{E:ENTHALPHYSQUARED}
	\upsigma := - (m^{-1})^{\alpha \beta} \partial_{\alpha} \Phi \partial_{\beta} \Phi,
\end{align}
and as usual, $(m^{-1})^{\alpha \beta} = \mbox{diag}(-1,1,1,1)$ is the standard 
inverse Minkowski metric expressed relative to rectangular coordinates.
In particular, in order for the corresponding Euler-Lagrange (wave) equation to have 
a fluid interpretation, Christodoulou considered Lagrangians
$\mathcal{L}(\upsigma)$ in a regime
where the following five positivity assumptions hold:
\begin{align}
	\upsigma,
		\,
	\mathscr{L}(\upsigma), 
		\,
	\frac{d \mathscr{L}}{d \upsigma},
		\,
	\frac{d}{d \upsigma}\left(\mathscr{L}/ \sqrt{\upsigma} \right),
		\,
	\frac{d^2\mathscr{L}}{d \upsigma^2} > 0. \label{E:FLUIDSINTERPRETATIONPOSITIVITY}
\end{align}
The assumptions \eqref{E:FLUIDSINTERPRETATIONPOSITIVITY}
imply that $\Phi$ can be interpreted as
a potential function for a physically reasonable irrotational
relativistic fluid
with desirable properties such as having a characteristic speed (of sound)
strictly in between $0$ and $1$ (speed of light).
The corresponding Euler-Lagrange equation, which is the main equation that he studies, 
is
\begin{align} \label{E:DCFLUIDEL}
		\partial_{\alpha} \left( \frac{\partial \mathcal{L}}{\partial (\partial_{\alpha} \Phi)} \right)
		= - 2 \partial_{\alpha} \left( \frac{\partial \mathcal{L}}{\partial \upsigma} 
			(m^{-1})^{\alpha \beta} \partial_{\beta} \Phi \right)
		& = 0.
\end{align}
As we mentioned in Subsect.~\ref{SS:CHRISTODOULOURESULTS},
Christodoulou studied perturbations of solutions of the form
$\Phi = kt,$ where $k$ is a non-zero constant. These are the solutions that correspond to the 
nontrivial constant states in relativistic fluid mechanics in Minkowski spacetime.
When expanded relative to rectangular coordinates, 
\eqref{E:DCFLUIDEL} becomes\footnote{We follow the conventions of \cite{dC2007}
and denote the dynamic metric by $h=h(\partial \Phi)$ in this section.}
\begin{align} \label{E:CHRISTODOULOUWAVEEQNEXPANDED}
	(h^{-1})^{\alpha \beta} \partial_{\alpha} \partial_{\beta} \Phi & = 0,
\end{align}
where the \emph{reciprocal acoustical metric} $h^{-1}$ is defined by
\begin{align}
	(h^{-1})^{\alpha \beta}
	= (h^{-1})^{\alpha \beta}(\partial \Phi)
	& := (m^{-1})^{\alpha \beta}
		- F
			(m^{-1})^{\alpha \kappa} (m^{-1})^{\beta \lambda} 
			\partial_{\alpha} \Phi \partial_{\beta} \Phi,
				\\
	F= F(\upsigma) 
	& := \frac{2}{G} \frac{d G}{d \upsigma},
		\\
	G= G(\upsigma) 
	& := 2 \frac{d \mathcal{L}}{d \upsigma}.
\end{align}

The characteristic speed of the background solution $\Phi = kt$ 
is not $1$ as in our work and that of Alinhac, but rather
\begin{align}
	\upeta_0 = \upeta(\upsigma = k^2),
\end{align}
where $\upeta > 0$ is the function defined by
\begin{align}
	\upeta^2 & = \upeta^2 (\upsigma) = 1 - \upsigma H,
		\\
	H & = H(\upsigma) := \frac{F}{1 + \upsigma F}.
		\label{E:CHRISTODOULOUSH}
\end{align}
More precisely, $\upeta$ is the \emph{speed of sound}, and
by virtue of \eqref{E:FLUIDSINTERPRETATIONPOSITIVITY}, 
it is straightforward to show that
$0 < \upeta < 1.$
Also, $(h^{-1})^{00}$ is not assumed to be equal to $-1$ as in our work, but rather there is a 
lapse function $\upalpha$ defined by
\begin{align}
	\upalpha^{-2}
	= \upalpha^{-2}(\partial \Phi)
	& := - (h^{-1})^{00}(\partial \Phi).
\end{align}

The proper analog of our background inverse Minkowski metric is in fact
the flat inverse metric
\begin{align}
	(h^{-1})^{\alpha \beta}(\partial_t \Phi = k, \partial_1 \Phi= \partial_2 \Phi= \partial_3 \Phi = 0),
\end{align}
which is equivalent to 
\begin{align}
	h(\partial_t \Phi = k, \partial_1 \Phi= \partial_2 \Phi= \partial_3 \Phi = 0)
	= - \upeta_0^2 dt^2 + \sum_{a=1}^3 (dx^a)^2.
\end{align}
The eikonal function corresponding to the background solution is
\begin{align}
	u_{(Flat)} & = 1 - r + \upeta_0 t,
\end{align}
where $r$ is the standard Euclidean radial coordinate.
The inverse foliation density corresponding to the background solution is
\begin{align}
	\upmu_{(Flat)} & = \upeta_0.
\end{align}
The outgoing and ingoing null vectorfields corresponding to the background solution are
\begin{align}
	\Lunit_{(Flat)} 
	& = \partial_t + \upeta_0 \partial_r,
	&&
	\uLgood_{(Flat)} 
	= \upeta_0^{-1} \partial_t - \partial_r.
\end{align}
The analog of the future null condition failure factor \eqref{E:OTHERFAILUREFACTOR} is
\begin{align} \label{E:CHRISTODOULOUNULLCONDITIONFAILUREFACTOR}
	\frac{d H}{d \upsigma}(\upsigma = k^2).
\end{align}
Note that unlike the general case of \eqref{E:OTHERFAILUREFACTOR},
the quantity in
\eqref{E:CHRISTODOULOUNULLCONDITIONFAILUREFACTOR} is a constant.
It was shown in \cite{dC2007} that 
$\frac{d H}{d \upsigma}(\upsigma = k^2)$ vanishes when $k \neq 0$ 
if and only if, up to
trivial normalization constants,
\begin{align} \label{E:EXCEPTIONALLAGRANGIAN}
	\mathcal{L}(\upsigma) = 1 - \sqrt{1 - \upsigma}.
\end{align}
The Lagrangian \eqref{E:EXCEPTIONALLAGRANGIAN} is therefore exceptional
in the sense that the quadratic nonlinearities that arise in 
expanding its wave equation \eqref{E:DCFLUIDEL} 
around the background $\Phi = kt$ verify the null condition. 

\bibliographystyle{amsalpha}
\bibliography{JBib}

\end{document}